\documentclass[a4paper, 10pt]{article}
\usepackage{amsmath}
\usepackage{amssymb,esint}
\usepackage{amscd}
\usepackage{xspace}
\usepackage{fancyhdr}
\usepackage{xcolor}
\newtheorem{theorem}{Theorem}[section]

\newtheorem{corollary}[theorem]{Corollary}

\newtheorem{definition}[theorem]{Definition}

\newtheorem{lemma}[theorem]{Lemma}

\newtheorem{proposition}[theorem]{Proposition}
\newtheorem{remark}[theorem]{Remark}

\newenvironment{proof}[1][Proof]{\textbf{#1.} }{\hfill\rule{0.5em}{0.5em}}
{\catcode`\@=11\global\let\AddToReset=\@addtoreset
\AddToReset{equation}{section}

\AddToReset{theorem}{section}

\begin{document}
\title{POTENTIAL ESTIMATES AND QUASILINEAR PARABOLIC EQUATIONS WITH MEASURE DATA}
\author{
 {\bf Quoc-Hung Nguyen\thanks{ New address: ShanghaiTech University, 393 Middle Huaxia Road, Pudong,
 		Shanghai, 201210, China. E-mail address: qhnguyen@shanghaitech.edu.cn}}\\[0.5mm]
{\small Laboratoire de Math\'ematiques et Physique Th\'eorique, }\\
{\small  Universit\'e Fran\c{c}ois Rabelais,  Tours,  FRANCE}}
\date{April 14, 2021}  
\maketitle
\begin{abstract}
In this paper, we study the existence and regularity of the quasilinear parabolic equations:
$$u_t-\operatorname{div}(A(x,t,\nabla u))=B(u,\nabla u)+\mu,$$
in either $\mathbb{R}^{N+1}$ or  $\mathbb{R}^N\times(0,\infty)$ or on  a bounded domain $\Omega\times (0,T)\subset\mathbb{R}^{N+1}$ where $N\geq 2$. In this paper, we shall assume that the nonlinearity $A$ fulfills standard growth conditions,  the function $B$  is a continuous and $\mu$ is a radon measure.
Our first task is to  establish the existence results with $B(u,\nabla u)=\pm|u|^{q-1}u$, for $q>1$. We next obtain global weighted-Lorentz, Lorentz-Morrey and Capacitary estimates on gradient of solutions with $B\equiv 0$, under minimal conditions on the boundary of domain and on nonlinearity $A$. Finally,
due to these estimates,
we solve the existence problems with $B(u,\nabla u)=|\nabla u|^q$
for $q>1$.
\end{abstract}
\tableofcontents 
\section{Introduction}
In this article, we study a class of quasilinear parabolic equations: 
\begin{align}\label{5hh060520141}
u_t-\operatorname{div}(A(x,t,\nabla u))=B(x,t,u,\nabla u)+\mu
\end{align} 
in $\mathbb{R}^{N+1}$ or $\mathbb{R}^N\times(0,\infty)$ or on a bounded domain $\Omega_T:=\Omega\times(0,T)\subset\mathbb{R}^{N+1}$
 where $N\geq2$. In this paper,  $\mu$  is a Radon measure,  $A:\mathbb{R}^N\times\mathbb{R}\times\mathbb{R}^N\to \mathbb{R}^N$ is a Carath\'eodory function which satisfies 
\begin{align}
\label{5hhconda}&|A(x,t,\zeta)|\le \Lambda_1 |\zeta|~\text{ and }~\\&
\label{5hhcondb}
 \left\langle {A(x,t,\zeta)-A(x,t,\lambda),\zeta-\lambda}\right\rangle\geq  \Lambda_2 |\zeta-\lambda|^2,
\end{align}   for every $(\lambda,\zeta)\in \mathbb{R}^N\times \mathbb{R}^N$ and a.e. $(x,t)\in \mathbb{R}^N\times \mathbb{R}$, for some   $\Lambda_1, \Lambda_2>0$; and  $B:\mathbb{R}^{N+1}\times\mathbb{R}\times\mathbb{R}^N\to\mathbb{R}$ is also a Carath\'eodory function. 

The regularity and singularity theory for the parabolic quasilinear equation \eqref{5hh060520141} were studied and developed intensely over the past 50 years in \cite{55Mo,55Lad,55Fu,55Li1,55Li2,55Di,55Li3,55Nau,55Zi1,55SaGa,55QuSo}. Moreover, we also refer to  \cite{55BOR}-\cite{55BW4} for $L^p-$gradient estimates theory in non-smooth domains and \cite{55Hu1} for Wiener criteria to existence of large solutions of nonlinear parabolic equations with absorption  in a non-cylindrical domain.

First, we are specially interested in the existence of solutions to  quasilinear parabolic equations with absorption, source terms and measure data:
\begin{align}\label{5hh300420141}
& u_t-\operatorname{div}(A(x,t,\nabla u))+|u|^{q-1}u=\mu, \quad (\text{absorption term})\\&
u_t-\operatorname{div}(A(x,t,\nabla u))=|u|^{q-1}u+\mu,\quad (\text{source term})\label{5hh300420142}
\end{align}
in $\mathbb{R}^{N+1}$ and 
\begin{align}\label{5hh300420143}
& u_t-\operatorname{div}(A(x,t,\nabla u))+|u|^{q-1}u=\mu,~~u(0)=\sigma, \quad (\text{absorption term})\\&
u_t-\operatorname{div}(A(x,t,\nabla u))=|u|^{q-1}u+\mu, ~~u(0)=\sigma,\quad (\text{source term})\label{5hh300420144}
\end{align}
in $\mathbb{R}^{N}\times (0,\infty)$ or a bounded domain $\Omega_T\subset \mathbb{R}^{N+1}$, where $q>1$ and $\mu,\sigma$ are Radon measures. 

The linear case $A(x,t,\nabla u )=\nabla u$ was studied in detail by Fujita, Brezis and Friedman, Baras and Pierre. 

 For the absorption case, in \cite{55BrFr}, they showed that if $\mu=0$  and $\sigma$ is a Dirac
    mass in $\Omega$, the problem \eqref{5hh300420143} in $\Omega_T$ (with Dirichlet boundary condition)  admits a (unique) solution if and only if $q<(N+2)/N$. Then, optimal results had been obtained in \cite{55BaPi1}. They proved that  for any $\mu\in\mathfrak{M}%
_{b}(\Omega_T)$ and $\sigma\in\mathfrak{M}_{b}(\Omega)$ there exists a unique solution of \eqref{5hh300420143} in $\Omega_T$ if and only if  $\mu\ll\text{Cap}_{2,1,q'}$ and $\sigma\ll\text{Cap}_{\mathbf{G}_{2/q},q'}$ i.e $\mu,\sigma$ are absolutely continuous with respect to the capacity $\text{Cap}_{2,1,q'}, \text{Cap}_{\mathbf{G}_{2/q},q'}$ (in $\Omega_T,\Omega$) respectively. Here  $q'$ is the conjugate exponent of $q$ and these two capacities will be defined in section 2.  

For the source case, in \cite{55BaPi2}, they showed that for any $\mu\in\mathfrak{M}%
_{b}^+(\Omega_T)$ and $\sigma\in\mathfrak{M}_{b}^+(\Omega)$, the problem \eqref{5hh300420144} in bounded domain $\Omega_T$  has a nonnegative solution if 
$$\mu(E)\leq C \text{Cap}_{2,1,q'}(E)~~\text{and}~~\sigma(O)\leq C \text{Cap}_{\mathbf{G}_{\frac{2}{q}},q'}(O)$$
      hold for every compact sets $E\subset \mathbb{R}^{N+1}$, $O\subset \mathbb{R}^{N}$ here $C=C(N,\text{diam}(\Omega),T)$ is small enough. Conversely, the existence holds then for compact subset $K\subset\subset\Omega$, one find $C_K>0$ such that 
      \begin{equation*}
             \mu(E\cap (K\times[0,T]))\leq C_K \text{Cap}_{2,1,q'}(E)~~\text{and}~~\sigma(O\cap K)\leq C_K \text{Cap}_{\mathbf{G}_{\frac{2}{q}},q'}(O)
             \end{equation*}hold for every compact sets $E\subset \mathbb{R}^{N+1}$, $O\subset \mathbb{R}^{N}$. 
       In unbounded domain $\mathbb{R}^N\times(0,\infty)$, in \cite{55Fu} they asserted that an inequality 
      \begin{align}\label{5hh300420147}
      u_t-\Delta u\geq u^q, u\geq 0~\text{ in } \mathbb{R}^N\times(0,\infty), 
      \end{align} 
      \begin{description}
      \item[i.]if $q<(N+2)/N$ then the only nonnegative global (in time) solution of above inequality is $u\equiv 0$,
      \item[ii.] if $q>(N+2)/N$ then there exists global positive solution of  above inequality. 
      \end{description}
      More generally, in \cite{55BaPi2},
      for $\mu\in\mathfrak{M}^+(\mathbb{R}^N\times(0,\infty))$ and $ \sigma\in\mathfrak{M}^+(\mathbb{R}^N)$, the equation \eqref{5hh300420144} has a nonnegative solution in $\mathbb{R}^N\times(0,\infty)$  (with $A(x,t,\nabla u )=\nabla u$) if and only if 
      \begin{equation}\label{5hh300420146}
             \mu(E)\leq C \text{Cap}_{\mathcal{H}_2,q'}(E)~~\text{and}~~\sigma(O)\leq C \text{Cap}_{\mathbf{I}_{\frac{2}{q}},q'}(O)
             \end{equation} 
            hold for every compact set $E\subset \mathbb{R}^{N+1}$, $O\subset \mathbb{R}^{N}$.  Here $C=C(N,q)$ is small enough, two capacities $\text{Cap}_{\mathcal{H}_2,q'}, \text{Cap}_{\mathbf{I}_{\frac{2}{q}},q'}$ will be defined in section 2. Note that a necessary and sufficient condition  for \eqref{5hh300420146} to  hold with  $\mu\in\mathfrak{M}^+(\mathbb{R}^N\times(0,\infty))\backslash\{0\}$ or $\sigma\in\mathfrak{M}^+(\mathbb{R}^N)\backslash\{0\}$ is $q\geq (N+2)/N$. In particular, \eqref{5hh300420147} has a (global) positive solution if and only if $q\geq (N+2)/N$.  It is known that conditions for data $\mu,\sigma$ in problems with absorption are  softer than those in problems with source. Recently, the exponential case, i.e $|u|^{q-1}u$ is replaced by $P(u)\sim\exp(a|u|^q)$, for $a>0$ and $q\geq 1$  was considered in \cite{55Ta}.

            We consider \eqref{5hh300420143} and \eqref{5hh300420144} in $\Omega_T$ with Dirichlet boundary conditions when $\operatorname{div}(A(x,t,\nabla u))$ is replaced by $\Delta_p u:= \operatorname{div}(|\nabla u|^{p-2}\nabla u )$ for $p\in(2-1/N,N)$. In \cite{55PePoPor}, they  showed that for any $q>p-1$, \eqref{5hh300420143} admits a (unique renormalized) solution provided 
            $\sigma\in L^1(\Omega)$ and $\mu\in\mathfrak{M}_b(\Omega_T)$ is diffuse measure i.e absolutely  continuous with respect to $C_p-$parabolic capacity in $\Omega_T$ defined on a compact set $K\subset \Omega_T$:  
 $$
 C_{p}(K,\Omega_T)=\inf\left\{||\varphi||_X:\varphi\geq \chi_K,\varphi\in C^\infty_c(\Omega_T)\right\},
$$
 where $X=\{\varphi:\varphi\in L^{p}(0,T;W_0^{1,p}(\Omega)),\varphi_t\in L^{p'}(0,T;W^{-1,p'}(\Omega))\}$ endowed with norm $||\varphi||_{X}=||\varphi||_{L^{p}(0,T;W_0^{1,p}(\Omega))}+||\varphi_t||_{L^{p'}(0,T;W^{-1,p'}(\Omega))}$ and $\chi_K$ is the characteristic function of $K$. An improving result was presented in \cite{55VHb} for measures that have good behavior in time, it is based on results of \cite{55VHV} relative to the elliptic case. That is,  \eqref{5hh300420143} has a (renormalized) solution for $q>p-1$ if $\sigma\in L^1(\Omega)$ and $|\mu|\leq f+\omega\otimes F$, where $f\in L^1_+(\Omega_T),F\in L^1_+((0,T))$ and $\omega\in\mathfrak{M}_b^+(\Omega)$ is absolutely continuous with respect to  $\text{Cap}_{\mathbf{G}_p,\frac{q}{q-p+1}}$ in $\Omega$. Also, \eqref{5hh300420144} has a (renormalized) nonnegative solution if $\sigma\in L^\infty_+(\Omega)$, $0\leq \mu\leq \omega\otimes\chi_{(0,T)}$ with $\omega\in\mathfrak{M}_b^+(\Omega)$ and 
 \begin{align*}
 \omega(E)\leq C_1\text{Cap}_{\mathbf{G}_p,\frac{q}{q-p+1}}(E)~\forall \text{ compact } E~ \subset\mathbb{R}^N,~~~||\sigma||_{L^\infty(\Omega)}\leq C_2 
 \end{align*}
for some $C_1,C_2$ small enough. 
 Another improving results are also stated in \cite{55Hu2}, especially if $q> p-1$, $p>2$, $\mu\in\mathfrak{M}_b(\Omega_T)$ and $\sigma\in\mathfrak{M}_b(\Omega)$ are absolutely continuous with respect to $\text{Cap}_{2,1,q'}$ in $\Omega_T$ and $\text{Cap}_{\mathbf{G}_{\frac{2}{q}},q'}$ in $\Omega$ then \eqref{5hh300420143} has a distributional solution. 
 
 In \cite{55Hu2},  we also obtain the existence of solutions for the porous medium equation with absorption and data measure. More precisely, for $q>m>\frac{N-2}{N},$ a sufficient condition to have existence of solution to the problem
 \begin{align*}
 u_t-\Delta (|u|^{m-1}u)+|u|^{q-1}u=\mu ~\text{ in }~\Omega_T,~~u=0~\text{on }~\partial\Omega\times (0,T),~~\text{ and }~ u(0)=\sigma ~\text{ in }~\Omega,
 \end{align*} 
  is $\mu\ll\text{Cap}_{2,1,q'}$,  $\sigma\ll\text{Cap}_{\mathbf{G}_{\frac{2}{q}},q'}$ if $m\geq 1$ and $\mu\ll\text{Cap}_{\mathcal{G}_2,\frac{2q}{2(q-1)+N(1-m)}}$,  $\sigma\ll\text{Cap}_{\mathbf{G}_{\frac{2-N(1-m)}{q}},\frac{2q}{2(q-1)+N(1-m)}}$ if $\frac{N-2}{N}<m\leq 1$. A necessary condition is $\mu\ll\text{Cap}_{2,1,\frac{q}{q-\max \{m,1\}}}$ and $\sigma\ll\text{Cap}_{\mathbf{G}_{\frac{2\max \{m,1\}}{q}},\frac{q}{q-\max \{m,1\}}}$. Moreover, if $\mu=\mu_1\otimes\chi_{[0,T]}$ with $\mu_1\in\mathfrak{M}_b(\Omega)$ and $\sigma\equiv 0$ then a condition $\mu_1\ll \text{Cap}_{\mathbf{G}_{2},\frac{q}{q-m}}$ is not only sufficient  but also 
 also necessary for existence of solutions to the above problem. 
 
We would like to make a brief survey on quasilinear elliptic equations with absorption, source terms and data measure. Namely, we shall study the following equations:
 \begin{align}\label{5hh300420148}
 &-\Delta_pu+|u|^{q-1}u=\omega, \\&
 -\Delta_p u=|u|^{q-1}u+\omega, u\geq 0,\label{5hh300420149}
 \end{align}
 in $\Omega$ with Dirichlet boundary conditions where $1<p<N$, $q>p-1$. In \cite{55VHV}, we proved that the existence solution of equation  \eqref{5hh300420148} holds if $\omega\in\mathfrak{M}_b(\Omega)$ is absolutely continuous with respect to  $\text{Cap}_{\mathbf{G}_p,\frac{q}{q-p+1}}$. Moreover, a necessary condition for existence  was also showed in \cite{55Bi2,55Bi3}. For the problem with source term, it was solved in \cite{55AdPi,55BaPi2} for $p=2$ (also see \cite{VHVjfa2015}) and \cite{55PhVe} for any $1<p\leq N$ (also see \cite{55PhVe2}). More precisely, if $\omega\in\mathfrak{M}_b^+(\Omega)$ has compact support in $\Omega$, then a sufficient and necessary condition for the existence of solutions to the problem \eqref{5hh300420149} is
  $$\omega(E)\leq C\text{Cap}_{\mathbf{G}_p,\frac{q}{q-p+1}}(E)\qquad\text{for all compact sets }E\subset\Omega,$$
  where $C$ is a constant depending  only on $N,p,q$ and $d(\text{supp}(\omega),\partial\Omega)$.
Their construction is based upon sharp estimates on the solutions of the problem
\begin{align*}
-\Delta_p u=\omega~~\text{ in }\Omega,~~
u=0~\text{ on } \partial\Omega,\end{align*}
for nonnegative Radon measures $\omega$ in $\Omega$ together with  a deep analysis of the Wolff potential. \\
Corresponding results in the case where $u^q$ term is changed by $P(u)\approx\exp (au^\lambda)$ for $a>0,\lambda>0$, were given in \cite{55VHV,55HV}.\medskip \\
There are many works for the Riccati equation 
\begin{equation}\label{Z12}
	-\Delta_p u=|\nabla u|^{q}+\omega
\end{equation}
in $\Omega$ with Dirichlet boundary conditions where $1<p\leq N$ and $q>p-1$. This problem was firstly studied in \cite{HaMa} in the case $p=2$. They proved that the problem has a solution if and only if 
$$\omega(E)\leq C\text{Cap}_{\mathbf{G}_1,\frac{q}{q-p+1}}(E)\qquad\text{for all compact sets }E,$$
for some $C>0$. Then, in  \cite{55Ph0, 55Ph3,55Ph2} they extended this result to $2-\frac{1}{N}<p\leq N$, see also \cite{VNV2020CVPDE}. Recently, in  \cite{H-P1,H-P2,H-P3,H-P4}, the authors considered this  problem in the singular case $1<p\leq 2-\frac{1}{N}$. \\
 In \cite{55DuzaMing}, Duzaar and Mingione gave a local pointwise estimate from above of the solutions  to the equation
 \begin{align}\label{5hh3004201412}
 u_t-\operatorname{div}(A(x,t,\nabla u))=\mu,
 \end{align} 
in $\Omega_T$ involving the Wolff parabolic potential $\mathbb{I}_2[|\mu|]$ defined by 
\begin{align*}
\mathbb{I}_2[|\mu|](x,t)=\int_{0}^{\infty}\frac{|\mu|(\tilde{Q}_\rho(x,t))}{\rho^N}\frac{d\rho}{\rho}~~\text{ for all }~~(x,t)\in\mathbb{R}^{N+1},
\end{align*}
here $\tilde{Q}_\rho(x,t):=  B_\rho(x)\times (t-\rho^2/2,t+\rho^2/2)$.  Specifically if $u\in L^2(0,T;H^1(\Omega))\cap C(\Omega_T)$ is a weak solution to the above equation with data $\mu \in L^2(\Omega_T)$, then 
\begin{align}
|u(x,t)|\leq C\fint_{\tilde{Q}_R(x,t)}|u|+C \int_{0}^{2R}\frac{|\mu|(\tilde{Q}_\rho(x,t))}{\rho^N}\frac{d\rho}{\rho},
\end{align}
for any $Q_{2R}(x,t):=B_{2R}(x)\times (t-(2R)^2,t)\subset\Omega_T$, where a constant $C$ only depends on $N$ and the structure of operator $A$.
In  this paper we show that if $u\geq 0,\mu\geq 0$ we also have  local pointwise estimate from below: 
\begin{equation}
u(y,s)\gtrsim\sum_{k=0}^{\infty}\frac{\mu(Q_{r_k/8}(y,s-\frac{35}{128}r_k^2))}{r_k^N},
\end{equation}
for any $Q_r(y,s)\subset \Omega_T$,  where $r_k=4^{-k}r$ (see section 5).

From the preceding two inequalities, we obtain  global pointwise estimates of solution to \eqref{5hh3004201412}. For example, if $\mu\in\mathfrak{M}(\mathbb{R}^{N+1})$ with $\mathbb{I}_2[|\mu|](x_0,t_0)<\infty$ for some $(x_0,t_0)\in\mathbb{R}^{N+1}$ then there exists a distributional solution to \eqref{5hh3004201412} in $\mathbb{R}^{N+1}$  such that 
\begin{align}\label{5hh3004201413}
-K\mathbb{I}_2[\mu^-](x,t)\leq u(x,t)\leq K \mathbb{I}_2[\mu^+](x,t) ~~\text{ for a.e } (x,t)\in\mathbb{R}^{N+1},
\end{align}
and we emphasize that if $u\geq 0, \mu\geq 0$ then 
\begin{align*}
          u(x,t)\geq K^{-1}\sum_{k=-\infty}^{\infty}\frac{\mu(Q_{2^{-2k-3}}(x,t-35\times 2^{-4k-7}))}{2^{-2Nk}} ~~\text{ for a.e } (x,t)\in\mathbb{R}^{N+1},
            \end{align*}
  and  for $q>1$,
  \begin{align*}
  ||u||_{L^q(\mathbb{R}^{N+1})}\sim ||\mathbb{I}_2[\mu]||_{L^q(\mathbb{R}^{N+1})}.
  \end{align*}
  Where the constant $K$  depends only  on $N$ and on the structure of the operator $A$.  \\

Our first aim is to verify that{\color{red},}
\begin{description}
\item[i. ] problems \eqref{5hh300420141} and \eqref{5hh300420143} have solutions if  $\mu,\sigma$ are absolutely continuous with respect to the capacity $\text{Cap}_{2,1,q'}, \text{Cap}_{\mathbf{G}_{\frac{2}{q}},q'}$  respectively,
\item[ii.] problems \eqref{5hh300420142} in $\mathbb{R}^{N+1}$ and \eqref{5hh300420144} in $\mathbb{R}^N\times (0,\infty)$ with signed measure data $\mu,\sigma$ admit a solution if 
 \begin{equation}\label{5hh3004201410}
             |\mu|(E)\leq C \text{Cap}_{\mathcal{H}_2,q'}(E)~~\text{and}~~|\sigma|(O)\leq C\text{Cap}_{\mathbf{I}_{\frac{2}{q}},q'}(O)
             \end{equation} 
            holds for every compact sets $E\subset \mathbb{R}^{N+1}$, $O\subset \mathbb{R}^{N}$ for some $C>0$. Also, 
              the equation \eqref{5hh300420144} in a bounded domain $\Omega_T$ has a solution if \eqref{5hh3004201410} holds where capacities $\text{Cap}_{2,1,q'},\text{Cap}_{\mathbf{G}_{\frac{2}{q}},q'}$ are exploited instead of $\text{Cap}_{\mathcal{H}_2,q'},\text{Cap}_{\mathbf{I}_{\frac{2}{q}},q'}$.              
\end{description}
It is worth mentioning that the solutions  of \eqref{5hh300420142} in $\mathbb{R}^{N+1}$ and \eqref{5hh300420144} in $\mathbb{R}^N\times (0,\infty)$ that we have obtained obey 
\begin{align*}
\int_E|u|^qdxdt \lesssim\text{Cap}_{\mathcal{H}_2,q'}(E)~~\text{ for all compact } E\subset \mathbb{R}^{N+1},
\end{align*}
 and we also have an analogous estimate for a solution of  \eqref{5hh300420144} in $\Omega_T$;
 \begin{align*}
 \int_E|u|^qdxdt \lesssim\text{Cap}_{2,1,q'}(E)~~\text{ for all compact } E\subset \mathbb{R}^{N+1}.
 \end{align*}

  In the case $\mu\equiv 0$, solutions of \eqref{5hh300420144} in $\mathbb{R}^N\times(0,\infty)$ and $\Omega_T$ verify the decay estimate
   \begin{align*}
   - Ct^{-\frac{1}{q-1}} \leq \inf_{x}u(x,t)\leq \sup_{x}u(x,t)\leq Ct^{-\frac{1}{q-1}} \text{ for any }~t>0. 
   \end{align*}
The strategy used to prove the above results is mainly based on some techniques from the two articles \cite{55VHV,55PhVe}  the global pointwise estimate \eqref{5hh3004201413} and some delicate estimates on Wolff parabolic potential and  the stability theorem see \cite{55VH} (and also Proposition \ref{5hh1203201412} of this paper). They will be proved in section 6. 

Then, we shall study the global regularity of solutions to quasilinear parabolic equations of the type:
\begin{align}\label{5hh3004201414}
 u_t- \operatorname{div}\left( {A(x,t,\nabla u)} \right)=\mu ~\text{ in }~\Omega_T,~u=0~\text{on}~\partial\Omega\times (0,T)~\text{ and }~ u(0)=\sigma ~\text{ in }~\Omega,
 \end{align}
 where the domain $\Omega_T$ and the nonlinearity $A$ are as mentioned at the beginning.
 
 Our aim is to find the minimal conditions on the boundary of $\Omega$ and on the nonlinearity $A$  so that the following statement holds
 \begin{align*}|||\nabla u|||_{\mathcal{K}}\lesssim || \mathbb{M}_1[\omega]||_{\mathcal{K}}.
 \end{align*}
  Here $\omega=|\mu|+|\sigma|\otimes\delta_{\{t=0\}}$ and  $\mathbb{M}_1$  is the first order fractional Maximal parabolic potential defined by 
      $$
      \mathbb{M}_{1}[\omega](x,t)=\sup_{\rho>0}\frac{\omega(\tilde{Q}_\rho(x,t))}{\rho^{N+1}}~~\forall ~(x,t)\in\mathbb{R}^{N+1},
      $$ the constant $C$ does not depend on $u$ and $\mu\in\mathfrak{M}_b(\Omega_T),\sigma\in\mathfrak{M}_b(\Omega)$ and $\mathcal{K}$ is a function space.
      The same type of question was answered in the elliptic framework  (see N. C. Phuc 
\cite{55Ph0,55Ph2,55Ph3} for more details). 

 First, we take $\mathcal{K}=L^{p,s}(\Omega_T)$ for  $1\leq p\leq \theta $ and $0<s\leq\infty$ under a capacity density condition on the domain $\Omega$ where $L^{p,s}(\Omega_T)$ is the Lorentz space.  The constant $\theta>2$ depends on the structure of this condition and on the nonlinearity $A$. It follows  the recent result in \cite{55BaCaPa}, see remark \ref{5hh070520145}. The capacity density condition is that, the complement of $\Omega$ satisfies the  {\it uniformly $2-$thick} condition, see section 2. We remark that under this condition,  the  Sobolev embedding $H_0^1(\Omega)\subset L^{\frac{2N}{N-2}}(\Omega)$ for $N>2$ is valid and it is fulfilled by any domain with Lipschitz boundary, or even of corkscrew type. This condition was used in the two papers \cite{55Ph0,55Ph3}. Also, it is essentially sharp for higher integrability results, presented in \cite[Remark 3.3]{55KiKo}. Furthermore, we also claim that if 
 $\frac{\gamma}{\gamma-1}<p<\theta$,  $2\leq\gamma<N+2$, $0<s \leq \infty$ and  $\sigma\equiv 0$ then
  $$                           |||\nabla u|||_{L_{*}^{p,s;(\gamma-1)p}(\Omega_T)}\lesssim||\mu||_{L_{*}^{\frac{(\gamma-1)p}{\gamma},\frac{(\gamma-1)s}{\gamma};(\gamma-1)p}(\Omega_T)},   $$
 where $L_{*}^{p,s;(\gamma-1)p}(\Omega_T), L_{*}^{\frac{(\gamma-1)p}{\gamma},\frac{(\gamma-1)s}{\gamma};(\gamma-1)p}(\Omega_T) $ are the Lorentz-Morrey spaces involving "calorie" introduced in section 2.  We would like to refer to  \cite{55Mi0} as the first paper where Lorentz-Morrey
  estimates for solutions of quasilinear elliptic equations via fractional operators have been obtained.  
  
  Next, in order to obtain shaper results,   we take $\mathcal{K}=L^{q,s}(\Omega_T,dw)$, the weighted Lorentz spaces with weight in the Muckenhoupht class $ A_\infty$ for $q\geq 1$,  $0<s\leq \infty$, we require some stricter conditions on the domain $\Omega$ and nonlinearity $A$. The condition on $\Omega$ is flat enough in the sense of Reifenberg. It essentially says that at the boundary point and at every scale, the boundary of the domain is between two hyperplanes at both sides (inside and outside) of the domain by a distance which depends on the scale. Conditions on $A$ are that  BMO type of $A$ with respect to the $x-$variable is small enough and the derivative of $A(x,t,\zeta)$ with respect to $\zeta$ is uniformly bounded.
   By choosing an appropriate weight we can establish the following important estimates: \\\\
  \textbf{a.} The Lorentz-Morrey estimates involving "calorie" for $0<\kappa\leq N+2$ is obtained
$$
                                                                                                                                           |||\nabla u|||_{L_{*}^{q,s;\kappa}(\Omega_T)}\lesssim ||\mathbb{M}_1[|\omega|]||_{L_{*}^{q,s;\kappa}(\Omega_T)}.
$$
                                                                                                                                                                                                                                                                                                     \textbf{b.} Another Lorentz-Morrey estimates is also obtained for $0<\vartheta\leq N$
$$
                                                                                                                                                                                                                ||\mathbb{M}(|\nabla u|)||_{L_{**}^{q,s;\vartheta}(\Omega_T)}\lesssim ||\mathbb{M}_1[|\omega|]||_{L_{**}^{q,s;\vartheta}(\Omega_T)},
                                                                               $$                                                           where $L_{**}^{q,s;\vartheta}(\Omega_T)$ is introduced in section 2. This estimate implies    global H\"older-estimate in space variable and $L^q-$estimate  in time, that is for all ball $B_\rho\subset\mathbb{R}^N$
                                                                                                                                                                                                                    \begin{equation*}                               
                                                                                                                                                                                                                                                                                                                                    \left(\int_0^T|\text{osc}_{B_\rho\cap\overline{\Omega}}u(t)|^qdt\right)^{\frac{1}{q}} \lesssim \rho^{1-\frac{\vartheta}{q}} ||\mathbb{M}_1[|\omega|]||_{L_{**}^{q;\vartheta}(\Omega_T)}~\text{provided }~0<\vartheta<\min\{q,N\}.
                                                                                                                                                                                                                                                                                                                                                                                                                                                \end{equation*}
 In particular, there holds                                                                                                                                                                                                                                                                                                                                                                                                                                               
 \begin{align*}
                 \left(\int_0^T|\text{osc}_{B_\rho\cap\overline{\Omega}}u(t)|^qdt\right)^{\frac{1}{q}}\lesssim \rho^{1-\frac{\vartheta}{q}}||\sigma||_{L^{\frac{\vartheta q}{\vartheta+2-q};\vartheta}(\Omega)}+\rho^{1-\frac{\vartheta}{q}}||\mu||_{L^{\frac{\vartheta qq_1}{(\vartheta+2+q)q_1-2q};\vartheta}(\Omega, L^{q_1}((0,T)))}
                  \end{align*}
        provided \begin{align*}
        & 1<q_1\leq q<2,\\& \max\left\{\frac{2-q}{q-1},\frac{1}{q-1}\left(2+q-\frac{2q}{q_1}\right)\right\}<\vartheta\leq N.
        \end{align*}
         Where $L^{\frac{\vartheta q}{\vartheta+2-q};\vartheta}(\Omega)$ is the standard Morrey space and 
                 \begin{align*}
                ||\mu||_{L^{q_2;\vartheta}(\Omega, L^{q_1}((0,T)))}=\sup_{\rho>0, x\in \Omega}\rho^{\frac{\vartheta-N}{q_2}}\left(\int_{B_\rho(x)\cap\Omega}\left(\int_{0}^T|\mu(y,t)|^{q_1}dt\right)^{\frac{q_2}{q_1}}dy\right)^{\frac{1}{q_2}},
                 \end{align*} 
with $q_2=\frac{\vartheta qq_1}{(\vartheta+2+q)q_1-2q}$. Besides, we also find 
  \begin{align*}
                   \left(\int_0^T|\text{osc}_{B_\rho\cap\overline{\Omega}}u(t)|^qdt\right)^{\frac{1}{q}}\lesssim\rho^{1-\frac{\vartheta}{q}}||\mu||_{L^{\frac{\vartheta qq_1}{(\vartheta+2+q)q_1-2q};\vartheta}(\Omega, L^{q_1}((0,T)))}
                    \end{align*}
          provided \begin{align*}
          &\sigma\equiv 0,~~ q\geq 2, 1<q_1\leq q,
         \\& \frac{1}{q-1}\left(2+q-\frac{2q}{q_1}\right)<\vartheta\leq N.
          \end{align*}      
      \textbf{c.} A global capacitary estimate is also  given 
      \begin{align*}
      \mathop {\sup }\limits_{\scriptstyle \text{compact } K\subset\mathbb{R}^{N+1} \hfill \atop 
        \scriptstyle \text{Cap}_{\mathcal{G}_1,q'}(K)>0 \hfill}\left(\frac{\int_{K}|\nabla u|^qdxdt}{\text{Cap}_{\mathcal{G}_1,q'}(K)}\right) \lesssim \mathop {\sup }\limits_{\scriptstyle \text{compact } K\subset\mathbb{R}^{N+1} \hfill \atop 
                \scriptstyle \text{Cap}_{\mathcal{G}_1,q'}(K)>0 \hfill}\left(\frac{|\omega|(K)}{\text{Cap}_{\mathcal{G}_1,q'}(K)}\right)^q.
      \end{align*}   
 To obtain this estimate we use deep techniques from nonlinear potential theory, see section 4 and Theorem \ref{5hh2410131'}.
                                           
                                           We  use some ideas (in the quasilinear elliptic framework) in articles of N.C. Phuc \cite{55Ph0, 55Ph3,55Ph2} to establish above estimates. Recently, in \cite{hungcvpde,THung2021} the author extended these results for \eqref{5hh3004201414} with distributional data.
                                           
We would like to emphasize that above estimates are also true for solutions to equation \eqref{5hh3004201414} in $\mathbb{R}^{N+1}$ with data $\mu$ (of course it is still true for \eqref{5hh3004201414} in $\mathbb{R}^{N}\times(0,\infty)$) with data $\mu$, provided  $\mathbb{I}_2[|\mu|](x_0,t_0)<\infty$ for some $(x_0,t_0)\in \mathbb{R}^{N+1}{\color{red}{,}}$ see Theorem  \ref{5hh040420149} and \ref{5hh040420145}. Moreover, a global pointwise estimate for the gradient of solutions is obtained when $A$ is independent of the space variable $x$, that is
\begin{align*}
|\nabla u(x,t)|\lesssim\mathbb{I}_1[|\mu|](x,t)~~\text{ a.e } (x,t)\in\mathbb{R}^{N+1},
\end{align*}
see Theorem \ref{5hh1203201417}.
                                            
Our final aim is to obtain existence results for the quasilinear Riccati type parabolic problems \eqref{5hh060520141} where $B(x,t,u,\nabla u)=|\nabla u|^q$ for $q>1$. The strategy we use in order to prove these existence results is a combination of the Schauder Fixed Point Theorem and all above estimates and  the stability Theorem see \cite{55VH}, Proposition \ref{5hh1203201412} in section 3. They will be carried out in section 9. 
   The method used in this paper allows to treat more general equations \eqref{5hh060520141}, namely
  $$
  | B(x,t,u,\nabla u)|\lesssim |u|^{q_1}+|\nabla u|^{q_2}, ~q_1,q_2>1.
   $$
  \\\\\\        
\textbf{Acknowledgements:}\\ The author wishes to express his deep gratitude to his advisors Professor Laurent V\'eron and Professor Marie-Fran\c{c}oise Bidaut-V\'eron for encouraging, taking care and giving many useful comments during the preparation of the paper. Besides the author would like to thank Omar Lazar and Nguyen Phuoc Tai for many interesting comments.  \\                         
 \section{Main Results}                                
Throughout the paper, we assume that $\Omega$ is a bounded open subset of $\mathbb{R}^N$, $N\geq2$ and $T>0$. Besides, we always denote  $\Omega_T=\Omega\times (0,T)$, $T_0=\text{diam}(\Omega)+T^{1/2}$ and $Q_\rho(x,t)=B_\rho(x)\times (t-\rho^2,t)$ $\tilde{Q}_\rho(x,t)=B_\rho(x)\times (t-\rho^2/2,t+\rho^2/2)$ for $(x,t)\in\mathbb{R}^{N+1}$ and $\rho>0$. We always assume that $A:\mathbb{R}^N\times\mathbb{R}\times \mathbb{R}^N\to \mathbb{R}^N$ is a Caratheodory vector valued function, i.e. $A$ is measurable in $(x,t)$ and continuous with respect to $\nabla u$ for each fixed $(x,t)$ and satisfies \eqref{5hhconda} and \eqref{5hhcondb}. This article is divided into three parts.  First part, we study the existence problems for the quasilinear parabolic equations with absorption and  source terms
\begin{equation}\label{5hh070120148}
                      \left\{
                                      \begin{array}
                                      [c]{l}%
                                      {u_{t}}-\operatorname{div}(A(x,t,\nabla u))+|u|^{q-1}u=\mu~\text{in }\Omega_T,\\ 
                                                            u=0~~~~~~~\text{on}~~
                                                                                             \partial\Omega \times (0,T),
                                                                                                \\
                                                                                                u(0) = \sigma~~~\text{in}~~ \Omega, \\ 
                                      \end{array}
                                      \right.  
                                      \end{equation}
                                        and                                  
\begin{equation}\label{5hh070120149}
                  \left\{
                  \begin{array}
                  [c]{l}%
                  {u_{t}}-\operatorname{div}(A(x,t,\nabla u))=|u|^{q-1}u+\mu~\text{in }\Omega_T,\\ 
                                      u=0~~~~~~~\text{on}~~
                                   \partial\Omega \times (0,T),
                                      \\
                                      u(0) = \sigma~~~\text{in}~~ \Omega,\\ 
                  \end{array}
                  \right.  
                  \end{equation}
where $q>1$,  and $\mu,\sigma$ are Radon measures.

                 In order to state our results, let us introduce some definitions and notations. If $D$ is either a bounded domain or whole $\mathbb{R}^{l}$ for $l\in\mathbb{N}$, we denote by $\mathfrak{M}(D)$ (resp. $\mathfrak{M}_b(D)$) the set of Radon measure (resp. bounded Radon measures) in $D$. Their positive cones are $\mathfrak{M}^+(D)$ and  $\mathfrak{M}_b^+(D)$ respectively. For $R\in (0,\infty]$, we define the 
                $R-$truncated Riesz parabolic potential $\mathbb{I}_{\alpha}$ and Fractional Maximal parabolic  potential $\mathbb{M}_{\alpha}$, $\alpha\in (0,N+2)$, on $\mathbb{R}^{N+1}$ of a measure $\mu\in\mathfrak{M}^+(\mathbb{R}^{N+1})$ by                 
                 \begin{equation}\label{5hh160420141}
                                                      \mathbb{I}_{\alpha}^R[\mu](x,t)=\int_{0}^{R}\frac{\mu(\tilde{Q}_\rho(x,t))}{\rho^{N+2-\alpha}}\frac{d\rho}{\rho}~\text{and}~\mathbb{M}_{\alpha}^R[\mu](x,t)=\sup_{0<\rho<R}\frac{\mu(\tilde{Q}_\rho(x,t))}{\rho^{N+2-\alpha}},
                                                     \end{equation}
                   for all $(x,t)$ in $\mathbb{R}^{N+1}$. If $R=\infty$, we drop it in expressions of \eqref{5hh160420141}.\\                                                       
                We denote by $\mathcal{H}_\alpha$ the Heat  kernel of order $\alpha\in (0,N+2)$:
                   \begin{equation*}
                   \mathcal{H}_\alpha(x,t)=C_\alpha
                   \frac{\chi_{(0,\infty)}(t)}{t^{(N+2-\alpha)/2}}\exp\left(-\frac{|x|^2}{4t}\right)~~\text{for}~~(x,t)~~\text{in}~~\mathbb{R}^{N+1},
                   \end{equation*}
                   and $\mathcal{G}_\alpha$
                   the parabolic Bessel kernel of order $\alpha>0$:
                      \begin{equation*}
                      \mathcal{G}_\alpha(x,t)=C_\alpha
                      \frac{\chi_{(0,\infty)}(t)}{t^{(N+2-\alpha)/2}}\exp\left(-t-\frac{|x|^2}{4t}\right)~~\text{for}~~(x,t)~~\text{in}~~\mathbb{R}^{N+1},
                      \end{equation*}
                     see \cite{55Bag}, where  $C_\alpha=\left((4\pi)^{N/2}\Gamma(\alpha/2)\right)^{-1}$.
                            It is known that
                            $
                           \mathcal{F}(\mathcal{H}_\alpha)(x,t)=(|x|^2+it)^{-\alpha/2}$ and  $ \mathcal{F}(\mathcal{G}_\alpha)(x,t)=(1+|x|^2+it)^{-\alpha/2}
                           $.  We define the parabolic  Riesz potential $\mathcal{H}_\alpha$ of a measure $\mu\in \mathfrak{M}^+(\mathbb{R}^{N+1})$ by                                  \begin{equation*}
                                 \mathcal{H}_\alpha[\mu](x,t)=(\mathcal{H}_\alpha*\mu)(x,t)=\int_{\mathbb{R}^{N+1}}\mathcal{H}_\alpha(x-y,t-s)d\mu(y,s)~~\text{for any }~~(x,t)~~\text{in}~~\mathbb{R}^{N+1},
                                 \end{equation*}      
                                 the parabolic  Bessel potential $\mathcal{G}_\alpha$ of a measure $\mu\in \mathfrak{M}^+(\mathbb{R}^{N+1})$ by 
                                       \begin{equation*}
                                       \mathcal{G}_\alpha[\mu](x,t)=(\mathcal{G}_\alpha*\mu)(x,t)=\int_{\mathbb{R}^{N+1}}\mathcal{G}_\alpha(x-y,t-s)d\mu(y,s)~\text{for any}~(x,t)~\text{in}~\mathbb{R}^{N+1}.
                                       \end{equation*}
               We also define $\mathbf{I}_\alpha, \mathbf{G}_\alpha, 0<\alpha<N$ the Riesz, Bessel potential of a measure $\mu\in\mathfrak{M}^+(\mathbb{R}^{N})$ by 
                \begin{equation*}\mathbf{I}_{\alpha}[\mu](x)=\int_{0}^{\infty}\frac{\mu(B_\rho(x))}{\rho^{N-\alpha}}\frac{d\rho}{\rho}~\text{and}~
                                                 \mathbf{G}_\alpha[\mu](x)=\int_{\mathbb{R}^{N}}\mathbf{G}_\alpha(x-y)d\mu(y)~\text{for any}~x~\text{in}~\mathbb{R}^{N},
                                                 \end{equation*}      where $\mathbf{G}_\alpha$ is the Bessel kernel of order $\alpha$, see \cite{55AH}.\\
Several  different capacities will be used throughout the paper. For $1<p<\infty$,   the $(\mathcal{H}_\alpha,p)$-capacity, $(\mathcal{G}_\alpha,p)$-capacity of a Borel set $E\subset \mathbb{R}^{N+1}$ are defined by  
 \begin{align*}
& \text{Cap}_{\mathcal{H}_\alpha,p}(E)=\inf\left\{\int_{\mathbb{R}^{N+1}}|f|^pdxdt: f\in L^p_+(\mathbb{R}^{N+1}), \mathcal{H}_\alpha*f\geq \chi_E\right\},\\&
 \text{Cap}_{\mathcal{G}_\alpha,p}(E)=\inf\left\{\int_{\mathbb{R}^{N+1}}|f|^pdxdt: f\in L^p_+(\mathbb{R}^{N+1}), \mathcal{G}_\alpha*f\geq \chi_E\right\}.
 \end{align*}  
 The  $W^{2,1}_p-$capacity of  compact set $E\subset\mathbb{R}^{N+1}$ is defined by    
           \begin{align*}
           \text{Cap}_{2,1,p}(E)=\inf\left\{||\varphi||^p_{W^{2,1}_p(\mathbb{R}^{N+1})}:\varphi\in S(\mathbb{R}^{N+1}), \varphi\geq 1 \text{ in a neighborhood of}~E\right \}, 
           \end{align*} 
           where      
          \begin{align*}
          ||\varphi||_{W^{2,1}_p(\mathbb{R}^{N+1})}=||\varphi||_{L^p(\mathbb{R}^{N+1})}+||\frac{\partial \varphi}{\partial t}||_{L^p(\mathbb{R}^{N+1})}+||\nabla \varphi||_{L^p(\mathbb{R}^{N+1})}+\sum\limits_{i,j=1,2,...,N} ||\frac{\partial^{2} \varphi}{\partial x_i\partial x_j}||_{L^p(\mathbb{R}^{N+1})}.
          \end{align*}                                                  We remark that thanks to Richard J. Bagby's result (see \cite{55Bag}) we obtain the equivalent of capacities $\text{Cap}_{2,1,p}$ and $\text{Cap}_{\mathcal{G}_2,p}$, i.e, 
          for any compact set $K\subset\mathbb{R}^{N+1}$ there holds 
$$
          \text{Cap}_{\mathcal{G}_2,p}(K)\sim\text{Cap}_{2,1,p}(K),$$
         see Corollary \eqref{5hh250220142} in section 4. \\
          The $(\mathbf{I}_{\alpha},p)$-capacity, $(\mathbf{G}_{\alpha},p)$-capacity of a Borel set $O\subset \mathbb{R}^{N}$ are defined by
\begin{align*}
& \text{Cap}_{\mathbf{I}_{\alpha},p}(O)=\inf\left\{\int_{\mathbb{R}^{N}}|g|^pdx: g\in L^p_+(\mathbb{R}^{N}), \mathbf{I}_{\alpha}*g\geq \chi_O\right\},\\&
\text{Cap}_{\mathbf{G}_{\alpha},p}(O)=\inf\left\{\int_{\mathbb{R}^{N}}|g|^pdx: g\in L^p_+(\mathbb{R}^{N}), \mathbf{G}_{\alpha}*g\geq \chi_O\right\}.
\end{align*}
In this paper, we often use the expression $A\lesssim_{\alpha,\beta} B$ to mean that there exists a universal constant $C$ depending on $\alpha,\beta$ and on fixed quantities ($N,\Lambda_1,\Lambda_2$) such that $a\leq C b$. The same convention is adopted for $\gtrsim_{\alpha,\beta}$ and $\sim_{\alpha,\beta}$.\\

In our first three Theorems, we present global pointwise potential estimates for solutions  to quasilinear parabolic problems
 \begin{equation}\label{5hhparabolic1}\left\{ \begin{array}{l}
                     {u_t} - \operatorname{div}\left({A(x,t,\nabla u)} \right) = \mu ~\text{in}~\Omega_T,  \\ 
                      u=0~~~~~~~\text{on}~~
                                                       \partial\Omega \times (0,T),
                                                          \\
                                                          u(0) = \sigma~~~\text{in}~~ \Omega, \\  
                     \end{array} \right.\end{equation}             
  \begin{equation}\label{5hhparabolic2}\left\{ \begin{array}{l}
                      {u_t} - \operatorname{div}\left({A(x,t,\nabla u)} \right) = \mu ~\text{in}~\mathbb{R}^{N+1}_+,  \\                     
                                                           u(0) = \sigma~~~\text{in}~~ \mathbb{R}^N, \\  
                      \end{array} \right.\end{equation}
                      and
                      \begin{align}\label{5hhparabolic2'}
                       {u_t} - \operatorname{div}\left( {A(x,t,\nabla u)} \right) = \mu ~\text{in}~\mathbb{R}^{N+1}.
                      \end{align}                                                                                                         
                 \begin{theorem}\label{5hh141013112} There exists a constant $K$ depending on $N,\Lambda_1,\Lambda_2$ such that for any  $\mu\in\mathfrak{M}_b(\Omega_T),\sigma\in\mathfrak{M}_b(\Omega)$ there is a distributional solution $u$ of \eqref{5hhparabolic1} which satisfies 
                          \begin{equation}\label{5hh14101314}
                           -K\mathbb{I}^{2T_0}_{2}[\mu^-+\sigma^-\otimes\delta_{\{t=0\}}] \leq u\leq K\mathbb{I}_{2}^{2T_0}[\mu^++\sigma^+\otimes\delta_{\{t=0\}}]~\text{in}~\Omega_T.
                             \end{equation}

                          \end{theorem}
                          \begin{remark}
                        Since $\sup_{x\in\mathbb{R}^{N}}\mathbb{I}_{\alpha}[\sigma^\pm\otimes\delta_{\{t=0\}}](x,t)\leq \frac{\sigma^\pm(\Omega)}{(N+2-\alpha)(2|t|)^{\frac{N+2-\alpha}{2}}}$ for any $t\not=0$ with $0<\alpha<N+2$. Thus, if $\mu\equiv0$, then we obtain the decay estimate:
\begin{equation*}
                  - \frac{K\sigma^-(\Omega)}{N(2t)^{\frac{N}{2}}} \leq \inf_{x\in\Omega}u(x,t)\leq \sup_{x\in\Omega}u(x,t)\leq \frac{K\sigma^+(\Omega)}{N(2t)^{\frac{N}{2}}}\text{ for any }~0<t<T.
                             \end{equation*}                        
                          \end{remark} 
                                          
                          \begin{theorem}\label{5hh270120145}For any $\mu\in\mathfrak{M}^+_b(\Omega_T),\sigma\in\mathfrak{M}^+_b(\Omega)$, there is a distributional solution $u$ of \eqref{5hhparabolic1} satisfying for a.e $(y,s)\in \Omega_T$ and $B_r(y)\subset\Omega$
                                                    \begin{equation}\label{5hh010420141}
                                                    u(y,s)\gtrsim\sum_{k=0}^{\infty}\frac{\mu(Q_{r_k/8}(y,s-\frac{35}{128}r_k^2))}{r_k^N}+\sum_{k=0}^{\infty}\frac{(\sigma\otimes\delta_{\{t=0\}})(Q_{r_k/8}(y,s-\frac{35}{128}r_k^2))}{r_k^N},
                                                    \end{equation}
                                                    where $r_k=4^{-k}r$.   
                                                    \end{theorem}
                         \begin{remark} The Theorem \ref{5hh270120145} is also true when we replace the assumption \eqref{5hhcondb} by  the following weaker one \begin{align*}                          \left\langle {A(x,t,\zeta),\zeta} \right\rangle \geq \Lambda_2 |\zeta|^2, ~~\left\langle {A(x,t,\zeta)-A(x,t,\lambda),\zeta-\lambda} \right\rangle>0,
                         \end{align*}   for every $(\lambda,\zeta)\in \mathbb{R}^N\times \mathbb{R}^N$, $\lambda\not=\zeta$ and a.e. $(x,t)\in \mathbb{R}^N\times \mathbb{R}$.
                                                    
                                                    \end{remark}
\begin{theorem}\label{5hh1203201417}
Let $K$ be the constant in Theorem \ref{5hh141013112}.  Let $\omega\in\mathfrak{M}(\mathbb{R}^{N+1})$ such that $I_2[|\omega|](x_0,t_0)<\infty$ for some $(x_0,t_0)\in \mathbb{R}^{N+1}$. Then,  there is a distributional solution $u$ to \eqref{5hhparabolic2'} with data $\mu=\omega$  satisfying
 \begin{equation}\label{5hh1203201414}
             -K\mathbb{I}_{2}[\omega^-] \leq u\leq K\mathbb{I}_{2}[\omega^+] \text{ in } \mathbb{R}^{N+1}
                               \end{equation}
 such that the following statements hold. 
                               \begin{description}
                               \item[a.] If $\omega\geq 0$, there holds for a.e $(x,t)\in\mathbb{R}^{N+1}$
                                  \begin{align}
                                  u(x,t)\gtrsim\sum_{k=-\infty}^{\infty}\frac{\omega(Q_{2^{-2k-3}}(x,t-35\times 2^{-4k-7}))}{2^{-2Nk}}. 
                                  \label{5hh1203201415}
                                  \end{align}  
                                  In particular, for any $q>\frac{N+2}{N}$ 
                                  \begin{align}\label{5hh130320145}
                               ||u||_{L^q(\mathbb{R}^{N+1})}\sim ||\mathcal{H}_2[\omega]||_{L^q(\mathbb{R}^{N+1})}.
                                  \end{align}
                              
                                  \item[b.] If $A$ is independent of space variable $x$ and  satisfies $\eqref{5hhcondc}$, then we have 
                                  \begin{align}\label{5hh310320146}
                                  |\nabla u|\lesssim\mathbb{I}_{1}[|\omega|]~\text{ in }\mathbb{R}^{N+1}.
                                  \end{align}
                                  \item[c.] If $\omega=\mu+\sigma\otimes\delta_{\{t=0\}}$ with $\mu\in\mathfrak{M}(\mathbb{R}^N\times (0,\infty))$ and $\sigma\in \mathfrak{M}(\mathbb{R}^N)$, then $u=0$ in $\mathbb{R}^{N}\times (-\infty,0)$ and ${\left. u \right|_{\mathbb{R}^N\times [0,\infty)}}$ is a distributional solution to \eqref{5hhparabolic2}. 
                               \end{description}

\end{theorem}
\begin{remark} For $q>\frac{N+2}{N}$, we alway have the following claim:
 \begin{align*}
 ||\mathcal{H}_2[\mu+\omega\otimes\delta_{\{t=0\}}]||_{L^q(\mathbb{R}^{N+1})}\sim ||\mathcal{H}_2[\mu]||_{L^q(\mathbb{R}^{N+1})}+||\mathbf{I}_{2/q}[\sigma]||_{L^q(\mathbb{R}^{N+1})},
                                  \end{align*}
                                  for every $\mu\in\mathfrak{M}^+(\mathbb{R}^N\times (0,\infty))$ and $\sigma\in \mathfrak{M}^+(\mathbb{R}^N)$.

\end{remark}
\begin{remark}\label{5hh020520141} For $\omega\in\mathfrak{M}^+(\mathbb{R}^{N+1})$, $0<\alpha<N+2$ if $\mathbb{I}_\alpha[\omega](x_0,t_0)<\infty$ for some $(x_0,t_0)\in\mathbb{R}^{N+1}$ then for any $0<\beta\leq \alpha$,  $\mathbb{I}_\beta[\omega]\in L^{s}_{\text{loc}}(\mathbb{R}^{N+1})$ for any $0<s<\frac{N+2}{N+2-\beta}$. However, for $0<\beta<\alpha<N+2$, one can find $\omega\in\mathfrak{M}^+(\mathbb{R}^{N+1})$ such that $\mathbb{I}_\alpha[\omega]\equiv \infty$ and $\mathbb{I}_\beta[\omega]<\infty$ in $\mathbb{R}^{N+1}$, see Appendix section. 
 \end{remark}                                                    
The next four theorems provide the existence of  solutions to quasilinear parabolic equations with absorption and  source terms. For convenience, we always denote by $q'$ the conjugate exponent of $q\in (1,\infty)$ i.e $q'=\frac{q}{q-1}$.
\begin{theorem}\label{5hh070120146}
Let $q>1$, $\mu\in\mathfrak{M}_b(\Omega_T)$ and $\sigma\in \mathfrak{M}_b(\Omega)$. Suppose that $\mu,\sigma$ are absolutely continuous with respect to the capacities $\text{Cap}_{2,1,q'}$, $\text{Cap}_{\mathbf{G}_{\frac{2}{q}},q'}$ in  $\Omega_T, \Omega$ respectively. 
Then there exists a distributional solution $u$ of \eqref{5hh070120148} satisfying 
                \begin{equation*}                      
                      -K\mathbb{I}_{2}[\mu^-+\sigma^-\otimes\delta_{\{t=0\}}] \leq u\leq K\mathbb{I}_{2}[\mu^++\sigma^+\otimes\delta_{\{t=0\}}]~\text{ in } \Omega_T.
                       \end{equation*}
                       Here the constant $K$ is in Theorem \ref{5hh141013112}. 
\end{theorem}

\begin{theorem}\label{5hh070120147}Let $K$ be the constant in Theorem \ref{5hh141013112}. Let $q>1$, $\mu\in\mathfrak{M}_b(\Omega_T)$ and $\sigma\in\mathfrak{M}_b(\Omega)$.  There exists a constant $\varepsilon_0=\varepsilon_0(N,q,\Lambda_1,\Lambda_2,\text{diam}(\Omega),T)$ such that  if 
       \begin{equation}\label{5hh020420143}
       |\mu|(E)\leq \varepsilon_0 \text{Cap}_{2,1,q'}(E)~~\text{and}~~|\sigma|(O)\leq \varepsilon_0 \text{Cap}_{\mathbf{G}_{\frac{2}{q}},q'}(O).
       \end{equation} 
      hold for every compact sets $E\subset \mathbb{R}^{N+1}$, $O\subset \mathbb{R}^{N}$, then the problem \eqref{5hh070120149} has  a distributional solution $u$ satisfying
                \begin{equation}  \label{5hh020420142}                     
                      -\frac{Kq}{q-1}\mathbb{I}_{2}[\mu^-+\sigma^-\otimes\delta_{\{t=0\}}]\leq u\leq \frac{Kq}{q-1}\mathbb{I}_{2}[\mu^++\sigma^+\otimes\delta_{\{t=0\}}]~\text{ in }~\Omega_T.
                       \end{equation}
            Besides, for every compact set $E\subset \mathbb{R}^{N+1}$ there holds
            \begin{equation}\label{5hh270420141}
            \int_E|u|^qdxdt\lesssim_{T_0,q} \text{Cap}_{2,1,q'}(E).
            \end{equation}                      
\end{theorem}
\begin{remark} From \eqref{5hh270420141} we get
 if $q>\frac{N+2}{N}$, 
                       \begin{align*}
                       \int_{\tilde{Q}_\rho(y,s)}|u|^qdxdt\lesssim_{T_0,q} \rho^{N+2-2q'}~~\text{ for any }~\tilde{Q}_\rho(y,s)\subset \mathbb{R}^{N+1},
                       \end{align*}
                       if $q=\frac{N+2}{N}$, 
                       \begin{align*}
                                              \int_{\tilde{Q}_\rho(y,s)}|u|^qdxdt\lesssim_{T_0,q} \left(\log(1/\rho)\right)^{-\frac{1}{q-1}} ~~\text{ for any }~\tilde{Q}_\rho(y,s)\subset \mathbb{R}^{N+1}, 0<\rho<1/2,
                                              \end{align*}
                                              see Remark \ref{5hh020520142}. 
\end{remark}
\begin{remark}\label{5hh180420141}
In the sub-critical case $1<q< \frac{N+2}{N}$, since the capacity $\text{Cap}_{2,1,q'}, \text{Cap}_{\mathbf{G}_{\frac{2}{q}},q'}$ of a single point are positive thus the conditions \eqref{5hh020420143} hold for some constant $\varepsilon_0>0$ provided $\mu\in\mathfrak{M}_b(\Omega_T),\sigma\in\mathfrak{M}_b(\Omega)$. Moreover, in the super-critical case $q>\frac{N+2}{N}$, we have 
\begin{align*}
\text{Cap}_{2,1,q'}(E)\gtrsim_q |E|^{1-\frac{2q'}{N+2}}~~\text{and}~~\text{Cap}_{\mathbf{G}_{\frac{2}{q}},q'}(O)\gtrsim_q |O|^{1-\frac{2}{(q-1)N}},
\end{align*} 
for every Borel sets $E\subset \mathbb{R}^{N+1}$, $O\subset \mathbb{R}^{N}$, thus if $\mu\in L^{\frac{N+2}{2q'},\infty}(\Omega_T)$ and $\sigma\in L^{\frac{(q-1)N}{2},\infty}(\Omega)$ then \eqref{5hh020420143} holds for some constant $\varepsilon_0>0$. In addition,  if $\mu\equiv0$, then \eqref{5hh020420142} implies for any $0<t<T$,
\begin{align*}
- c(T_0,q)t^{-\frac{1}{q-1}} \leq \inf_{x\in\Omega}u(x,t)\leq \sup_{x\in\Omega}u(x,t)\leq c(T_0,q)t^{-\frac{1}{q-1}} ,
\end{align*}
 since $|\sigma|(B_\rho(x))\lesssim_{T_0,q}\rho^{N-\frac{2}{q-1}}$ for all  $x\in\mathbb{R}^N$, $0<\rho<2T_0$.  
\end{remark}

\begin{theorem}\label{5hh1303201411}Let $K$ be the constant in Theorem \ref{5hh141013112} and $q>1$. If $\omega\in\mathfrak{M}(\mathbb{R}^{N+1})$  is absolutely continuous with respect to the capacity $\text{Cap}_{2,1,q'}$ in $\mathbb{R}^{N+1}$, then there exists a distributional solution $u\in L^\gamma_{\text{loc}
               }(\mathbb{R};W^{1,\gamma}_{\text{loc}}(\mathbb{R}^N))$ for any $1\leq \gamma<\frac{2q}{q+1}$ to problem
\begin{align}\label{5hhparabolic3'}
 {u_t} - \operatorname{div}\left({A(x,t,\nabla u)} \right) +|u|^{q-1}u= \omega ~\text{in}~\mathbb{R}^{N+1},
\end{align} which satisfies
                \begin{equation} \label{5hh020420145}                     
                      -K\mathbb{I}_{2}[\omega^-]\leq u\leq K\mathbb{I}_{2}[\omega^+]~\text{ in } \mathbb{R}^{N+1}.
                       \end{equation} 
Furthermore, when $\omega=\mu+\sigma\otimes \delta_{\{t=0\}}$ with       $\mu\in\mathfrak{M}(\mathbb{R}^{N}\times (0,\infty))$, $\sigma\in \mathfrak{M}(\mathbb{R}^N)$  then $u=0$ in $\mathbb{R}^N\times(-\infty,0)$ and   ${\left. u \right|_{\mathbb{R}^N\times [0,\infty)}}$ is a distributional solution to problem                 
 \begin{equation}\label{5hhparabolic3}\left\{ \begin{array}{l}
                       {u_t} - \operatorname{div}\left( {A(x,t,\nabla u)} \right) +|u|^{q-1}u= \mu ~\text{in}~\mathbb{R}^N\times (0,\infty),  \\                     
                                                            u(0) = \sigma~~~\text{in}~~ \mathbb{R}^N. \\  
                       \end{array} \right.\end{equation}
                                  
\end{theorem}
\begin{remark}The measure $\omega=\mu+\sigma\otimes \delta_{\{t=0\}}$ is absolutely continuous with respect to the capacity $\text{Cap}_{2,1,q'}$ in $\mathbb{R}^{N+1}$ if and only if $\mu,\sigma$ are absolutely continuous with respect to the capacities $\text{Cap}_{2,1,q'}$, $\text{Cap}_{\mathbf{G}_{\frac{2}{q}},q'}$ in $\mathbb{R}^{N+1},\mathbb{R}^N$ respectively.
\end{remark}
Existence result of the problem \eqref{5hh070120149} on $\mathbb{R}^{N+1}$ or on $\mathbb{R}^N\times(0,\infty)$ is similar to Theorem \ref{5hh070120147} presented in the following Theorem, where the capacities $\text{Cap}_{\mathcal{H}_2,q'}, \text{Cap}_{\mathbf{I}_{\frac{2}{q}},q'}$ are used in place of respectively $\text{Cap}_{2,1,q'}, \text{Cap}_{\mathbf{G}_{\frac{2}{q}},q'}$.
\begin{theorem}\label{5hh130320142} Let $K$ be the constant in Theorem \ref{5hh141013112} and  $q>\frac{N+2}{N}$, $\omega\in\mathfrak{M}(\mathbb{R}^{N+1})$. There exists a constant $\varepsilon_0=\varepsilon_0(N,q,\Lambda_1,\Lambda_2)$ such that if  
\begin{equation}\label{5hh130320143'}
       |\omega|(E)\leq \varepsilon_0 \text{Cap}_{\mathcal{H}_2,q'}(E),
       \end{equation} 
for every compact set $E\subset \mathbb{R}^{N+1}$, then the problem 
  \begin{align}\label{5hhparabolic4'}
  {u_t} - \operatorname{div}\left({A(x,t,\nabla u)} \right) =|u|^{q-1}u+ \omega ~\text{in}~\mathbb{R}^{N+1}
  \end{align}    has a distributional solution $u\in L^\gamma_{\text{loc}
                 }(\mathbb{R};W^{1,\gamma}_{\text{loc}}(\mathbb{R}^N))$ for any $1\leq \gamma<\frac{2q}{q+1}$ satisfying
                  \begin{equation}
                        \label{5hh130320144}
                         -\frac{Kq}{q-1}\mathbb{I}_{2}[\omega^-]\leq u\leq \frac{Kq}{q-1}\mathbb{I}_{2}[\omega^+]~\text{ in }~\mathbb{R}^{N+1}.
                         \end{equation} 
    Moreover, when  $\omega=\mu+\sigma\otimes \delta_{\{t=0\}}$ with       $\mu\in\mathfrak{M}(\mathbb{R}^{N}\times (0,\infty))$, $\sigma\in \mathfrak{M}(\mathbb{R}^N)$  then $u=0$ in $\mathbb{R}^N\times(-\infty,0)$ and   ${\left. u \right|_{\mathbb{R}^N\times [0,\infty)}}$ is a distributional solution to problem  \begin{equation}\label{5hhparabolic4}\left\{ \begin{array}{l}
                           {u_t} - \operatorname{div}\left( {A(x,t,\nabla u)} \right) =|u|^{q-1}u+ \mu ~\text{in}~\mathbb{R}^N\times (0,\infty),  \\                     
                                                                u(0) = \sigma~~~\text{in}~~ \mathbb{R}^N. \\  
                           \end{array} \right.\end{equation}                               
In addition, for any compact set $E\subset \mathbb{R}^{N+1}$ there holds
\begin{equation}\label{5hh280420141}
\int_E|u|^qdxdt\lesssim_q \text{Cap}_{\mathcal{H}_2,q'}(E).
\end{equation}
\end{theorem}
\begin{remark} The measure $\omega=\mu+\sigma\otimes \delta_{\{t=0\}}$ satisfies \eqref{5hh130320143'} if and only if 
 \begin{equation*}
       |\mu|(E)\lesssim \text{Cap}_{\mathcal{H}_2,q'}(E)~~\text{and}~~|\sigma|(O)\lesssim \text{Cap}_{\mathbf{I}_{\frac{2}{q}},q'}(O),
       \end{equation*}    
       for every compact sets $E\subset \mathbb{R}^{N+1}$ and $O\subset \mathbb{R}^{N}$. 
\end{remark}
\begin{remark}\label{5hh180420142}
If $\omega\in L^{\frac{N+2}{2q'},\infty}(\mathbb{R}^{N+1})$  then \eqref{5hh130320143'} holds for some constant $\varepsilon_0>0$. Moreover, if $\omega=\sigma\otimes\delta_{\{t=0\}}$ with $\sigma\in \mathfrak{M}_b(\mathbb{R}^N)$, then from \eqref{5hh130320144} we get the decay estimate: 
$$
                  - ct^{-\frac{1}{q-1}} \leq \inf_{x\in\mathbb{R}^N}u(x,t)\leq \sup_{x\in\mathbb{R}^N}u(x,t)\leq ct^{-\frac{1}{q-1}}\text{ for any }~t>0,
$$
                             since $|\sigma|(B_\rho(x))\lesssim_q \rho^{N-\frac{2}{q-1}}$ for any $B_\rho(x)\subset\mathbb{R}^N$. 
\end{remark}

Second part, we establish global regularity in weighted-Lorentz space and Lorentz-Morrey {space for the gradient of solutions to problem 
\eqref{5hhparabolic1}.  For this purpose, we need a capacity density condition imposed on $\Omega$. That is,  the complement of $\Omega$  satisfies  {\it uniformly $p$-thick with constants $c_0,r_0$}, i.e,  for all $0<r\le r_0$ and all $x\in \mathbb{R}^N\backslash \Omega$ there holds 
      \begin{equation}
      \text{Cap}_p(\overline{B_r(x)}\cap(\mathbb{R}^N\backslash \Omega), B_{2r}(x))\geq c_0 \text{Cap}_p(\overline{B_r(x)}, B_{2r}(x)),
      \end{equation}
  where the involved capacity of a compact set $K\subset B_{2r}(x)$ is given as follows 
      \begin{equation}
      \text{Cap}_p(K,B_{2r}(x))=\inf\left\{\int_{B_{2r}(x)}|\nabla \phi|^pdy: \phi\in C^\infty_c(B_{2r}(x)), \phi\geq \chi_K\right\}.
      \end{equation}
      
 In order to obtain better regularity we need a stricter condition on $\Omega$ which is expressed in the following way. We say that $\Omega$ is a $(\delta,R_0)-$Reifenberg flat domain for $\delta\in (0,1)$ and $R_0>0$ if for every $x_0\in\partial \Omega$ and every $r\in(0,R_0]$, there exists a system of coordinates $\{z_1,z_2,...,z_n\}$, which may depend on $r$ and $x_0$, so that in this coordinate system $x_0=0$ and that 
         \begin{equation}
         B_r(0)\cap \{z_n>\delta r\}\subset B_r(0)\cap \Omega\subset B_r(0)\cap\{z_n>-\delta r\}.
         \end{equation}
         
 We remark that this class of flat domains is rather wide since it includes $C^1$, Lipschitz domains with sufficiently small Lipschitz constants and fractal domains. Besides, it has many important roles in the theory of minimal surfaces and free boundary problems,  this class was first appeared in a work of Reifenberg (see \cite{55Re}) in the context of a Plateau problem. Its properties can be found in \cite{55KeTo1,55KeTo2,55To}.
 
         On the other hand, it is well-known that in general, conditions \eqref{5hhconda} and \eqref{5hhcondb} on the nonlinearity $A(x,t,\zeta)$ are not enough to ensure higher integral of gradient of solutions to problem \eqref{5hhparabolic1},  we need to  assume that $A$ satisfies \begin{equation}\label{5hhcondc}
              \left \langle A_\zeta(x,t,\zeta)\xi ,\xi\right\rangle\geq \Lambda_2 |\xi|^2,~~~~ |A_\zeta(x,t,\zeta)|\le \Lambda_1,
               \end{equation}
               for every $(\xi,\zeta)\in \mathbb{R}^N\times \mathbb{R}^N\backslash \{(0,0)\}$ and a.e $(x,t)\in \mathbb{R}^N\times\mathbb{R}$, where $\Lambda_1,\Lambda_2$ are constants in \eqref{5hhconda} and \eqref{5hhcondb}.
               We also require  that the nonlinearity $A$
               satisfies a smallness condition of BMO type in the $x$-variable. We say that $A(x,t,\zeta)$ satisfies a $(\delta,R_0)$-BMO condition for some $\delta, R_0>0$ with exponent $s>0$ if 
                                   \begin{equation*}
                                   [A]^{R_0}_s:=\mathop {\sup }\limits_{(y,s)\in \mathbb{R}^N\times\mathbb{R},0<r\leq R_0}\left(\fint_{Q_r(y,s)}\left(\Theta(A,B_r(y))(x,t)\right)^sdxdt\right)^{\frac{1}{s}} \leq \delta,
                                   \end{equation*}
                                   where 
                                   \begin{equation*}
                                   \Theta(A,B_r(y))(x,t):=\mathop {\sup }\limits_{\zeta\in\mathbb{R}^N\backslash\{0\}}\frac{|A(x,t,\zeta)-\overline{A}_{B_r(y)}(t,\zeta)|}{|\zeta|},
                                   \end{equation*}
                                   and $\overline{A}_{B_r(y)}(t,\zeta)$ is denoted the average of $A(t,.,\zeta)$ over the cylinder $B_r(y)$, i.e,
                                   \begin{equation*}
                                   \overline{A}_{B_r(y)}(t,\zeta):=\fint_{B_r(y)}A(x,t,\zeta)dx=\frac{1}{|B_r(y)|}\int_{B_r(y)}A(x,t,\zeta)dx.
                                   \end{equation*}
                                   
                                      The above condition was observed in  \cite{55BW2}. It is easy to  see that the $(\delta,R_0)-$BMO condition on $A$ is satisfied when $A$ is continuous or has small jump discontinuities with respect to $x$.  
                                                         
               In this paper, $\mathbb{M}$ denotes the Hardy-Littlewood maximal function defined for each locally integrable function  $f$ in $\mathbb{R}^{N+1}$ by
                     \begin{equation*}
                     \mathbb{M}(f)(x,t)=\sup_{\rho>0}\fint_{\tilde{Q}_\rho(x,t)}|f(y,s)|dyds~~\forall (x,t)\in\mathbb{R}^{N+1}.
                     \end{equation*}
                                          
     It is  by now rather standard to verify that $\mathbb{M}$ is bounded operator from $L^1(\mathbb{R}^{N+1})$ to $L^{1,\infty}(\mathbb{R}^{N+1})$ and $L^s(\mathbb{R}^{N+1})$ ($L^{s,\infty}(\mathbb{R}^{N+1})$) to itself for $s>1$, see  \cite{55Stein2,55Stein3}.
     
     We recall that a positive function $w\in L^1_{\text{loc}}(\mathbb{R}^{N+1})$ is called an $A_{\infty}$ weight if there are two positive constants $C$ and $\nu$ such that
                         $$w(E)\le C \left(\frac{|E|}{|Q|}\right)^\nu w(Q),
                         $$
                          for all cylinder $Q=\tilde{Q}_\rho(x,t)$ and all measurable subsets $E$ of $Q$. The pair $(C,\nu) $ is called the $A_\infty$ constant of $w$ and is denoted by $[w]_{A_\infty}$.
                          
      For  a weight function $w\in A_{\infty}$, the weighted Lorentz spaces $L^{q,s}(D,dw)$ with $0<q<\infty$, $0<s\leq\infty$ and a Borel set $D\subset \mathbb{R}^{N+1}$, is the set of measurable functions $g$ on $D$ such that 
      \begin{equation*}
   ||g||_{L^{q,s}(D,dw)}:=\left\{ \begin{array}{l}
            \left(q\int_{0}^{\infty}\left(\rho^qw\left(\{(x,t)\in D:|g(x,t)|>\rho\}\right)\right)^{\frac{s}{q}}\frac{d\rho}{\rho}\right)^{1/s}<\infty~\text{ if }~s<\infty, \\ 
             \sup_{\rho>0}\rho w\left(\{(x,t)\in D:|g(x,t)|>\rho\}\right)^{1/q}<\infty~~\text{ if }~s=\infty. \\ 
             \end{array} \right.
      \end{equation*}              
             Here we write $w(E)=\int_{E}w(x,t)dxdt$ for a measurable set $E\subset \mathbb{R}^{N+1}$.  Obviously,  
             $
                     ||g||_{L^{q,q}(D,dw)}=||g||_{L^q(D,dw)}
                    $,
                    thus we have $L^{q,q}(D,dw)=L^{q}(D,dw)$.               As usual, when $w \equiv 1$  we simply write $L^{q,s}(D)$ instead of $L^{q,s}(D,dw)$. 
                    \\
                             
                    We now state the next results of the paper.                   
                    \begin{theorem}\label{5hh0701201411}Let $\mu\in\mathfrak{M}_b(\Omega_T)$, $\sigma\in\mathfrak{M}_b(\Omega)$, set $\omega=|\mu|+|\sigma|\otimes\delta_{\{t=0\}}$. There exists a distributional solution of \eqref{5hhparabolic1} with data $\mu$ and $\sigma$ such that if $\mathbb{R}^N\backslash\Omega$ satisfies uniformly $2-$thick with  constants $c_0,r_0$ then for any $1\leq p< \theta $ and $0<s \leq \infty$, 
                    \begin{align}
                    \label{5hh15101311}
                                  ||\mathbb{M}(|\nabla u|)||_{L^{p,s}(\Omega_T)}\lesssim C ||\mathbb{M}_1[\omega]||_{L^{p,s}(Q)}.
                    \end{align}                
Here                                                 $\theta=\theta(N,\Lambda_1,\Lambda_1,c_0)>2$ and  $C=C(p,s,c_0,T_0/r_0)$ 
                            and $Q=B_{\text{diam}(\Omega)}(x_0)\times (0,T)$ which $\Omega\subset B_{\text{diam}(\Omega)}(x_0)$.\\
                           Especially, when  $1< p<2$, then 
\begin{align}\label{5hh200320142}                                                                                                                                                                                                         ||\mathbb{M}(|\nabla u|)||_{L^{p}(\Omega_T)}\lesssim C\left( ||\mathcal{G}_1[|\mu|]||_{L^{p}(\mathbb{R}^{N+1})}+||\mathbf{G}_{\frac{2}{p}-1}[|\sigma|]||_{L^{p}(\mathbb{R}^{N})}\right).
                                                                                                                                                                                \end{align}                             
                       \end{theorem}
   \begin{remark}\label{5hh070520145} If $\frac{N+2}{N+1}<p<2$, there hold
   \begin{align*}
   ||\mathcal{G}_1[|\mu|]||_{L^p(\mathbb{R}^{N+1})}\lesssim_p||\mu||_{L^{\frac{p(N+2)}{N+2+p}}(\Omega_T)}~\text{ and }~
   ||\mathbf{G}_{\frac{2}{p}-1}[|\sigma|]||_{L^p(\mathbb{R}^{N})}\lesssim_p ||\sigma||_{L^{\frac{pN}{N+2-p}}(\Omega)}.
   \end{align*}
   From \eqref{5hh200320142} we obtain
    \begin{align*}                                                                                                                                                                                                        |||\nabla u|||_{L^{p}(\Omega_T)}\lesssim_p||\mu||_{L^{\frac{p(N+2)}{N+2+p}}(\Omega_T)}+||\sigma||_{L^{\frac{pN}{N+2-p}}(\Omega)}~\text{ provided } \frac{N+2}{N+1}<p<2.
                                                                                                                                                                                                             \end{align*}  
   We should mention that if $\sigma\equiv 0$, then for any $0<s\leq \infty, p>\frac{N+2}{N+1}$  
    \begin{align*}
    ||\mathbb{M}_1[\omega]||_{L^{p,s}(\mathbb{R}^{N+1})}\lesssim_{p,s} ||\mu||_{L^{\frac{p(N+2)}{N+2+p},s}(\Omega_T)},
    \end{align*}
    and we get \cite[Theorem 1.2]{55BaCaPa} from estimate \eqref{5hh15101311}.
     \end{remark}  
                         
  In order to state the next results, we need to introduce  Lorentz-Morrey spaces  $L^{q,s;\theta}_{*}(D)$ involving "calorie"  with a Borel set $D\subset\mathbb{R}^{N+1}$, is the set of measurable functions $g$ on $D$ such that 
               \begin{equation*}
               ||g||_{L_{*}^{q,s;\kappa}(D)}:=\sup_{0<\rho< \text{diam}(D), (x,t)\in D}\rho^{\frac{\kappa-N-2}{q}}||g||_{L^{q,s}(\tilde{Q}_\rho(x,t)\cap D)}<\infty,
               \end{equation*}
               where $0<\kappa\leq N+2$, $0<q<\infty$, $0<s\leq \infty$.                
               Clearly, $L_{*}^{q,s;N+2}(D)=L^{q,s}(D)$. Moreover, when $q=s$ the space $L_{*}^{q,s;\theta}(D)$ will be denoted by $L_{*}^{q;\theta}(D)$. \\                     
                       The following theorem provides an estimate on gradient in Lorentz-Morrey spaces.   
                        \begin{theorem}\label{5hh0701201412} Let $\mu\in\mathfrak{M}_b(\Omega_T)$, $\sigma\in\mathfrak{M}_b(\Omega)$, set $\omega=|\mu|+|\sigma|\otimes\delta_{\{t=0\}}$. There exists a distributional solution of \eqref{5hhparabolic1} with data $\mu$ and $\sigma$ such that if $\mathbb{R}^N\backslash\Omega$ satisfies uniformly $2-$thick with  constants $c_0,r_0$ then for any  $1 \leq p< \theta $ and $0<s \leq \infty$,                        
                       $2-\gamma_0<\gamma<N+2$, $\gamma\leq \frac{N+2}{p}+1,$ 
 \begin{align}
                       \nonumber&||\mathbb{M}\left(|\nabla u|\right)||_{L_{*}^{p,s;p(\gamma-1)}(\Omega_T)}\lesssim C_1||\mathbb{M}_{\gamma}[\omega]||_{L^\infty(\Omega_T)}\\&~~~~~+C_2\sup_{0<R\leq T_0, (y_0,s_0)\in \Omega_T}\left(R^{\frac{p(\gamma-1)-N-2}{p}}||\mathbb{M}_1[\chi_{\tilde{Q}_{R}(y_0,s_0)}\omega]||_{L^{p,s}(\tilde{Q}_{R}(y_0,s_0))}\right).\label{5hh070520141}
                       \end{align} 
Here $\theta$ is in Theorem \ref{5hh0701201411}, $\gamma_0=\gamma_0(N,\Lambda_1,\Lambda_1,c_0)\in (0,1/2]$ and $C_1=C_1(p,s,\gamma,c_0,T_0/r_0)$, $C_2=C_2(p,s,\gamma,c_0)$.  Besides, if $\frac{\gamma}{\gamma-1}<p<\theta$,  $2-\gamma_0<\gamma<N+2$, $0<s \leq \infty$ and $\mu\in L_{*}^{\frac{(\gamma-1)p}{\gamma},\frac{(\gamma-1)s}{\gamma};(\gamma-1)p}(\Omega_T)$, $\sigma\equiv 0$, then  $u$ is a unique renormalized solution satisfying   

 \begin{equation}                            ||\mathbb{M}\left(|\nabla u|\right)||_{L_{*}^{p,s;(\gamma-1)p}(\Omega_T)}\lesssim C_1||\mu||_{L_{*}^{\frac{(\gamma-1)p}{\gamma},\frac{(\gamma-1)s}{\gamma};(\gamma-1)p}(\Omega_T)}.\label{5hh070520142} 
                            \end{equation}                                               
                                         \end{theorem}

The following Theorem is about higher regularity for solutions of \eqref{5hhparabolic1}. We need more assumptions on boundary of $\Omega$ and on the nonlinearity $A$. 
                   \begin{theorem} \label{5hh2410131}  Suppose that $A$ satisfies \eqref{5hhcondc}. Let $\mu\in\mathfrak{M}_b(\Omega_T)$, $\sigma\in\mathfrak{M}_b(\Omega)$, set $\omega=|\mu|+|\sigma|\otimes\delta_{\{t=0\}}$. There exists a distributional solution of \eqref{5hhparabolic1} with data $\mu,\sigma$ such that the following holds. For any $w\in A_\infty$, $1\leq q<\infty$, $0<s\leq\infty$ we find  $\delta=\delta(N,\Lambda_1,\Lambda_2,q,s, [w]_{A_\infty})\in (0,1)$ and $s_0=s_0(N,\Lambda_1,\Lambda_2)>0$ such that if $\Omega$ is  $(\delta,R_0)$-Reifenberg flat domain $\Omega$ and $[A]_{s_0}^{R_0}\le \delta$ for some $R_0$ then                                       
                        \begin{equation}\label{5hh16101312}
                        ||\mathbb{M}(|\nabla u|)||_{L^{q,s}(\Omega_T,dw)}\lesssim C ||\mathbb{M}_1[\omega]||_{L^{q,s}(\Omega_T,dw)}.
                        \end{equation} 
                         Here $C$ depends  on $q,s, [w]_{A_\infty}$ and $T_0/R_0$.                       
                       \end{theorem} 
                                             
                       Next results are actually consequences of  Theorem \ref{5hh2410131}. For our purpose, we  introduce
                      another Lorentz-Morrey spaces  $L_{**}^{q,s;\theta}(O_1\times O_2)$, is the set of measurable functions $g$ on $O_1\times O_2$ such that 
                                                                     \begin{equation*}
                                                                     ||g||_{L_{**}^{q,s;\vartheta}(O_1\times O_2)}:=\sup_{0<\rho< \text{diam}(O_1), x\in O_1}\rho^{\frac{\vartheta-N}{q}}||g||_{L^{q,s}((B_\rho(x)\cap O_1)\times O_2))}<\infty,
                                                                     \end{equation*}
 where $O_1,O_2$ are Borel sets in $\mathbb{R}^N$ and $\mathbb{R}$ respectively, $0<\vartheta\leq N$, $0<q<\infty$, $0<s\leq \infty$. Obviously,    $L_{**}^{q,s;N}(D)=L^{q,s}(D)$. For simplicity of notation, we write $L_{**}^{q;\vartheta}(D)$ instead of $L_{**}^{q,s;\vartheta}(D)$ when $q=s$. Moreover, $$
 ||g||_{L_{**}^{q,q;\vartheta}(O_1\times O_2)}=||G||_{L^{q;\vartheta}(O_1)},$$ 
 where  $G(x)=||g(x,.)||_{L^{q}(O_1)}$ and $L^{q;\vartheta}(O_1)$ is the usual Morrey space, i.e the spaces of all measurable functions $f$ on $O_1$ with
                                                   $$||f||_{L^{q;\vartheta}(O_1)}:=\sup_{0<\rho< \text{diam}(O_1),y\in O_1}\rho^{\frac{\vartheta-N}{q}}||f||_{L^{q}(B_\rho(y)\cap O_1)}<\infty.$$                     \begin{theorem} \label{5hh190320143}  Suppose that $A$ satisfies \eqref{5hhcondc}. Let $\mu\in\mathfrak{M}_b(\Omega_T)$, $\sigma\in\mathfrak{M}_b(\Omega)$, set $\omega=|\mu|+|\sigma|\otimes\delta_{\{t=0\}}$. Let $s_0$ be in  Theorem \ref{5hh2410131}.  There exists a distributional solution of \eqref{5hhparabolic1} with data $\mu, \sigma$ such that the following holds.
                                                                                                           \begin{description}
                                                                                                           \item[a.]For any $1\leq q<\infty$, $0<s\leq\infty$ and $0<\kappa\leq N+2$ we find  $\delta=\delta(N,\Lambda_1,\Lambda_2,q,s,\kappa)\in (0,1)$ such that if $\Omega$ is  $(\delta,R_0)$-Reifenberg flat domain $\Omega$ and $[A]_{s_0}^{R_0}\le \delta$ for some $R_0$ then 
                                                                                                                                                                      \begin{equation}\label{5hh190320144}                                 
                                                                                                                                                                      ||\mathbb{M}(|\nabla u|)||_{L_{*}^{q,s;\kappa}(\Omega_T)}\lesssim C_1 ||\mathbb{M}_1[|\omega|]||_{L_{*}^{q,s;\kappa}(\Omega_T)}.
                                                                                                                                                                      \end{equation} 
                                                                                                            Here $C_1$ depends  on $q,s,\kappa$ and $T_0/R_0$.

                                                                                                                     \item[b.]For any $1\leq q<\infty$, $0<s\leq\infty$ and $0<\vartheta\leq N$ we find  $\delta=\delta(N,\Lambda_1,\Lambda_2,q,s,\vartheta)\in (0,1)$  such that if $\Omega$ is  $(\delta,R_0)$-Reifenberg flat domain $\Omega$ and $[A]_{s_0}^{R_0}\le \delta$ for some $R_0$ then 
                                                                                                                                                                                                        \begin{equation}\label{5hh050520143}                                 
                                                                                                                                                                                                        ||\mathbb{M}(|\nabla u|)||_{L_{**}^{q,s;\vartheta}(\Omega_T)}\lesssim C_2 ||\mathbb{M}_1[|\omega|]||_{L_{**}^{q,s;\vartheta}(\Omega_T)},
                                                                                                                                                                                                        \end{equation} 
                                                                                                                                              for some  $C_2=C_2(q,s,\vartheta, T_0/R_0)$. Especially, when $q=s$ and $0<\vartheta<\min\{N,q\}$,  there holds for any ball $B_\rho\subset\mathbb{R}^N$
                                                                                                          \begin{equation}\label{5hh050520141}                                 
                                                                                                                \left(\int_0^T|\text{osc}_{B_\rho\cap\overline{\Omega}}u(t)|^qdt\right)^{\frac{1}{q}} \lesssim C_3 \rho^{1-\frac{\vartheta}{q}} ||\mathbb{M}_1[|\omega|]||_{L_{**}^{q;\vartheta}(\Omega_T)},
                                                                                                                                                                                                                                                               \end{equation} 
                                                                                                                                                                                        for some  $C_3=C_3(q,\vartheta, T_0/R_0)$.                                                                           
                                                                                                           \end{description}                           
                                                                                                                                           \end{theorem}

                                                        \begin{remark} Above results also hold when $[A]_{s}^{R_0}$ is replaced by $\{A\}^{R_0}_s$:
                                                                          \begin{equation*}
                                                                    \{A\}^{R_0}_s:=\mathop {\sup }\limits_{(y,s)\in \mathbb{R}^N\times\mathbb{R},0<r\leq R_0}\left(\fint_{Q_r(y,s)}\left(\Theta(A,Q_r(y,s))(x,t)\right)^sdxdt\right)^{\frac{1}{s}} \leq \delta,
                                                                                                                              \end{equation*}
                                                                                                                               where 
                                                                                                                                                                 \begin{equation*}
                                                                                                                                                                 \Theta(A,Q_r(y,s))(x,t):=\mathop {\sup }\limits_{\zeta\in\mathbb{R}^N\backslash\{0\}}\frac{|A(x,t,\zeta)-\overline{A}_{Q_r(y,s)}(\zeta)|}{|\zeta|},
                                                                                                                                                                 \end{equation*}
                                                                                                                                                                 and  $\overline{A}_{Q_r(y,s)}(\zeta)$ is denoted the average of $A(.,.,\zeta)$ over the cylinder $Q_r(y,s)$, i.e,
                                                                                                                                                                 \begin{equation*}
                                                                                                                                                                 \overline{A}_{Q_r(y,s)}(\zeta):=\fint_{Q_r(y,s)}A(x,t,\zeta)dxdt=\frac{1}{|Q_r(y,s)|}\int_{Q_r(y,s)}A(x,t,\zeta)dxdt.
                                                                                                                                                                \end{equation*} 
                                                                                                                                                   
                                                                                           \end{remark}
                                                                                           Next results are corresponding estimates of gradient for domain $\mathbb{R}^N\times (0,\infty)$ or whole $\mathbb{R}^{N+1}$.  
 \begin{theorem}\label{5hh040420149}
 Let $\theta\in (2,N+2)$ be in Theorem \ref{5hh0701201411} and $\omega\in\mathfrak{M}(\mathbb{R}^{N+1})$. There exists a distributional solution $u$ of \eqref{5hhparabolic2'}  with data $\mu=\omega$ such that the following statements hold
 \begin{description}
 \item[a.] For any $\frac{N+2}{N+1}< p<\theta$ and $0<s\leq\infty$, \begin{equation}
                      \label{5hh040420141}
             |||\nabla u|||_{L^{p,s}(\mathbb{R}^{N+1})}\lesssim_{p,s} ||\mathbb{M}_1[|\omega|]||_{L^{p,s}(\mathbb{R}^{N+1}                                                                )}.
 \end{equation}                                                                                                                                                          \item[b.] For any  $\frac{N+2}{N+1}< p< \theta $ and $0<s \leq \infty$,                        
                                                                                                                                                                                                                                 $2-\gamma_0<\gamma<N+2$ and $\gamma\leq \frac{N+2}{p}+1,$
                                                                                                                                                                                                           \begin{align}
                                                                                                                                                                                                                                 \nonumber&|||\nabla u|||_{L_{*}^{p,s;p(\gamma-1)}(\mathbb{R}^{N+1})}\lesssim_{p,s,\gamma}||\mathbb{M}_{\gamma}[|\omega|]||_{L^\infty(\mathbb{R}^{N+1})}\\&~~~~~+\sup_{R>0, (y_0,s_0)\in \mathbb{R}^{N+1}}\left(R^{\frac{p(\gamma-1)-N-2}{p}}||\mathbb{M}_1[\chi_{\tilde{Q}_{R}(y_0,s_0)}|\omega|]||_{L^{p,s}(\tilde{Q}_{R}(y_0,s_0))}\right),
                             \label{5hh040420143}                       \end{align} 
provided $\mathbb{I}_2[|\omega|](x_0,t_0)<\infty$ for some $(x_0,t_0)\in \mathbb{R}^{N+1}$.\\                             
 Also, if  $\omega\in L_{*}^{\frac{(\gamma-1)p}{\gamma},\frac{(\gamma-1)s}{\gamma};(\gamma-1)p}(\mathbb{R}^{N+1})$ with $p>\frac{\gamma}{\gamma-1}$ then   
                                                                                                                                              \begin{equation} \label{5hh040420144}                           |||\nabla u|||_{L_{*}^{p,s;(\gamma-1)p}(\mathbb{R}^{N+1})}\lesssim_{p,s,\gamma}||\omega||_{L_{*}^{\frac{(\gamma-1)p}{\gamma},\frac{(\gamma-1)s}{\gamma};(\gamma-1)p}(\mathbb{R}^{N+1})},                     
                                                                                                                                              \end{equation} 
                                                                                                                                                                                                                                                                                                                                                           for some $\gamma_0=\gamma_0(N,\Lambda_1,\Lambda_2)\in (0,\frac{1}{2}]$. 
                                                                                                                                                                                                   \item[c.] The statement \textbf{c} in Theorem \ref{5hh1203201417} is true.
 \end{description}

 \end{theorem}               
             \begin{remark}\label{5hh0404201410} Let $s>1$. For  $\omega\in\mathfrak{M}^+(\mathbb{R}^{N+1})$, $\mathbb{I}_1[\omega]\in L^{s,\infty}(\mathbb{R}^{N+1})$ implies  $ \mathbb{I}_2[|\omega|]<\infty$ a.e in $\mathbb{R}^{N+1}$ if and only if $s\leq N+2$.
                \end{remark}                                           
              \begin{theorem} \label{5hh040420145}  Suppose that $A$ satisfies  \eqref{5hhcondc}. Let $s_0$ be in Theorem \ref{5hh2410131}. Let $\omega\in\mathfrak{M}(\mathbb{R}^{N+1})$ with $\mathbb{I}_2[|\omega|](x_0,t_0)<\infty$ for some $(x_0,t_0)\in \mathbb{R}^{N+1}$. There exists a distributional solution of \eqref{5hhparabolic2'} with data $\mu=\omega$  such that following statements hold,
             \begin{description}
             \item[a.]For any $w\in A_\infty$, $1\leq q<\infty$, $0<s\leq\infty$ we find  $\delta=\delta(N,\Lambda_1,\Lambda_2,q,s, [w]_{A_\infty})\in (0,1)$ such that if  $[A]_{s_0}^\infty\le \delta$ then                                       
                                                                                             \begin{equation}\label{5hh040420146}
                                                                                             |||\nabla u|||_{L^{q,s}(\mathbb{R}^{N+1},dw)}\lesssim C_1 ||\mathbb{M}_1[|\omega|]||_{L^{q,s}(\mathbb{R}^{N+1},dw)}.
                                                                                             \end{equation}                                                                                               Here $C_1$ depends  on $q,s, [w]_{A_\infty}$.                                                                                                                                                                                   \item[b.] For any $\frac{N+2}{N+1}<q<\infty$, $0<s\leq\infty$ and $0<\kappa\leq N+2$ we find  $\delta=\delta(N,\Lambda_1,\Lambda_2,q,s,\kappa)\in (0,1)$  such that if  $[A]_{s_0}^{\infty}\le \delta$ then 
                                                                                                                                                                                                                                                        \begin{equation}\label{5hh040420147}                                 
                                                                                                                                                                                                                                                        |||\nabla u|||_{L_{*}^{q,s;\kappa}(\mathbb{R}^{N+1})}\lesssim_{q,s,\kappa} ||\mathbb{M}_1[|\omega|]||_{L_{*}^{q,s;\kappa}(\mathbb{R}^{N+1})}.
                                                                                                                                                                                                                                                        \end{equation} 
                                                                                                                                                                                             
                                                                                                                                                                                               \item[c.] For any $\frac{N+2}{N+1}<q<\infty$, $0<s\leq\infty$ and $0<\vartheta\leq N$ one find  $\delta=\delta(N,\Lambda_1,\Lambda_2,q,s,\vartheta)\in (0,1)$  such that if  $[A]_{s_0}^{\infty}\le \delta$ then 
                                                                                                                                                                                                                                                                                                                                                       \begin{equation}\label{5hh05052014113}                                 
                                                                                                                                                                                                                                                                                                                                                       |||\nabla u|||_{L_{**}^{q,s;\vartheta}(\mathbb{R}^{N+1})}\lesssim_{q,s,\vartheta} ||\mathbb{M}_1[|\omega|]||_{L_{**}^{q,s;\vartheta}(\mathbb{R}^{N+1})}.
                                                                                                                                                                                                                                                                                                                                                       \end{equation} 
                                                                                                                                                                                                                                    Especially, when $q=s$ and $0<\vartheta<\min\{N,q\}$,  there holds for any ball $B_\rho\subset\mathbb{R}^N$
                                                                                                                                                                                                                                                                                                                                                                                                       \begin{equation}\label{5hh0505201413}                                 
                                                                                                                                                                                                                                                                                                                                                                                                             \left(\int_{\mathbb{R}}|\text{osc}_{B_\rho}u(t)|^qdt\right)^{\frac{1}{q}} \lesssim_{q,s,\vartheta} \rho^{1-\frac{\vartheta}{q}} ||\mathbb{M}_1[|\omega|]||_{L_{**}^{q;\vartheta}(\mathbb{R}^{N+1})}.
                                                                                                                                                                                                                                                                                                                                                                                                                                                                                                                                                            \end{equation} 
                                                                                                                                                                                                                                                                                                                                                                                                                                                                                                                                                                                                 
                         \item[d.] The statement \textbf{c} in Theorem \ref{5hh1203201417} is true.              
             \end{description} 
                                                                       
                                                                               \end{theorem}
                                                                           We will present some  estimates for norms of $\mathbb{M}_1[\omega]$ in $L_{*}^{q;\kappa}(\mathbb{R}^{N+1})$ and $L_{**}^{q;\vartheta}(\mathbb{R}^{N+1})$ in section 8.

Final part, we prove the  existence solutions for the quasilinear Riccati type parabolic problems
    \begin{equation}\label{5hh0701201410}
                      \left\{
                      \begin{array}
                      [c]{l}%
                      {u_{t}}-\operatorname{div}(A(x,t,\nabla u))=|\nabla u|^q+\mu~~~\text{in }\Omega_T,\\
                      {u}=0\qquad\text{on }\partial\Omega\times (0,T),\\
                      u(0)=\sigma~~~~\text{in }~\Omega,
                      \end{array}
                      \right.  
                      \end{equation}
               \begin{align}\label{5hh270420142}
               {u_t} - \operatorname{div}\left( {A(x,t,\nabla u)} \right) = |\nabla u|^q+\mu ~\text{in }~\mathbb{R}^{N+1},
               \end{align}
 where $q>1$. 
 
 The following result is considered in subcritical case this means $1<q<\frac{N+2}{N+1}$, to obtain existence solutions in this case we need data $\mu,\sigma$ to be  finite measures and small enough.
                       \begin{theorem}\label{5hh260320144} Let $1<q<\frac{N+2}{N+1}$ and $\mu\in\mathfrak{M}_b(\Omega_T)$, $\sigma\in\mathfrak{M}_b(\Omega)$.  There exists   $\varepsilon_0=\varepsilon_0(N,\Lambda_1,\Lambda_2,q)>0$  such that if $$|\Omega_T|^{-1+\frac{q'}{N+2}}\left(|\mu|(\Omega_T)+|\omega|(\Omega)\right)\leq \varepsilon_0,$$  the  problem \eqref{5hh0701201410}
                                             has a distribution  solution $u$, satisfied
                                             \begin{align*}
                                             |||\nabla u|||_{L^{\frac{N+2}{N+1},\infty}(\Omega_T)}\lesssim_q|\mu|(\Omega_T)+|\omega|(\Omega).
                                             \end{align*}
                                            
                                              \end{theorem} 
In the next results are concerned in critical and supercritical case.                                               
                      \begin{theorem}\label{5hh0701201413}
                              Suppose that  $\mathbb{R}^N\backslash\Omega$ satisfies uniformly $2-$thick with  constants $c_0,r_0$. Let $\theta$ be in Theorem \ref{5hh0701201411},  $q\in\left(\frac{N+2}{N+1},\frac{N+2+\theta}{N+2}\right)$, $\mu\in \mathfrak{M}_b(\Omega_T)$ and $\sigma\in \mathfrak{M}_b(\Omega)$. Assume that $\sigma\equiv0$ when $q\geq\frac{N+4}{N+2}$. There exists $\varepsilon_0=\varepsilon_0(N,\Lambda_1,\Lambda_2,q,c_0,T_0/r_0)>0$ such that if $$ ||\mathbb{I}_1[|\mu|]||_{L^{(N+2)(q-1),\infty}(\mathbb{R}^{N+1})}+ ||\mathbf{I}_{\frac{2}{(N+2)(q-1)}-1}[|\sigma|]||_{L^{(N+2)(q-1)}(\mathbb{R}^{N})}\leq \varepsilon_0,$$ then the problem \eqref{5hh0701201410}                              
                                                has a  distributional solution u satisfying
                                  \begin{equation}
                                  |||\nabla u|||_{L^{(q-1)(N+2),\infty}(\Omega_T)}\lesssim C,
                                  \end{equation}
                                  for some $C=C(q,c_0,T_0/r_0)$. 
                                 
                              \end{theorem}
   We remark that a  necessary  condition the existence of $\sigma\in \mathfrak{M}_b(\Omega)\backslash\{0\}$ with $\mathbb{M}_1[|\sigma|\otimes\delta_{\{t=0\}}]\in L^{(N+2)(q-1),\infty}(\mathbb{R}^{N+1})$ is $\frac{N+2}{N+1}\leq q<\frac{N+4}{N+2}$.  
                 
 \begin{theorem}\label{5hh2410136}
            Suppose that $A$ satisfies \eqref{5hhcondc}.  Let $s_0$ be the constant in Theorem \ref{5hh2410131}. Let $q\geq\frac{N+2}{N+1}$ and $\mu\in\mathfrak{M}_b(\Omega_T),\sigma\in\mathfrak{M}_b(\Omega)$, set $\omega=|\mu|+|\sigma|\otimes\delta_{\{t=0\}}$. There exists   $\delta=\delta(N,\Lambda_1,\Lambda_2,q)\in (0,1)$  such that $\Omega$ is $(\delta,R_0)$-Reifenberg flat domain $\Omega$ and $[A]_{s_0}^{R_0}\le \delta$ for some $R_0$  and the following holds. The problem  \eqref{5hh0701201410}           
           has a distributional solution $u$ if  \begin{equation}\label{5hh090120146}
                                 \omega(K)\leq \varepsilon_0\text{Cap}_{\mathcal{G}_1,q'}(K)~~\text{ for every compact subset } K\subset \mathbb{R}^{N+1},
                                 \end{equation}
                                  with a constant $\varepsilon_0>0$ small enough.   
           Moreover, $u$   satisfies
           $$\int_K |\nabla u|^{q}dxdt\lesssim C \text{Cap}_{\mathcal{G}_1,q'}(K)~~\text{ for every compact subset } K\subset \mathbb{R}^{N+1},
$$where $C$ depends on $q,T_0/R_0,T_0$.
                         
            \end{theorem} 
            Since $\text{
            Cap}_{\mathcal{G}_1,s}(B_r(0)\times\{t=0\})=0$ for all $r>0$ and $0<s\leq 2$, see Remark \ref{5hh270320146} thus if there is $\sigma\in \mathfrak{M}_b(\Omega)\backslash\{0\}$ satisfying $ (|\sigma|\otimes\delta_{\{t=0\}})(E)\leq  \text{Cap}_{\mathcal{G}_1,s}(E)$  for every compact subsets $E\subset \mathbb{R}^{N+1}$ then we must have $s>2$. \\\\
    We have an \textbf{important} Proposition.
    \begin{proposition}
    All the existence results (in the bounded domain case) that we have obtained in the above Theorems are renormalized solutions if we further assume that $\sigma\in L^1(\Omega).$
    \end{proposition}                                 The following theorem is an existence result to the problem  \eqref{5hh270420142}.                         
     \begin{theorem}\label{5hh0404201412} Let  $\theta\in (2,N+2)$ be in Theorem \ref{5hh0701201411},  $q\in\left(\frac{N+2}{N+1},\frac{N+2+\theta}{N+2}\right)$ and $\omega\in\mathfrak{M}(\mathbb{R}^{N+1})${\color{red}{.}}  There exists $\varepsilon_0=\varepsilon_0(N,\Lambda_1,\Lambda_2,q)>0$ such that if \begin{align*}
                                  ||\mathbb{I}_1[|\omega|]||_{L^{(N+2)(q-1),\infty}(\mathbb{R}^{N+1})}\leq \varepsilon_0,
                                  \end{align*} then the problem  \eqref{5hh270420142} 
                                                                        has a  distributional solution $u\in L^1_{loc}(\mathbb{R};W^{1,1}_{loc}(\mathbb{R}^N))$ such that 
                                     \begin{equation}
                |||\nabla u|||_{L^{(q-1)(N+2),\infty}(\mathbb{R}^{N+1})}\lesssim_q 1\label{5hh0404201415}.
                                     \end{equation}
         Furthermore, if $\text{supp}(\mu)\subset \mathbb{R}^{N}\times [0,\infty)$ then $u=0$ in $\mathbb{R}^N\times (-\infty,0)$.             
        
                                  \end{theorem}                                                                                                 
\begin{theorem}\label{5hh0404201417}
  Suppose that $A$ satisfies \eqref{5hhcondc}.  Let $s_0$ be the constant in Theorem \ref{5hh2410131}.  Let $q>\frac{N+2}{N+1}$ and $\mu\in\mathfrak{M}(\mathbb{R}^{N+1})$. There exists   $\delta=\delta(N,\Lambda_1,\Lambda_2,q)\in (0,1)$  such that  $[A]_{s_0}^{\infty}\le \delta$ for some $R_0$  and the following holds. The problem  \eqref{5hh270420142}           
  has a distributional solution $u$ if  \begin{equation}\label{Z7}
  	|\mu|(K)\leq \varepsilon_0\text{Cap}_{\mathcal{H}_1,q'}(K)~~\text{ for every compact subset } K\subset \mathbb{R}^{N+1},
  \end{equation}
  with a constant $\varepsilon_0>0$ small enough.     Moreover, $u$   satisfies
  $$\int_K |\nabla u|^{q}dxdt\lesssim_q \text{Cap}_{\mathcal{H}_1,q'}(K)~~\text{ for every compact subset } K\subset \mathbb{R}^{N+1}.
  $$
                                                                                              Furthermore, if $\text{supp}(\mu)\subset \mathbb{R}^{N}\times [0,\infty)$ then $u=0$ in $\mathbb{R}^N\times (-\infty,0)$.                                                                                                                                                                                     \end{theorem}                                                                                                                                                                                  
    \section{The notion of solutions and some properties}
    
    ~~Although the notion of renormalized solutions  becomes more and more familiar in the theory of quasilinear parabolic  equations with measure data, it is still necessary to present below some main aspects concerning this notion. Let $\Omega$ be a bounded domain in $\mathbb{R}^N$, $(a,b)\subset\subset \mathbb{R}$. If $\mu\in \mathfrak{M}_b(\Omega\times (a,b))$, we denote by $\mu^+$ and $\mu^-$ respectively its positive and negative part. We denote by $\mathfrak{M}_0(\Omega\times (a,b))$ the space of measures in $\Omega
    \times (a,b)$ which are absolutely continuous with respect to the $C_{2}$-capacity defined on a compact set $K\subset\Omega\times(a,b)$ by
    \begin{equation}\label{5hhD1}
    C_{2}(K,\Omega\times (a,b))=\inf\left\{||\varphi||_W:\varphi\geq \chi_K,\varphi\in C^\infty_c(\Omega\times (a,b))\right\}.
    \end{equation}
    Here $W=\{z:z\in L^2(a,b,H^1_0(\Omega)),z_t\in L^2(a,b,H^{-1}(\Omega))\}$ is endowed  with the  norm $||\varphi||_{W}=||\varphi||_{L^2(a,b,H^1_0(\Omega))}+||\varphi_t||_{L^2(a,b,H^{-1}(\Omega))}$ and $\chi_K$ is the characteristic function of $K$.
    
    We also denote $\mathfrak{M}_s(\Omega\times (a,b))$ the space of measures in $\Omega\times (a,b)$ with support on a set of zero $C_{2}$-capacity. Classically, any $\mu\in \mathfrak{M}_b(\Omega\times (a,b))$ can be written in a unique way under the form $\mu=\mu_0+\mu_s$ where $\mu_0\in (\mathfrak{M}_0\cap\mathfrak{M}_b)(\Omega\times (a,b))$ and $\mu_s\in \mathfrak{M}_s(\Omega\times (a,b))$.
    We recall that any $\mu_0\in (\mathfrak{M}_0\cap \mathfrak{M}_b)(\Omega\times (a,b))$ can be decomposed  under the form $\mu_0=f-\operatorname{div}g+h_t$ where $f\in L^1(\Omega\times (a,b))$, $g\in L^{2}(\Omega\times (a,b),\mathbb{R}^N)$ and $h\in L^2(a,b,H^1_0(\Omega))$ and $(f,g,h)$ is said to be decomposition of $\mu_0$.
    Set ${{\widehat{\mu_{0}%
    }}}=\mu_{0}-h_{t}=f-\operatorname{div}g$. In the general case ${{\widehat{\mu
    _{0}}}}\notin\mathfrak{M}(\Omega\times (a,b)),$ but we write, for convenience,
    \begin{align*}
    \int_{\Omega\times (a,b)}{wd{\widehat{\mu_{0}}}}:=\int_{\Omega\times (a,b)}(fw+g.\nabla w)dxdt,\qquad\forall w\in
    L^2(a,b,H^1_0(\Omega)){\cap}L^{\infty}(\Omega\times (a,b)).
    \end{align*}
    
    However, for $\sigma\in\mathfrak{M}_b(\Omega)$ and $t_0\in (a,b)$ then $\sigma\otimes\delta_{\{t=t_0\}}\in\mathfrak{M}_0(\Omega\times (a,b))$ if and only if $\sigma\in L^1(\Omega)$, see \cite{55DrPoPr}. We also have that for $\sigma\in\mathfrak{M}_b(\Omega)$, $\sigma\otimes \chi_{[a,b]}\in \mathfrak{M}_0(\Omega\times (a,b))$ if and only if $\sigma$ is absolutely continuous with respect to the $\text{Cap}_{\mathbf{G}_1,2}$-capacity, see \cite{55DrPoPr}.
     
    For $k>0$ and $s\in\mathbb{R}$ we set $T_k(s)=\max\{\min\{s,k\},-k\}$. We recall that if $u$ is a measurable function defined and finite a.e. in $\Omega\times (a,b)$, such that $T_k(u)\in L^2(a,b,H^1_0(\Omega))$ for any $k>0$, there exists a measurable function $v:\Omega\times (a,b)\to \mathbb{R}^N$ such that $\nabla T_k(u)=\chi_{ |u|\leq k}v$ 
    a.e. in $\Omega\times (a,b)$ and for all $k>0$. We define the gradient of $u$ by $v=\nabla u$.\medskip\\ We recall the definition of a renormalized solution given in \cite{55Pe08}.
    \begin{definition}
    \label{5hhdefin}\label{5hh100320142}Suppose that $B\in C(\mathbb{R}\times\mathbb{R}^N,\mathbb{R})$. Let  $\mu=\mu_{0}+\mu_{s}%
    \in\mathfrak{M}_{b}(\Omega\times (a,b))$ and $\sigma\in L^1(\Omega)$. A measurable function $u$ is a
    \textbf{renormalized solution} of
    \begin{equation}\label{5hh050420141}
                                          \left\{
                                          \begin{array}
                                          [c]{l}%
                                          {u_{t}}-\operatorname{div}(A(x,t,\nabla u))=B(u,\nabla u)+\mu~\text{in }\Omega\times (a,b),\\ 
                                                                u=0~~~~~~~\text{on}~~
                                                                                                 \partial\Omega \times (a,b),
                                                                                                    \\
                                                                                                    u(a) = \sigma~~~\text{in}~~ \Omega, \\ 
                                          \end{array}
                                          \right.  
                                          \end{equation}   if
    there exists a decomposition $(f,g,h)$ of $\mu_{0}$ such that
    \begin{align}
    \nonumber &v=u-h\in L^{s}(a,b,W_{0}^{1,s}(\Omega))\cap L^{\infty}(a,b,L^{1}%
    (\Omega))~\forall s\in\left[  1,\frac{N+2}{N+1}\right),\\&~~~~ T_{k}(v)\in
    L^2(a,b,H^1_0(\Omega))~\forall k>0,B(u,\nabla u)\in L^1(\Omega\times (a,b)), \label{5hhdefv}%
    \end{align}
    and:\medskip
    
    (i) for any $S\in W^{2,\infty}(\mathbb{R})$ such that $S^{\prime}$ has compact
    support on $\mathbb{R}$, and $S(0)=0$,%
    \begin{align}\nonumber
    &-\int_{\Omega}S(\sigma)\varphi(a)dx-\int_{\Omega\times (a,b)}{{\varphi_{t}}S(v)}dxdt+\int%
    _{\Omega\times (a,b)}{S^{\prime}(v)A(x,t,\nabla u)\nabla\varphi}dxdt\\&+\int_{\Omega\times (a,b)}{S^{\prime\prime
    }(v)\varphi A(x,t,\nabla u).\nabla v}dxdt=\int_{\Omega\times (a,b)}S^{\prime}(v)\varphi B(u,\nabla u)dxdt+\int_{\Omega\times (a,b)}{S^{\prime}(v)\varphi
    d{\widehat{\mu_{0}}}}, \label{5hhrenor}%
    \end{align}
    for any $\varphi\in L^2(a,b,H^1_0(\Omega))\cap L^{\infty}(\Omega\times (a,b))$ such that $\varphi_{t}\in L^2(a,b,H^{-1}(\Omega))+L^{1}(\Omega\times (a,b))$ and $\varphi(.,b)=0$;\medskip
    
    (ii) for any $\phi\in C(\overline{\Omega}\times [a,b]),$%
    \begin{equation}
    \label{5hhrenor2}\lim_{m\rightarrow\infty}\frac{1}{m}\int\limits_{\left\{  m\leq
    v<2m\right\}  }{\phi A(x,t,\nabla u)\nabla v}dxdt =\int_{\Omega\times (a,b)}\phi d\mu_{s}^{+}~~\text{ and }
    \end{equation}
    \begin{equation}
    \label{5hhrenor3}\lim_{m\rightarrow\infty}\frac{1}{m}\int\limits_{\left\{  -m\geq
    v>-2m\right\}  }{\phi A(x,t,\nabla u)\nabla v}dxdt =\int_{\Omega\times (a,b)}\phi d\mu_{s}^{-}.
    \end{equation}
    
    \end{definition}
    \begin{remark}\label{5hh070420143} If $\mu\in L^1(\Omega\times(a,b))$, then we have the following estimates: 
    \begin{align*}
    &||u||_{L^{\frac{N+2}{N},\infty}(\Omega\times(a,b))}\lesssim||\sigma||_{L^1(\Omega)}+|\mu|(\Omega\times(a,b)),\\&
    |||\nabla u|||_{L^{\frac{N+2}{N+1},\infty}(\Omega\times(a,b))}\lesssim||\sigma||_{L^1(\Omega)}+|\mu|(\Omega\times(a,b)),
    \end{align*}
   see \cite[Remark 4.9]{55VH}.\\
    In particular, 
    \begin{align*}
    &||u||_{L^1(\Omega\times(a,b))}\lesssim(\text{diam}(\Omega)+(b-a)^{1/2})^2\left(||\sigma||_{L^1(\Omega)}+|\mu|(\Omega\times(a,b))\right),\\&
    |||\nabla u|||_{L^1(\Omega\times(a,b))}\lesssim(\text{diam}(\Omega)+(b-a)^{1/2})\left(||\sigma||_{L^1(\Omega)}+|\mu|(\Omega\times(a,b))\right).
    \end{align*}
    \end{remark}
    \begin{remark} \label{5hh240220142}It is easy to see that $u$ is a weak solution to the problem \eqref{5hh050420141} in $\Omega\times (a,b)$ with $\mu\in L^2(\Omega\times (a,b))$, $\sigma\in H^1_0(\Omega)$ and $B\equiv 0$ then $U=\chi_{[a,b]} u$ is a unique  renormalized solution of 
     \begin{equation*}\left\{ \begin{array}{l}
                         {U_t} - \operatorname{div}\left( {A(x,t,\nabla U)} \right) = \chi_{(a,b)}\mu+(\chi_{[a,b)}\sigma)_t ~\text{in}~\Omega\times (c,b), \\ 
                         U=0~~~~~~~~~~
                         \text{ on }\partial\Omega \times (c,b),
                         \\
                         U(c) = 0~~~~~~~
                                       \text{ in }\Omega, \\ 
                         \end{array} \right.\end{equation*}
                         for any $c<a$.
    \end{remark}
    \begin{remark}\label{5hh120320145}
    Let $\Omega'\subset\subset\Omega$ and $a<a'<b'<b$. For a nonnegative function $\eta\in C^\infty_c(\Omega'\times(a',b'))$, from \eqref{5hhrenor} we have 
    \begin{align*}
    &\left(\eta S(v)\right)_t-\eta_tS(v)+S'(v)A(x,t,\nabla u)\nabla \eta-\operatorname{div}\left(S'(v)\eta A(x,t,\nabla u)\right)\\&~~~+S''(v)\eta A(x,t,\nabla u)\nabla v=S'(v)\eta f+\nabla\left(S'(v)\eta\right).g-\operatorname{div}\left(S'(v)\eta g\right)
    \end{align*}
    in $\mathcal{D}'(\Omega'\times (a',b'))$. 
    Thus, $\left(\eta S(v)\right)_t\in L^{2}(a',b',H^{-1}(\Omega'))+L^1(D)$ and we have the following estimate 
    \begin{align}\nonumber
    &||\left(\eta S(v)\right)_t||_{L^{2}(a',b',H^{-1}(\Omega'))+L^1(D)}\lesssim||S||_{W^{2,\infty}(\mathbb{R})}\left( ||\eta_t v||_{L^1(D)} \right.
    \\\nonumber&~~~~~~+|||\nabla u||\nabla \eta|||_{L^1(D)}+||\eta|\nabla u|\chi_{|v|\leq M}||_{L^2(D)}+||\eta|\nabla u||\nabla v|\chi_{|v|\leq M}||_{L^2(D)}
    \\&~~~~~~+||\eta f||_{L^1(D)}+||\eta|\nabla u|^2\chi_{|v|\leq M|}||_{L^1(D)}+||\eta|g|^2||_{L^1(D)}\left.+||\eta|g|||_{L^2(D)}  \right)\label{5hh100320141}
    \end{align}
    with $D=\Omega'\times (a',b')$ and $\text{supp}(S')\subset [-M,M]$.
    \end{remark}
    We recall the following important results, see in \cite{55VH}.
     \begin{proposition}
             \label{5hhmun} Let $\{\mu_{n}\}$ be a bounded in $\mathfrak{M}_b(\Omega\times (a,b))$ and $\sigma_{n}$ be a bounded in $L^1(\Omega)$. 
             Let $u_{n}$ be a renormalized solution of \eqref{5hhparabolic1} with data $\mu_{n}=\mu_{n,0}%
              +\mu_{n,s}$ relative to a decomposition $(f_{n},g_{n},h_{n})$
              of $\mu_{n,0}$ and initial data $\sigma_{n}$. If $\{f_{n}\}$ is bounded in
              $L^{1}(\Omega_T)$, $\{g_{n}\}$ is  bounded in $L^{2}(\Omega\times(a,b),\mathbb{R}^N)$ and $\{h_{n}\}$ is
              convergent in $L^2(a,b,H^1_0(\Omega))$,    
             then, up to a subsequence, $\{u_{n}\}$ converges  to a
             function $u$ in $L^1(\Omega\times(a,b))$. Moreover,  if $\{\mu_{n}\}$ is a bounded in $L^1(\Omega\times(a,b))$ then  $\{u_{n}\}$ is convergent in  $L^s(a,b,W^{1,s}_0(\Omega))$ for any $s\in\left[1,\frac{N+2}{N+1}\right)$.
             \end{proposition}
              We say that a sequence of bounded measures $\{\mu_n\}$ in $\Omega\times (a,b)$ converges to a bounded measure $\mu$ in $\Omega\times (a,b)$ in the {\it narrow topology} of measures if 
              \begin{equation*}
              \lim_{n\to\infty}\int_{\Omega\times(a,b)}\varphi d\mu_n=\int_{\Omega\times(a,b)}\varphi d\mu ~~\text{ for all } \varphi\in C(\Omega\times(a,b))\cap L^\infty(\Omega\times(a,b))).
              \end{equation*}         
    We recall the following fundamental stability result of \cite{55VH}.
    \begin{theorem}
             \label{5hhsta} Suppose that $B\equiv 0$. Let $\sigma\in L^1(\Omega)$ and 
             \[
             \mu=f-\operatorname{div}g+h_{t}+\mu_{s}^{+}-\mu_{s}^{-}\in\mathfrak{M}_{b}%
             ({\Omega\times(a,b)}),
             \]
             with $f\in L^{1}(\Omega\times(a,b)),g\in L^{2}(\Omega\times(a,b),\mathbb{R}^N)$, $h\in L^2(a,b,H^1_0(\Omega))$ and $\mu_{s}%
             ^{+},\mu_{s}^{-}\in\mathfrak{M}_{s}^{+}(\Omega\times(a,b))$. Let $\sigma_{n}\in L^1(\Omega)$ and
             \begin{align*}         
             \mu_{n}=f_{n}-\operatorname{div} g_{n}+(h_{n})_{t}+\rho_{n}-\eta_{n}%
                      \in\mathfrak{M}_{b}(\Omega\times(a,b))
             \end{align*}       
             with \ $f_{n}\in L^{1}(\Omega\times(a,b)),g_{n}\in L^{2}(\Omega\times(a,b),\mathbb{R}^N),h_{n}\in L^2(a,b,H^1_0(\Omega)),$ and
             $\rho_{n},\eta_{n}\in\mathfrak{M}_{b}^{+}(\Omega\times(a,b)),$ such that
             \begin{align*}
              \rho_{n}=\rho_{n}^{1}-\operatorname{div}\rho_{n}^{2}+\rho_{n,s},\qquad\eta
                      _{n}=\eta_{n}^{1}-\operatorname{div}   \eta_{n}^{2}+\eta_{n,s},
             \end{align*}       
             with $\rho_{n}^{1},\eta_{n}^{1}\in L^{1}(\Omega\times(a,b)),\rho_{n}^{2},\eta_{n}^{2}%
             \in L^{2}(\Omega\times(a,b),\mathbb{R}^N)$ and $\rho_{n,s},\eta_{n,s}\in\mathfrak{M}_{s}%
             ^{+}(\Omega\times(a,b)).$\\ Assume that $\{\mu_n\}$ is a bounded in $\mathfrak{M}_b(\Omega\times(a,b))$,  $\{\sigma_n\}, \{ f_{n}\}, \{ g_{n}\}, \{h_{n}\} $ converge to $\sigma, f, g, h$ in $L^1(\Omega)$,weakly in
             $L^{1}(\Omega\times(a,b))$,in
             $L^{2}(\Omega\times(a,b),\mathbb{R}^N)$,in $L^2(a,b,H^1_0(\Omega))$ respectively and  $\{  \rho_{n}\}, \{\eta_{n}\}$ converge to $\mu_{s}^{+}, \mu_{s}^{-}$  in the narrow
             topology of measures; and $\left\{  \rho_{n}%
             ^{1}\right\}  ,\left\{  \eta_{n}^{1}\right\}  $ are bounded in $L^{1}(\Omega\times(a,b))$, and
             $\left\{  \rho_{n}^{2}\right\}  ,\left\{  \eta_{n}^{2}\right\}  $ bounded in
             $L^{2}(\Omega\times(a,b),\mathbb{R}^N)$. \\Let $\left\{  u_{n}\right\}  $ be a sequence of
             renormalized solutions of
             \begin{equation}
             \left\{
             \begin{array}
             [c]{l}%
             {(u_{n})_t}-\operatorname{div}(A(x,t,\nabla u_n))=\mu_{n}~\text{in }\Omega\times (a,b),\\ 
                                 u_n=0~~~~~~\text{ on }\partial\Omega \times (a,b),
                                 \\
                                 u_n(a) = \sigma_{n}~\text{ in } \Omega, \\         
             \end{array}
             \right.  \label{5hhpmun}%
             \end{equation}
             relative to the decomposition $(f_{n}+\rho_{n}^{1}-\eta_{n}%
             ^{1},g_{n}+\rho_{n}^{2}-\eta_{n}^{2},h_{n})$ of $\mu_{n,0}.$  Let $v_{n}=u_{n}-h_{n}.$ Then up to a subsequence, $\left\{
             u_{n}\right\}  $ converges $a.e.$ in $\Omega\times(a,b)$ to a renormalized solution $u$ of \eqref{5hh050420141},
             and $\left\{  v_{n}\right\}  $ converges $a.e.$ in $\Omega\times(a,b)$ to $v=u-h.$ Moreover,
             $\left\{  \nabla u_{n}\right\}  ,\left\{  \nabla v_{n}\right\}  $ converge
             respectively to $\nabla u,\nabla v$ a.e in $\Omega\times(a,b),$ and $\left\{  T_{k}(v_{n})\right\}  $ converges to 
             $T_{k}(v)$ strongly in $L^2(a,b,H^1_0(\Omega))$ for any $k>0$.\bigskip
             \end{theorem}
             In order to apply above theorem, we need some the following properties concerning approximate measures of $\mu
             \in\mathfrak{M}_{b}^{+}(\Omega\times (a,b))$, see in \cite{55VH}.
             \begin{proposition}
             \label{5hhatt}Let $\mu=\mu_0+\mu_{s}\in\mathfrak{M}_{b}^{+}(\Omega\times(a,b))$ with $\mu_0\in \mathfrak{M}_0(\Omega\times(a,b))\cap\mathfrak{M}_b^+(\Omega\times(a,b))$ and $\mu_s\in \mathfrak{M}_s^+(\Omega\times(a,b))$. Let $\left\{  \varphi_{n}\right\}$ be sequence of standard mollifiers in $\mathbb{R}^{N+1}$. Then, 
              there exist a decomposition $(f,g,h)$ of $\mu_0$ and $f_{n},g_{n},h_{n}\in C_{c}^{\infty}(\Omega\times(a,b))$, 
                      $\mu_{n,s}\in (C_{c}^{\infty}\cap\mathfrak{M}^{+}_b)(\Omega\times(a,b))$ such that
                     $\left\{  f_{n}\right\}  ,\left\{  g_{n}\right\}  ,\left\{  h_{n}\right\}
                      $ strongly converge to $f,g,h$ in $L^{1}(\Omega\times(a,b)),L^{2}(\Omega\times(a,b),\mathbb{R}^N)$ and $L^2(a,b,H^1_0(\Omega))$;  $\mu_{n}=f_{n}-\operatorname{div}%
                                        g_{n}+(h_{n})_{t}+\mu_{n,s}, \mu_{n,s}$  converge to $\mu,\mu_s$ in the narrow topology respectively;
           $0\leq\mu_n \leq \varphi_n*\mu$ and                    \begin{equation*}
             ||f_{n}||_{L^{1}(\Omega\times(a,b))}+\left\Vert g_{n}\right\Vert _{L^{2}(\Omega\times(a,b),\mathbb{R}^N)}+||h_{n}||_{L^2(a,b,H^1_0(\Omega))}+\mu_{n,s}(\Omega\times(a,b))
             \leq 2\mu(\Omega\times(a,b)).\end{equation*}
             \end{proposition}
             \begin{proposition}
             \label{5hhP5}Let $\mu=\mu_0+\mu_{s}, \mu_n=\mu_{n,0}+\mu_{n,s}\in\mathfrak{M}_{b}^{+}(\Omega\times(a,b))$ with $\mu_0,\mu_{n,0}\in (\mathfrak{M}_0\cap\mathfrak{M}_b^+)(\Omega\times(a,b))$ and $\mu_{n,s}, \mu_s\in \mathfrak{M}_s^+(\Omega\times(a,b))$ such that
             $\left\{  \mu_{n}\right\}  $  nondecreasingly  converges to $\mu$ in $\mathfrak{M}_{b}(\Omega\times(a,b)).$ Then, $\left\{  \mu_{n,s}\right\}  $ is nondecreasing and converging to $\mu_s$ in $\mathfrak{M}_{b}(\Omega\times(a,b))$ and  there exist decompositions $(f,g,h)$ of $\mu_0$, $(f_n,g_n,h_n)$ of $\mu_{n,0}$ 
             such that
            $\left\{  f_{n}\right\}  ,\left\{  g_{n}\right\}  ,\left\{  h_{n}\right\}
             $ strongly converge to $f,g,h$ in $L^{1}(\Omega\times(a,b)),L^{2}(\Omega\times(a,b),\mathbb{R}^N)$ and $L^2(a,b,H^1_0(\Omega))$
             respectively satisfying 
             \begin{equation*}
           ||f_{n}||_{L^{1}(\Omega\times(a,b))}+\left\Vert g_{n}\right\Vert _{L^{2}(\Omega\times(a,b),\mathbb{R}^N)}+||h_{n}||_{L^2(a,b,H^1_0(\Omega))}+\mu_{n,s}(\Omega\times(a,b))
                    \leq 2\mu(\Omega\times(a,b)). 
             \end{equation*}                          
                      \end{proposition}
             \begin{remark}\label{5hh1203201410}  For $0<\rho\leq\frac{1}{3}\min\{\sup_{x\in\Omega}d(x,\partial\Omega),(b-a)^{1/2}\}$, set
             \begin{align*}
             \Omega_{\rho}^{j}=\{x\in\Omega: d(x,\partial\Omega)>j\rho\}\times (a+(j\rho)^2,a+((b-a)^{1/2}-j\rho)^2)~\text{ for }~j=0,...,k_\rho,
             \end{align*}  where $k_\rho=\left[\frac{\min\{\sup_{x\in\Omega}d(x,\partial\Omega),(b-a)^{1/2}\}}{2\rho}\right]$. \\
             We can choose $f_n,g_n,h_n$ in above two Propositions  such that for any $j=1,...,k_\rho$,
             \begin{align}\label{5hh120320149}
             ||f_{n}||_{L^1(\Omega_{\rho}^{j})}+\left\Vert g_{n}\right\Vert _{L^{2}(\Omega_{\rho}^{j},\mathbb{R}^N)}+|||h_n|+|\nabla h_{n}|||_{L^2(\Omega_{\rho}^{j})}\leq 2\mu(\Omega_{\rho}^{j-1})~\forall n\in\mathbb{N}
             \end{align}
             In fact, set $\mu_j=\chi_{\Omega_{\rho}^{k_\rho-j}\backslash \Omega_{\rho}^{k_\rho-j+1}}\mu$ if $j=1,...,k_\rho-1$, $\mu_j=\chi_{\Omega\times(a,b)\backslash \Omega^{1}_{\rho}}\mu$ if $j=k_\rho$ and  $\mu_j=\chi_{\Omega^{k_\rho}_{\rho}}\mu$ if $j=0$.
               From the proof of above two Propositions in \cite{55VH}, for any $\varepsilon>0$ we can assume  supports of $f_n,g_n,h_n$  containing in $\text{supp}(\mu)+\tilde{Q}_\varepsilon(0,0)$. Thus, for any $\mu=\mu_j$ we have $f_n^j,g_n^j,h_n^j$  correspondingly such that their supports contain in $\Omega_{\rho,T}^{k_\rho-j-1/2}\backslash \Omega_{\rho,T}^{k_\rho-j+3/2}$ if $j=1,...,k_\rho-1$ and $\Omega_{T}\backslash \Omega^{3/2}_{\rho,T}$ if $j=k_\rho$ and  $\Omega^{k_\rho-1/2}_{\rho,T}$ if $j=0$. By $\mu=\sum_{j=0}^{k_\rho}\mu_j$, thus it is allowed to choose $f_n=\sum_{j=0}^{k_\rho}f^j_n,f_n=\sum_{j=0}^{k_\rho}g^j_n$ and $h_n=\sum_{j=0}^{k_\rho}h^j_n$
               and \eqref{5hh120320149} satisfies since 
               \begin{align*}
                        &||f_{n}||_{L^1(\Omega_{\rho}^{j})}+\left\Vert g_{n}\right\Vert _{L^{2}(\Omega_{\rho}^{j},\mathbb{R}^N)}+|||h_n|+|\nabla h_{n}|||_{L^2(\Omega_{\rho}^{j})}
                        \\&~~~~~~~~~~\leq  \sum_{i=0}^{k_\rho-j+1}\left(||f_{n}^i||_{L^1(\Omega_{\rho}^{j})}+\left\Vert g_{n}^i\right\Vert _{L^{2}(\Omega_{\rho}^{j},\mathbb{R}^N)}+|||h_n^i|+|\nabla h_{n}^i|||_{L^2(\Omega_{\rho}^{j})}\right) 
                        \\&~~~~~~~~~~\leq \sum_{i=j-1}^{k_\rho-j+1}2\mu_j(\Omega\times(a,b))
                       = 2\mu(\Omega_{\rho}^{j-1}).
                        \end{align*}

             \end{remark}
             \begin{definition} Let $\mu\in\mathfrak{M}_b(\Omega\times(a,b))$ and $\sigma\in\mathfrak{M}_b(\Omega)$. A measurable function $u$ is a distributional solution to the problem \eqref{5hh050420141}
             if $u\in L^s(a,b,W_0^{1,s}(\Omega))$ for any $s\in\left[1,\frac{N+2}{N+1}\right)$ and $B(u,\nabla u)\in L^1(\Omega\times (a,b))$ such that 
               \begin{align*}
              & -\int_{\Omega\times(a,b)}u\varphi_tdxdt+\int_{\Omega\times(a,b)}A(x,t,\nabla u)\nabla \varphi dxdt\\&~~~~~~~~=\int_{\Omega\times(a,b)}B(u,\nabla u)\varphi dxdt+\int_{\Omega\times(a,b)}\varphi d\mu+\int_{\Omega}\varphi(a)d \sigma
               \end{align*}
               for every $\varphi\in C_c^1(\Omega\times [a,b))$. 
             \end{definition}
             \begin{remark}\label{5hh060420141} Let $\sigma'\in\mathfrak{M}_b(\Omega)$ and $a'\in (a,b)$, set $\omega=\mu+\sigma'\otimes\delta_{\{t=a'\}}$. If $u$ is a distributional solution to the problem \eqref{5hh050420141} with data $\omega$ and $\sigma=0$ such that $\text{supp}(\mu)\subset \overline{\Omega}\times [a',b]$, and $u=0,B(u,\nabla u)=0$ in $\Omega\times (a,a')$, then $\tilde{u}:=\left. u \right|_{\Omega\times[a',b)}$ is a distributional solution to problem \eqref{5hh050420141} in $\Omega\times (a',b)$ with data $\mu$ and $\sigma'$. Indeed, for any $\varphi\in C_c^1(\Omega\times [a',b))$ we define 
             \begin{equation*}\tilde{\varphi}(x,t)=\left\{ \begin{array}{l}
                                  \varphi(x,t) ~\text{if}~(x,t)\in \Omega\times [a',b), \\                     
                                  (1+\varepsilon_0)(t-a')\varphi_t(x,a')+\varphi(x,(1+\varepsilon_0)a'-\varepsilon_0t)~\text{if}~(x,t)\in \Omega\times [a,a'), \\ 
                                  \end{array} \right.\end{equation*}
                                  where $\varepsilon_0\in \left(0,\frac{b-a'}{a'-a}\right)$.\\
                                  Clearly, $\tilde{\varphi}\in C_c^1(\Omega\times [a,b))$, thus we have
                                  \begin{align*}
                                    &-\int_{\Omega\times(a,b)}u\tilde{\varphi}_tdxdt+\int_{\Omega\times(a,b)}A(x,t,\nabla u)\nabla \tilde{\varphi} dxdt\\&~~~~~~~~~~~~~~~~~~~~~=\int_{\Omega\times(a,b)}B(u,\nabla u)\tilde{\varphi} dxdt+\int_{\Omega\times(a,b)}\tilde{\varphi} d\omega,                      
                                    \end{align*}
                          which implies
                          \begin{align*}
                                                 &-\int_{\Omega\times(a',b)}\tilde{u}\varphi_tdxdt+\int_{\Omega\times(a',b)}A(x,t,\nabla \tilde{u})\nabla \varphi dxdt\\&~~~~~~~~~~~~~~~=\int_{\Omega\times(a',b)}B(\tilde{u},\nabla \tilde{u})\varphi dxdt+\int_{\Omega\times(a',b)}\varphi d\mu
                                                 + \int_{\Omega}\varphi(a')d\sigma'.
                                                 \end{align*}        
             \end{remark}
             \begin{definition}Let $\mu\in\mathfrak{M}(\mathbb{R}^{N}\times [a,+\infty))$, for $a\in \mathbb{R}$ and $\sigma\in\mathfrak{M}(\mathbb{R}^{N})$. A measurable function $u$ is a distributional solution to problem 
             \begin{equation}\label{5hh120320148}\left\{ \begin{array}{l}
                                  {u_t} - \operatorname{div}\left( {A(x,t,\nabla u)} \right) = B(u,\nabla u)+\mu ~\text{in}~\mathbb{R}^{N}\times (a,+\infty), \\                     
                                  u(a) = \sigma\quad ~in~ \mathbb{R}^{N}, \\ 
                                  \end{array} \right.\end{equation}
             if $u\in L^s_{\text{loc}}(a,\infty,W_{\text{loc}}^{1,s}(\mathbb{R}^N))$  for any $s\in\left[1,\frac{N+2}{N+1}\right)$ and $B(u,\nabla u)\in L^1_{\text{loc}}(\mathbb{R}^N\times[a,\infty))$ such that 
               \begin{align*}
               &-\int_{\mathbb{R}^N\times (a,\infty)}u\varphi_tdxdt+\int_{\mathbb{R}^N\times (a,\infty)}A(x,t,\nabla u)\nabla \varphi dxdt\\&~~~~~~~~~~~~~~~~=\int_{\mathbb{R}^N\times (a,\infty)}B(u,\nabla u)\varphi dxdt+\int_{\mathbb{R}^N\times (a,\infty)}\varphi d\mu+\int_{\mathbb{R}^N}\varphi(a)d \sigma
               \end{align*}
               for every $\varphi\in C_c^1(\mathbb{R}^N\times [a,\infty))$. 
             \end{definition}
             \begin{definition}Let $\omega\in\mathfrak{M}(\mathbb{R}^{N+1})$. A measurable function $u$ is a distributional solution to problem 
             \begin{equation}\label{5hh190420141}
             {u_t} - \operatorname{div}\left( {A(x,t,\nabla u)} \right) = B(u,\nabla u)+\omega ~\text{in}~\mathbb{R}^{N+1},
             \end{equation}                  
                      if $u\in L^s_{\text{loc}}(\mathbb{R};W_{\text{loc}}^{1,s}(\mathbb{R}^N))$  for any $s\in\left[1,\frac{N+2}{N+1}\right)$ and $B(u,\nabla u)\in L^1_{\text{loc}}(\mathbb{R}^{N+1})$ such that 
$$
                        -\int_{\mathbb{R}^{N+1}}u\varphi_tdxdt+\int_{\mathbb{R}^{N+1}}A(x,t,\nabla u)\nabla \varphi dxdt=\int_{\mathbb{R}^{N+1}}B(u,\nabla u)\varphi dxdt+\int_{\mathbb{R}^{N+1}}\varphi d\omega,
$$
                        for every $\varphi\in C_c^1(\mathbb{R}^{N+1})$. 
                      \end{definition}
             \begin{remark}\label{5hh060420142}Let $\mu\in\mathfrak{M}(\mathbb{R}^{N}\times [a,+\infty))$, for $a\in \mathbb{R}$ and $\sigma\in\mathfrak{M}(\mathbb{R}^{N})$.  If $u$ is a distributional solution to problem \eqref{5hh190420141} with data $\omega=\mu+\sigma\otimes\delta_{\{t=a\}}$  such that  $u=0,B(u,\nabla u)=0$ in $\mathbb{R}^N\times (-\infty,a)$, then $\tilde{u}:=\left. u \right|_{\mathbb{R}^N\times[a,\infty)}$ is a distributional solution to problem \eqref{5hh120320148} in $\mathbb{R}^N\times (a,\infty)$ with data $\mu$ and $\sigma$, see Remark \ref{5hh060420141}.        
             \end{remark}
             To prove the existence distributional solution of problem \eqref{5hh120320148} we need the following results. 
    First, we have local estimates of the renormalized solution which get from \cite[Proposition 2.8 ]{55VH}.
    \begin{proposition} \label{5hh120320146}
     Let $u,v$ be in Definition \ref{5hh100320142}.  There holds for $k\geq 1$ and $0\leq\eta\in C_c^\infty(\Omega\times (a,b))$
    \begin{equation}\label{5hh120320143}
    \int_{|v|\leq k}\eta |\nabla u|^2dxdt+\int_{|v|\leq k}\eta |\nabla v|^2dxdt\lesssim k A
    \end{equation}
    where  
    \begin{align*}
    &A=|| v\eta_t||_{L^1(\Omega\times (a,b))}+|||\nabla u||\nabla\eta|||_{L^1(\Omega\times (a,b))}+||\eta f||_{L^1(\Omega\times (a,b))}+||\eta|g|^2||_{L^1(\Omega\times (a,b))}\\&~~~~~~~~~~~~+|||\nabla\eta||g|||_{L^1(\Omega\times (a,b))}+||\eta|\nabla h|^2||_{L^1(\Omega\times (a,b))}+\int_{\Omega\times (a,b)}\eta d|\mu_s|.
    \end{align*}
    \end{proposition}
    For our purpose, we recall the Landes-time approximation  of functions $w$ belonging to $L^2(a,b,H_0^1(\Omega))$, introduced in \cite{55Lan}, used in \cite{55DaOr,55BDGO97,55BlPo1}. For $\nu>0$ we define
    \begin{align*}
    \langle w\rangle_{\nu}(x,t)=\nu\int_{a}^{\min\{t,b\}}w(x,s)e^{\nu(s-t)}ds~~\text{ for all }~ (x,t)\in \Omega\times(a,b).
    \end{align*}
    We have that $\langle w\rangle_{\nu}$ converges to $w$ strongly in $L^2(a,b,H_0^1(\Omega))$ and $||\langle w\rangle_{\nu}||_{L^q(\Omega\times(a,b))}\leq ||w||_{L^q(\Omega\times(a,b))}$ for every $q\in [1,\infty]$. Moreover, 
    \begin{align*}
    (\langle w\rangle_{\nu})_{t}=\nu\left(  w-\langle
    w\rangle_{\nu}\right)  \quad\text{in the distributional sense.}
    \end{align*}
    \begin{proposition}\label{5hh1203201411} 
    Let $q_0>1$ and $0<\alpha<1/2$ such that $q_0>\alpha+1$. Let $L:\mathbb{R}\to\mathbb{R}$  be continuous and nondecreasing such that $L(0)=0$.  If $u$ is a solution of \begin{equation}\label{5hh060420143}
                                          \left\{
                                          \begin{array}
                                          [c]{l}%
                                          {u_{t}}-\operatorname{div}(A(x,t,\nabla u))+L(u)=\mu~\text{in }\Omega\times (a,b),\\ 
                                                                u=0~~~~~~~\text{on}~~
                                                                                                 \partial\Omega \times (a,b),
                                                                                                    \\
                                                                                                    u(a) =0~~~\text{in}~~ \Omega, \\ 
                                          \end{array}
                                          \right.  
                                          \end{equation} 
                                           with $\mu\in C^\infty_c(\Omega\times(a,b))$  then  for   $0\le \eta \in C^\infty_c(D)$ 
    \begin{equation}
    \frac{1}{k}\int_{D}|\nabla T_k(u)|^2\eta + \int_{D}\frac{|\nabla u|^2\eta}{(|u|+1)^{\alpha+1}}+|||\nabla u| |\nabla \eta|||_{L^1(D)}+||L(u)\eta||_{L^1(D)}\lesssim_{\alpha,q_0}B,
    \label{5hh110320141}
    \end{equation}
    where $q_1=\frac{q_0-\alpha-1}{2q_0}$, $D=\Omega'\times (a',b')$, $\Omega'\subset\subset\Omega$ and   $a<a'<b'<b$; 
$$
    B=||\eta_t (|u|+1)||_{L^1(D)}+\int_{D}(|u|+1)^{q_0}\eta dxdt+\int_{D}|\nabla \eta^{1/q_1}|^{q_1}dxdt+\int_{D}\eta d|\mu|.$$
    Furthermore, for $T_k(w)\in L^2(a',b',H^1_0(\Omega'))$, the Landes-time approximation $\langle
    T_{k}(w)\rangle_{\nu}$  of the truncated function $T_{k}(w)$ in $D$ then for any $\varepsilon\in (0,1)$ and $\nu>0$
     \begin{align}\nonumber
     &\nu\int_{D} \eta\left( T_{k}(w)-\langle T_{k}(w)\rangle_\nu\right)T_\varepsilon(T_k(u)-\langle T_{k}(w)\rangle_\nu) dxdt\\&~~~~~~+\int_{D}\eta A(x,t,\nabla T_k(u))\nabla T_{\varepsilon}(T_k(u)-\langle T_{k}(w)\rangle_\nu) dxdt \lesssim\varepsilon(1+k)B.\label{5hh110320147}
     \end{align}
    \end{proposition}
    
    \begin{proposition}\label{5hh1203201412} Let $q_0>1$, $\mu_n=\mu_{n,0}+\mu_{n,s}\in \mathfrak{M}_b(B_n(0)\times (-n^2,n^2))$. Let $u_n$ be a renormalized solution of 
    \begin{equation}
             \left\{
             \begin{array}
             [c]{l}%
             {(u_{n})_t}-\operatorname{div}(A(x,t,\nabla u_n))=\mu_{n}~~\text{in }B_n(0)\times (-n^2,n^2),\\ 
                                 u_n=0
                                 ~~~~~~\text{on}~~\partial B_n(0)\times (-n^2,n^2),
                                 \\
                                 u_n(-n^2) = 0~\text{in}~ B_n(0), \\         
             \end{array}
             \right.  \label{5hh120320142}%
             \end{equation}
             relative to the decomposition $(f_n,g_n,h_n)$ of $\mu_{n,0}$ satisfying  \eqref{5hh110320147} in Proposition \ref{5hh1203201411} with $L\equiv 0$. Assume that for any $m\in\mathbb{N}$ and $\alpha\in (0,1/2)$, $D_m:=B_m(0)\times (-m^2,m^2)$
             \begin{align}\nonumber
             &\frac{1}{k}|||\nabla T_k(u)|^2||_{L^1(D_m)}+|||\nabla u|^2(|u|+1)^{-\alpha-1}||_{L^1(D_m)}+|||\nabla u| ||_{L^1(D_m)}+|\mu_n|(D_m)\\&\nonumber~~+||f_n||_{L^1(D_m)}+||g_n||_{L^2(D_m,\mathbb{R}^N)}+|||h_n|+|\nabla h_n|||_{L^2(D_m)}+||u_n||_{L^{q_0}(D_m)}\leq  C(m,\alpha)\label{5hh120320144}
             \end{align}
               for all $n\geq m$ and $h_n$ is convergent in $L^1_{\text{loc}}(\mathbb{R}^{N+1})$. 
      Then, there exists a subsequence of $\{u_n\}$, still denoted by $\{u_n\}$  such that $u_n$ converges to $u$ a.e in $\mathbb{R}^{N+1}$ and  in $L^s_{\text{loc}}(\mathbb{R};W^{1,s}_{loc}(\mathbb{R}^N))$ for any $s\in [1,\frac{N+2}{N+1})$.
    \end{proposition}
    Proofs of above two propositions are given in  Appendix. The following result is as a consequence of Proposition \ref{5hh1203201412}.
    \begin{corollary}\label{5hh090420144}
    Let  $\mu_n\in L^1(B_n(0)\times (-n^2,n^2))$. Let $u_n$ be a unique renormalized solution to the problem \ref{5hh120320142}.  Assume that for any $m\in\mathbb{N}$,
    \begin{align*}
    \sup_{n\geq m}|\mu_n|(B_m(0)\times (-m^2,m^2))<\infty~~\text{and}~~\sup_{n\geq m}\int_{B_m(0)\times (-m^2,m^2)}|u_n|^{q_0}dxdt<\infty.
    \end{align*}  then there exists a subsequence of $\{u_n\}$ converging to $u$ a.e in $\mathbb{R}^{N+1}$ and  in $L^s_{\text{loc}}(\mathbb{R};W^{1,s}_{loc}(\mathbb{R}^N))$ for any $s\in [1,\frac{N+2}{N+1})$.
    \end{corollary} 
    Finally, we would like to present a technical lemma which will be used several times in the paper, especially in the proof of Theorem \ref{5hh0701201411}, \ref{5hh0701201412} and \ref{5hh2410131}.   It is a consequence of Vitali Covering Lemma (see \cite{55BW4,55BW2,55MePh2}). 
     \begin{lemma}\label{5hhvitali2} Let $\Omega$ be a $(R_0,\delta)$- Reifenberg flat domain with $\delta<1/4$ and let $w$ be an $A_\infty$ weight. Suppose that the sequence of balls $\{B_r(y_i)\}_{i=1}^L$ with centers $y_i\in\overline{\Omega}$ and a common radius $r\leq R_0/4$ covers $\Omega$. Set $s_i=T-ir^2/2$ for all $i=0,1,...,[\frac{2T}{r^2}]$. Let $E\subset F\subset \Omega_T$ be measurable sets for which there exists $0<\varepsilon<1$ such that  $w(E)<\varepsilon w(\tilde{Q}_r(y_i,s_j))$ for all $i=1,...,L$, $j=0,1,...,[\frac{2T}{r^2}]$; and  for all $(x,t)\in \Omega_T$, $\rho\in (0,2r]$, we have
             $\tilde{Q}_\rho(x,t)\cap \Omega_T\subset F$      
             if $w(E\cap \tilde{Q}_\rho(x,t))\geq \varepsilon w(\tilde{Q}_\rho(x,t))$. Then $
             w(E)\leq C\varepsilon w(F)$         
             for a constant $C$ depending only on $N$ and $[w]_{A_\infty}$.
            \end{lemma}           
Clearly,  the lemma implies the following two consquences that we state as Lemmas. More precisely, we have the following two Lemmas.          
              \begin{lemma}\label{5hhvitali1}
                Let $0<\varepsilon<1, R>0$ and  consider the cylinder $\tilde{Q}_R:=\tilde{Q}_R(x_0,t_0)$ for some $(x_0,t_0)\in \mathbb{R}^{N+1}$ and $w\in A_{\infty}$. Let $E\subset F\subset \tilde{Q}_R$ be two measurable sets in $\mathbb{R}^{N+1}$ with $w(E)<\varepsilon w(\tilde{Q}_R)$
                 satisfying
                 the following property: for all $(x,t)\in \tilde{Q}_R$ and $r\in (0,R]$, we have
                $\tilde{Q}_r(x,t)\cap \tilde{Q}_R\subset F$
                 provided $w(E\cap \tilde{Q}_r(x,t))\geq \varepsilon w(\tilde{Q}_r(x,t)).
                   $
                 Then $w(E)\leq C\varepsilon w(F)$ for some $C=C(N,[w]_{A_\infty})$.
                \end{lemma}  
                \begin{lemma}\label{5hhvitali3}
                      Let $0<\varepsilon<1$ and  $R>R^{\prime}>0$ and let $E\subset F\subset Q=B_{R}(x_0)\times(a,b)$ be two measurable sets in $\mathbb{R}^{N+1}$ with $|E|<\varepsilon |\tilde{Q}_{R^{\prime}}|$
                       which satisfy
                       the following property: for all $(x,t)\in Q$ and $r\in (0,R^{\prime}]$, we have 
                       $Q_r(x,t)\cap Q\subset F$
                       if  $ |E\cap \tilde{Q}_r(x,t)|\geq \varepsilon |\tilde{Q}_r(x,t)|$.
                       Then  $|E|\leq C\varepsilon |F|$ for a constant $C$ depending only on $N$.
                      \end{lemma} 
                
\section{Estimates on Potential}
In this section, we will develop the nonlinear potential theory corresponding to quasilinear parabolic equations. \\\\
First we introduce the Wolff parabolic potential of $\mu\in\mathfrak{M}^+(\mathbb{R}^{N+1})$ by
\begin{equation*}
 \mathbb{W}^R_{\alpha,p}[\mu](x,t)=\int_{0}^{R}\left(\frac{\mu(\tilde{Q}_\rho(x,t))}{\rho^{N+2-\alpha p}}\right)^{\frac{1}{p-1}}\frac{d\rho}{\rho}~~\text{ for any }~ (x,t)\in \mathbb{R}^{N+1},
\end{equation*}
where $\alpha>0, 1<p<\alpha^{-1}(N+2)$ and $0<R\leq\infty$. For convenience, $\mathbb{W}_{\alpha,p}[\mu]:=\mathbb{W}^\infty_{\alpha,p}[\mu]$.\\\\
The following result is an extension of \cite[Theorem 1.1]{55HoJa}, \cite[Proposition 2.2]{55VHV} to Parabolic potential. 
\begin{theorem}\label{5hh100420141}Let $\alpha>0$, $1<p<\alpha^{-1}(N+2)$ and $w\in A_\infty$, $\mu\in\mathfrak{M}^+(\mathbb{R}^{N+1})$. There exist constants $C>0$ and $\varepsilon_0\in (0,1)$ depending on $N,\alpha,p,[w]_{A_\infty}$ such that for any $\lambda>0$ and $\varepsilon\in (0,\varepsilon_0)$
\begin{equation}\label{5hh100420142}
 w(\{\mathbb{W}^R_{\alpha,p}[\mu]>a\lambda,(\mathbb{M}^R_{\alpha p}[\mu])^{\frac{1}{p-1}}\le \varepsilon \lambda \})\le C\exp(-1/(C\varepsilon)) w(\{\mathbb{W}^R_{\alpha ,p}[\mu]>\lambda\})
                     \end{equation}
 where $a=2+3^{\frac{N+2-\alpha p}{p-1}}$.
\end{theorem}
\begin{proof}[Proof of Theorem \ref{5hh100420141}]
We only consider the case $R<\infty$. Let $\{\tilde{Q}_{R}(x_j,t_j)\}$ be a cover of $\mathbb{R}^{N+1}$ such that 
  $
  \sum_j \chi_{\tilde{Q}_{R}(x_j,t_j)}\leq M$  in $\mathbb{R}^{N+1}
 $
  for some constant $M=M(N)>0$. It is enough to show that 
 there exist constants $c>0$ and $\varepsilon_0\in (0,1)$ depending on $N,\alpha,p,[w]_{A_\infty}$ such that for any $Q\in\{\tilde{Q}_{R}(x_j,t_j)\}$, $\lambda>0$ and $\varepsilon\in (0,\varepsilon_0)$
   \begin{equation}\label{5hh170220141}
                     w(Q\cap\{\mathbb{W}^R_{\alpha,p}[\mu]>a\lambda,(\mathbb{M}^R_{\alpha p}[\mu])^{\frac{1}{p-1}}\le \varepsilon \lambda \})\le c\exp(-(c\varepsilon)^{-1}) w(Q\cap\{\mathbb{W}^R_{\alpha ,p}[\mu]>\lambda\}).
                     \end{equation}
                    Fix $\lambda>0$ and $0<\varepsilon<1/10$. We set 
\begin{align*}
E=Q\cap\{\mathbb{W}^R_{\alpha,p}[\mu]>a\lambda,(\mathbb{M}^R_{\alpha p}[\mu])^{\frac{1}{p-1}}\le \varepsilon \lambda \}~\text{ and }~F=Q\cap\{\mathbb{W}^R_{\alpha ,p}[\mu]>\lambda\}.
\end{align*}
Thanks to Lemma \ref{5hhvitali1} we will get \eqref{5hh170220141} if we verify the following two claims:
 \begin{equation}\label{5hh170220142}
 w(E)\leq c\exp(-(c\varepsilon)^{-1})w(Q),
 \end{equation}
 and for any $(x,t)\in Q$, $0<r\leq R$, 
 \begin{equation}\label{5hh170220143}
 w(E\cap \tilde{Q}_r(x,t))< c\exp(-(c\varepsilon)^{-1})w(\tilde{Q}_r(x,t)),
 \end{equation}
   provided that $\tilde{Q}_r(x,t)\cap Q\cap F^c\not =\emptyset$ and $E\cap \tilde{Q}_r(x,t)\not = \emptyset,$
   where constants $c_3,c_4,c_5$ and $c_6$ depend on $N,\alpha,p$ and $[w]_{A_\infty}$. \\
 \textbf{Claim} \eqref{5hh170220142}: Set 
 $$g_k(x,t)=\int_{2^{-k}R}^{2^{-k+1}R}\left(\frac{\mu(\tilde{Q}_\rho(x,t))}{\rho^{N+2-\alpha p}}\right)^{\frac {1}{p-1}}\frac{d\rho}{\rho}.
  $$
 We have for $m\in\mathbb{N}$ and $(x,t)\in E$
 \begin{align*}
 \mathbb{W}^R_{\alpha,p}[\mu](x,t)&=\sum_{k=m+1}^{\infty}g_k(x,t)+\int_{2^{-m}R}^{R}\left(\frac{\mu(\tilde{Q}_\rho(x,t))}{\rho^{N+2-\alpha p}}\right)^{\frac {1}{p-1}}\frac{d\rho}{\rho}
\\& \leq \sum_{k=m+1}^{\infty}g_k(x,t)+m(\mathbb{M}^R_{\alpha p}[\mu](x,t))^{\frac{1}{p-1}}
\\&\leq \sum_{k=m+1}^{\infty}g_k(x,t)+m\varepsilon\lambda.
 \end{align*} 
 We deduce that for $\beta>0$, $m\in\mathbb{N}$
 \begin{align*}
 |E|&\leq |Q\cap\{\sum_{k=m+1}^{\infty}g_k>(1-m\varepsilon)\lambda \}|
 \\&= |Q\cap\{\sum_{k=m+1}^{\infty}g_k>\sum_{k=m+1}^{\infty}2^{-\beta(k-m-1)}(1-2^{-\beta})(1-m\varepsilon)\lambda\}|
 \\&\leq \sum_{k=m+1}^{\infty}|Q\cap\{g_k>2^{-\beta(k-m-1)}(1-2^{-\beta})(1-m\varepsilon)\lambda\}|.
 \end{align*}
 We can  assume that $(x_0,t_0)\in Q$, $(\mathbb{M}^R_{\alpha p}[\mu](x_0,t_0))^{\frac{1}{p-1}}\leq \varepsilon \lambda$. Thus, by computing,  see \cite[Proof of Proposition 2.2 ]{55VHV} we have for any $k\in\mathbb{N}$
 \begin{align*}
 |{Q\cap \{ g_k>s\}}|\leq \frac{c}{s^{p-1}}2^{-k\alpha p}|{Q}|(\varepsilon \lambda)^{p-1}. 
 \end{align*}  
 Consequently, 
 \begin{align*}
 |E|&\leq \sum_{k=m+1}^{\infty}\frac{c_7}{\left(2^{-\beta(k-m-1)}(1-2^{-\beta})(1-m\varepsilon)\lambda \right)^{p-1}}2^{-k\alpha p}|{Q}|(\varepsilon \lambda)^{p-1} 
 \\&\leq c2^{-(m+1)\alpha p}\left(\frac{\varepsilon}{1-m\varepsilon}\right)^{p-1}|{Q}|\left(1-2^{-\beta}\right)^{-p+1}
   \displaystyle\sum_{k=m+1}^\infty 2^{(\beta(p-1)-\alpha p)(k-m-1)}.
 \end{align*}
  If we choose $\varepsilon^{-1}-2<m\leq \varepsilon^{-1}-1$ and $\beta=\beta(\alpha,p)$ so that $\beta (p-1)-\alpha p<0$, we obtain 
  \begin{align*}
  |E|\leq c \exp(-\alpha p \ln(2) \varepsilon^{-1})|{Q}|.
  \end{align*}
  Thus, we get \eqref{5hh170220142}.\\
  \textbf{Claim} \eqref{5hh170220143}. 
   Take $(x,t)\in Q$ and $0<r\leq R$.
           Now assume that $\tilde{Q}_r(x,t)\cap Q\cap F^c\not= \emptyset$ and $E\cap \tilde{Q}_r(x,t)\not = \emptyset$ i.e, there exist $(x_1,t_1),(x_2,t_2)\in \tilde{Q}_r(x,t)\cap Q$ such that $\mathbb{W}^R_{\alpha,p}[\mu](x_1,t_1)\leq \lambda$ and $(\mathbb{M}^R_{\alpha p}[\mu](x_2,t_2))^{\frac{1}{p-1}}\le \varepsilon \lambda$.
            We need to prove that
            \begin{equation*}
                   w(E\cap \tilde{Q}_r(x,t))< c\exp(-(c\varepsilon)^{-1}) w(\tilde{Q}_r(x,t)). 
                                    \end{equation*}
         To do this, for all $(y,s)\in E\cap\tilde{Q}_r(x,t)$, $\tilde{Q}_\rho(y,s)\subset \tilde{Q}_{3\rho}(x_1,t_1)$ if $\rho> r$.\\
         If $r\leq R/3$, one has\begin{align*}
        \mathbb{W}^R_{\alpha,p}[\mu](y,s)&=\mathbb{W}^r_{\alpha,p}[\mu](y,s)+\int_{r}^{R/3}\left(\frac{\mu(\tilde{Q}_\rho(y,s))}{\rho^{N+2-\alpha p}}\right)^{\frac{1}{p-1}}\frac{d\rho}{\rho}
         +\int_{R/3}^{R}\left(\frac{\mu(\tilde{Q}_\rho(y,s))}{\rho^{N+2-\alpha p}}\right)^{\frac{1}{p-1}}\frac{d\rho}{\rho}
         \\&\leq 
\mathbb{W}^r_{\alpha,p}[\mu](y,s)+\int_{r}^{R/3}\left(\frac{\mu(\tilde{Q}_{3\rho}(x_1,t_1))}{\rho^{N+2-\alpha p}}\right)^{\frac{1}{p-1}}\frac{d\rho}{\rho}
         +2(\mathbb{M}^R_{\alpha p}[\mu](y,s))^{\frac{1}{p-1}}    
\\&\leq 
\mathbb{W}^r_{\alpha,p}[\mu](y,s)+3^{\frac{N+2-\alpha p}{p-1}}\lambda
         +2\varepsilon\lambda.              
         \end{align*}
     This gives $\mathbb{W}^r_{\alpha,p}[\mu](y,s)>\lambda$.\\
     If $r\geq R/3$, one has 
     \begin{align*}
           \mathbb{W}^R_{\alpha,p}[\mu](y,s)&\leq \mathbb{W}^r_{\alpha,p}[\mu](y,s)
              +\int_{R/3}^{R}\left(\frac{\mu(\tilde{Q}_\rho(y,s))}{\rho^{N+2-\alpha p}}\right)^{\frac{1}{p-1}}\frac{d\rho}{\rho}               
     \\&\leq      \mathbb{W}^r_{\alpha,p}[\mu](y,s)+2\varepsilon\lambda,              
              \end{align*}
 This gives $\mathbb{W}^r_{\alpha,p}[\mu](y,s)>\lambda$.\\ 
 Thus,  
 \begin{align*}
 w(E\cap \tilde{Q}_r(x,t))\leq w(\tilde{Q}_r(x,t)\cap\{\mathbb{W}_{\alpha,p}^r[\mu]>\lambda\}).
 \end{align*} 
 Since  $(x_2,t_2)\in \tilde{Q}_r(x,t)$, $(\mathbb{M}^R_{\alpha p}[\mu](x_2,t_2))^{\frac{1}{p-1}}\le \varepsilon \lambda$, so as above we also obtain 
 \begin{equation*}
 w(\tilde{Q}_r(x,t)\cap\{\mathbb{W}_{\alpha,p}^r[\mu]>\lambda\})\leq c\exp(-(c\varepsilon)^{-1})w(\tilde{Q}_r(x,t)),
 \end{equation*}
 which implies \eqref{5hh170220143}.
This completes the proof.
\end{proof}
   \begin{theorem}\label{5hh051120131} Let $\alpha>0$, $1<p<\alpha^{-1}(N+2)$,  $p-1<q<\infty$ and $0<s\leq \infty$ and $w\in A_\infty$. There holds 
    \begin{equation}\label{5hh180220142}
    ||\mathbb{W}^R_{\alpha,p}[\mu]||_{L^{q,s}(\mathbb{R}^{N+1},dw)}\sim_C ||(\mathbb{M}^R_{\alpha p}[\mu])^{\frac{1}{p-1}}||_{L^{q,s}(\mathbb{R}^{N+1},dw)},
    \end{equation}
    for all $\mu\in \mathfrak{M}^+(\mathbb{R}^{N+1})$ and $R\in (0,\infty] $ where $C$ is a positive constant only depending on $\alpha,p,q,s$ and $[w]_{A_\infty}$.
     \end{theorem}
      \begin{proof} Thanks to \eqref{5hh100420142} in Theorem \eqref{5hh100420141}, we have     
             for $0<s<\infty$ 
                 \begin{align*}
                  &||\mathbb{W}^R_{\alpha,p}[\mu]||_{L^{q,s}(\mathbb{R}^{N+1},dw)}^s= a^sq\int_{0}^{\infty}\lambda^sw(\{\mathbb{W}^R_{\alpha,p}[\mu]>a\lambda \})^{\frac{s}{q}}\frac{d\lambda}{\lambda}\\&~\leq 
                            c\exp(-\frac{1}{c\varepsilon})q\int_{0}^{\infty}\lambda^sw(\{\mathbb{W}^R_{\alpha,p}[\mu]>\lambda \})^{\frac{s}{q}}\frac{d\lambda}{\lambda}+ cs\int_{0}^{\infty}\lambda^sw(\{(\mathbb{M}^R_{\alpha p}[\mu])^{\frac{1}{p-1}}>\varepsilon\lambda \})^{\frac{s}{q}}\frac{d\lambda}{\lambda}\\&~= c\exp(-\frac{1}{c\varepsilon})||\mathbb{W}^R_{\alpha,p}[\mu]||_{L^{q,s}(\mathbb{R}^{N+1},dw)}^s+ c \varepsilon^{-s}||(\mathbb{M}^R_{\alpha p}[\mu])^{\frac{1}{p-1}}||_{L^{q,s}(\mathbb{R}^{N+1},dw)}^s.
                 \end{align*}           
                 Choose $0<\varepsilon<\varepsilon_0$ such that $ c\exp(-\frac{1}{c\varepsilon})<1/2$ we get 
$$
                             ||\mathbb{W}^R_{\alpha,p}[\mu]||_{L^{q,s}(\mathbb{R}^{N+1},dw)}^s\lesssim||(\mathbb{M}^R_{\alpha p}[\mu])^{\frac{1}{p-1}}||_{L^{q,s}(\mathbb{R}^{N+1},dw)}^s. $$
                             Similarly, we also get above inequality in case $s=\infty$. So, we proved the right-hand side inequality of \eqref{5hh180220142}.\\                     
            To complete the proof, we prove the left-hand side inequality of \eqref{5hh180220142}. Since for every $(x,t)\in\mathbb{R}^{N+1}$  
            \begin{align*}
            &(\mathbb{M}^R_{\alpha p}[\mu](x,t))^{\frac{1}{p-1}}\lesssim \mathbb{W}^R_{\alpha,p}[\mu](x,t)+\left(\frac{\mu(\tilde{Q}_{2R}(x,t))}{R^{N+2-\alpha p}}\right)^{\frac{1}{p-1}},\\&~~~~~~~~~~~~~~~
            \left(\frac{\mu(\tilde{Q}_{R/2}(x,t))}{R^{N+2-\alpha p}}\right)^{\frac{1}{p-1}}\lesssim\mathbb{W}^R_{\alpha,p}[\mu](x,t).
            \end{align*}   
             Thus it is enough to show that  for any $\lambda>0$
             \begin{align}\label{5hh180220143}
             w\left(\left\{(x,t): \left(\frac{\mu(\tilde{Q}_{2R}(x,t))}{R^{N+2-\alpha p}}\right)^{\frac{1}{p-1}}>\lambda\right\}\right)\lesssim w\left(\left\{(x,t): \left(\frac{\mu(\tilde{Q}_{R/2}(x,t))}{R^{N+2-\alpha p}}\right)^{\frac{1}{p-1}}\gtrsim\lambda\right\}\right).
             \end{align}
             Let $\{Q_{j}\}=\{\tilde{Q}_{R/4}(x_j,t_j)\}$ be a cover of $\mathbb{R}^{N+1}$ such that 
                for any $Q_j\in\{Q_{j}\}$, there exist $Q_{j,1},...,Q_{j,M_1}\in \{Q_{j}\}$  with $\sum_{j}\sum_{k=1}^{M_1}\chi_{Q_{j,k}}\leq M_2$ and 
               $
                        Q_j+\tilde{Q}_{2R}(0,0)\subset\bigcup\limits_{k = 1}^{{M_1}} {{Q_{j,k}}} 
                         $
                         for some integer constants $M_i=M_i(N), i=1,2$.        
               Then,
     \begin{align*}
            & w\left(\left\{(x,t): \left(\frac{\mu(\tilde{Q}_{2R}(x,t))}{R^{N+2-\alpha p}}\right)^{\frac{1}{p-1}}>\lambda\right\}\right)\leq \sum_{j}w\left(\left\{(x,t): \left(\frac{\mu(\tilde{Q}_{2R}(x,t))}{R^{N+2-\alpha p}}\right)^{\frac{1}{p-1}}>\lambda\right\}\cap Q_j\right)
            \\&~~~~~~~~~\leq \sum_{j}w\left(\left\{(x,t):\sum_{k=1}^{M_1} \frac{\mu(Q_{j,k})}{R^{N+2-\alpha p}}>\lambda^{p-1}\right\}\cap Q_j\right)
            \\&~~~~~~~~~\leq \sum_{j}\sum_{k=1}^{M_1}w\left(\left\{(x,t): \left(\frac{\mu(Q_{j,k})}{R^{N+2-\alpha p}}\right)^{\frac{1}{p-1}}>M_1^{-1/(p-1)}\lambda\right\}\cap Q_j\right)
            \\&~~~~~~~~~= \sum_{j}\sum_{k=1}^{M_1}a_{j,k}w(Q_j),
             \end{align*}                                  where $a_{j,k}=1$ if $\left(\frac{\mu(Q_{j,k})}{R^{N+2-\alpha p}}\right)^{\frac{1}{p-1}}>M_1^{-1/(p-1)}\lambda$ and $a_{j,k}=0$ if otherwise. \\
             Using the strong doubling property of $w$, there is $c_0=c_0(N,[w]_{A_\infty})$ such that 
             $w(Q_j)\leq c_0w(Q_{j,k})$. 
             On the other hand, if $a_{j,k}=1$ then $Q_{j,k}\subset\left\{(x,t): \left(\frac{\mu(\tilde{Q}_{R/2}(x,t))}{R^{N+2-\alpha p}}\right)^{\frac{1}{p-1}}>M_1^{-1/(p-1)}\lambda\right\}$.\\
             Therefore,
     \begin{align*}
             &w\left(\left\{(x,t): \left(\frac{\mu(\tilde{Q}_{2R}(x,t))}{R^{N+2-\alpha p}}\right)^{\frac{1}{p-1}}>\lambda\right\}\right)\leq  \sum_{j}\sum_{k=1}^{M_1}c_{0}a_{j,k}w(Q_{j,k})\\&~~~~~~\leq \sum_{j}\sum_{k=1}^{M_1}c_{0}w\left(\left\{(x,t): \left(\frac{\mu(\tilde{Q}_{R/2}(x,t))}{R^{N+2-\alpha p}}\right)^{\frac{1}{p-1}}>M_1^{-1/(p-1)}\lambda\right\}\cap Q_{j,k}\right),            
             \end{align*}        
            which implies \eqref{5hh180220143} since $\sum_{j}\sum_{k=1}^{M_1}\chi_{Q_{j,k}}\leq M_2$ in $\mathbb{R}^{N+1}$.      
       \end{proof}
     \begin{theorem}\label{5hh170220145} Let $0<\alpha p <N+2$ and $w\in A_\infty$
         There exists $C>0$ depending on $N,\alpha,p$ and $[w]_{A_\infty}$ such that for any
         $\mu\in\mathfrak{M}^+(\mathbb{R}^{N+1})$,  any  cylinder $\tilde{Q}_\rho\subset \mathbb{R}^{N+1}$  there holds   
        \begin{equation}
        \label{5hh121120131}\frac{1}{w({\tilde{Q}_{2\rho}})}\int_{\tilde{Q}_{2\rho}}
           \exp\left(C^{-1}\mathbb{W}^R_{\alpha,p}[\mu_{\tilde{Q}_\rho}](x,t) \right) dw(x,t)
           \leq C
        \end{equation}
        provided 
       $
        ||{\mathbb{M}^R_{\alpha p}[\mu_{\tilde{Q}_\rho}]}||_{L^\infty(\tilde{Q}_\rho)}\leq 1,
       $
        where  $\mu_{\tilde{Q}_\rho}=\chi _{\tilde{Q}_\rho}\mu$.
        \end{theorem}
        \begin{proof}
        Assume that $||\mathbb{M}^R_{\alpha p}[\mu_{\tilde{Q}_\rho}]||_{L^\infty(\tilde{Q}_\rho)}\leq 1$. We apply Theorem \eqref{5hh100420141} to $\mu_{\tilde{Q}_\rho}$. Then, choose $\varepsilon=\lambda^{-1}$ for all $\lambda \geq \lambda_0:=\max\{\varepsilon_0^{-1},\frac{N+2-\alpha p}{p-1}\}$,                                 
              we obtain
    \begin{equation*}w(\{\mathbb{W}^R_{\alpha,p}[\mu_{\tilde{Q}_\rho}]>a\lambda \}\cap \tilde{Q}_{2\rho})\le c\exp(-\lambda/c) w(\{\mathbb{W}^R_{\alpha,p}[\mu_{\tilde{Q}_\rho}]>\lambda\})~~\forall~ \lambda\geq \lambda_0.
                                \end{equation*}
On the other hand, if $\rho>R$, clearly we have  $\mathbb{W}^R_{\alpha,p}[\mu_{\tilde{Q}_\rho}]\equiv0$ in $\mathbb{R}^{N+1}\backslash \tilde{Q}_{2\rho}$, if $\rho\leq R$, for any $(x,t)\in \mathbb{R}^{N+1}\backslash \tilde{Q}_{2\rho}$
\begin{equation*}
\mathbb{W}^R_{\alpha,p}[\mu_{\tilde{Q}_\rho}](x,t)=\int_{\rho}^{R}\left(\frac{\mu_{\tilde{Q}_\rho}(\tilde{Q}_r(x,t))}{r^{N+2-\alpha p}}\right)^{\frac{1}{p-1}}\frac{dr}{r}\leq \frac{N+2-\alpha p}{p-1}\left(\frac{\mu(\tilde{Q}_\rho)}{\rho^{N+2-\alpha p}}\right)^{\frac{1}{p-1}}\leq \lambda_0.
\end{equation*}                                                     So, we get $\{\mathbb{W}^R_{\alpha,p}[\mu_{\tilde{Q}_\rho}]>\lambda\}\subset  \tilde{Q}_{2\rho}$ for all $\lambda\geq\lambda_0$.
         This can be written under the form 
         \begin{equation*}
               w(\{\mathbb{W}^R_{\alpha,p}[\mu_{\tilde{Q}_\rho}]>a\lambda \}\cap \tilde{Q}_{2\rho})\le \left(\chi_{(0,\lambda_0]}+ c\exp(-\lambda/c)\right) w(\tilde{Q}_{2\rho}),                                              \end{equation*}                                             for all $\lambda>0$. Therefore, we get \eqref{5hh121120131}. 
        \end{proof}\\\\
  In what follows, we need some estimates on Wolff parabolic potential:
  \begin{proposition}\label{5hh23101315} Let $p>1, 0<\alpha p <N+2$ and $q>1, \alpha p q<N+2$. There hold \begin{equation}\label{5hh2210133}
      ||\mathbb{W}_{\alpha,p}[\mu]||_{L^{\frac{(N+2)(p-1)}{N+2-\alpha p},\infty}(\mathbb{R}^{N+1})}\lesssim(\mu (\mathbb{R}^{N+1}))^{\frac{1}{p-1}}~~\forall ~\mu\in \mathfrak{M}_b^+(\mathbb{R}^{N+1}),
      \end{equation}
      \begin{equation}\label{5hh2210134}
          ||\mathbb{W}_{\alpha,p}[\mu]||_{L^{\frac{q(N+2)(p-1)}{N+2-\alpha pq},\infty}(\mathbb{R}^{N+1})}\lesssim ||\mu||^{\frac{1}{p-1}}_{L^{q,\infty}(\mathbb{R}^{N+1})}~~\forall ~\mu\in L^{q,\infty}(\mathbb{R}^{N+1}),\mu\geq 0,
          \end{equation}
          \begin{equation}\label{5hh2210135}
                  ||\mathbb{W}_{\alpha,p}[\mu]||_{L^{\frac{q(N+2)(p-1)}{N+2-\alpha pq}}(\mathbb{R}^{N+1})}\lesssim ||\mu||^{\frac{1}{p-1}}_{L^{q}(\mathbb{R}^{N+1})}~~\forall ~\mu\in L^{q}(\mathbb{R}^{N+1}),\mu\geq 0.
                  \end{equation}
    In particular, for $s>\frac{(p-1)(N+2)}{N+2-\alpha p}$, we define $F(\mu):=\left(\mathbb{W}_{\alpha,p}[\mu]\right)^{s}$  for all $\mu\in \mathfrak{M}_b^+(\mathbb{R}^{N+1})$. Then,
    \begin{align*}
   & ||F(\mu)||_{L^{\frac{(N+2)(s-p+1)}{\alpha s p}}(\mathbb{R}^{N+1})}\lesssim||\mu||^{\frac{s}{p-1}}_{L^{\frac{(N+2)(s-p+1)}{\alpha s p}}(\mathbb{R}^{N+1})},\\&
   ||F(\mu)||_{L^{\frac{(N+2)(s-p+1)}{\alpha s p},\infty}(\mathbb{R}^{N+1})}\lesssim ||\mu||^{\frac{s}{p-1}}_{L^{\frac{(N+2)(s-p+1)}{\alpha s p},\infty}(\mathbb{R}^{N+1})}.
    \end{align*}
    \end{proposition}
    \begin{proof}
    Let $s\geq 1$ be such that $\alpha s p <N+2$. It is known that if $\mu\in L^{s,\infty}(\mathbb{R}^{N+1})$ then 
    \begin{equation*}
    |\mu|(\tilde{Q}_\rho (x,t))\lesssim ||\mu||_{L^{s,
    \infty}(\mathbb{R}^{N+1})}\rho^{\frac{N+2}{s'}}~~\forall~~\rho>0.
    \end{equation*} 
    Thus for $\delta=||\mu||^{\frac{s}{N+2}}_{L^{s,\infty}(\mathbb{R}^{N+1})} \left(\mathbb{M}(\mu)(x,t)\right)^{-\frac{s}{N+2}}$, we have
    \begin{align*}
     \mathbb{W}_{\alpha,p}[\mu](x,t)&=\int_{0}^{\delta}\left(\frac{\mu(\tilde{Q}_\rho(x,t))}{\rho^{N+2-\alpha p}}\right)^{\frac{1}{p-1}}\frac{d\rho}{\rho}+\int_{\delta}^{\infty}\left(\frac{\mu(\tilde{Q}_\rho(x,t))}{\rho^{N+2-\alpha p}}\right)^{\frac{1}{p-1}}\frac{d\rho}{\rho}
      \\&\lesssim \left(\mathbb{M}(\mu)(x,t)\right)^{\frac{1}{p-1}}\delta^{\frac{\alpha p}{p-1}}+||\mu||^{\frac{1}{p-1}}_{L^{s,\infty}(\mathbb{R}^{N+1})}\delta^{-\frac{N+2-\alpha s p}{s(p-1)}}
      \\&=  \left(\mathbb{M}(\mu)(x,t)\right)^{\frac{N+2-\alpha s p}{(p-1)(N+2)}}||\mu||^{\frac{\alpha s p}{(p-1)(N+2)}}_{L^{s,\infty}(\mathbb{R}^{N+1})}.
    \end{align*}
    So, for any $\lambda>0$
    $$|\{\mathbb{W}_{\alpha,p}[\mu]>\lambda\}|\leq |\{\mathbb{M}(\mu) \gtrsim ||\mu||^{-\frac{\alpha s p}{N+2-\alpha s p}}_{L^{s,\infty}(\mathbb{R}^{N+1})}\lambda^{\frac{(p-1)(N+2)}{N+2-\alpha s p}}\}|.$$
    Since $\mathbb{M}$ is bounded from $\mathfrak{M}^+_b(\mathbb{R}^{N+1})$ to $L^{1,\infty}(\mathbb{R}^{N+1})$ and $L^{q}(\mathbb{R}^{N+1})$ ($L^{q,\infty}(\mathbb{R}^{N+1})$~resp.) to itself, we get  the result.
    \end{proof}\\ 
    \begin{remark}\label{5hh140420141}Assume that $\alpha p=N+2$ and $R>0$. As above we also have for any $\varepsilon>0$
$$
    \mathbb{W}_{\alpha,p}^R[\mu](x,t)\lesssim_{\varepsilon} \max\left\{(|\mu|(\mathbb{R}^{N+1}))^{\frac{1}{p-1}},\left((\mathbb{M}(\mu)(x,t))^{\varepsilon}(|\mu|(\mathbb{R}^{N+1}))^{\frac{\alpha p}{p-1}}R^{\varepsilon \alpha p}\right)^{\frac{1}{\alpha p+\varepsilon (p-1)}}\right\}.$$
      Therefore, for any $\lambda\gtrsim_\varepsilon (|\mu|(\mathbb{R}^{N+1}))^{\frac{1}{p-1}}$,
     \begin{align}\label{5hh140420142}
     |\{\mathbb{W}_{\alpha,p}^R[\mu]>\lambda\}|\lesssim_{\varepsilon} \left(\frac{(|\mu|(\mathbb{R}^{N+1}))^{\frac{1}{p-1}}}{\lambda}\right)^{\frac{\alpha p+\varepsilon(p-1)}{\varepsilon}}R^{\alpha p},
     \end{align}
     In particular, if $\mu\in\mathfrak{M}^+_b(\mathbb{R}^{N+1})$ then  $\mathbb{W}_{\alpha,p}^R[\mu]\in L^s_{\text{loc}}(\mathbb{R}^{N+1})$ for all  $s>0$.
    \end{remark}     
    \begin{remark} Assume that $p,q>1, 0<\alpha pq <N+2$. 
    As in \cite[Theorem 3]{55MuWh}, it is easy to prove that
    if $w\in A_{\frac{q(N+2-\alpha)}{N+2-\alpha pq}}$, i.e, $0<w\in L^1_{\text{loc}}(\mathbb{R}^{N+1})$ and for any $\tilde{Q}_\rho(y,s)\subset\mathbb{R}^{N+1}$
    \begin{align*}
    \sup_{\tilde{Q}_\rho(y,s)\subset\mathbb{R}^{N+1}}\left(\left(\fint_{\tilde{Q}_\rho(y,s)}wdxdt\right)\left(\fint_{\tilde{Q}_\rho(y,s)}w^{-\frac{N+2-\alpha pq}{(q-1)(N+2)}}dxdt\right)^{\frac{(q-1)(N+2)}{N+2-\alpha pq}}\right)= C_1<\infty,
    \end{align*}
    then 
    \begin{align*}
    \left(\int_{\mathbb{R}^{N+1}}\left(\mathbb{M}_{\alpha p}[|f|]\right)^{\frac{(N+2)q}{N+2-\alpha pq}}wdxdt\right)^{\frac{N+2-\alpha pq}{(N+2)q}}\leq C_2\left(\int_{\mathbb{R}^{N+1}}|f|^qw^{1-\frac{\alpha pq}{N+2}}dxdt\right)^{\frac{1}{q}},
    \end{align*}
    for some a constant $C_2=C_2(N,\alpha p,q,C_1)$.\\
    Therefore, from \eqref{5hh180220142} in Theorem \ref{5hh051120131} we get a weighted version of \eqref{5hh2210135}:
     \begin{align*}
        \left(\int_{\mathbb{R}^{N+1}}\left(\mathbb{W}_{\alpha, p}[|f|]\right)^{\frac{(N+2)(p-1)q}{N+2-\alpha pq}}wdxdt\right)^{\frac{N+2-\alpha pq}{(N+2)q}}\leq C_2\left(\int_{\mathbb{R}^{N+1}}|f|^pw^{1-\frac{\alpha p}{N+2}}dxdt\right)^{\frac{1}{p}}.
        \end{align*}     
    \end{remark}
    In the following proposition, we give another version of \eqref{5hh2210135} in the Lorentz-Morrey spaces involving calorie.  
  \begin{proposition}\label{5hh190320148}
  Let $p,q>1$, and $0<\alpha pq<\theta\leq N+2$. There  holds
  \begin{align}\label{5hh190320142}
  ||\left(\mathbb{W}_{\alpha,p}[|\mu|]\right)^{p-1}||_{L^{\frac{\theta q}{\theta -\alpha p q};\theta}(\mathbb{R}^{N+1})}\lesssim||\mu||_{L^{q;\theta}(\mathbb{R}^{N+1})}~~\forall \mu\in L^{q;\theta}(\mathbb{R}^{N+1}).
  \end{align}
  \end{proposition}  
    \begin{proof}
  As the proof of Proposition \ref{5hh23101315} we have 
$$
  \mathbb{W}_{\alpha,p}[|\mu|]\lesssim \left(\mathbb{M}_{\theta/q}[|\mu|]\right)^{\frac{\alpha p q}{\theta (p-1)}}
  \left(\mathbb{M}[|\mu|]\right)^{\frac{\theta-\alpha p q}{\theta (p-1)}}.$$
  Since $\mathbb{M}_{\theta/q}[|\mu|]\lesssim \left(\mathbb{M}_{\theta}[|\mu|^q]\right)^{1/q}$, above inequality becomes
  \begin{align}\label{5hh190320141}
  \mathbb{W}_{\alpha,p}[\mu]\lesssim \left(\mathbb{M}_{\theta}[|\mu|^q]\right)^{\frac{\alpha p}{\theta (p-1)}}
  \left(\mathbb{M}[\mu]\right)^{\frac{\theta-\alpha p q}{\theta (p-1)}}.\end{align}
  Take $\tilde{Q}_\rho(y,s)\subset\mathbb{R}^{N+1}$, we have 
  \begin{align*}
  \int_{\tilde{Q}_\rho(y,s)}\left(\mathbb{W}_{\alpha,p}[\mu]\right)^{\frac{\theta q (p-1)}{\theta-\alpha p q}}dxdt&\lesssim \int_{\tilde{Q}_\rho(y,s)}\left(\mathbb{W}_{\alpha,p}[\chi_{\tilde{Q}_{2\rho}(y,s)}\mu]\right)^{\frac{\theta q (p-1)}{\theta-\alpha p q}}dxdt\\&~~~~~~~~~+\int_{\tilde{Q}_\rho(y,s)}\left(\mathbb{W}_{\alpha,p}[\chi_{(\tilde{Q}_{2\rho}(y,s))^c}\mu]\right)^{\frac{\theta q (p-1)}{\theta-\alpha p q}}dxdt
  \\&=A+B.
  \end{align*}
  Using  \eqref{5hh190320141} and boundedness of $\mathbb{M}$ from $L^q(\mathbb{R}^{N+1})$ to itself, yield
  \begin{align*}
  A&\lesssim \int_{\mathbb{R}^{N+1}}\left(\mathbb{M}_{\theta}[|\mu|^q]\right)^{\frac{\alpha pq}{\theta-\alpha p q}}
  \left(\mathbb{M}[\chi_{\tilde{Q}_{2\rho}(y,s)}\mu]\right)^{q}dxdt
  \\&\lesssim ||\mu||_{L^{q;\theta}(\mathbb{R}^{N+1})}^{\frac{\alpha pq^2}{\theta-\alpha p q}} \int_{\chi_{\tilde{Q}_{2\rho}(y,s)}}
  |\mu|^{q}dxdt
  \\&\lesssim||\mu||_{L^{q;\theta}(\mathbb{R}^{N+1})}^{\frac{\theta q}{\theta-\alpha p q}} \rho^{N+2-\theta}.
  \end{align*}
  On the other hand, since $|\mu|(\tilde{Q}_r(x,t))\lesssim ||\mu||_{L^{q;\theta}(\mathbb{R}^{N+1})}r^{N+2-\frac{\theta}{q}}$ for all $\tilde{Q}_r(x,t)\subset\mathbb{R}^{N+1}$,
  \begin{align*}
  B&\leq\int_{\tilde{Q}_\rho(y,s)}\left(\int_{\rho}^{\infty}\left(\frac{|\mu|(\tilde{Q}_r(x,t))}{r^{N+2-\alpha p}}\right)^{\frac{1}{p-1}}\frac{dr}{r}\right)^{\frac{\theta q (p-1)}{\theta-\alpha p q}}dxdt
  \\&\lesssim \int_{\tilde{Q}_\rho(y,s)}\left(\int_{\rho}^{\infty}\left(||\mu||_{L^{q;\theta}(\mathbb{R}^{N+1})}r^{-\frac{\theta}{q}+\alpha}\right)^{\frac{1}{p-1}}\frac{dr}{r}\right)^{\frac{\theta q (p-1)}{\theta-\alpha p q}}dxdt
  \\&\lesssim ||\mu||_{L^{q;\theta}(\mathbb{R}^{N+1})}^{\frac{\theta q}{\theta-\alpha p q}} \rho^{N+2-\theta}.
  \end{align*}
  Therefore, 
$$
  \int_{\tilde{Q}_\rho(y,s)}\left(\mathbb{W}_{\alpha,p}[\mu]\right)^{\frac{\theta q (p-1)}{\theta-\alpha p q}}dxdt\lesssim||\mu||_{L^{q;\theta}(\mathbb{R}^{N+1})}^{\frac{\theta q}{\theta-\alpha p q}} \rho^{N+2-\theta},$$
  which follows \eqref{5hh190320142}.
    \end{proof}\\
    In the next result we state a series of equivalent norms concerning potentials $\mathbb{I}_\alpha,\mathbb{I}_\alpha^R, \mathcal{H}_\alpha,\mathcal{G}_\alpha$. 
   \begin{proposition}\label{5hh230120143}
     Let  $q>1$, $0<\alpha<N+2$ and $R>0$. Then, the following statements hold
     \begin{description}
     \item[a.] for any $\mu\in \mathfrak{M}^+(\mathbb{R}^{N+1})$
     \begin{align}
     \label{5hh230120141}
             & ||\mathcal{H}_\alpha[\mu]||_{L^q(\mathbb{R}^{N+1})}\sim_{\alpha,q}||\mathbb{I}_\alpha[\mu]||_{L^q(\mathbb{R}^{N+1})},\\&\label{5hh230120142}
          ||{\mathop \mathcal{H}\limits^ \vee}_\alpha[\mu]||_{L^q(\mathbb{R}^{N+1})}\sim_{\alpha,q}||\mathbb{I}_\alpha[\mu]||_{L^q(\mathbb{R}^{N+1})}.             
     \end{align}                        
        \item[b.]  for any $\mu\in \mathfrak{M}^+(\mathbb{R}^{N+1})$ 
        \begin{align}
        \label{5hh230120141'}
                      &  ||\mathcal{G}_\alpha[\mu]||_{L^q(\mathbb{R}^{N+1})}\sim_{\alpha,q,R}||\mathbb{I}^R_\alpha[\mu]||_{L^q(\mathbb{R}^{N+1})},\\&\label{5hh230120142'}
                                          ||{\mathop \mathcal{G}\limits^ \vee}_\alpha[\mu]||_{L^q(\mathbb{R}^{N+1})}\sim_{\alpha,q,R}||\mathbb{I}^R_\alpha[\mu]||_{L^q(\mathbb{R}^{N+1})}.
        \end{align}                         
     \end{description} 
     where   ${\mathop \mathcal{H}\limits^ \vee}_\alpha[\mu]$ is the backward parabolic  Riesz potential, defined by 
                $$
                 {\mathop \mathcal{H}\limits^ \vee}_\alpha[\mu](x,t)=({\mathop \mathcal{H}\limits^ \vee}_\alpha*\mu)(x,t)=\int_{\mathbb{R}^{N+1}}\mathcal{H}_\alpha(y-x,s-t)d\mu(y,s),
                $$        
                 and ${\mathop \mathcal{G}\limits^ \vee}_\alpha[\mu]$ is the backward parabolic  Bessel potential:
                            $$
                            {\mathop \mathcal{G}\limits^ \vee}_\alpha[\mu](x,t)=({\mathop \mathcal{G}\limits^ \vee}_\alpha*\mu)(x,t)=\int_{\mathbb{R}^{N+1}}\mathcal{G}_\alpha(y-x,s-t)d\mu(y,s).
                            $$ 
     \end{proposition}
     \begin{proof}\textbf{a.} We have:
$$
     \frac{1 }{t^{\frac{N+2-\alpha}{2}}}\chi_{t>0}\chi_{|x|\leq2 \sqrt{t}}\lesssim \mathcal{H}_\alpha(x,t)\lesssim \frac{1}{\max\{|x|,\sqrt{2|t|}\}^{N+2-\alpha}},
$$
     which implies
$$
         \int_{0}^{\infty}\frac{\chi_{B_r(0)\times (\frac{r^2}{4},r^2)}(x,t)}{r^{N+2-\alpha}}\frac{dr}{r}\lesssim \mathcal{H}_\alpha(x,t)\lesssim \int_{0}^{\infty}\frac{\chi_{\tilde{Q}_r(0,0)}(x,t)}{r^{N+2-\alpha}}\frac{dr}{r}.
$$
  Thus,  
     \begin{equation}\label{5hh230120147}
    \int_{0}^{\infty}\frac{\mu\left(B(x,r)\times (t-r^2,t-\frac{r^2}{4})\right)}{r^{N+2-\alpha}}\frac{dr}{r}\lesssim \mathcal{H}_\alpha[\mu](x,t) \lesssim\mathbb{I}_\alpha[\mu](x,t).
     \end{equation}
   Thanks to Theorem \ref{5hh051120131} we will finish the proof of \eqref{5hh230120141} if we show that 
$$
   \int_{\mathbb{R}}\left(\int_{0}^{\infty}\frac{\mu\left(B(x,r)\times (t-r^2,t-\frac{r^2}{4})\right)}{r^{N+2-\alpha}}\frac{dr}{r}\right)^qdt\gtrsim \int_{\mathbb{R}}\int_{0}^{+\infty}\left(\frac{\mu(\tilde{Q}_r(x,t))}{r^{N+2-\alpha}}\right)^q\frac{dr}{r}dt.
$$
  Indeed, we have for $r_k=(\frac{2}{\sqrt{3}})^{-k}$, 
$$
  \left(\int_{0}^{\infty}\frac{\mu\left(B(x,r)\times (t-r^2,t-r^2/4)\right)}{r^{N+2-\alpha}}\frac{dr}{r}\right)^q
  \gtrsim \sum_{k=-\infty}^{\infty}\left(\frac{\mu\left(B(x,r_{k})\times (t-r_{k}^2,t-\frac{1}{3}r_k^2)\right)}{r_k^{N+2-\alpha}}\right)^q.$$
   Thus,
  \begin{align*}
  &\int_{\mathbb{R}}\left(\int_{0}^{\infty}\frac{\mu\left(B(x,r)\times (t-r^2,t-\frac{1}{4}r^2)\right)}{r^{N+2-\alpha}}\frac{dr}{r}\right)^qdt\\&~~~~\gtrsim \sum_{k=-\infty}^{\infty}\int_{\mathbb{R}}\left(\frac{\mu\left(B(x,r_{k})\times (t-\frac{1}{3}r_{k}^2,t+\frac{1}{3}r_k^2)\right)}{r_k^{N+2-\alpha}}\right)^qdt
  \\&~~~~
  \gtrsim \int_{\mathbb{R}}\int_{0}^{+\infty}\left(\frac{\mu(\tilde{Q}_r(x,t))}{r^{N+2-\alpha}}\right)^q\frac{dr}{r}dt.
  \end{align*} 
  Similarly, we also can prove \eqref{5hh230120142}. \\
  \textbf{b.} Obviously   \begin{align*}
         &\frac{ \exp(-4R^2) }{t^{\frac{N+2-\alpha}{2}}}\chi_{0<t<4R^2}\chi_{|x|\leq2 \sqrt{t}}\lesssim \mathcal{G}_\alpha(x,t)\\&\quad\quad\quad\quad\lesssim \frac{\chi_{\tilde{Q}_{R/2}(0,0)}(x,t)}{\max\{|x|,\sqrt{2|t|}\}^{N+2-\alpha}}+\frac{\exp\left(-\max\{|x|,\sqrt{2|t|}\}\right)}{R^{N+2-\alpha}}.
          \end{align*}
     Thus, we can assert that 
     \begin{align*}
       &\int_{0}^{2R}\frac{\chi_{B_r(0)\times (\frac{r^2}{4},r^2)}(x,t)}{r^{N+2-\alpha}}\frac{dr}{r}\lesssim_R \mathcal{G}_\alpha(x,t)\lesssim \int_{0}^{R}\frac{\chi_{\tilde{Q}_r(0,0)}(x,t)}{r^{N+2-\alpha}}\frac{dr}{r}\\&~~~~~~~~~~~~~+c(R)\int_{\mathbb{R}^{N+1}}\exp\left(-\max\{|y|,\sqrt{2|s|}\}\right)\chi_{\tilde{Q}_{R/2}(0,0)}(x-y,t-s)dyds.
                       \end{align*}
Immediately, we get 
\begin{align}
      &\int_{0}^{2R}\frac{\mu \left(B(x,r)\times(t-r^2,t-\frac{r^2}{4})\right)}{r^{N+2-\alpha}} \frac{dr}{r} \lesssim_R \mathcal{G}_\alpha[\mu](x,t) \lesssim \mathbb{I}_\alpha^R[\mu](x,t) +c(R)\mathbf{F}(x,t),
                       \end{align}                                     where $$\mathbf{F}(x,t)=\int_{\mathbb{R}^{N+1}}\exp\left(-\max\{|y|,\sqrt{2|s|}\}\right)\mu\left(\tilde{Q}_{R/2}(x-y,t-s)\right) dyds.$$    
  As above, we can show that
  \begin{align*}
  \int_{0}^{\infty}\left(\int_{0}^{2R}\frac{\mu \left(B(x,r)\times(t-r^2,t-\frac{r^2}{4})\right)}{r^{N+2-\alpha}} \frac{dr}{r}\right)^qdt\gtrsim \int_{0}^{\infty}\int_{0}^{R}\left(\frac{\mu(\tilde{Q}_r(x,t))}{r^{N+2-\alpha}}\right)^q\frac{dr}{r}.
  \end{align*}
  Thus, thanks to Theorem \ref{5hh051120131} we get the left-hand side inequality of \eqref{5hh230120141'}. \\To obtain the right-hand side of \eqref{5hh230120141'}, we use  $\mu\left(\tilde{Q}_{R/2}(x-y,t-s)\right)\lesssim R^{-(N+2-\alpha)}\mathbb{I}_\alpha^R[\mu] (x-y,t-s)$ and Young's inequality
  \begin{align*}
 &||\mathcal{G}_\alpha[\mu]||_{L^q(\mathbb{R}^{N+1})}\lesssim_R ||\mathbb{I}^R_\alpha[\mu]||_{L^q(\mathbb{R}^{N+1})}+||\mathbf{F}||_{L^q(\mathbb{R}^{N+1})}\\&~~~~~~\lesssim_R ||\mathbb{I}^R_\alpha[\mu]||_{L^q(\mathbb{R}^{N+1})}+||\mathbb{I}^R_\alpha[\mu]||_{L^q(\mathbb{R}^{N+1})}\int_{\mathbb{R}^{N+1}}\exp\left(-\max\{|x|,\sqrt{2|t|}\}\right)dxdt
 \\&~~~~~~\lesssim_R||\mathbb{I}^R_\alpha[\mu]||_{L^q(\mathbb{R}^{N+1})}. \end{align*}
   Similarly, we also can obtain \eqref{5hh230120142'}. This completes the proof.
     \end{proof}
      \begin{remark} Assume that $0<\alpha<N+2$.
         From \eqref{5hh2210133} in Proposition \ref{5hh23101315} and $
          ||\mathcal{G}_\alpha[\mu]||_{L^{1}(\mathbb{R}^{N+1})}\lesssim \mu(\mathbb{R}^{N+1})
          $ we deduce that for $1\leq s<\frac{N+2}{N+2-\alpha}$
          \begin{align*}
          ||\mathcal{G}_\alpha[\mu]||_{L^{s}(\mathbb{R}^{N+1})}\lesssim \mu(\mathbb{R}^{N+1})~~\forall ~\mu\in \mathfrak{M}_b^+(\mathbb{R}^{N+1}).
          \end{align*}
         \end{remark}
         Next, we introduce the following kernel:
         \begin{align*}
         E_{\alpha}^{R}(x,t)=\max\{|x|,\sqrt{2|t|}\}^{-(N+2-\alpha)}\chi_{\tilde{Q}_R(0,0)}(x,t)
         \end{align*}  
         where $0<\alpha<N+2$ and $0<R\leq\infty$. We denote $E_\alpha^\infty$ by $E_\alpha$. It is easy to see that $E_\alpha*\mu=(N+2-\alpha)\mathbb{I}_\alpha[\mu]$ and 
         $||E^R _\alpha*\mu||_{L^{s}(\mathbb{R}^{N+1})}$ is equivalent to $||\mathbb{I}^R _\alpha[\mu]||_{L^{s}(\mathbb{R}^{N+1})}$ for every $\mu\in\mathfrak{M}^+(\mathbb{R}^{N+1})$ where $1\leq s<\infty$.\\          
We obtain equivalences of capacities $\text{Cap}_{E_\alpha,p}, \text{Cap}_{E_\alpha^R,p},\text{Cap}_{\mathcal{H}_\alpha,p}$ and $\text{Cap}_{\mathcal{G}_\alpha,p}$. 
     \begin{corollary}\label{5hh230120148}
     Let $p>1$, $1<\alpha<N+2$ and $R>0$. Then, the following statements hold
     \begin{description}
     \item[a.]for any compact $E\subset \mathbb{R}^{N+1}$
                                     \begin{equation}\label{5hh230120144}
              \text{Cap}_{E_\alpha,p}(E)\lesssim \text{Cap}_{\mathcal{H}_\alpha,p}(E), \end{equation}
                                                \item[b.]for any compact $E\subset \mathbb{R}^{N+1}$
                                                \begin{equation}\label{5hh230120144'}
                                \text{Cap}_{E^R_\alpha,p}(E)\lesssim_R \text{Cap}_{\mathcal{G}_\alpha,p}(E), \end{equation}
   \item[c.]for any compact $E\subset \mathbb{R}^{N+1}$
   \begin{equation}\label{5hh23012014'}
                \text{Cap}_{\mathcal{H}_\alpha,p}(E) \leq  \text{Cap}_{\mathcal{G}_\alpha,p}(E) \lesssim \text{Cap}_{\mathcal{H}_\alpha,p}(E)+\left(\text{Cap}_{\mathcal{H}_\alpha,p}(E)\right)^{\frac{N+2}{N+2-\alpha p}}\end{equation}
                                                 provided $1<\alpha p <N+2$.                                              
     \end{description}                                          
                                              
     \end{corollary}
    \begin{proof}
    By \cite[Chapter 2]{55AH}, we have 
    \begin{align*}
     &\text{Cap}_{E_\alpha,p}(E)^{1/p}=\sup\{\mu(E):\mu\in\mathfrak{M}^+(E), ||E _\alpha*\mu||_{L^{p'}(\mathbb{R}^{N+1})}\leq 1\},\\&
      \text{Cap}_{E^R_\alpha,p}(E)^{1/p}=\sup\{\mu(E):\mu\in\mathfrak{M}^+(E),||E^R _\alpha*\mu||_{L^{p'}(\mathbb{R}^{N+1})}\leq 1\},\\&
      \text{Cap}_{\mathcal{H}_\alpha,p}(E)^{1/p}=\sup\{\mu(E):\mu\in\mathfrak{M}^+(E),||{\mathop \mathcal{H}\limits^ \vee}_\alpha[\mu]||_{L^{p'}(\mathbb{R}^{N+1})}\leq 1\},\\& \text{Cap}_{\mathcal{G}_\alpha,p}(E)^{1/p}=\sup\{\mu(E):\mu\in\mathfrak{M}^+(E),||{\mathop \mathcal{G}\limits^ \vee}_\alpha[\mu]||_{L^{p'}(\mathbb{R}^{N+1})}\leq 1\}.
    \end{align*}                                         
     Thanks to \eqref{5hh230120142}, \eqref{5hh230120142'}  in Proposition \ref{5hh230120143} and  $\mathbb{I}_\alpha[\mu]= E_\alpha*\mu$ and $||E^R _\alpha*\mu||_{L^{s}(\mathbb{R}^{N+1})}\sim||\mathbb{I}^R _\alpha[\mu]||_{L^{s}(\mathbb{R}^{N+1})}$,  we get \eqref{5hh230120144} and \eqref{5hh230120144'}.\\
     Since $\mathcal{G}_\alpha\leq \mathcal{H}_\alpha$, thus $\text{Cap}_{\mathcal{H}_\alpha,p}(E) \leq  \text{Cap}_{\mathcal{G}_\alpha,p}(E)$ for any compact $E\subset \mathbb{R}^{N+1}$. Set $\text{Cap}_{E_\alpha,p}(E)=a>0$. We need to prove that
     \begin{equation}\label{5hh160220141}
       \text{Cap}_{E_\alpha^1,p}(E) \lesssim a+a^{\frac{N+2}{N+2-\alpha p}}.\end{equation}       
       We will follow a proof of Yu.V. Netrusov in \cite[Chapter 5]{55AH}. First, we can find $f\in L^p_+(\mathbb{R}^{N+1})$ such that $||f||_{L^p(\mathbb{R}^{N+1})}\leq 2a$ and $E_\alpha*f\geq \chi_E$. 
       Set $F_\alpha=E_\alpha-E^1_\alpha$, we have $c_1 F_\alpha\leq E_\alpha^1*F_\alpha$ for some $c_1>0$. Thus, 
       $
       E\subset \{E^1_\alpha*f\geq 1/2\}\cup \{E^1_\alpha*(F_\alpha*f)\geq c_1/2\} 
       $. \\
       Since $||E^1_\alpha||_{L^1(\mathbb{R}^{N+1})}<\infty$, for $c_2=c_1(4 ||E^1_\alpha||_{L^1(\mathbb{R}^{N+1})})^{-1}$  $$
       E^1_\alpha*(F_\alpha*f)\leq c_1/4+E^1_\alpha*g \text{ with } g=\chi_{F_\alpha*f\geq c_2}F_\alpha*f,
       $$
      which follows
      $
           E\subset \{E^1_\alpha*f\geq 1/2\}\cup \{E^1_\alpha*g\geq c_1/4\} 
           $. \\
      Using the subadditivity of capacity, we have
      \begin{align*}
      \text{Cap}_{E_\alpha^1,p}(E) &\leq  \text{Cap}_{E_\alpha^1,p}(\{E^1_\alpha*f\geq 1/2\})+\text{Cap}_{E_\alpha^1,p}(\{E^1_\alpha*g\geq c_1/4\})
      \\&\lesssim ||f||_{L^{p}(\mathbb{R}^{N+1})}^p+||g||_{L^{p}(\mathbb{R}^{N+1})}^p
      \\&\lesssim ||f||_{L^{p}(\mathbb{R}^{N+1})}^p+||E_\alpha*f||_{L^{p*}(\mathbb{R}^{N+1})}^{p*}, \text{ with } p*=\frac{(N+2)p}{N+2-\alpha p}.
      \end{align*}  
On the other hand, from \eqref{5hh2210135} in Proposition \ref{5hh23101315} we have 
$$||E_\alpha*f||_{L^{p*}(\mathbb{R}^{N+1})}\lesssim||f||_{L^{p}(\mathbb{R}^{N+1})}.$$
  Hence, we get \eqref{5hh160220141}. The proof is complete.   
    \end{proof} \\
    \begin{remark}\label{5hh240320148}
    Since $\mathcal{G}_\alpha\in L^1(\mathbb{R}^{N+1})$, 
$$
    \int_{\mathbb{R}^{N+1}}\left(\mathcal{G}_\alpha*f\right)^pdxdt\leq ||\mathcal{G}_\alpha||_{L^1(\mathbb{R}^{N+1})}^p\int_{\mathbb{R}^{N+1}}f^pdxdt~~\forall f\in L^p_+(\mathbb{R}^{N+1}).
$$
    Thus, for any Borel set $E\subset \mathbb{R}^{N+1}$\begin{equation}\label{5hh240320147}
    Cap_{\mathcal{G}_\alpha,p}(E)\geq C|E|~\text{with }C=||\mathcal{G}_\alpha||_{L^1(\mathbb{R}^{N+1})}^{-p}.
    \end{equation}
    \end{remark}
    \begin{remark}  It is well-known that $\mathcal{H}_2$ is the fundamental solution of the heat operator $\frac{\partial}{\partial t}-\Delta$. In \cite{55GZ}, R. Gariepy and W. P. Ziemer introduced the following capacity:
$$
      \text{C}_{\mathcal{H}_2}(K)=\sup\left\{\mu(K):\mu\in \mathfrak{M}^+(K), \mathcal{H}_2[\mu]\leq \right\},$$
       whenever $K\subset \mathbb{R}^{N+1}$ is compact. 
       Thanks to  \cite[Theorem 2.5.5]{55AH}, we obtain $$
    \text{Cap}_{\mathcal{H}_1,2}(K)=\text{C}_{\mathcal{H}_2}(K).$$
     \end{remark}
     \begin{remark}\label{5hh270320146}
     For any Borel set $E\subset \mathbb{R}^{N}$, then we always have $\text{Cap}_{\mathcal{G}_1,2}(E\times\{t=0\})=0$
     In fact, for $B_1=B_1(0)$
     \begin{align*}
     \text{Cap}_{E^1_1,2}(B_1\times\{t=0\})=\sup\{ \omega(B_1):\omega\in\mathfrak{M}^+(B_1),||E^1 _1*(\omega\otimes\delta_0)||_{L^{2}(\mathbb{R}^{N+1})}\leq 1\}.
     \end{align*}
  Since $||E^1 _1*(\omega\otimes\delta_0)||_{L^{2}(\mathbb{R}^{N+1})}=\infty$ if $\omega\not=0$, thus $\text{Cap}_{\mathcal{G}_1,2}(B_1\times\{t=0\})=\text{Cap}_{E^1_1,2}(B_1\times\{t=0\})=0$.
  In particular, $\text{Cap}_{\mathcal{G}_1,2}$ is not absolutely continuous with respect to capacity $C_{1,2}(.,{\Omega\times (a,b)})$. This capacity will be defined in next section. 
     \end{remark}
     \begin{remark}\label{5hh020520142}
     Let $p>1,\alpha>0$ and $\tilde{Q}_\rho=\tilde{Q}_\rho(0,0)$ for $\rho>0$. Case  $\alpha p\geq N+1$, we always have $||\mathcal{H}_\alpha[\mu]||_{L^{p'}(\mathbb{R}^N)}=\infty$ for any $\mu\in\mathfrak{M}^+(\mathbb{R}^N)\backslash\{0\}$. This implies  $\text{Cap}_{\mathcal{H}_\alpha,p}(\tilde{Q}_1)=0$. If $0<\alpha p<N+2$, $\text{Cap}_{\mathcal{H}_\alpha,p}(\tilde{Q}_\rho)=c\rho^{N+2-\alpha p}$ for some constant $c>0$.  From \eqref{5hh23012014'} in Corollary \ref{5hh230120148} we get $\text{Cap}_{\mathcal{G}_\alpha,p}(\tilde{Q}_\rho)\sim \rho^{N+2-\alpha p}$ for any $0<\rho<1$ if $\alpha p< N+2$. Since $||\mathcal{G}_\alpha[\delta_{(0,0)}]||_{L^{p'}(\mathbb{R}^{N+1})}<\infty$ thus $\text{Cap}_{\mathcal{G}_\alpha,p}((0,0))>0$ if $\alpha p>N+2$.\\ If $\alpha p= N+2$,  $\text{Cap}_{\mathcal{G}_\alpha,p}(\tilde{Q}_\rho)\sim \left(\log(1/\rho)\right)^{1-p}$ for any $0<\rho<1/2$. In fact, 
     we can prove that $||\mathbb{I}^{1/2}_{\alpha}[\mu]||_{L^{p'}(\mathbb{R}^N)}\lesssim 1$  for any $d\mu(x,t)=\left(\text{log}(1/\rho)\right)^{-1/p'}\rho^{-N-2}\chi_{\tilde{Q}_\rho}dxdt$ it follows $\text{Cap}_{\mathcal{G}_\alpha,p}(\tilde{Q}_\rho)\gtrsim \left(\log(1/\rho)\right)^{1-p}$. Moreover, for $\mu\in\mathfrak{M}^+(\tilde{Q}_\rho)$, if $||\mathbb{I}_\alpha^3[\mu]||_{L^{p'}(\mathbb{R}^{N+1})}^{p'}\leq 1$, 
     \begin{align*}
     1&\geq \int_{\tilde{Q}_1\backslash \tilde{Q}_\rho }\left(\int_{2\max\{|x|,|2t|^{1/2}\}}^{3}\frac{\mu(\tilde{Q}_r(x,t))}{r^{N+2-\alpha}}\frac{dr}{r}\right)^{p'}dxdt
      \\&\geq \int_{\tilde{Q}_1\backslash \tilde{Q}_\rho }\left(\int_{2\max\{|x|,|2t|^{1/2}\}}^{3}\frac{1}{r^{N+2-\alpha}}\frac{dr}{r}\right)^{p'}dxdt\mu(\tilde{Q}_\rho)^{p'}
      \\& \gtrsim\log(1/\rho)\mu(\tilde{Q}_\rho)^{p'}. 
     \end{align*}
     So
     $\text{Cap}_{\mathcal{G}_\alpha,p}(\tilde{Q}_\rho)\lesssim \mu(\tilde{Q}_\rho)^{p}\lesssim\left(\log(1/\rho)\right)^{1-p}$. 
     \end{remark}
     \begin{definition}
   The parabolic Bessel potential $\mathcal{L}_\alpha^p(\mathbb{R}^{N+1})$, $\alpha>0$ and $p>1$ is defined by 
$$
         \mathcal{L}_\alpha^p(\mathbb{R}^{N+1})= \left\{f:f=\mathcal{G}_\alpha*g, g\in L^p(\mathbb{R}^{N+1})\right\}
$$
         with the norm $||f||_{\mathcal{L}_\alpha^p(\mathbb{R}^{N+1})}:=||g||_{L^p(\mathbb{R}^{N+1})}$.
   We denote its dual space by $\left(\mathcal{L}_\alpha^p(\mathbb{R}^{N+1})\right)^*$.
     \end{definition}
  \begin{definition}
  Let $k$  be a positive integer, the Sobolev space $W^{2k,k}_p(\mathbb{R}^{N+1})$ is defined by 
$$
      W^{2k,k}_p(\mathbb{R}^{N+1})=\left\{\varphi:\frac{\partial^{i_1+...+i_N+i} \varphi}{\partial x_1^{i_1}...\partial x_N^{i_N}\partial t^{i}}\in L^p(\mathbb{R}^{N+1})~\text{for any}~i_1+...+i_N+2i\leq 2k\right\}
$$
      with the norm
$$
         ||\varphi||_{W^{2k,k}_p(\mathbb{R}^{N+1})}=\sum\limits_{i_1+...+i_N+2i\leq 2k} ||\frac{\partial^{i_1+...+i_N+i} \varphi}{\partial x_1^{i_1}...\partial x_N^{i_N}\partial t^{i}}||_{L^p(\mathbb{R}^{N+1})}.
$$
     We denote its dual space by $\left(W^{2k,k}_p(\mathbb{R}^{N+1})\right)^*$.     
        We also define a corresponding capacity on compact set $E\subset\mathbb{R}^{N+1}$,
$$
          Cap_{2k,k,p}(E)=\inf\left\{||\varphi||^p_{W^{2k,k}_p(\mathbb{R}^{N+1})}:\varphi\in S(\mathbb{R}^{N+1}), \varphi\geq 1 \text{ in a neighborhood of}~E\right \}. 
$$
  \end{definition} Let us recall Richard J. Bagby's result, proved in \cite{55Bag}.
     \begin{theorem}\label{5hh040220141} Let $p>1$ and $k$ be a positive integer. Then, there holds for any $u\in \mathcal{L}_{2k}^p(\mathbb{R}^{N+1})$,
      \begin{equation*}
       ||u||_{\mathcal{L}_{2k}^p(\mathbb{R}^{N+1})}\sim||u||_{W^{2k,k}_p(\mathbb{R}^{N+1})}.
      \end{equation*}
      \end{theorem}
      This Theorem gives the assertion of  equivalence of capacity $\text{Cap}_{2k,k,p},\text{Cap}_{\mathcal{G}_{2k},p}:$
     \begin{corollary}\label{5hh250220142}
     Let $p>1$ and $k$ be a positive integer. There exists a constant $C$ depending on $N,k,p$ such that for any compact set $E\subset\mathbb{R}^{N+1}$
$$\text{Cap}_{\mathcal{G}_{2k},p}(E)\sim\text{Cap}_{2k,k,p}(E).$$
        
     \end{corollary}
     Next result provides  some relations of Riesz, Bessel parabolic potential and Riesz, Bessel potential. 
  \begin{proposition}\label{5hh240120146}
   Let $q>1$ and $\frac{2}{q'}<\alpha<N+\frac{2}{q'}$. There hold  for any $\omega\in\mathfrak{M}^+(\mathbb{R}^N)$
   \begin{equation} ||\mathcal{H}_\alpha[\omega\otimes\delta_{\{t=0\}}]||_{L^q(\mathbb{R}^{N+1})}\sim||{\mathop \mathcal{H}\limits^ \vee}_\alpha[\omega\otimes\delta_{\{t=0\}}]||_{L^q(\mathbb{R}^{N+1})}\sim||{\bf I}_{\alpha-\frac{2}{q'}}[\omega]||_{L^q(\mathbb{R}^N)},\label{5hh240120145}
   \end{equation}
   \begin{equation} ||\mathcal{G}_\alpha[\omega\otimes\delta_{\{t=0\}}]||_{L^q(\mathbb{R}^{N+1})}\sim ||{\mathop \mathcal{G}\limits^ \vee}_\alpha[\omega\otimes\delta_{\{t=0\}}]||_{L^q(\mathbb{R}^{N+1})}\sim||{\bf G}_{\alpha-\frac{2}{q'}}[\omega]||_{L^q(\mathbb{R}^N)}\label{5hh240120145'}
    \end{equation}
     where $\delta_{\{t=0\}}$ is the Dirac mass in time at $0$. 
  \end{proposition}
  \begin{proof} 
  We have
  \begin{align*}
  \mathbb{I}_\alpha[\omega\otimes\delta_{\{t=0\}}](x,t)=\int_{\sqrt{2|t|}}^{\infty}\frac{\omega(B(x,r))}{r^{N+2-\alpha}}\frac{dr}{r},~\mathbb{I}^1_\alpha[\omega\otimes\delta_{\{t=0\}}](x,t)=\int_{\min\{1,\sqrt{2|t|}\}}^{1}\frac{\omega(B(x,r))}{r^{N+2-\alpha}}\frac{dr}{r}.
  \end{align*}
  By \cite[Theorem 2.3 ]{55VHV} and Proposition \ref{5hh230120143}, thus it is enough to show that
  \begin{equation}
  \int_{\mathbb{R}}\left(\int_{\sqrt{2|t|}}^{\infty}\frac{\omega(B(x,r))}{r^{N+2-\alpha}}\frac{dr}{r}\right)^qdt\sim \int_{0}^{\infty}\left(\frac{\omega(B(x,r))}{r^{N+2-\alpha-\frac{2}{q}}}\right)^q\frac{dr}{r},\label{5hh240120144}
  \end{equation}
  \begin{equation}
    \int_{0}^{1/2}\left(\frac{\omega(B(x,r))}{r^{N+2-\alpha-\frac{2}{q}}}\right)^q\frac{dr}{r}\lesssim \int_{\mathbb{R}}\left(\int_{\min\{1,\sqrt{2|t|}\}}^{1}\frac{\omega(B(x,r))}{r^{N+2-\alpha}}\frac{dr}{r}\right)^qdt\lesssim \int_{0}^{1}\left(\frac{\omega(B(x,r))}{r^{N+2-\alpha-\frac{2}{q}}}\right)^q\frac{dr}{r}.\label{5hh240120144'}
    \end{equation}
  Indeed, by changing of variables
 $$
  \int_{-\infty}^{\infty}\left(\int_{\sqrt{2|t|}}^{\infty}\frac{\omega(B(x,r))}{r^{N+2-\alpha}}\frac{dr}{r}\right)^qdt=2\int_{0}^{\infty}t\left(\int_{t}^{\infty}\frac{\omega(B(x,r))}{r^{N+2-\alpha}}\frac{dr}{r}\right)^qdt.$$
   Using Hardy's inequality, we have
$$
  \int_{0}^{\infty}t\left(\int_{t}^{\infty}\frac{\omega(B(x,r))}{r^{N+2-\alpha}}\frac{dr}{r}\right)^qdt\lesssim \int_{0}^{\infty}r\left(\frac{\omega(B(x,r))}{r^{N+2-\alpha}}\right)^qdr $$
  and using the fact that
$$
  \int_{t}^{\infty}\frac{\omega(B(x,r))}{r^{N+2-\alpha}}\frac{dr}{r}\gtrsim\frac{\omega(B(x,t))}{t^{N+2-\alpha}},$$
  we get
$$
  \int_{0}^{\infty}t\left(\int_{t}^{\infty}\frac{\omega(B(x,r))}{r^{N+2-\alpha}}\frac{dr}{r}\right)^qdt\gtrsim \int_{0}^{\infty}r\left(\frac{\omega(B(x,r))}{r^{N+2-\alpha}}\right)^qdr.$$
  Thus, we get \eqref{5hh240120144}. Likewise, we also obtain \eqref{5hh240120144'}.
  \end{proof}\medskip\\
 We have comparisons of $\text{Cap}_{\mathcal{H}_\alpha,p},\text{Cap}_{\mathcal{G}_\alpha,p},\text{Cap}_{{\bf I}_{\alpha-\frac{2}{p}},p},\text{Cap}_{{\bf G}_{\alpha-\frac{2}{p}},p}$. 
  \begin{corollary}\label{5hh250320147}
  Let $p>1$ and $\frac{2}{p}<\alpha<N+\frac{2}{p}$. There exists a positive constant $C$ depending on $N,q,\alpha$ such that for any compact $K\subset \mathbb{R}^{N}$
  \begin{align}
  	\label{5hh240120147}
  &	 \text{Cap}_{\mathcal{H}_\alpha,p}(K\times\{0\})\sim \text{Cap}_{{\bf I}_{\alpha-\frac{2}{p}},p}(K),\\&\label{5hh240120147'}
  \text{Cap}_{\mathcal{G}_\alpha,p}(K\times\{0\})\sim \text{Cap}_{{\bf G}_{\alpha-\frac{2}{p}},p}(K).
  \end{align}
  \end{corollary}
  \begin{proof}
  By \cite[Chapter 2]{55AH}, we have 
  \begin{align*}
   &\text{Cap}_{\mathcal{H}_\alpha,p}(K\times\{0\})^{1/p}=\sup\{ \omega(K):\omega\in\mathfrak{M}^+(K),||{\mathop \mathcal{H}\limits^ \vee}_\alpha[\omega\otimes\delta_{\{t=0\}}]||_{L^{p'}(\mathbb{R}^{N+1})}\leq 1\},
   \\&\text{Cap}_{\mathcal{G}_\alpha,p}(K\times\{0\})^{1/p}
  =\sup\{ \omega(K):\omega\in\mathfrak{M}^+(K),||{\mathop \mathcal{G}\limits^ \vee}_\alpha[\omega\otimes\delta_0]||_{L^{p'}(\mathbb{R}^{N+1})}\leq 1\},\\& \text{Cap}_{{\bf I}_{\alpha-\frac{2}{p}},p}(K)^{1/p}=\sup\{\omega(K):\omega\in\mathfrak{M}^+(K),||{\bf I}_{\alpha-\frac{2}{p}}[\omega]||_{L^{p'}(\mathbb{R}^{N+1})}\leq 1\},\\&
  \text{Cap}_{{\bf G}_{\alpha-\frac{2}{p}},p}(K)^{1/p}=\sup\{\omega(K):\omega\in\mathfrak{M}^+(K),||{\bf G}_{\alpha-\frac{2}{p}}[\omega]||_{L^{p'}(\mathbb{R}^{N+1})}\leq 1\}.
  \end{align*}                         
                       Therefore, thanks to Proposition \eqref{5hh240120146} we get the results.
  \end{proof}
   \begin{corollary}\label{5hh250220143}
     Let $p>1$ and $k$ be a positive integer such that $2k<N+2/p$. There exists a positive constant $C$ depending on $N,k,p$ such that for any compact set $K\subset\mathbb{R}^{N}$, we have
$$
         \text{Cap}_{2k,k,p}(K\times\{0\})\sim\text{Cap}_{{\bf G}_{2k-\frac{2}{p}},p}(K).$$
     \end{corollary}
     We also have comparisons of $\text{Cap}_{\mathcal{G}_\alpha,p}, \text{Cap}_{\mathbf{G}_\alpha,p}$. 
  \begin{proposition} Let $0<\alpha<N$, $p>1$. For $a>0$ there holds for any compact $K\subset\mathbb{R}^N$,
$$
  \text{Cap}_{\mathcal{G}_\alpha,p}(K\times [-a,a])\sim_a \text{Cap}_{\mathbf{G}_\alpha,p}(K).$$
  \end{proposition}
  \begin{proof} By \cite{55AH}, we have 
 $$
  \text{Cap}_{\mathbf{I}_\alpha^{\frac{\sqrt{a}}{2}},p}(K)\lesssim_a\text{Cap}_{\mathbf{G}_\alpha,p}(K).
 $$
  So, we can find $f\in L^p_+(\mathbb{R}^N)$ such that  $\mathbf{I}_\alpha^{\frac{\sqrt{a}}{2}}*f\geq \chi_K$ and
  $$
  \int_{\mathbb{R}^N}|f|^pdx\lesssim\text{Cap}_{\mathbf{G}_\alpha,p}(K).
 $$
  Note that $(E_\alpha^{\sqrt{a}}*\tilde{f})(x,t)\gtrsim (\mathbf{I}_\alpha^{\frac{\sqrt{a}}{2}}*f)(x,t)$ for all $(x,t)\in\mathbb{R}^N\times[-a,a]$ where $\tilde{f}(x,t)=f(x)\chi_{[-2a,2a]}(t)$. So, 
$$
  \text{Cap}_{E_\alpha^{\sqrt{a}},p}(K\times [-a,a])\lesssim\int_{\mathbb{R}^{N+1}}|\tilde{f}|^pdxdt=2a\int_{\mathbb{R}^{N}}|f|^pdx.$$
  By Corollary \ref{5hh230120148}, one has  $$\text{Cap}_{\mathcal{G}_\alpha,p}(K\times [-a,a])\lesssim \text{Cap}_{E_\alpha^{\sqrt{a}},p}(K\times [-a,a]).$$
   Thus, we get 
   $$\text{Cap}_{\mathcal{G}_\alpha,p}(K\times [-a,a])\lesssim \text{Cap}_{\mathbf{G}_\alpha,p}(K).$$
   Finally, we prove other one. It is easy to see that 
$$
   ||\mathbb{I}_\alpha^{\sqrt{\frac{a}{2}}}[\omega\otimes \chi_{[-a,a]}]||_{L^{p'}(\mathbb{R}^{N+1})}\lesssim ||\mathbf{I}_\alpha^{\sqrt{\frac{a}{2}}}[\omega]||_{L^{p'}(\mathbb{R}^{N})}~~\forall~\omega\in\mathfrak{M}^+(\mathbb{R}^{N}),$$
  which implies
$$
    ||\mathcal{G}_\alpha[\omega\otimes \chi_{[-a,a]}]||_{L^{p'}(\mathbb{R}^{N+1})}\lesssim ||\mathbf{G}_\alpha[\omega]||_{L^{p'}(\mathbb{R}^{N})}~~\forall~\omega\in\mathfrak{M}^+(\mathbb{R}^{N+1}).$$
    It follows,
    $$\text{Cap}_{\mathcal{G}_\alpha,p}(K\times [-a,a])\gtrsim \text{Cap}_{\mathbf{G}_\alpha,p}(K).$$
   The proof is complete. 
  \end{proof}  \\\\
  The following proposition is useful for proving that many operators of classical analysis are bounded in the space of functions $f$ such that 
  \begin{equation*}
  \int_K |f|^pdxdt\lesssim \text{Cap} (K)
  \end{equation*}
  for every compact set $K\subset\mathbb{R}^{N+1}$, $(1<p<\infty)$, if they are bounded in $L^q(\mathbb{R}^{N+1},dw)$ with $w\in A_\infty$.
\begin{proposition}\label{5hh140420143}
          Let  $0<R\leq\infty$, $1< p\leq \alpha^{-1}(N+2)$, $0<\delta<\alpha$  and  $f,g\in L^1_{loc}(\mathbb{R}^{N+1})$. Suppose that 
          \begin{description}
          \item[1.] For any compact sets $K\subset\mathbb{R}^{N+1}$
$$
          \int_K|f|dxdt\lesssim \text{Cap}_{E_\alpha^{R,\delta},p}(K).$$
          \item[2.] For all weights $w\in A_1$,
$$
          \int_{\mathbb{R}^{N+1}} |g|wdxdt\lesssim_{[w]_{A_1}}\int_{\mathbb{R}^{N+1}} |f|wdxdt.$$
          \end{description}
          Then, $$
             \int_K|g| dxdt\lesssim_{\alpha, p,\delta}\text{Cap}_{E^{R,\delta}_\alpha,p}(K)~\text{ for any compact set }K\subset\mathbb{R}^{N+1}.$$
          \end{proposition}
          The capacity is mentioned in Proposition \eqref{5hh140420143}, that is $(E^{R,\delta}_\alpha,p)$-capacity defined by
\begin{equation*}
                                       \text{Cap}_{E^{R,\delta}_\alpha,p}(E)=\inf\left\{\int_{\mathbb{R}^{N+1}}|f|^pdxdt: f\in L^p_+(\mathbb{R}^{N+1}), E_\alpha^{R,\delta}*f\geq \chi_E\right\},                                       \end{equation*}
for all measurable sets $E\subset \mathbb{R}^{N+1}$, where $0<R\leq\infty$, $0<\delta<\alpha<N+2$,                                                   
  \begin{align*}
                     E_{\alpha}^{R,\delta}(x,t)=\max\left\{|x|,\sqrt{2|t|}\right\}^{-(N+2-\alpha)}\min\left\{1,\left(\frac{\max\{|x|,\sqrt{2|t|}\}}{R}\right)^{-\delta}\right\}.
                     \end{align*}                
       \begin{remark}\label{5hh080120141}For $0<\alpha q<N+2$, the inequality \eqref{5hh2210135} in Proposition \ref{5hh23101315} implies 
 \begin{equation}\label{5hh160220142}
    \left(\int_{\mathbb{R}^{N+1}}\left(E_{\alpha}^{R,\delta}*f\right)^{\frac{q(N+2)}{N+2-\alpha q}}dxdt\right)^{1-\frac{\alpha q}{N+2}} \lesssim\int_{\mathbb{R}^{N+1}}f^qdxdt ~~\forall f\in L^q(\mathbb{R}^{N+1}), f\geq 0.
             \end{equation}
      Hence, we get the isoperimetric inequality:
     \begin{equation}\label{5hh080120142}|E|^{1-\frac{\alpha p}{N+2}}\lesssim \text{Cap}_{E^{R,\delta}_\alpha,p}(E),
                                               \end{equation}
     for all measurable sets $E\subset \mathbb{R}^{N+1}$.\\
                                         \end{remark} 
           Also, we recall that a positive function $w\in L^1_{loc}(\mathbb{R}^{N+1})$ is called  an $A_1$ weight, 
            if 
            \begin{equation*}
            [w]_{A_1}:=\sup\left(\left(\fint_Qwdyds \right) \mathop {ess\sup }\limits_{(x,t)\in Q}\frac{1}{w(x,t)} \right)<\infty,
            \end{equation*} 
            where the supremum is taken over all cylinder $Q=\tilde{Q}_R(x,t)\subset \mathbb{R}^{N+1}$. The constant $[w]_{A_1}$ is called the $A_1$ constant of $w$.\\\\
            To prove the Proposition \eqref{5hh140420143},  we need to introduce the $(R,\delta)-$Wolff parabolic potential,  
\begin{align*}
\mathbb{W}_{\alpha,p}^{R,\delta}[\mu](x,t)=\int_{0}^{\infty}\left(\frac{\mu(\tilde{Q}_\rho(x,t))}{\rho^{N+2-\alpha p}}\right)^{\frac{1}{p-1}}\min\left\{1,\left(\frac{\rho}{R}\right)^{-\delta}\right\}\frac{d\rho}{\rho}~~\text{for any} ~(x,t)\in \mathbb{R}^{N+1},
\end{align*}
where $p>1$, $0<\alpha p \leq N+2$, $0<\delta<\alpha p'$ and $0<R\leq \infty$ and $\mu\in\mathfrak{M}^+(\mathbb{R}^{N+1})$. \\
It is easy to see that 
\begin{equation}\label{5hh240220141}
               \mathbb{W}^{R,\delta}_{\alpha,p}[\mu](x,t)\lesssim_{\alpha,p,\delta}\mathop {\sup}\limits_{(y,s)\in \text{supp} \mu }  \mathbb{W}^{R,\delta}_{\alpha,p}[\mu](y,s).
               \end{equation}
           
\begin{remark}\label{5hh150420142} We easily verify that the Theorem \ref{5hh100420141} also holds for $\mathbb{W}^{R,\delta,R_1}_{\alpha,p}[\mu]$ and $\mathbb{M}^{R,\delta,R_1}_{\alpha p}[\mu]$:
\begin{align*}
&\mathbb{W}_{\alpha,p}^{R,\delta,R_1}[\mu](x,t)=\int_{0}^{R_1}\left(\frac{\mu(\tilde{Q}_\rho(x,t))}{\rho^{N+2-\alpha p}}\right)^{\frac{1}{p-1}}\min\left\{1,\left(\frac{\rho}{R}\right)^{-\delta}\right\}\frac{d\rho}{\rho},\\&
\mathbb{M}_{\alpha,p}^{R,\delta/(p-1),R_1}[\mu](x,t)=\sup_{0<\rho<R_1}\left(\frac{\mu(\tilde{Q}_\rho(x,t))}{\rho^{N+2-\alpha p}}\min\left\{1,\left(\frac{\rho}{R}\right)^{-\delta(p-1)}\right\}\right)~~\text{for any} ~(x,t)\in \mathbb{R}^{N+1},
\end{align*}
where $0<\delta<\alpha p'$, $1< p<\alpha^{-1}(N+2)$ and $R_1>R>0$. This means, 
for $w\in A_\infty$, $\mu\in\mathfrak{M}^+(\mathbb{R}^{N+1})$, there exist positive constants $C>0$ and $\varepsilon_0\in (0,1)$ depending on $N,\alpha,p,\delta, [w]_{A_\infty}$ such that for any $\lambda>0$ and $\varepsilon\in (0,\varepsilon_0)$
\begin{equation}
 w(\{\mathbb{W}^{R,\delta,R_1}_{\alpha,p}[\mu]>a\lambda,(\mathbb{M}^{R,\delta(p-1),R_1}_{\alpha p}[\mu])^{\frac{1}{p-1}}\le \varepsilon \lambda \})\le C\exp(-(C\varepsilon)^{-1}) w(\{\mathbb{W}^{R,\delta,R_1}_{\alpha ,p}[\mu]>\lambda\}),
                     \end{equation}
 where $a=2+3^{\frac{N+2-\alpha p+\delta(p-1)}{p-1}}$.\\
 Therefore, for $q>p-1$ 
 \begin{equation*}
     ||\mathbb{W}^{R,\delta,R_1}_{\alpha,p}[\mu]||_{L^{q}(\mathbb{R}^{N+1},dw)}\lesssim C_1 ||(\mathbb{M}^{R,\delta(p-1),R_1}_{\alpha p}[\mu])^{\frac{1}{p-1}}||_{L^{q}(\mathbb{R}^{N+1},dw)},
     \end{equation*}
     where $C_1=C_1(\alpha,p,\delta,q)$.
Letting $R_1\to\infty$, we get
   \begin{equation}\label{5hh150420141}
       ||\mathbb{W}^{R,\delta}_{\alpha,p}[\mu]||_{L^{q}(\mathbb{R}^{N+1},dw)}\lesssim C_1 ||(\mathbb{M}^{R,\delta(p-1)}_{\alpha p}[\mu])^{\frac{1}{p-1}}||_{L^{q}(\mathbb{R}^{N+1},dw)},
       \end{equation}
       where   $\mathbb{M}^{R,\delta(p-1)}_{\alpha p}[\mu]:=\mathbb{M}^{R,\delta(p-1),\infty}_{\alpha p}[\mu]$. 
\end{remark}   
We will need the following three Lemmas to prove the Proposition \eqref{5hh140420143}.           
     \begin{lemma} \label{5hh1310136}Let $0<p\leq \alpha^{-1}(N+2)$ and $0<\beta <\frac{(N+2)(p-1)}{N+2-\alpha p+\delta (p-1)}$. There holds for each $\tilde{Q}_r=\tilde{Q}_r(x,t)$
    \begin{equation}\label{5hh09262}
   \fint_{\tilde{Q}_r} (\mathbb{W}^{R,\delta}_{\alpha,p}[\mu](y,s))^\beta dyds \lesssim_\delta (\mathbb{W}^{R,\delta}_{\alpha,p}[\mu](x,t))^\beta.
    \end{equation}
     \end{lemma}
     \begin{proof}
     We set 
     \begin{align*}
     & U^{r}_{\alpha,p}[\mu](y,s)= \int_{r}^{\infty}\left(\frac{|\mu|(\tilde{Q}_\rho(y,s))}{\rho^{N+2-\alpha p}}\right)^{\frac{1}{p-1}}\min\left\{1,\left(\frac{\rho}{R}\right)^{-\delta}\right\} \frac{d\rho}{\rho},\\&
     L^{r}_{\alpha,p}[\mu](y,s)= \int_{0}^{r}\left(\frac{\mu(\tilde{Q}_\rho(y,s))}{\rho^{N+2-\alpha p}}\right)^{\frac{1}{p-1}} \min\left\{1,\left(\frac{\rho}{R}\right)^{-\delta}\right\}\frac{d\rho}{\rho}.
     \end{align*}One has,
$$
              \fint_{\tilde{Q}_r} (\mathbb{W}^{R,\delta}_{\alpha,p}[\mu](y,s))^\delta dyds \lesssim  \fint_{\tilde{Q}_r} (U^r_{\alpha,p}[\mu](y,s))^\delta dyds+  \fint_{\tilde{Q}_r} (L^r_{\alpha,p}[\mu](y,s))^\delta dyds.
$$
               Since for each $(y,s)\in \tilde{Q}_r$ and $\rho\geq r$ we have $\tilde{Q}_\rho(y,s)\subset \tilde{Q}_{2\rho}(x,t)$, thus for each $(y,s)\in \tilde{Q}_r$,
$$U^r_{\alpha,p}[\mu](y,s)\leq  \int_{r}^{\infty}\left(\frac{\mu(\tilde{Q}_{2\rho}(x,t))}{\rho^{N+2-\alpha p}}\right)^{\frac{1}{p-1}}\left(\max\{1,\frac{\rho}{R}\}\right)^{-\delta} \frac{d\rho}{\rho}\lesssim \mathbb{W}^{R,\delta}_{\alpha,p}[\mu](x,t),$$
           which implies     
$$
                 \fint_{\tilde{Q}_r} (U^r_{\alpha,p}[\mu](y,s))^\delta dyds\lesssim (\mathbb{W}^{R,\delta}_{\alpha,p}[\mu](x,t))^\delta.$$
           Since   for each $(y,s)\in \tilde{Q}_r$ and $\rho\leq r$ we have $\tilde{Q}_\rho(y,s)\subset \tilde{Q}_{2 r}(x,t)$ thus, $L^r_{\alpha,p}[\mu]=L^r_{\alpha,p}[\mu\chi_{\tilde{Q}_{2 r}(x,t)}]\leq \mathbb{W}^{R,\delta}_{\alpha,p}[\mu\chi_{\tilde{Q}_{2 r}(x,t)}]$ in  $\tilde{Q}_{r}(x,t)$. 
           We now consider two cases. \\
           Case 1: $r\leq R$. We have for $a>0$,
           \begin{align*}
           \fint_{\tilde{Q}_r} (L^r_{\alpha,p}[\mu](y,s))^\beta dyds&\leq \fint_{\tilde{Q}_r} (\mathbb{W}^{r}_{\alpha,p}[\mu\chi_{\tilde{Q}_{2 r}(x,t)}](y,s))^\beta dyds \\&=\frac{1}{|\tilde{Q}_r|} \beta\int_{0}^{\infty}\lambda^{\beta-1}|\{\mathbb{W}_{\alpha,p}^r[\mu \chi_{\tilde{Q}_{2 r}(x,t)}]>\lambda \}\cap\tilde{Q}_r | d\lambda 
         \\&\lesssim \lambda_0^{\beta}+ r^{-N-2}\int_{\lambda_0}^{\infty}\lambda^{\beta-1} |\{\mathbb{W}_{\alpha,p}^r[\mu \chi_{\tilde{Q}_{2 r}(x,t)}]>\lambda \}| d\lambda.
           \end{align*}
      If $\alpha p=N+2$, we use \eqref{5hh140420142} in Remark \ref{5hh140420141} with $\varepsilon=\frac{\alpha p}{\beta}$ and take $\lambda_0=(\mu(\tilde{Q}_{2 r}(x,t)))^{\frac{1}{p-1}}$
      \begin{align*}
                 \fint_{\tilde{Q}_r} (L^r_{\alpha,p}[\mu](y,s))^\beta dyds&\lesssim  \lambda_0^{\beta}+r^{-N-2} \int_{\lambda_0}^{\infty}\lambda^{\beta-1}  \left(\frac{(\mu(\tilde{Q}_{2 r}(x,t)))^{\frac{1}{p-1}}}{\lambda}\right)^{\frac{\alpha p+\varepsilon(p-1)}{\varepsilon}}r^{\alpha p} d\lambda\\&\lesssim
                 (\mu(\tilde{Q}_{2 r}(x,t)))^{\frac{\beta}{p-1}} 
               \lesssim   (\mathbb{W}^{R,\delta}_{\alpha,p}[\mu](x,t))^\beta.
                 \end{align*}
     If $\alpha p<N+2$, we use \eqref{5hh2210133} in Proposition \ref{5hh23101315} and take $\lambda_0 =\mu(\tilde{Q}_{2 r}(x,t))^{\frac{1}{p-1}}r^{-\frac{N+2-\alpha p}{p-1}}$, we get  
$$
             \fint_{\tilde{Q}_r} (L^r_{\alpha,p}[\mu](y,s))^\beta dyds\lesssim \left(\mu(\tilde{Q}_{2 r}(x,t))^{\frac{1}{p-1}}r^{-\frac{N+2-\alpha p}{p-1}}\right)^{\beta}\lesssim (\mathbb{W}^{R,\delta}_{\alpha,p}[\mu](x,t))^\beta.
$$
 Case 2: $r\geq R$. As above case, we have
$$\fint_{\tilde{Q}_r} (\mathbb{W}_{\alpha-\frac{\delta}{p'},p}[\mu\chi_{\tilde{Q}_{2 r}(x,t)}](y,s))^\beta dyds\lesssim \left(\mu(\tilde{Q}_{2 r}(x,t))^{\frac{1}{p-1}}r^{-\frac{N+2-\alpha p+\delta(p-1)}{p-1}}\right)^{\beta}.$$
Since $\mathbb{W}^{R,\delta}_{\alpha,p}[\mu\chi_{\tilde{Q}_{2 r}(x,t)}]\leq R^{\delta} \mathbb{W}_{\alpha-\frac{\delta}{p'},p}[\mu\chi_{\tilde{Q}_{2 r}(x,t)}]$, thus  
$$ \fint_{\tilde{Q}_r} (L^r_{\alpha,p}[\mu](y,s))^\beta dyds\lesssim \left(\mu(\tilde{Q}_{2 r}(x,t))^{\frac{1}{p-1}}r^{-\frac{N+2-\alpha p+\delta (p-1)}{p-1}}R^\delta\right)^{\beta}\lesssim(\mathbb{W}^{R,\delta}_{\alpha,p}[\mu](x,t))^\beta.$$                                               Therefore, we get \eqref{5hh09262}. The proof is complete.
     \end{proof}
     \begin{remark} It is easy to see that
     the inequality \eqref{5hh09262} does not hold true for $\mathbb{W}_{\alpha,p}^R[\delta_{(0,0)}]$  where $\delta_{(0,0)}$ is the Dirac mass at $(x,t)=(0,0)$.
     \end{remark}   
     \begin{remark}\label{5hh120320141} From Lemma \eqref{5hh1310136}, we have, if there exists $(x_0,t_0)\in\mathbb{R}^{N+1}$ such that $\mathbb{W}^{R,\delta}_{\alpha,p}[\mu](x_0,t_0)<\infty$ then $\mathbb{W}^{R,\delta}_{\alpha,p}[\mu]\in L^\beta_{\text{loc}}(\mathbb{R}^{N+1})$ for any $0<\beta <\frac{(N+2)(p-1)}{N+2-\alpha p+\delta (p-1)}$.
     \end{remark}

 \begin{lemma}
     \label{5hh1310133}
                  Let  $R\in (0,\infty]$, $1< p\leq \alpha^{-1}(N+2)$ and $0<\delta<\alpha p'$. Assume that $\alpha p<N+2$ if $R=\infty$. Then, for any compact set $K\subset\mathbb{R}^{N+1}$ there exists a $\mu\in \mathfrak{M}^+(K)$, called a capacitary measure of $K$ such that
                  \begin{align*}
                  \mu(K)\sim_{\alpha,p,\delta}\text{Cap}_{E_\alpha^{R,\delta/p'},p}(K)
                  \end{align*} and $\mathbb{W}^{R,\delta}_{\alpha,p}[\mu](x,t)\gtrsim_{\alpha,p,\delta} 1 $ a.e in $K$ and $ \mathbb{W}^{R,\delta}_{\alpha,p}[\mu]\lesssim_{\alpha,p,\delta} 1$ a.e in  $\mathbb{R}^{N+1}$.
     \end{lemma}
     \begin{proof}          
     We consider a measure $\nu$ on $M=\mathbb{R}^{N+1}\times \mathbb{Z}$ as follows 
     \begin{equation*}
     \nu=m\otimes\sum_{n=-\infty}^{\infty}\delta_n,
     \end{equation*}
     where $m$ is Lebesgue measure, and $\delta_n$ denotes unit mass at $n$. Thus, $f\in L^p(M,d\nu)$, means $f=\{f_n\}_{-\infty}^{\infty}$, with 
     \begin{equation*}
     ||f||_{L^p(M,d\nu)}^p=\sum_{n=-\infty}^{\infty}||f_n||_{L^p(\mathbb{R}^{N+1})}^p.
     \end{equation*}
     Let $n_R
     \in\mathbb{Z}\cup \{+\infty\}$ such that $2^{-n_R}\leq R <2^{-n_R+1}$ if $R<+\infty$ and $n_R=-\infty$ if $R=+\infty$. We define a kernel $\mathbb{P}_\alpha$ in $\mathbb{R}^{N+1}\times M=\mathbb{R}^{N+1}\times \mathbb{R}^{N+1}\times \mathbb{Z}$ by 
     \begin{equation*}
     \mathbb{P}_\alpha(x,t,x',t',n)=\min\{1,2^{(n-n_R)\delta/p'}\}2^{n(N+2-\alpha)}\chi_{\tilde{Q}_{2^{-n}}}(x-x',t-t').
     \end{equation*}    
        If $f$ is $\nu-$measurable and nonnegative and $\mu\in\mathfrak{M}^+(\mathbb{R}^{N+1})$, the corresponding potentials $\mathcal{P}_\alpha f$,  ${\mathop \mathcal{P}\limits^ \vee} _\alpha\mu$ and $V_{\mathbb{P}_\alpha,p}^\mu$ are everywhere well defined and given by 
        \begin{align*}
        &(\mathcal{P}_\alpha f)(x,t)=\int_{M}\mathbb{P}_\alpha(x,t,x',t',n)f(x',t',n)d\nu(x',t',n)\\&~~~~~~~~~~~~~~ =\sum_{n=-\infty}^{\infty}\min\{1,2^{(n-n_R)\delta/p'}\}2^{n(N+2-\alpha)}(\chi_{\tilde{Q}_{2^{-n}}}*f_n)(x,t),\\
        &({\mathop \mathcal{P}\limits^ \vee} _\alpha\mu)(x',t',n)=\int_{\mathbb{R}^{N+1}}\mathbb{P}_\alpha(x,t,x',t',n)d\mu(x,t)\\&~~~~~~~~~~~~~~~~~~~=\min\{1,2^{(n-n_R)\delta/p'}\}2^{n(N+2-\alpha)}(\chi_{\tilde{Q}_{2^{-n}}}*\mu)(x',t'),
        \\  & V_{P_\alpha,p}^\mu(x,t)=(\mathcal{P}_\alpha({\mathop \mathcal{P}\limits^ \vee} _\alpha\mu)^{p'-1})(x,t)\\&~~~~~~~~~~~~~~ = \sum_{n=-\infty}^{\infty}\min\{1,2^{(n-n_R)\delta}\}2^{np'(N+2-\alpha)}\left(\chi_{\tilde{Q}_{2^{-n}}}*\left(\chi_{\tilde{Q}_{2^{-n}}}*\mu\right)^{p'-1}\right)(x,t).
        \end{align*}    
        for any $(x,t,x',t',n)\in \mathbb{R}^{N+1}\times M$.    \\      
        Since for all $(x,t)\in \mathbb{R}^{N+1}$,
        \begin{align*}
        |\tilde{Q}_1|2^{-(n+1)(N+2)}(\mu(\tilde{Q}_{2^{-n-1}}(x,t)))^{p'-1}& \leq \left(\chi_{\tilde{Q}_{2^{-n}}}*\left(\chi_{\tilde{Q}_{2^{-n}}}*\mu\right)^{p'-1}\right)(x,t)\\&\leq |\tilde{Q}_1|2^{-n(N+2)}(\mu(\tilde{Q}_{2^{-n+1}}(x,t)))^{p'-1},
        \end{align*} 
        thus, 
        \begin{equation}\label{5hh090920141}  \mathbb{W}^{R,\delta}_{\alpha,p}[\mu]\sim V_{\mathbb{P}_\alpha,p}^\mu.                
        \end{equation} 
         We now define the $L^p-$capacity with $1<p<\infty$
                \begin{equation*}
                \text{Cap}_{\mathbb{P}_\alpha,p}(E)=\inf\{||f||^p_{L^p(M,d\nu)}: f\in L^p_+(M,d\nu), \mathcal{P}_\alpha f\geq \chi_E\}.
                \end{equation*}  
                for any Borel set $E\subset\mathbb{R}^{N+1}$.       
                  By \cite[Theorem 2.5.1]{55AH},  for any compact set $K\subset \mathbb{R}^{N+1}$
                  \begin{equation*}
                  \text{Cap}_{\mathbb{P}_\alpha,p}(K)^{1/p}=\sup\{\mu(K):\mu\in\mathfrak{M}^+(K),||{\mathop \mathcal{P}\limits^ \vee} _\alpha\mu||_{L^{p'}(M,d\nu)}\leq 1\}.
                  \end{equation*}
                   By \cite[Theorem 2.5.6]{55AH}, for any compact set $K$ in $\mathbb{R}^{N+1}$, there exists $\mu\in \mathfrak{M}^+(K)$, called a capacitary measure for $K$, such that $V_{\mathbb{P}_\alpha,p}^\mu\geq 1$  $\text{Cap}_{\mathbb{P}_\alpha,p}-\text{q.e. in } K$, $V_{P_\alpha,p}^\mu\leq 1$ a.e in $\text{supp}(\mu)$ and 
                            $\mu(K)=\text{Cap}_{\mathbb{P}_\alpha,p}(K)$. 
             Thanks to \eqref{5hh090920141} and \eqref{5hh240220141}, we have $\mathbb{W}^{R,\delta}_{\alpha,p}[\mu]\gtrsim 1$ $\text{Cap}_{\mathbb{P}_\alpha,p}-\text{q.e. in } K$, $\mathbb{W}^{R,\delta}_{\alpha,p}[\mu]\lesssim 1$ a.e in $\mathbb{R}^{N+1}$ and 
                                               $\mu(K)=\text{Cap}_{\mathbb{P}_\alpha,p}(K)$. \\
On the other hand,
        \begin{align*}
         ||{\mathop \mathcal{P}\limits^ \vee} _\alpha\mu||_{L^{p'}(M,d\nu)}^{p'}&=\sum_{n=-\infty}^{\infty}||\min\{1,2^{(n-n_R)\delta/p'}\}2^{n(N+2-\alpha)}\chi_{\tilde{Q}_{2^{-n}}}*\mu||_{L^{p'}(\mathbb{R}^{N+1})}^{p'}\\&=\sum_{n=-\infty}^{\infty}\min\{1,2^{(n-n_R)\delta}\}2^{np'(N+2-\alpha)}\int_{\mathbb{R}^{N+1}}(\chi_{\tilde{Q}_{2^{-n}}}*\mu)^{p'}dxdt,
        \end{align*}
This quantity is equivalent to 
$$
 \int_{\mathbb{R}^{N+1}}\int_{0}^{\infty}\left(\frac{\mu(\tilde{Q}_\rho(x,t))}{\rho^{N+2-\alpha}}\right)^{p'}\min\{1,\left(\frac{\rho}{R}\right)^{-\delta}\}\frac{d\rho}{\rho}dxdt.$$
So, thanks to \eqref{5hh150420141} in Remark \ref{5hh150420142}, we obtain
$$ ||{\mathop \mathcal{P}\limits^ \vee} _\alpha\mu||_{L^{p'}(M,d\nu)}\sim ||E_\alpha^{R,\delta/p'}*\mu||_{L^{p'}(\mathbb{R}^{N+1})}.$$
It follows that the two capacities $\text{Cap}_{\mathbb{P}_\alpha,p}$ and $ \text{Cap}_{E_\alpha^{R,\delta/p'},p} $ are equivalent.         
        Therefore, we obtain the desired results.  
       \end{proof}      
        \begin{lemma}\label{5hh1310135} Let $R\in (0,\infty]$, $1<p\leq \alpha^{-1}(N+2)$ and $0<\delta<\alpha p'$. Assume that $\alpha p<N+2$ if $R=\infty$. There holds for any $\mu\in\mathfrak{M}^+_b(\mathbb{R}^{N+1})$
        \begin{equation}\label{5hh150420143}
                        \text{Cap}_{E_\alpha^{R,\delta/p'},p}(\{\mathbb{W}_{\alpha,p}^{R,\delta}[\mu]>\lambda\})\lesssim_{\alpha,p,\delta}\lambda^{-p+1}\mu (\mathbb{R}^{N+1})~~\forall ~\lambda>0.
                        \end{equation}
                        In particular, $\mathbb{W}^{R,\delta}_{\alpha,p}[\mu]<\infty$~ $\text{Cap}_{E_\alpha^{R,\delta/p'},p}-$q.e. in $\mathbb{R}^{N+1}$.
        \end{lemma}
        \begin{proof} By Lemma \ref{5hh1310133}, there is a capacitary measure $\sigma$ for a compact subset $K$ of $\{\mathbb{W}^{R,\delta}_{\alpha,p}[\mu]>\lambda\}$ such that  $\mathbb{W}^{R,\delta}_{\alpha,p}[\sigma](x,t)\lesssim_{\alpha,p,\delta} 1$ on  $\text{supp} \sigma$ and 
        $\text{Cap}_{E^{R,\delta/p'}_{\alpha},p}(K)\sim \sigma (K)$. \\
        Set 
        $
        \mathbb{M}[\mu,\sigma](x,t)= \mathop {\sup}\limits_{\rho>0 }\frac{\mu(\tilde{Q}_{\rho}(x,t))}{\sigma(\tilde{Q}_{3\rho}(x,t))}
       $ for any $(x,t)\in \text{supp} \sigma$.
        Then, for any $(x,t)\in \text{supp} \sigma$
        \begin{align*}
        \lambda< \mathbb{W}^{R,\delta}_{\alpha,p}[\mu](x,t)&\leq \left( \mathbb{M}[\mu,\sigma](x,t)\right)^{\frac{1}{p-1}}
              \int_{0}^{\infty}\left(\frac{\sigma(\tilde{Q}_{3\rho}(x,t))}{\rho^{N+2-\alpha p}}\right)^{\frac{1}{p-1}}\min\{1,\left(\frac{\rho}{R}\right)^{-\delta}\}\frac{d\rho}{\rho}              
                    \\&\lesssim_{\alpha,p,\delta}\left( \mathbb{M}[\mu,\sigma](x,t)\right)^{\frac{1}{p-1}}.
        \end{align*}      
      Thus, for any $\lambda>0$, $
      \text{supp} \sigma \subset \{\left( \mathbb{M}[\mu,\sigma]\right)^{\frac{1}{p-1}}\gtrsim_{\alpha,p,\delta}\lambda\}$.
      By Vitali Covering Lemma one can cover $\text{supp}\sigma$ with a union of $\tilde{Q}_{3\rho_i}(x_i,t_i)$ for $i=1,...,m(K)$ so that $\tilde{Q}_{\rho_i}(x_i,t_i)$ are disjoint and $\sigma(\tilde{Q}_{3\rho_i}(x_i,t_i))\lesssim \lambda^{-p+1}\mu(\tilde{Q}_{\rho_i}(x_i,t_i))$.     It follows that 
$$
      \text{Cap}_{E^{R}_{\alpha},p}(K)\lesssim \sum_{i=1}^{m(K)}\sigma (\tilde{Q}_{3\rho_i}(x_i,t_i))\lesssim   \lambda^{-p+1}\sum_{i=1}^{m(K)}\mu (\tilde{Q}_{\rho_i}(x_i,t_i))\lesssim \lambda^{-p+1}\mu (\mathbb{R}^{N+1}).
    $$           So, for all compact subset $K$ of $\{\mathbb{W}^{R,\delta}_{\alpha,p}[\mu]>\lambda\}$,$$
       \text{Cap}_{E^{R,\delta/p'}_{\alpha},p}(K)\lesssim \lambda^{-p+1}\mu (\mathbb{R}^{N+1}).$$
       Therefore, we obtain \eqref{5hh150420143}.
        \end{proof}\medskip\\
Now we are ready to  prove Proposition \ref{5hh140420143}.\\
   \begin{proof}[Proof of Proposition \ref{5hh140420143}]
    By Lemma \ref{5hh1310136},   \ref{5hh1310133} and \ref{5hh1310135}, there exists a  capacitary measure $\mu$ of a compact subset  $K\subset \mathbb{R}^{N+1}$ such that $\mathbb{W}^{R,\delta p'}_{\alpha,p}[\mu]\gtrsim 1$ a.e in $K$, $\mathbb{W}^{R,\delta p'}_{\alpha,p}[\mu]\lesssim 1$ a.e in $\mathbb{R}^{N+1}$ and $\text{Cap}_{E^{R,\delta}_\alpha,p}(\{\mathbb{W}^{R,\delta p'}_{\alpha,p}[\mu]>\lambda\})\lesssim\lambda^{-p+1}Cap_{E^{R,\delta}_\alpha,p}(K)$ for all $\lambda>0$, $(\mathbb{W}^{R,\delta p'}_{\alpha,p}[\mu])^\beta\in A_1$ for any $0<\beta <\frac{(N+2)(p-1)}{N+2-\alpha p+\delta p}$. 
From the second assumption we have
 $$ 
      \int_{\mathbb{R}^{N+1}} |g|(\mathbb{W}^{R,\delta p'}_{\alpha,p}[\mu])^\beta dxdt\lesssim\int_{\mathbb{R}^{N+1}} |f|(\mathbb{W}^{R,\delta p'}_{\alpha,p}[\mu])^\beta dxdt.
     $$
 Thus
      \begin{align*}
      \int_{K} |g|dxdt&\lesssim \int_{\mathbb{R}^{N+1}}|g|(\mathbb{W}^{R,\delta p'}_{\alpha,p}[\mu])^\beta dxdt
            \lesssim \int_{\mathbb{R}^{N+1}} |f|(\mathbb{W}^{R,\delta p'}_{\alpha,p}[\mu])^\beta dxdt\\&=  \beta \int_{0}^{c}\int_{\mathbb{W}^{R,\delta p'}_{\alpha,p}[\mu]> \lambda}|f|dxdt\lambda^{\beta-1}d\lambda.
      \end{align*}      
By the first assumption we get
$$
\int_{\mathbb{W}^{R,\delta p'}_{\alpha,p}[\mu]> \lambda}|f|dxdt\lesssim\text{Cap}_{E^{R,\delta}_\alpha,p}(\{\mathbb{W}^{R,\delta p'}_{\alpha,p}[\mu]> \lambda\})\lesssim \lambda^{-p+1} \text{Cap}_{E^{R,\delta}_\alpha,p}(K).$$
Therefore,
$$
      \int_{K} |g|dxdt\lesssim \int_{0}^{c}\lambda^{-p+1} \text{Cap}_{E^{R,\delta}_\alpha,p}(K) \lambda^{\delta-1}d\lambda\lesssim \text{Cap}_{E^{R,\delta}_\alpha,p}(K),$$
      since one can choose $\delta>p-1$. This completes the proof.
   \end{proof}
\begin{corollary} \label{keylam}Let $f,g\in L^1_{\text{loc}}(\mathbb{R}^{N+1})$ be such that 
$$
		\int_{\mathbb{R}^{N+1}} |g|wdxdt\lesssim_{[w]_{A_1}}\int_{\mathbb{R}^{N+1}} |f|wdxdt$$
for any weight $w\in A_1$. Then, 
\begin{equation}\label{Z1}
	\mathbb{I}_{\beta}[|g|]\lesssim_\beta 	\mathbb{I}_{\beta}[|f|]
\end{equation}
for any $\beta\in (0,N+2)$. 
\end{corollary}
The inequality \eqref{Z1} for elliptic version was proved in \cite{phucompde,55Ph2,H-P1}.\medskip\\
\begin{proof} Let $\varphi_n$ be the standard mollifiers in $\mathbb{R}^{N+1}$. Thanks to Lemma \ref{5hh1310136}, one gets $\mathbb{I}_{\beta}[\varphi_{n}]\in A_1 $ with $\sup_{n}[\mathbb{I}_{\beta}[\varphi_{n}]]_{A_1}\lesssim_\beta 1$. So, for any $(x_0,t_0)\in \mathbb{R}^{N+1}$, 
$$
\int_{\mathbb{R}^{N+1}} |g| \mathbb{I}_{\beta}[\varphi_{n}((x_0,t_0)+.)]\lesssim_\beta\int_{\mathbb{R}^{N+1}} |f|\mathbb{I}_{\beta}[\varphi_{n}((x_0,t_0)+.)].$$	
Thus,
$$
\int_{\mathbb{R}^{N+1}}\varphi_{n}((x_0,t_0)+.) \mathbb{I}_{\beta}[|g|] \lesssim_\beta\int_{\mathbb{R}^{N+1}} \varphi_{n}((x_0,t_0)+.)\mathbb{I}_{\beta}[|f|]. $$	
Letting $n\to\infty$, one has 
	$$	\mathbb{I}_{\beta}[|g|](x_0,t_0)\lesssim_\beta	\mathbb{I}_{\beta}[|f|](x_0,t_0).$$
This implies \eqref{Z1}. The proof is complete. 	
\end{proof}
   \begin{definition} 
      Let $s>1$, $\alpha>0$. We define the space $\mathfrak{M}^{\mathcal{H}_\alpha,s}(\mathbb{R}^{N+1})$ ($\mathfrak{M}^{\mathcal{G}_\alpha,s}(\mathbb{R}^{N+1})$ resp.) to be the set of all measure $\mu\in \mathfrak{M}(\mathbb{R}^{N+1})$ such that        
        \begin{align*}
        &[\mu]_{\mathfrak{M}^{\mathcal{H}_\alpha,s}(\mathbb{R}^{N+1})}:=\sup\left\{|\mu|(K)/\text{Cap}_{\mathcal{H}_\alpha,s}(K):~\text{Cap}_{\mathcal{H}_\alpha,s}(K)>0\right\}<\infty,
        \\&\left( [\mu]_{\mathfrak{M}^{\mathcal{G}_\alpha,s}(\mathbb{R}^{N+1})}:=\sup\left\{|\mu|(K)/\text{Cap}_{\mathcal{G}_\alpha,s}(K):\text{Cap}_{\mathcal{G}_\alpha,s}(K)>0\right\}<\infty~\text{resp.}\right)
        \end{align*}
where the supremum is taken all compact sets $K\subset\mathbb{R}^{N+1}$.\\         
        For simplicity, we will write $\mathfrak{M}^{\mathcal{H}_\alpha,s},\mathfrak{M}^{\mathcal{G}_\alpha,s}$ to denote  $\mathfrak{M}^{\mathcal{H}_\alpha,s}(\mathbb{R}^{N+1}),\mathfrak{M}^{\mathcal{G}_\alpha,s}(\mathbb{R}^{N+1})$ resp.
      \end{definition}
      We see that if $\alpha s\geq N+2$, $\mathfrak{M}^{\mathcal{H}_\alpha,s}(\mathbb{R}^{N+1})=\{0\}$, if $\alpha s<N+2$, $\mathfrak{M}^{\mathcal{H}_\alpha,s}(\mathbb{R}^{N+1})\subset \mathfrak{M}^{\mathcal{G}_\alpha,s}(\mathbb{R}^{N+1})$. 
       On the other hand, $\mathfrak{M}^{\mathcal{G}_\alpha,s}(\mathbb{R}^{N+1})\supset\mathfrak{M}_b(\mathbb{R}^{N+1})$ if $\alpha s>N+2$.\\
   We now have the following two remarks: 
    \begin{remark}\label{5hh2403201410} There holds
     \begin{align}\label{5hh240320149}
     [f]_{\mathfrak{M}^{\mathcal{G}_\alpha,p}}\lesssim_{\alpha,s} [|f|^s]^{1/s}_{\mathfrak{M}^{\mathcal{G}_\alpha,p}}~~\text{ for all function } f.
     \end{align}
     Indeed,  set $a=[|f|^s]_{\mathfrak{M}^{\mathcal{G}_\alpha,p}}$, so for any compact set $K$ in $\mathbb{R}^{N+1}$,
    $$
     \int_{K}|f|^sdxdt\lesssim_{\alpha,s} \text{Cap}_{\mathcal{G}_\alpha,p}(K).
     $$
     This gives $2a \text{Cap}_{\mathcal{G}_\alpha,p}(K)\gtrsim_{\alpha,s} \int_{K}\left(|f|^s+a\right)
          \gtrsim_{\alpha,s} a^{1-1/s}\int_{K}|f|,$
     here we used \eqref{5hh240320147} in Remark \ref{5hh240320148} at the first inequality  and H\"older's inequality  at the second one. 
     It follows \eqref{5hh240320149}.
    \end{remark}
    \begin{remark}\label{5hh130320146} Assume that $p>1$ and $\frac{2}{p}<\alpha<N+\frac{2}{p}$. 
      Clearly, from Corollary \ref{5hh250320147} we obtain  that for $\omega\in\mathfrak{M}^+(\mathbb{R}^{N})$
$$
           \left[\omega\otimes\delta_{\{t=0\}}\right]_{\mathfrak{M}^{\mathcal{H}_\alpha,p}}\sim_{\alpha,p}\left[\omega\right]_{\mathfrak{M}^{\mathbf{I}_{\alpha-2/p},p}}, \quad \left[\omega\otimes\delta_{\{t=0\}}\right]_{\mathfrak{M}^{\mathcal{G}_\alpha,p}}\sim_{\alpha,p} \left[\omega\right]_{\mathfrak{M}^{\mathbf{G}_{\alpha-2/p},p}}.
$$
        Here $\mathfrak{M}^{\mathbf{I}_{\alpha-2/p},p}:=\mathfrak{M}^{\mathbf{I}_{\alpha-2/p},p}(\mathbb{R}^N)$ , $\mathfrak{M}^{\mathbf{G}_{\alpha-2/p},p}:=\mathfrak{M}^{\mathbf{G}_{\alpha-2/p},p}(\mathbb{R}^N)$ and 
        \begin{align*}
       & [\omega]_{\mathfrak{M}^{\mathbf{I}_{\alpha-2/p},p}(\mathbb{R}^N)}:=\sup\left\{\omega(K)/\text{Cap}_{\mathbf{I}_{\alpha-2/p},p}(K):~\text{Cap}_{\mathbf{I}_{\alpha-2/p},p}(K)>0\right\},\\& [\omega]_{\mathfrak{M}^{\mathbf{G}_{\alpha-2/p},p}(\mathbb{R}^N)}:=\sup\left\{\omega(K)/\text{Cap}_{\mathbf{G}_{\alpha-2/p},p}(K):~\text{Cap}_{\mathbf{G}_{\alpha-2/p},p}(K)>0\right\},       
        \end{align*}
             where the supremum is taken all compact sets $K\subset\mathbb{R}^{N}$.             
      \end{remark}
      Clearly, Theorem \ref{5hh051120131} and Proposition \ref{5hh140420143} lead to the following result.
      \begin{proposition} Let $q>p-1$, $s>1$ and $0<\alpha p<N+2$. Then the following quantities are equivalent
      \begin{equation*}
      \left[\left(\mathbb{W}^R_{\alpha,p}[\mu]\right)^{q}\right]_{\mathfrak{M}^{\mathcal{H}_\alpha,s}},~~\left[\left(\mathbb{I}^R_{\alpha p}[\mu]\right)^{\frac{q}{p-1}}\right]_{\mathfrak{M}^{\mathcal{H}_\alpha,s}}~\text{ and }~\left[\left(\mathbb{M}^R_{\alpha p}[\mu]\right)^{\frac{q}{p-1}}\right]_{\mathfrak{M}^{\mathcal{H}_\alpha,s}},
      \end{equation*}
       for every $\mu\in \mathfrak{M}^{+}(\mathbb{R}^{N+1})$ and $0<R\leq\infty$.       
      \end{proposition}
           In the next result, we present a characterization of the following trace inequality:                    
  \begin{equation}\label{5hh04101}
        ||E^{R,\delta}_{\alpha}*f||_{L^p(\mathbb{R}^{N+1},d\mu)}\leq C_1||f||_{L^p(\mathbb{R}^{N+1})}~~\forall f\in L^{p}(\mathbb{R}^{N+1}).
        \end{equation} 
    \begin{theorem}\label{5hh1410136}
      Let $0<R\leq\infty$,$1< p<\alpha^{-1}(N+2)$, $0<\delta<\alpha$ and $\mu\in \mathfrak{M}^{+}(\mathbb{R}^{N+1})$. Then the following statements are equivalent. 
      \begin{description}
      \item[1.] The trace inequality \eqref{5hh04101} holds.
      
      \item[2.] There holds
        \begin{equation}\label{5hh04102}
         ||E^{R,\delta}_{\alpha}*f||_{L^p(\mathbb{R}^{N+1},d\omega)}\leq C_2||f||_{L^p(\mathbb{R}^{N+1})}~~\forall f\in L^{p}(\mathbb{R}^{N+1}),
         \end{equation} 
         where $d\omega= (\mathbb{I}^{R,\delta}_{\alpha}\mu)^{p'}dxdt $.
    \item[3.]There holds
       \begin{equation}\label{5hh04103}
       ||E^{R,\delta}_{\alpha}*f||_{L^{p,\infty}(\mathbb{R}^{N+1},d\mu)}\leq C_3||f||_{L^p(\mathbb{R}^{N+1})}~~\forall f\in L^{p}(\mathbb{R}^{N+1}).
       \end{equation} 
       \item[4.] For every compact set $E\subset \mathbb{R}^{N+1}$, 
       \begin{equation}\label{5hh04104}
       \mu(E)\leq C_4 Cap_{E^{R,\delta}_\alpha,p}(E).
       \end{equation}
       \item[5.] $\mathbb{I}_\alpha^{R,\delta}[\mu] <\infty$ a.e and 
       \begin{equation}\label{5hh04105}
       \mathbb{I}^{R,\delta}_\alpha[(\mathbb{I}^{R,\delta}_\alpha[\mu])^{p'}]\leq C_5 \mathbb{I}^{R,\delta}_\alpha [\mu] ~~a.e.
       \end{equation}
       \item[6.] For every compact set $E\subset \mathbb{R}^{N+1}$,
       \begin{equation}\label{5hh04106}
       \int_{E}(\mathbb{I}^{R,\delta}_\alpha[\mu])^{p'}dxdt\leq C_6 Cap_{E^{R,\delta}_\alpha,p}(E).
       \end{equation}
       \item[7.] For every compact set $E\subset \mathbb{R}^{N+1}$,
           \begin{equation}\label{5hh04107}
           \int_{\mathbb{R}^{N+1}}(\mathbb{I}^{R,\delta}_\alpha[\mu\chi_E])^{p'}dxdt\leq C_7 \mu(E).
           \end{equation}
       \item[8.] For every compact set $E\subset \mathbb{R}^{N+1}$,
               \begin{equation}\label{5hh04108}
               \int_{E}(\mathbb{I}^{R,\delta}_\alpha[\mu\chi_E])^{p'}dxdt\leq C_8 \mu(E).
               \end{equation}    
      \end{description}    
      \end{theorem}         
      We can find a simple sufficient condition on $\mu$ so that  trace inequality  \eqref{5hh04101} is satisfied from the isoperimetric inequality \eqref{5hh080120142}. \\    
   \begin{proof}[Proof of Theorem \ref{5hh1410136}]As in \cite{55Verbi} we can show that $1 \Leftrightarrow 2\Leftrightarrow3 \Leftrightarrow4 \Leftrightarrow 6 \Leftrightarrow7$ and $7\Rightarrow 8, 5\Rightarrow 2$. Thus, it is enough   to show that $8.\Rightarrow 5$.\\
       First, we need to show that 
        \begin{equation}\label{5hh08101}
              \left(\int_{r}^{\infty}\frac{\mu(\tilde{Q}_\rho(x,t))}{\rho^{N+2-\alpha}}\min\{1,\left(\frac{\rho}{R}\right)^{-\delta}\}\frac{d\rho}{\rho}\right)^{p'-1}\lesssim r^{-\alpha}\left(\min\{1,\left(\frac{r}{R}\right)^{-\delta}\}\right)^{-1}
              \end{equation}
       We have for any $(y,s)\in \tilde{Q}_r(x,t)$ 
       \begin{align*}
       \mathbb{I}^{R,\delta}_\alpha[\mu\chi_{\tilde{Q}_r(x,t)}](y,s)&\geq  \int_{2r}^{4r}\frac{\mu(\tilde{Q}_r(x,t)\cap \tilde{Q}_\rho(y,s))}{\rho^{N+2-\alpha}}\min\left\{1,\left(\frac{\rho}{R}\right)^{-\delta}\right\}\frac{d\rho}{\rho}
              \\&\gtrsim\frac{\mu(\tilde{Q}_r(x,t))}{r^{N+2-\alpha}}\min\left\{1,\left(\frac{r}{R}\right)^{-\delta}\right\}.
       \end{align*}       
       In \eqref{5hh04108}, we take $E=\tilde{Q}_r(x,t)$
$$
        \mu(\tilde{Q}_r(x,t))\gtrsim \int_{\tilde{Q}_r(x,t)}(\mathbb{I}_\alpha[\mu\chi_{\tilde{Q}_r(x,t)}])^{p'}\gtrsim  \left(\frac{\mu(\tilde{Q}_r(x,t))}{r^{N+2-\alpha}}\min\left\{1,\left(\frac{r}{R}\right)^{-\delta}\right\}\right)^{p'}|\tilde{Q}_r|.$$
      So
       $
       \mu(\tilde{Q}_r(x,t))\lesssim r^{N+2-\alpha p}\left(\min\left\{1,\left(\frac{r}{R}\right)^{-\delta}\right\}\right)^{-p}
       $
       which implies \eqref{5hh08101}.\\
      Next we set 
      \begin{equation*}
      L_r[\mu](x,t)=\int_{r}^{+\infty}\frac{\mu(\tilde{Q}_\rho(x,t))}{\rho}\min\{1,\left(\frac{\rho}{R}\right)^{-\delta}\}\frac{d\rho}{\rho},
      \end{equation*}
      \begin{align*}
      U_r[\mu](x,t)=\int_{0}^{r}\frac{\mu(\tilde{Q}_\rho(x,t))}{\rho}\min\{1,\left(\frac{\rho}{R}\right)^{-\delta}\}\frac{d\rho}{\rho},
      \end{align*}
            and \begin{equation*}
           d\omega= (I_{\alpha}\mu)^{p'}dxdt,~~~ d\sigma_{1,r}=\left(L_r[\mu]\right)^{p'}dxdt,~~~d\sigma_{2,r}=\left(U_r[\mu]\right)^{p'}dxdt .
            \end{equation*}
            We have 
            $
            d\omega\leq 2^{p'-1}\left(d\sigma_{1,r}+d\sigma_{2,r}\right).
           $
            To prove \eqref{5hh04105} we need to show that
            \begin{align}
\label{5hh08102}&\int_{0}^{\infty}\frac{\sigma_{1,r}(\tilde{Q}_{r}(x,t))}{r^{N+2-\alpha}}\min\{1,\left(\frac{r}{R}\right)^{-\delta}\}\frac{dr}{r}\lesssim \mathbb{I}^{R,\delta}_\alpha[\mu](x,t), \\& \label{5hh08103} \int_{0}^{\infty}\frac{\sigma_{2,r}(\tilde{Q}_{r}(x,t))}{r^{N+2-\alpha}}\min\{1,\left(\frac{r}{R}\right)^{-\delta}\}\frac{dr}{r}\lesssim \mathbb{I}^{R,\delta}_\alpha[\mu](x,t).        \end{align}             
            Since, for all $r>0$, $0<\rho<r$ and $(y,s)\in \tilde{Q}_r(x,t)$ we have $\tilde{Q}_\rho(y,s)\subset \tilde{Q}_{2r}(x,t)$. So, 
            \begin{equation*}
            \sigma_{2,r}(\tilde{Q}_r(x,t))=\int_{\tilde{Q}_r(x,t)}\left(U_r[\mu]\right)^{p'}=\int_{\tilde{Q}_r(x,t)}\left(U_r[\mu\chi_{\tilde{Q}_{2r}(x,t)}]\right)^{p'}.
            \end{equation*}
            Thus, from \eqref{5hh04108} we get for $\tilde{Q}_{2r}=\tilde{Q}_{2r}(x,t)$
$$
            \sigma_{2,r}(\tilde{Q}_r(x,t))\leq  \int_{\tilde{Q}_{2r}}\left(U_r[\mu\chi_{\tilde{Q}_{2r}}]\right)^{p'}\leq  \int_{\tilde{Q}_{2r}}\left(\mathbb{I}^{R,\delta}_\alpha[\mu\chi_{\tilde{Q}_{2r}}]\right)^{p'}
                 \lesssim \mu (\tilde{Q}_{2r}).$$
            Therefore, \eqref{5hh08103} follows.\\
            Since, for all $r>0$, $\rho\geq r$ and $(y,s)\in \tilde{Q}_r(x,t)$ we have $\tilde{Q}_\rho(y,s)\subset \tilde{Q}_{2\rho}(x,t)$. So, for all $(y,s)\in \tilde{Q}_r(x,t)$ we have 
$$
             L_r[\mu](y,s)\leq \int_{r}^{+\infty}\frac{\mu(\tilde{Q}_{2\rho}(x,t))}{\rho^{N+2-\alpha}}\min\left\{1,\left(\frac{\rho}{R}\right)^{-\delta}\right\}\frac{d\rho}{\rho}\lesssim  L_r[\mu](x,t).$$
            Hence, 
$$
            \sigma_{1,r}(\tilde{Q}_r(x,t))= \int_{\tilde{Q}_r(x,t)}\left(L_r[\mu](y,s)\right)^{p'}dyds
                      \lesssim r^{N+2}\left(L_r[\mu](x,t)\right)^{p'}.$$
            Since $r^{\alpha-1}\min\{1,\left(\frac{r}{R}\right)^{-\delta}\}\leq\frac{1}{\alpha-\delta}\frac{d}{dr}\left(r^{\alpha}\min\{1,\left(\frac{r}{R}\right)^{-\delta}\}\right)$, we deduce that 
            \begin{align*}
            &\int_{0}^{\infty}\frac{\sigma_{1,r}(\tilde{Q}_{r}(x,t))}{r^{N+2-\alpha}}\min\left\{1,\left(\frac{r}{R}\right)^{-\delta}\right\}\frac{dr}{r}\\&\quad\quad\quad\quad\lesssim\int_{0}^{\infty}r^{\alpha-1}\left(L_r[\mu](x,t)\right)^{p'}\min\{1,\left(\frac{r}{R}\right)^{-\delta}\}dr\\&\quad\quad\quad\quad\lesssim   \int_{0}^{\infty}\frac{d}{dr}\left(r^{\alpha}\min\{1,\left(\frac{r}{R}\right)^{-\delta}\}\right)\left(L_r[\mu](x,t)\right)^{p'}dr\\&\quad\quad\quad\quad\lesssim \int_{0}^{\infty}r^{\alpha}\left(L_r[\mu](x,t)\right)^{p'-1}\frac{\mu(\tilde{Q}_{r}(x,t))}{r^{N+2-\alpha}}\min\left\{1,\left(\frac{r}{R}\right)^{-\delta}\right\}^2\frac{dr}{r}.
            \end{align*}   
        Combining this with  \eqref{5hh08101}, one gets \eqref{5hh08102}.
            The proof is complete.
   \end{proof}
   \begin{remark}\label{5hh260320141}It is easy to assert that if $\textbf{8.}$ holds then for any $0<\beta<N+2$
   \begin{align}\label{5hh260320142}
  \mathbb{I}_{\beta}\left[\left(\mathbb{I}^{R,\delta}_\alpha[\mu]\right)^{p'}\right] \lesssim \mathbb{I}_{\beta}[\mu].
   \end{align}
 
   \end{remark}
   \begin{corollary}\label{5hh250320145}
   Let $p>1, \alpha>0$ be such that $0<\alpha p<N+2$. There holds
            \begin{equation}\label{5hh250320143}
           \left[\left(\mathbb{I}_{\alpha }[\mu]\right)^{p'}\right]_{\mathfrak{M}^{\mathcal{H}_\alpha,p}}\sim\left[\mu\right]^{p'}_{\mathfrak{M}^{\mathcal{H}_\alpha,p}}
            \end{equation}
            for all $\mu\in\mathfrak{M}^+(\mathbb{R}^{N+1} )$.
            Furthermore,
            \begin{align}
            \label{5hh250320141}
            \left[\varphi_n*\mu\right]_{\mathfrak{M}^{\mathcal{H}_\alpha,p}}\lesssim \left[\mu\right]_{\mathfrak{M}^{\mathcal{H}_\alpha,p}}
            \end{align}for $n\in\mathbb{N}$,  $\mu\in\mathfrak{M}^+(\mathbb{R}^{N+1} )$ where $\left\{  \varphi_{n}\right\}$ is a sequence of mollifiers in $\mathbb{R}^{N+1}$.
   \end{corollary}
   \begin{proof}
   For $R=\infty$ we have $\mathbb{I}_\alpha^{R,\delta}[\mu]=\mathbb{I}_\alpha[\mu]$  and $E_\alpha^{R,\delta}=E_\alpha$. Thus, by \eqref{5hh230120144} in Corollary \ref{5hh230120148} and Theorem  \ref{5hh1410136} we get 
for every compact set $E\subset \mathbb{R}^{N+1}$, 
     $$
       \mu(E)\leq c \text{Cap}_{\mathcal{H}_\alpha,p}(E)
     $$
     if and only if for every compact set $E\subset \mathbb{R}^{N+1}$, 
          $$
            \int_{E}\left(\mathbb{I}_\alpha[\mu]\right)^{p'}dxdt\leq c' \text{Cap}_{\mathcal{H}_\alpha,p}(E).
          $$
 It follows \eqref{5hh250320143}.     \\
 Since $\mathbb{I}_\alpha[\varphi_n*\mu]=\varphi_n*\mathbb{I}_\alpha[\mu]\leq \mathbb{M}\left(\mathbb{I}_\alpha[\mu]\right)$ and $\mathbb{M}$ is bounded in $L^{p'}(\mathbb{R}^{N+1},dw)$ with $w\in A_{p'}$ ,one gets$$
 \int_{\mathbb{R}^{N+1}}\left(\mathbb{I}_\alpha[\varphi_n*\mu]\right)^{p'}dw\lesssim_{[w]_{A_{p'}}} \int_{\mathbb{R}^{N+1}}\left(\mathbb{I}_\alpha[\mu]\right)^{p'}dw.$$
 Thanks to Proposition \ref{5hh140420143} we have
 $$
  \left[\left(\mathbb{I}_{\alpha }[\varphi_n*\mu]\right)^{p'}\right]_{\mathfrak{M}^{\mathcal{H}_\alpha,p}}\lesssim \left[\left(\mathbb{I}_{\alpha }[\mu]\right)^{p'}\right]_{\mathfrak{M}^{\mathcal{H}_\alpha,p}},$$
 which implies \eqref{5hh250320141}. 
   \end{proof}
   \begin{corollary}\label{5hh250320146}
   Let $p>1$, $\alpha>0$ with $0<\alpha p\leq N+2$,  $0<\delta<\alpha$ and $R,d>0$. There holds
               \begin{equation}\label{5hh250320144}
              \left[\left(\mathbb{I}_{\alpha }^{R,\delta}[\mu]\right)^{p'}\right]_{\mathfrak{M}^{\mathcal{G}_\alpha,p}}\lesssim_{d,R} \left[\mu\right]^{p'}_{\mathfrak{M}^{\mathcal{G}_\alpha,p}}
               \end{equation}
               for all $\mu\in\mathfrak{M}^+(\mathbb{R}^{N+1} )$ with $\text{diam}(\text{supp}(\mu))\leq d$.  Furthermore, 
                           \begin{align}
                           \label{5hh250320142}
                           \left[\varphi_n*\mu\right]_{\mathfrak{M}^{\mathcal{G}_\alpha,p}}\lesssim_{d} \left[\mu\right]_{\mathfrak{M}^{\mathcal{G}_\alpha,p}}
                           \end{align}for $n\in\mathbb{N}$,  $\mu\in\mathfrak{M}^+(\mathbb{R}^{N+1} )$ with $\text{diam}(\text{supp}(\mu))\leq d$ where $\left\{  \varphi_{n}\right\}$ is a sequence of standard mollifiers in $\mathbb{R}^{N+1}$.
   \end{corollary} 
  \begin{proof}
  It is easy to see that
$$
       ||E_\alpha^R[\mu]||_{L^{p'}(\mathbb{R}^{N+1})}\lesssim_{d/R} ||E_\alpha^{R,\delta}*\mu||_{L^{p'}(\mathbb{R}^{N+1})}\lesssim_{d/R} ||E_\alpha^R[\mu]||_{L^{p'}(\mathbb{R}^{N+1})}$$
        for any $\mu\in\mathfrak{M}^+(\mathbb{R}^{N+1})$ with $\text{diam}(\text{supp}(\mu))\leq d$,
        thus  $\text{Cap}_{E_\alpha^{R,\delta},p}(E)\sim_{d/R} \text{Cap}_{E_\alpha^{R},p}(E)$  for every compact set $E\subset\mathbb{R}^{N+1}$, $\text{diam}(E)\leq d$.
        Therefore, by Corollary
        \ref{5hh230120148} we get $\text{Cap}_{E_\alpha^{R,\delta},p}(E)\sim_{d,R} \text{Cap}_{\mathcal{G}_\alpha,p}(E)$ for every compact set $E\subset\mathbb{R}^{N+1}$, $\text{diam}(E)\leq d$. Thus, by Theorem  \ref{5hh1410136} and $\text{diam}(\text{supp}(\mu))\leq d$ we get that if 
        for every compact set $E\subset \mathbb{R}^{N+1}$, 
             $$
               \mu(E)\lesssim_{d,R} \text{Cap}_{\mathcal{G}_\alpha,p}(E),
             $$
           then  for every compact set $E\subset \mathbb{R}^{N+1}$, 
$$ \int_{E}\left(\mathbb{I}_\alpha^{R,\delta}[\mu]\right)^{p'}dxdt\lesssim_{d,R} \text{Cap}_{E_\alpha^{R,\delta},p}(E)\lesssim_{d,R}  \text{Cap}_{\mathcal{G}_\alpha,p}(E).$$
                   It follows \eqref{5hh250320144}. 
                   As in the Proof of Corollary \ref{5hh250320145}  we also have  for
                   $w\in A_{p'}$ 
$$
                    \int_{\mathbb{R}^{N+1}}\left(\mathbb{I}^{1,\delta}_\alpha[\varphi_n*\mu]\right)^{p'}dw\lesssim_{[w]_{A_{p'}}} \int_{\mathbb{R}^{N+1}}\left(\mathbb{I}^{1,\delta}_\alpha[\mu]\right)^{p'}dw.$$
                    Thanks to Proposition \ref{5hh140420143} and Theorem \ref{5hh1410136} we obtain \eqref{5hh250320142}.
  \end{proof}
\begin{remark}\label{5hh270320141} Likewise (see \cite[Lemma 5.7]{55Ph2}), we can verify that for $\frac{2}{p}<\alpha<N+\frac{2}{p}$,
$$\left[\varphi_{1,n}\star\omega_1\right]_{\mathfrak{M}^{\mathbf{I}_{\alpha-2/p},p}}\lesssim \left[\omega_1\right]_{\mathfrak{M}^{\mathbf{I}_{\alpha-2/p},p}},\quad\quad \left[\varphi_{1,n}\star\omega_2\right]_{\mathfrak{M}^{\mathbf{G}_{\alpha-2/p},p}}\lesssim_d \left[\omega_2\right]_{\mathfrak{M}^{\mathbf{G}_{\alpha-2/p},p}}, 
    $$
     for $n\in\mathbb{N}$ and $\omega_1,\omega_2\in\mathfrak{M}^+(\mathbb{R}^{N} )$ with $\text{diam}(\text{supp}(\omega_2))\leq d$ where  $\left\{  \varphi_{1,n}\right\}$ is a sequence of standard mollifiers in $\mathbb{R}^{N}$ and $[.]_{\mathfrak{M}^{\mathbf{I}_{\alpha-2/p},p}},[.]_{\mathfrak{M}^{\mathbf{G}_{\alpha-2/p},p}}$ was defined in Remark \ref{5hh130320146}.
    Hence, we obtain for $\frac{2}{p}<\alpha<N+\frac{2}{p}$
\begin{align*}
&\left[(\varphi_{1,n}*\omega_1)\otimes\delta_{\{t=0\}}\right]_{\mathfrak{M}^{\mathcal{H}_\alpha,p}}\lesssim \left[\omega_1\otimes\delta_{\{t=0\}}\right]_{\mathfrak{M}^{\mathcal{H}_\alpha,p}},\\&  \left[(\varphi_{1,n}*\omega_2)\otimes\delta_{\{t=0\}}\right]_{\mathfrak{M}^{\mathcal{G}_\alpha,p}}\lesssim_d \left[\omega_2\otimes\delta_{\{t=0\}}\right]_{\mathfrak{M}^{\mathcal{G}_\alpha,p}},
    \end{align*}for $n\in\mathbb{N}$ and $\omega_1,\omega_2\in\mathfrak{M}^+(\mathbb{R}^{N+1} )$, $\text{diam}(\text{supp}(\mu))\leq d$. 
\end{remark}   
      
   \begin{proposition} \label{5hh14101315}Let $q>1$, $0<\alpha q<N+2$, $0<R\leq \infty, 0<\delta<\alpha$ and $K>0$. Let $0\leq f\in L^q_{\text{loc}}(\mathbb{R}^{N+1})$. Let $C_4,C_5$ be constants in inequalities \eqref{5hh04104} and \eqref{5hh04105} in Theorem \ref{5hh1410136} with $p=q'$. Suppose that $\{u_n\}$ is a sequence of nonnegative measurable functions in $\mathbb{R}^{N+1}$  satisfying 
   \begin{align}
   \nonumber
     & u_{n+1}\leq K \mathbb{I}^{R,\delta}_\alpha[u_n^q]+f~~\forall n\in\mathbb{N},\\&~~~~~~~u_0\leq f.\label{5hh1410139}
   \end{align}
   Then, if for every compact set $E\subset \mathbb{R}^{N+1}$, 
       \begin{equation}\label{5hh1410133}
    \int_{E}f^qdxdt\leq C \text{Cap}_{E^{R,\delta}_\alpha,q'}(E)
       \end{equation}
       with 
       \begin{equation}
       \label{5hh1410134} C\leq C_4\left(\frac{2^{-q+1}}{C_5(q-1)}\left(\frac{q-1}{qK2^{q-1}}\right)^q\right)^{q-1},
       \end{equation}
       we have  
       \begin{equation}
       \label{5hh1410135}
       u_n\leq \frac{Kq2^{q-1}}{q-1}\mathbb{I}^{R,\delta}_{\alpha}[f^q]+f~~\forall n\in\mathbb{N}.
       \end{equation}
   \end{proposition}
   \begin{proof}
   From \eqref{5hh04104} and \eqref{5hh04105} in Theorem \ref{5hh1410136}, we see that \eqref{5hh1410133} implies  
       \begin{equation}\label{5hh1410137'}
              \mathbb{I}^{R,\delta}_\alpha[(I^{R,\delta}_\alpha[f^q])^{q}]\leq \left(\frac{C}{C_4}\right)^{\frac{1}{q-1}}C_5 \mathbb{I}^{R,\delta}_\alpha [f^q].
              \end{equation}
       Now we prove \eqref{5hh1410135} by induction. Clearly, \eqref{5hh1410135} holds with $n=0$. Next we assume that \eqref{5hh1410135} holds with $n=m$. Then, by \eqref{5hh1410134},  \eqref{5hh1410137'} and \eqref{5hh1410139} we have
       \begin{align*}
       u_{m+1}&\leq K\mathbb{I}^{R,\delta}_\alpha[u_n^q]+f\\&\leq  
              K2^{q-1}\left(\frac{Kq2^{q-1}}{q-1}\right)^q\mathbb{I}^{R,\delta}_\alpha[(\mathbb{I}^{R,\delta}_\alpha[f^q])^q]+K2^{q-1}\mathbb{I}^{R,\delta}_\alpha[f^q]+f\\&\leq  K2^{q-1}\left(\frac{Kq2^{q-1}}{q-1}\right)^q\left(\frac{C}{C_4}\right)^{\frac{1}{q-1}}C_5\mathbb{I}^{R,\delta}_\alpha[f^q]+K2^{q-1}\mathbb{I}^{R,\delta}_\alpha[f^q]+f\\&\leq  \frac{Kq2^{q-1}}{q-1}\mathbb{I}^{R,\delta}_{\alpha}[f^q]+f.
       \end{align*}     
    Therefore \eqref{5hh1410135} also holds true with $n=m+1$. The proof is complete. 
   \end{proof}
   \begin{corollary} Let $q>\frac{N+2}{N+2-\alpha}$, $\alpha>0$ and $f\in L^q_+(\mathbb{R}^{N+1})$. There exists a constant $\varepsilon_0>0$ depending on $N,\alpha,q$ such that if for every compact set $E\subset \mathbb{R}^{N+1}$,
      $\int_Ef^qdxdt\leq \varepsilon_0 \text{Cap}_{\mathcal{H}_\alpha,q'}(E)$,
      then $u=\mathcal{H}_\alpha[u^q]+f$ admits a positive solution $u\in L^q_{loc}(\mathbb{R}^{N+1})$.
      \end{corollary}
      \begin{proof}
         Consider the sequence $\{u_n\}$ of nonnegative functions defined by $u_0=f$ and 
      $
        u_{n+1}=\mathcal{H}_\alpha[u_n^q]+f~~\forall~n\geq 0
       $. 
        It is easy to see that 
       $
        u_{n+1}\leq c_1\mathbb{I}_2[u_n^q]+f~~\forall n\geq 0
      $.        
       By Proposition \ref{5hh14101315} and Corollary \ref{5hh250320145}, there exists a constant $c_2=c_2(N,\alpha,q)>0$ such that 
       if   for every compact set $E\subset \mathbb{R}^{N+1}$,
          $\int_Ef^qdxdt\leq c_2 \text{Cap}_{\mathcal{H}_\alpha,q'}(E)$ then $u_n$ is well defined and
         $$
          u_n\leq \frac{c_1q3^{q-1}}{q-1}\mathbb{I}_{\alpha}[f^q]+f~~\forall n\geq 0.
         $$
      Since $\{u_n\}$ is nondecreasing, thus thanks to the dominated convergence theorem we obtain $u(x,t)=\lim\limits_{n\to\infty}u_n(x,t)$ is a solution of $u=\mathcal{H}_\alpha[u^q]+f$ which $u\in L^q_{\text{loc}}(\mathbb{R}^{N+1})$. The proof is complete.   
         \end{proof}
      \begin{corollary} Let $q>1$, $\alpha>0$, $0<R\leq \infty, 0<\delta<\alpha$ and $\mu\in\mathfrak{M}^+(\mathbb{R}^{N+1})$. The following two statements are equivalent.
      \begin{description}
      \item[a.] for every compact set $E\subset \mathbb{R}^{N+1}$, we have 
               $\int_Ef^qdxdt\leq C \text{Cap}_{E^{R,\delta}_\alpha,q'}(E)$
               for some  $C>0$,
       \item[b.] There exists a function $u\in L^q_{loc}(\mathbb{R}^{N+1})$ such that $u=\mathbb{I}_\alpha^{R,\delta}[u^q]+\varepsilon f$ for some $\varepsilon>0$.  
      \end{description}
         \end{corollary}
         \begin{proof}
         We will prove $b.\Rightarrow a.$ Set $d\omega (x,t)=\left(\left(\mathbb{I}_\alpha^{R,\delta}[u^q]\right)^q+\varepsilon^qf^q\right)dxdt$, thus we have $ dw(x,t)\geq \left(I_\alpha^{R,\delta}[\omega]\right)^qdxdt$. Let $\mathbb{M}_\omega$ denote the centered Hardy-littlewoood maximal function which is defined for $g\in L^1_{loc}(\mathbb{R}^{N+1}, d\omega)$,
         $$
         \mathbb{M}_\omega g(x,t)=\sup_{\rho>0}\frac{1}{\omega(\tilde{Q}_\rho(x,t))}\int_{\tilde{Q}_\rho(x,t)}|g|d\omega.
         $$
        We have for compact set $E\subset\mathbb{R}^{N+1}$
         \begin{align*}
         \int_{\mathbb{R}^{N+1}}\left(\mathbb{M}_\omega \chi_E\right)^q\left(\mathbb{I}_\alpha^{R,\delta}[\omega]\right)^qdxdt\leq \int_{\mathbb{R}^{N+1}}\left(\mathbb{M}_\omega \chi_E\right)^qd\omega(x,t).
         \end{align*}         
         Since $\mathbb{M}_\omega$ is bounded on $L^s(\mathbb{R}^{N+1},d\omega)$ for $s>1$ and $\left(\mathbb{M}_\omega \chi_E\right)^q\left(\mathbb{I}_\alpha^{R,\delta}[\omega]\right)^q\geq \left(\mathbb{I}_\alpha^{R,\delta}[\omega\chi_E]\right)^q$, thus
        $$
                  \int_{\mathbb{R}^{N+1}}\left(\mathbb{I}_\alpha^{R,\delta}[\omega\chi_E]\right)^qdxdt\lesssim \omega(E).
                $$
By Theorem \ref{5hh1410136}, we get for any compact set $E\subset\mathbb{R}^{N+1}$
\begin{align*}
\omega(E)\lesssim \text{Cap}_{E^{R,\delta}_\alpha,q'}(E).
\end{align*}
It follows the result.                   
         \end{proof}\\
\begin{remark} In \cite{55QH1}, we also use Theorem \ref{5hh1410136} to show the existence  of mild solutions to the   Navier-Stokes Equations
\begin{equation}\label{5hh090920142}
  \left\{ 
  \begin{array}{l}
  \partial_tu-{\Delta }u+\mathbb{P}\operatorname{div}(u\otimes u)=\mathbb{P}F ~~\text{in }\mathbb{R}^N\times (0,\infty),
  \\ 
  u(0)=u_0~~\text{in }\mathbb{R}^N,
  \end{array}%
  \right.
  \end{equation}
  where $u,F\in \mathbb{R}^N$ , $\mathbb{P}=id-\nabla \Delta^{-1}\nabla.$ is the Helmholtz Leray projection onto the vector fields of zero divergence, i.e, for $f\in\mathbb{R}^N$,  $\mathbb{P}f=f-\nabla u$ and $\Delta u=\operatorname{div}f$. Namely, there exists $\varepsilon_0=\varepsilon_0(N)>0$ such that if $\operatorname{div}(u_0)=0$ and 
     \begin{equation}\label{060920141}
            \int_{K}|D(x,t)|^2dxdt\leq \varepsilon_0 \text{Cap}_{\mathcal{H}_1,2}(K),
            \end{equation}
           for any compact set $K\subset \mathbb{R}^{N+1}$, where if $(x,t)\in \mathbb{R}^N\times [0,+\infty)$, 
            \begin{equation*}
            D(x,t)=(e^{t\Delta}u_0)(x)+\int_{0}^{t}(e^{(t-s)\Delta}\mathbb{P}F)(x)ds,
            \end{equation*}
            and  $D(x,t)=0$ otherwise. 
            Then, the equation \eqref{5hh090920142} has globally  solution $u$ satisfying 
                      \begin{equation}\label{070920148}
                      |u(x,t)|\leq |D(x,t)|+c\mathbb{I}_1[|D|^2](x,t),
                      \end{equation}
                      for all $(x,t)\in \mathbb{R}^N\times (0,\infty)$ for some $c=c(N)$.
\end{remark}

     \section{Global point wise estimates of solutions to the parabolic equations}  
   
    First, we recall Duzzar and Mingione's result \cite{55DuzaMing},  see also \cite{55KuMi1,55KuMi2} which involves local pointwise estimates for solutions of equations \eqref{5hhparabolic1}. 
  \begin{theorem} \label{5hh14101313} If $u\in L^2(0,T,H^1(\Omega))\cap C(\Omega_T)$ is a weak solution to \eqref{5hhparabolic1} with $\mu\in L^2(\Omega_T)$ and $u(0)=0$, we have 
     \begin{equation}
     |u(x,t)|\lesssim \fint_{\tilde{Q}_R(x,t)}|u|+\mathbb{I}_{2}^{2R}[|\mu|](x,t)
     \end{equation}
     for all $Q_{2R}(x,t)\subset \Omega\times (-\infty,T)$.\\
     Furthermore, if $A$ is independent of space variable $x$,  $\eqref{5hhcondc}$ is satisfied and $\nabla u \in C(\Omega_T)$ then 
     \begin{align}
     \label{5hh310320147}
     |\nabla u(x,t)|\lesssim \fint_{\tilde{Q}_R(x,t)}|\nabla u|dyds+\mathbb{I}_{1}^{2R}[|\mu|](x,t)
     \end{align}
      for all $Q_{2R}(x,t)\subset \Omega\times (-\infty,T)$.
     \end{theorem}                          
\begin{proof}[Proof of Theorem \ref{5hh141013112}]Let $\mu=\mu_{0}+\mu_{s}\in\mathfrak{M}_{b}(\Omega_T)$,
with $\mu_{0}\in\mathfrak{M}_{0}(\Omega_T),\mu_{s}\in\mathfrak{M}_{s}(\Omega_T)$.  By Proposition
\ref{5hhatt}, there exist  sequences of nonnegative measures $\mu_{n,0,i}=(f_{n,i}%
         ,g_{n,i},h_{n,i})$ and $\mu_{n,s,i}$ such that 
   \begin{description}
   	\item[i)]  $f_{n,i},g_{n,i},h_{n,i}\in C_{c}^{\infty}(\Omega_T)$ 
   	strongly converge to some $f_i,g_i,h_i$ in $L^{1}(\Omega_T),L^{2}(\Omega_T,\mathbb{R}^N)$ and $L^2(0,T,H^1_0(\Omega))$
   	respectively;
   	\item[ii)]$\mu_{n,1},\mu_{n,2}, \mu_{n,s,1},\mu_{n,s,2}\in C_{c}^{\infty}(\Omega_T)$ converge to $\mu^{+},\mu^-,\mu^+_s,\mu^-_s$ resp. in the narrow topology with $\mu_{n,i}=\mu_{n,0,i}+\mu_{n,s,i}$, for $i=1,2$ and
   	satisfying $\mu_{0}^{+}=(f_{1},g_{1},h_{1}),$ $\mu_{0}^{-}=(f_{2},g_{2},h_{2})$ and  $0\leq\mu_{n,1} \leq \varphi_n*\mu^{+}, 0\leq\mu_{n,2} \leq \varphi_n*\mu^{-}$, where $\{\varphi_n\}$ is a sequence of standard mollifiers in $\mathbb{R}^{N+1}$.
   \end{description}      
   Let $\sigma_{1,n},\sigma_{2,n}\in C^\infty_c(\Omega)$ be convergent to $\sigma^+$ and $\sigma^-$ in the narrow topology and in  $L^1(\Omega)$ if $\sigma\in L^1(\Omega)$ resp. such that $0\leq\sigma_{1,n}\leq \varphi_{1,n}*\sigma^+, 0\leq\sigma_{2,n}\leq \varphi_{1,n}*\sigma^-$ where $\{\varphi_{1,n}\}$ is a sequence of standard mollifiers in $\mathbb{R}^{N}$. Set $\mu_{n}=\mu_{n,1}-\mu_{n,2}$ and $\sigma_n=\sigma_{1,n}-\sigma_{2,n}$.\\
         Let $u_n,u_{n,1},u_{n,2}$ be solutions of equations 
         \begin{equation}
                  \left\{
                  \begin{array}
                  [c]{l}%
                  {(u_{n})_t}-\operatorname{div}(A(x,t,\nabla u_n))=\mu_{n}~~\text{in }\Omega_T,\\
                  u_n=0~~~\text{on}~\partial\Omega \times (0,T),\\
                  {u}_{n}(0)=\sigma_{n}~~\text{on }\Omega,
                  \end{array}
                  \right.  
                  \end{equation}               
                   \begin{equation}
                                    \left\{
                                    \begin{array}
                                    [c]{l}%
                                    {(u_{n,1})_t}-\operatorname{div}(A(x,t,\nabla u_{n,1}))=\chi_{\Omega_T}\mu_{n,1}~~\text{in }B_{2T_0}(x_0)\times (0,2T_0^2),\\
                                    u_{n,1}=0~~~\text{on}~\partial B_{2T_0}(x_0)\times (0,2T_0^2),\\
                                    {u}_{n,1}(0)=\sigma_{1,n}~~\text{on }B_{2T_0}(x_0),
                                    \end{array}
                                    \right.  
                                    \end{equation}
\begin{equation}
                                    \left\{
                                    \begin{array}
                                    [c]{l}%
                                    {(u_{n,2})_t}+\operatorname{div}(A(x,t,-\nabla u_{n,2}))=\chi_{\Omega_T}\mu_{n,2}~~\text{in }B_{2T_0}(x_0)\times (0,2T_0^2),\\
                                    u_{n,2}=0~~~\text{on}~\partial B_{2T_0}(x_0)\times (0,2T_0^2),\\
                                    {u}_{n,2}(0)=\sigma_{2,n}~~\text{on }B_{2T_0}(x_0),
                                    \end{array}
                                    \right.                                      \end{equation}                                                                        where $\Omega\subset B_{T_0}(x_0)$ for $x_0\in\Omega$.\\
                                                                        We see that $u_{n,1},u_{n,2}\geq 0$ in $B_{2T_0}(x_0)\times (0,2T_0^2)$ and  $-u_{n,2}\leq u_n\leq u_{n,1}$ in $\Omega_T$. \\ 
                                                     Now, we estimate $u_{n,1}$. By Remark \ref{5hh240220142} and Theorem \ref{5hhsta}, a sequence $\{u_{n,1,m}\}$ of solutions to equations
 \begin{equation}
       \left\{
    \begin{array}
                             [c]{l}%
                                    {(u_{n,1,m})_t}-\operatorname{div}(A(x,t,\nabla u_{n,1,m}))=\left(g_{n,m}\right)_t+\chi_{\Omega_T}\mu_{n,1}~~\text{in }B_{2T_0}(x_0)\times (-2T_0^2,2T_0^2),\\
                                    u_{n,1,m}=0~~~\text{on}~\partial B_{2T_0}(x_0)\times (-2T_0^2,2T_0^2),\\
                                    {u}_{n,1,m}(-2T^2_0)=0~~\text{on }B_{2T_0}(x_0),
                                    \end{array}
                                    \right.  
                    \end{equation}   
                    converges to $u_{n,1}$ in $B_{2T_0}(x_0)\times (0,2T_0^2)$,                
                    where $g_{n,m}(x,t)=\sigma_{1,n}(x)\int_{-2T^2_0}^{t}\varphi_{2,m}(s)ds$ and $\{\varphi_{2,m}\}$ is a sequence of mollifiers in $\mathbb{R}$.\\                         
                    By Remark \ref{5hh070420143}, we have
  \begin{equation} \label{5hh240220143}                                                                    ||u_{n,1,m}||_{L^1(\tilde{Q}_{2T_0}(x_0,0))}\lesssim T_0^2 A_{n,m},
        \end{equation} 
  where                                                          $
 A_{n,m}=\mu_{n,1}(\Omega_T)+\int_{\tilde{Q}_{2T_0}(x_0,0)}\sigma_{1,n}(x)\varphi_{2,m}(t)dxdt.$   \\                                      
Hence, thanks to Theorem \ref{5hh14101313} we have for $(x,t)\in \Omega_T$
\begin{align*}
 u_{n,1,m}(x,t)&\lesssim T_0^{-N-2}||u_{n,1,m}||_{L^1(\tilde{Q}_{2T_0}(x_0,0))}+\mathbb{I}_{2}[\mu_{n,1}](x,t)+c\mathbb{I}_{2}[\sigma_{1,n}\varphi_m](x,t)
 \\&\lesssim \mathbb{I}_{2}[\mu_{n,1}](x,t)+\mathbb{I}_{2}[\sigma_{1,n}\varphi_m](x,t).
\end{align*}
 Since $0\leq\mu_{n,1} \leq \varphi_n*\mu^{+}$, $\sigma_{1,n}\leq \varphi_{1,n}*\sigma^+$,
$$
                      u_{n,1,m}(x,t)\leq c\varphi_n*\mathbb{I}_{2}[\mu^+](x,t)+c(\varphi_{1,n}\varphi_{2,m})*\mathbb{I}_{2}[\sigma^+\otimes \delta_{\{t=0\}}](x,t)~\forall ~(x,t)\in \Omega_T.                    $$
   Letting $m\to\infty$, we get
   $$u_{n,1}(x,t)\lesssim\varphi_n*\mathbb{I}_{2}[\mu^+](x,t)+\varphi_{1,n}*\left(\mathbb{I}_{2}[\sigma^+\otimes \delta_{\{t=0\}}](.,t)\right)(x)~\forall ~(x,t)\in \Omega_T.$$
 Similarly, we also get 
$$
      u_{n,2}(x,t)\lesssim\varphi_n*\mathbb{I}_{2}[\mu^-](x,t)+\varphi_{1,n}*\left(\mathbb{I}_{2}[\sigma^-\otimes \delta_{\{t=0\}}](.,t)\right)(x)~\forall ~(x,t)\in \Omega_T. $$                  
     Consequently, by Proposition \ref{5hhmun} and Theorem \ref{5hhsta} , up to a subsequence,
           $\{u_{n}\}$ converges to a  distributional solution (a renormalized solution if $\sigma\in L^1(\Omega)$) $u$ of \eqref{5hhparabolic1} and satisfied    \eqref{5hh14101314}.                                    
\end{proof}
\begin{remark}\label{5hh070420145} Obviously, if $\sigma\equiv0$ and $\text{supp}(\mu)\subset\overline{\Omega}\times[a,T]$, $a>0$ then $u=0$ in $\Omega\times(0,a)$.
\end{remark}
\begin{remark} \label{5hh070420146}
 If $A$ is independent of space variable $x$,  $\eqref{5hhcondc}$ is satisfied then
 \begin{align}
 \label{5hh070420144}
 |\nabla u(x,t)|\lesssim_{T_0/d} \mathbb{I}_1^{2T_0}[|\mu|+|\sigma|\otimes\delta_{\{t=0\}}](x,t)
 \end{align}
 for any $(x,t)\in\Omega^d\times (0,T)$ and $ 0<d\leq\frac{1}{2}\min\{\sup_{x\in\Omega}d(x,\partial\Omega),T_0^{1/2}\}$ where  $\Omega^d=\{x\in\Omega: d(x,\partial\Omega)>d\}$. 
 Indeed, by Remark \ref{5hh240220142} and Theorem \ref{5hhsta}, a sequence $\{v_{n,m}\}$ of solutions to equations
  \begin{equation}
        \left\{
     \begin{array}
                              [c]{l}%
                                     {(v_{n,m})_t}-\operatorname{div}(A(t,\nabla u_{n,m}))=\left(g_{n,m}\right)_t+\chi_{\Omega_T}\mu_{n}~~\text{in }\Omega\times (-2T_0^2,T),\\
                                     v_{n,m}=0~~~\text{on}~\partial \Omega\times (-2T_0^2,T),\\
                                     {v}_{n,m}(-2T^2_0)=0~~\text{on }\Omega,
                                     \end{array}
                                     \right.  
                     \end{equation}   
                     converges to $u_{n}$ in $L^1(0,T,W^{1,1}_0(\Omega))$,                
                     where $g_{n,m}(x,t)=\sigma_{n}(x)\int_{-2T^2_0}^{t}\varphi_{2,m}(s)ds$ and $\{\varphi_{2,m}\}$ is a sequence of mollifiers in $\mathbb{R}$.\\
                     By Theorem \ref{5hh14101313}, we have for any $(x,t)\in\Omega^d\times (0,T)$
$$
                     |\nabla v_{n,m}(x,t)|\lesssim
                      \fint_{\tilde{Q}_{d/2}(x,t)}|\nabla v_{n,m}|+\mathbb{I}_{1}^{d}[|\mu_n|+|\sigma_n|\otimes\varphi_{2,m}](x,t).
$$
                     On the other hand, by remark \ref{5hh070420143}, $$|||\nabla v_{n,m}|||_{L^1(\Omega\times(-T_0^2,T))}\lesssim T_0 (|\mu_n|+|\sigma_n|\otimes\varphi_{2,m})(\Omega\times(-T_0^2,T)).$$
                     Therefore, for any $(x,t)\in\Omega^d\times (0,T)$
                     $$|\nabla v_{n,m}(x,t)|\lesssim_{T_0/d}\mathbb{I}_{1}[|\mu_n|+|\sigma_n|\otimes\varphi_{2,m}](x,t).
                                       $$
Finally, letting $m\to \infty$ and $n\to\infty$ we get for any $(x,t)\in\Omega^d\times (0,T)$
$$|\nabla u(x,t)|\lesssim\mathbb{I}_{1}[|\mu|+|\sigma|\otimes\delta_{\{t=0\}}](x,t).$$     We  conclude \eqref{5hh070420144} since    $\mathbb{I}_1[|\mu|+|\sigma|\otimes\delta_{\{t=0\}}] \lesssim \mathbb{I}_1^{2T_0}[|\mu|+|\sigma|\otimes\delta_{\{t=0\}}] $ in  $\Omega_T$.                          
\end{remark}
Next, we will establish pointwise estimates from below for solutions of equations \eqref{5hhparabolic1}. 
\begin{theorem}\label{5hh260220143}If $u\in C(Q_r(y,s))\cap L^2(s-r^2,s, H^1(B_r(y)))$ is a nonnegative weak solution of \eqref{5hhparabolic1} with data $\mu\in \mathfrak{M}^+(Q_r(y,s))$ and $u(s-r^2)\geq0$, then we have
\begin{equation}
u(y,s)\gtrsim\sum_{k=0}^{\infty}\frac{\mu(Q_{r_k/8}(y,s-\frac{35}{128}r_k^2))}{r_k^N},
\end{equation}
where $r_k=4^{-k}r$. 
\end{theorem}

\begin{proof} It is enough to show that for $\rho\in (0,r)$
\begin{equation}\label{5hh270120144}
\frac{\mu(Q_{\rho/8}(y,s-\frac{35}{128}\rho^2))}{\rho^N}\lesssim \inf_{Q_{\rho/4}(y,s)}u-\inf_{Q_{\rho}(y,s)}u.
\end{equation} By \cite[Theorem 6.18, p. 122 ]{55Li3}, we have for any $\theta\in (0,1+2/N)$,
\begin{equation}\label{5hh270120141}
\left(\fint_{Q_{\rho/4}(y,s-\rho^2/4)}(u-a)^\theta\right)^{1/\theta}\lesssim b-a,
\end{equation}
where $b=\inf_{Q_{\rho/4}(y,s)}u $, $a=\inf_{Q_{\rho}(y,s)}u.$

Let $\eta\in C^\infty_c(Q_\rho(y,s))$ such that $0\leq \eta\leq 1$,  $\text{supp} \eta\subset Q_{\rho/4}(y,s-\frac{1}{4}\rho^2)$, $\eta=1$ in $Q_{\rho/8}(y,s-\frac{35}{128}\rho^2)$ and $|\nabla\eta|\lesssim 1/\rho$, $|\eta_t|\lesssim 1/\rho^2$. We have 
\begin{align*}
\mu(Q_{\rho/8}(y,s-\frac{35}{128}\rho^2))&\leq\int_{Q_\rho(y,s)}\eta^2d\mu(x,t)\\&= \int_{Q_\rho(y,s)}u_t\eta^2dxdt+2\int_{Q_\rho(y,s)}\eta A(x,t,\nabla u)\nabla \eta dxdt 
\\&= -2\int_{Q_\rho(y,s)}(u-a)\eta_t\eta dxdt+2\int_{Q_\rho(y,s)}\eta A(x,t,\nabla u)\nabla \eta dxdt 
\\ &\lesssim r^{-2}\int_{Q_{\rho/4}(y,s-\frac{1}{4}\rho^2)}(u-a) dxdt+\int_{Q_\rho(y,s)}\eta |\nabla u||\nabla \eta| dxdt
\\ &\lesssim r^N(b-a)+\int_{Q_\rho(y,s)}\eta |\nabla u||\nabla \eta| dxdt.
\end{align*}
Here we used \eqref{5hh270120141} with $\theta=1$ in the last inequality. It remains to show that  
\begin{equation}\label{5hh270120142}
\int_{Q_r(y,s)}\eta |\nabla u||\nabla \eta| dxdt\lesssim r^{N}(b-a).
\end{equation}
First, we verify that for $\varepsilon\in(0,1)$
\begin{equation}\label{5hh270120143}
\int_{Q_\rho(y,s)} |\nabla u|^2(u-a)^{-\varepsilon-1}\eta^2dxdt
\lesssim\int_{Q_\rho(y,s)}(u-a)^{1-\varepsilon}\left(\eta|\eta_t|+|\nabla \eta|^2\right)dxdt.
\end{equation}
Indeed, 
for $\delta\in(0,1)$  we choose $\varphi=(u-a+\delta)^{-\varepsilon}\eta^2$ as test function in \eqref{5hhparabolic1}, 
\begin{align*}
0&\leq \int_{Q_\rho(y,s)} u_t(u-a+\delta)^{-\varepsilon}\eta^2 dxdt+\int_{Q_\rho(y,s)} A(x,t,\nabla u)\nabla \left((u-a+\delta)^{-\varepsilon}\eta^2\right) dxdt\\
&\leq 2(1-\varepsilon)\int_{Q_\rho(y,s)} (u-a+\delta)^{1-\varepsilon}|\eta_t|\eta dxdt -\varepsilon \Lambda_2\int_{Q_\rho(y,s)} |\nabla u|^2(u-a+\delta)^{-\varepsilon-1}\eta^2dxdt\\&~~~~+ 2\Lambda_1\int_{Q_\rho(y,s)}\eta|\nabla u|(u-a+\delta)^{-\varepsilon}|\nabla \eta|dxdt.
\end{align*}
So, we deduce \eqref{5hh270120143} from using the H\"older's inequality and letting $\delta\to 0$.\\
Therefore, for $\varepsilon\in (0,2/N)$ using the H\"older's inequality,  we get
\begin{align*}&\int_{Q_r(y,s)}\eta |\nabla u||\nabla \eta| dxdt\\&~~~\leq \left(\int_{Q_\rho(y,s)} |\nabla u|^2(u-a)^{-\varepsilon-1}\eta^2dxdt\right)^{1/2}\left(\int_{Q_\rho(y,s)} (u-a)^{\varepsilon+1}|\nabla \eta |^2dxdt\right)^{1/2}
\\&~~~\overset{\eqref{5hh270120143}}\lesssim\left(\int_{Q_\rho(y,s)}(u-a)^{1-\varepsilon}\left(\eta|\eta_t|+|\nabla \eta|^2\right)dxdt\right)^{1/2}\left(\int_{Q_\rho(y,s)} (u-a)^{\varepsilon+1}|\nabla \eta |^2dxdt\right)^{1/2}
\\&~~~\lesssim\rho^{-2}\left(\int_{Q_{\rho/4}(y,s-\frac{1}{4}\rho^2)}(u-a)^{1-\varepsilon}dxdt\right)^{1/2}\left(\int_{Q_{\rho/4}(y,s-\frac{1}{4}\rho^2)} (u-a)^{\varepsilon+1}dxdt\right)^{1/2}.
\end{align*}
Hence, \eqref{5hh270120144}  follows from this and \eqref{5hh270120141}.
\end{proof}\\\\
\begin{proof}[Proof of Theorem \ref{5hh270120145}]Let $\mu_{n}\in (C_c^\infty (\Omega_T))^+,\sigma_n\in (C_c^\infty(\Omega))^+$ be in the proof of  Theorem \ref{5hh141013112}. 
         Let $u_n$ be a weak solution of equation
                  \begin{equation*}
                           \left\{
                           \begin{array}
                           [c]{l}%
                           {(u_{n})_t}-\operatorname{div}(A(x,t,\nabla u_n))=\mu_{n}~~\text{in }\Omega_T,\\
                           u_n=0~~~\text{on}~\partial\Omega \times (0,T),\\
                           {u}_{n}(0)=\sigma_n~~\text{on }\Omega.
                           \end{array}
                           \right.  
                           \end{equation*} 
As the proof of  Theorem \ref{5hh141013112}, thanks to Theorem \ref{5hh260220143} we get for any $Q_r(y,s)\subset\Omega\times (-\text{diam}(\Omega),T)$ and $r_k=4^{-k}r$
$$u_{n}(y,s)\gtrsim\sum_{k=0}^{\infty}\frac{\mu_n(Q_{r_k/8}(y,s-\frac{35}{128}r_k^2))}{r_k^N}+\sum_{k=0}^{\infty}\frac{(\sigma_{n}\otimes\delta_{\{t=0\}})(Q_{r_k/8}(y,s-\frac{35}{128}r_k^2))}{r_k^N}.
$$
Finally, by Proposition \ref{5hhmun} and Theorem \ref{5hhsta} we get the result.                         
\end{proof}\\
\begin{remark} If $u\in L^q(\Omega_T)$ satisfies \eqref{5hh010420141} then $\mathcal{G}_2[\chi_E\mu]\in L^q(\mathbb{R}^{N+1})$ and $\mathbf{G}_{\frac{2}{q}}[\chi_{F}\sigma]\in L^q(\mathbb{R}^{N})$ for every $E\subset\subset \Omega\times [0,T)$ and $F\subset\subset\Omega$. Indeed, for $E\subset\subset\Omega\times [0,T)$, $\varepsilon=\text{dist}\left(E,(\Omega\times (0,T))\cup (\Omega\times \{t=T\})\right)>0$, we can see that for any $(y,s)\in \Omega_T$, $r_k=4^{-k}\varepsilon/4$
$$
    u(y,s)\gtrsim\sum_{k=0}^{\infty}\frac{\tilde{\mu}(E\cap Q_{r_k/8}(y,s-\frac{35}{128}r_k^2))}{r_k^N},$$
         where $\tilde{\mu}=\mu+\sigma\otimes\delta_{\{t=0\}}$.\\
     Moreover, for any $(y,s)\notin \Omega_T$
     \begin{align*}
     \sum_{k=0}^{\infty}\frac{\tilde{\mu}(E\cap Q_{r_k/8}(y,s-\frac{35}{128}r_k^2))}{r_k^N}=0.
     \end{align*}
    Thus, 
\begin{align*}
\infty &> \int_{\mathbb{R}^{N+1}}\sum_{k=0}^{\infty}\left(\frac{\tilde{\mu}(E\cap Q_{r_k/8}(y,s-\frac{35}{128}r_k^2))}{r_k^N}\right)^qdyds \geq \int_{\mathbb{R}^N}\sum_{k=0}^{\infty}\int_{\mathbb{R}}\left(\frac{\tilde{\mu}(E\cap \tilde{Q}_{r_k/8}(y,s))}{r_k^N}\right)^qdsdy
\\& \gtrsim\int_{\mathbb{R}^{N+1}}\int_0^{\varepsilon/64}\left(\frac{\tilde{\mu}(E\cap\tilde{Q}_\rho(y,s))}{\rho^N}\right)^q\frac{d\rho}{\rho}dsdy\gtrsim_\varepsilon\int_{\mathbb{R}^{N+1}}\left(\mathcal{G}_2[\tilde{\mu}\chi_E]\right)^qdsdy.
\end{align*}
Thus, from Proposition \ref{5hh240120146}, we get the results.
\end{remark}
\begin{proof}[Proof of Theorem \ref{5hh1203201417}]Set $D_n=B_n(0)\times(-n^2,n^2)$. For $n\geq 4$, by Theorem \ref{5hh141013112}, there exists a renormalized solution $u_n$ to problem 
         \begin{equation*}
                  \left\{
                  \begin{array}
                  [c]{l}%
                  {(u_{n})_t}-\operatorname{div}(A(x,t,\nabla u_n))=\chi_{D_{n-1}}\omega~~\text{in }D_n,\\
                  u_n=0~~~\text{on}~\partial B_n(0)\times (-n^2,n^2),\\
                  {u}_{n}(-n^2)=0~~\text{on }B_n(0).
                  \end{array}
                  \right.  
                  \end{equation*} 
              relative to  a decomposition $(f_n,g_n,h_n)$ of $\chi_{D_{n-1}}\omega_0$              satisfying  \begin{equation}
                      -K\mathbb{I}_{2}[\omega^-](x,t) \leq u_n(x,t)\leq K\mathbb{I}_{2}[\omega^+](x,t)~~ \forall ~(x,t)\in D_{n}.
                                                 \end{equation}
From the proof of Theorem \ref{5hh141013112} and Remark \ref{5hh1203201410}, we can assume that $u_n$  satisfies  \eqref{5hh110320141} and \eqref{5hh110320147} in Proposition \ref{5hh1203201411} with $1<q_0<\frac{N+2}{N}$, $L\equiv 0$. Moreover, there holds
\begin{align}
||f_n||_{L^1(D_i)}+||g_n||_{L^2(D_i)}+|||h_n|+|\nabla h_n|||_{L^2(D_i)}\leq 2|\omega|(D_{i+1})\label{5hh130320148}
\end{align}
 for any $i=1,...,n-1$ and $h_n$ is convergent in $L^1_{\text{loc}}(\mathbb{R}^{N+1})$.\\
On the other hand, by Lemma \ref{5hh1310136} we have for any $s\in (1,\frac{N+2}{N})$ 
\begin{equation}
\int_{D_m}|u_n|^{s}dxdt\leq K^s\int_{D_m}(I_2[|\omega|])^{s}dxdt\leq K^s\int_{\tilde{Q}_{4m}(x_0,t_0)}(I_2[|\omega|])^{s}dxdt\lesssim m^{N+2}, \label{5hh1203201413}
\end{equation}
for  $n\geq m \geq |x_0|+|t_0|^{1/2}$.
Consequently, we can apply Proposition \ref{5hh1203201412} and obtain that $u_n$ converges to some $u$ in $L^1_{loc}(\mathbb{R};W^{1,1}_{loc}(\mathbb{R}^N))$.\\ Since  for any $\alpha\in(0,1/2)$
\begin{align*}
\int_{D_m}\frac{|\nabla u_n|^2}{(|u_n|+1)^{\alpha+1}}dxdt\lesssim_{\alpha,m}1 ~~\forall~ n\geq m,
\end{align*}
thus using \eqref{5hh1203201413} and H\"older's inequality, we get  for any $1\leq s_1<\frac{N+2}{N+1}$ 
\begin{align*}
\int_{D_m}|\nabla u_n|^{s_1}dxdt\lesssim_{s_1,m} 1~~\text{ for all }~ n\geq m\geq |x_0|+|t_0|^{1/2}.
\end{align*}
This yields $u_n\to u$ in $L^{s_1}_{\text{loc}}(\mathbb{R};W^{1,s_1}_{\text{loc}}(\mathbb{R}^N))$. \\
Take $\varphi\in C_c^\infty(\mathbb{R}^{N+1})$ and $m_0\in\mathbb{N}$ with $\text{supp} (\varphi)\subset D_{m_0}$, we have for $n\geq m_0+1$
\begin{align*}
  -\int_{\mathbb{R}^{N+1}}u_n\varphi_tdxdt+\int_{\mathbb{R}^{N+1}}A(x,t,\nabla u_n)\nabla \varphi dxdt=\int_{\mathbb{R}^{N+1}}\varphi d\omega
  \end{align*}
Letting $n\to\infty$, we conclude that $u$ is a distributional solution to problem  \eqref{5hhparabolic2'} with data $\mu=\omega$ which satisfies \eqref{5hh1203201414}.\\
       \textbf{Claim 1.} If $\omega\geq 0$. By Theorem \ref{5hh270120145}, we have for $n\geq 4^{k_0+1}$, $(y,s)\in B_{4^{k_0}}\times (0,n^2)$
     $$
                            u_n(y,s)\gtrsim\sum_{k=0}^{\infty}\frac{\omega(Q_{r_k/8}(y,s-\frac{35}{128}r_k^2)\cap D_{n-1})}{r_k^N},
              $$ where $r_k=4^{-k+k_0}$.
       This gives 
$$
                               u_n(y,s)\gtrsim\sum_{k=-k_0}^{\infty}\frac{\omega(Q_{2^{-2k-3}}(y,s-35\times 2^{-4k-7})\cap B_{n-1}(0)\times(0,(n-1)^2))}{2^{-2Nk}}.$$
       Letting $n\to\infty$ and $k_0\to\infty$ we have   \eqref{5hh1203201415}. 
       Finally, thanks to Proposition \ref{5hh230120143} and Theorem \ref{5hh051120131}, we will assert \eqref{5hh130320145} if we show that for $q>\frac{N+2}{N}$
$$ B:=\int_{\mathbb{R}}\left(\sum_{k=-\infty}^{\infty}\frac{\omega(Q_{2^{-2k-3}}(x,t-35\times 2^{-4k-7}))}{2^{-2Nk}}\right)^qdxdt\gtrsim \int_{\mathbb{R}}\int_{0}^{+\infty}\left(\frac{\omega(\tilde{Q}_\rho(x,t))}{\rho^N}\right)^q\frac{d\rho}{\rho}dxdt.$$
       Indeed, 
       \begin{align*}
       B&\geq  \sum_{k=-\infty}^{\infty}\int_{\mathbb{R}}\left(\frac{\omega(Q_{2^{-2k-3}}(x,t-35\times 2^{-4k-7}))}{2^{-2Nk}}\right)^qdtdx
       \\&~~~~~= \sum_{k=-\infty}^{\infty}\int_{\mathbb{R}}\left(\frac{\omega(\tilde{Q}_{2^{-2k-3}}(x,t))}{2^{-2Nk}}\right)^qdt       
        \gtrsim \int_{\mathbb{R}^{N+1}}\int_{0}^{+\infty}\left(\frac{\omega(\tilde{Q}_\rho(x,t))}{\rho^N}\right)^q\frac{d\rho}{\rho}dxdt.
       \end{align*}
      \textbf{Claim 2.}  If $A$ is independent of space variable $x$ and  $\eqref{5hhcondc}$ is satisfied. By Remark \ref{5hh070420146} we get for any $(x,t)\in D_{n/4} $
      $$|\nabla u_n(x,t)|\lesssim\mathbb{I}_{1}[|\omega|](x,t).$$
                                                                       Letting $n\to \infty$, we get \eqref{5hh310320146}. \\
               \textbf{Claim 3.}  If $\omega=\mu+\sigma\otimes\delta_{\{t=0\}}$ with $\mu\in\mathfrak{M}(\mathbb{R}^N\times (0,\infty))$ and $\sigma\in \mathfrak{M}(\mathbb{R}^N)$, then by  Remark \eqref{5hh070420145} we  get  that $u_n=0$ in $B_n(0)\times (-n^2,0)$. So, $u=0$ in $\mathbb{R}^N\times(-\infty,0)$. Therefore, clearly ${\left. u \right|_{\mathbb{R}^N\times [0,\infty)}}$ is a distributional solution to \eqref{5hhparabolic2}.  The proof is complete.                                                                                              
\end{proof}
\begin{remark}\label{5hh280320141} If $\omega\in \mathfrak{M}_b(\mathbb{R}^{N+1})$ then $u$ satisfies 
\begin{align*}
|||\nabla u|||_{L^{\frac{N+2}{N+1},\infty}(\mathbb{R}^{N+1})}\lesssim|\omega|(\mathbb{R}^{N+1}).
\end{align*} 
Moreover, $I_2[|\omega|]\in L^{\frac{N+2}{N},\infty}(\mathbb{R}^{N+1})$ and  $I_2[|\omega|]<\infty$ a.e in $\mathbb{R}^{N+1}$.

\end{remark} 
\section{Quasilinear Lane-Emden Type Parabolic Equations}
\subsection{Quasilinear Lane-Emden Parabolic Equations in $\Omega_T$} 
To prove Theorem \ref{5hh070120146} we need the following proposition which was proved in \cite{55BaPi2}.
\begin{proposition}\label{5hh040220142}
 Assume that $O$ is an open subset of $\mathbb{R}^{N+1}$. Let $p>1$ and $\mu\in \mathfrak{M}^+(O)$. If $\mu$ is absolutely continuous with respect to $\text{Cap}_{2,1,p}$ in $O$, there exists a nondecreasing sequence
$\{\mu_n\}\subset \mathfrak{M}_b^+(O)\cap \left(W^{2,1}_p(\mathbb{R}^{N+1})\right)^{*} $,  with compact support in $O$ which converges to 
$\mu$  weakly in $\mathfrak{M}(O)$. Moreover, if $\mu\in\mathfrak{M}_b^+(O)$ then $||\mu_n-\mu||_{\mathfrak{M}_b(O)}\to 0$ as $n\to\infty$.
\end{proposition}
\begin{remark} By Theorem \ref{5hh040220141}, $W^{2,1}_p(\mathbb{R}^{N+1})=\mathcal{L}_2^p(\mathbb{R}^{N+1})$, it follows $\{\mu_n\}\subset \mathfrak{M}_b^+(O)\cap \left(\mathcal{L}_2^p(\mathbb{R}^{N+1})\right)^{*}$. Since $||\mu_n||_{\left(\mathcal{L}_2^p(\mathbb{R}^{N+1})\right)^{*}}=|| {\mathop \mathcal{G}\limits^ \vee}_2[\mu_n]||_{L^{p'}(\mathbb{R}^{N+1})}$,  so $ {\mathop \mathcal{G}\limits^ \vee}_2[\mu_n]\in L^{p'}(\mathbb{R}^{N+1})$.\medskip\\ Consequently, from \eqref{5hh230120142'} in Proposition \ref{5hh230120143}, we obtain $\mathbb{I}_2^R[\mu_n]\in L^{p'}(\mathbb{R}^{N+1})$ for any $n\in\mathbb{N}$ and $R>0$. In particular, $\mathbb{I}_2[\mu_n]\in L^{p'}_{\text{loc}}(\mathbb{R}^{N+1})$ for any $n\in\mathbb{N}$.
\end{remark}
\begin{remark}
As in the proof of Theorem 2.5 in \cite{55VHV}, we can prove a general version of Proposition \ref{5hh040220142}, that is:  for $p>1$, if $\mu$ is absolutely continuous with respect to $\text{Cap}_{\mathcal{G}_\alpha,p}$ in $O$, there exists a nondecreasing sequence
$\{\mu_n\}\subset \mathfrak{M}_b^+(O)\cap \left(\mathcal{L}_\alpha^p(\mathbb{R}^{N+1})\right)^{*} $,  with compact support in $O$ which converges to 
$\mu$ weakly in $\mathfrak{M}(O)$. Furthermore, $\mathbb{I}_\alpha[\mu_n]\in L^{p'}_{\text{loc}}(\mathbb{R}^{N+1})$ for all $n\in\mathbb{N}$. Besides, we also obtain that for $\mu\in\mathfrak{M}_b(O)$ is absolutely continuous with respect to $\text{Cap}_{\mathcal{G}_\alpha,p}$ in $O$ if and only if $\mu=f+\nu$ where $f\in L^1(O)$ and $\nu\in  \left(\mathcal{L}_\alpha^p(\mathbb{R}^{N+1})\right)^{*}.$ 
\end{remark}
\begin{proof}[Proof of Theorem \ref{5hh070120146}]First, assume that $\sigma\in L^1(\Omega)$. Because $\mu$ is absolutely continuous with respect to the capacity 
$\text{Cap}_{2,1,q'}$, so are $\mu^+$ and $\mu^-$. Applying Proposition \ref{5hh040220142} there exist two nondecreasing sequences $\{\mu_{1,n}\}$ and $\{\mu_{2,n}\}$ of positive bounded measures with compact support in $\Omega_T$ which converge to $\mu^+$ and $\mu^-$ in $\mathfrak{M}_b(\Omega_T)$ respectively and such that $\mathbb{I}_2[\mu_{1,n}], \mathbb{I}_2[\mu_{2,n}]\in L^{q}(\Omega_T)$.\medskip\\
For $i=1,2$, set $\tilde{\mu}_{i,1}=\mu_{i,1}$ and $\tilde{\mu}_{i,j}=\mu_{i,j}-\mu_{i,j-1}\geq 0$, so $\mu_{i,n}=\sum_{j=1}^{n}\tilde{\mu}_{i,j}$. We write $\mu_{i,n}=\mu_{i,n,0}+\mu_{i,n,s}$,  $\tilde{\mu}_{i,j}=\tilde{\mu}_{i,j,0}+\tilde{\mu}_{i,j,s}$ with  $\mu_{i,n,0},\tilde{\mu}_{i,n,0}\in\mathfrak{M}_0(\Omega_T)$, $\mu_{i,n,s},\tilde{\mu}_{i,n,s}\in\mathfrak{M}_s(\Omega_T)$.\\ As in the proof of Theorem \ref{5hh141013112}, for any $j\in\mathbb{N}$ and $i=1,2$, there exist  sequences of nonnegative measures $\tilde{\mu}_{m,i,j,0}=(f_{m,i,j}%
         ,g_{m,i,j},h_{m,i,j})$ and $\tilde{\mu}_{m,i,j,s}$ such that
         \begin{description}
         	\item[i)] $f_{m,i,j},g_{m,i,j},h_{m,i,j}\in C_{c}^{\infty}(\Omega_T)$
         	strongly converge to some $f_{i,j},g_{i,j},h_{i,j}$ in $L^{1}(\Omega_T),L^{2}(\Omega_T,\mathbb{R}^N)$ and $L^2(0,T,H^1_0(\Omega))$
         	respectively;
         	\item[ii)] $\tilde{\mu}_{m,i,j},\tilde{\mu}_{m,i,j,s}\in C_{c}^{\infty}(\Omega_T)$ converge to $\tilde{\mu}_{i,j},\tilde{\mu}_{i,j,s}$ resp. in the narrow topology with $\tilde{\mu}_{m,i,j}=\tilde{\mu}_{m,i,j,0}+\tilde{\mu}_{m,i,j,s}$ which 
         	satisfy $\tilde{\mu}_{i,j,0}=(f_{i,j},g_{i,j},h_{i,j})$ and  $0\leq\tilde{\mu}_{m,i,j} \leq \varphi_m*\tilde{\mu}_{i,j}$ and                    \begin{equation}\label{5hh130320141}
         		||f_{m,i,j}||_{L^1(\Omega_T)}+\left\Vert g_{m,i,j}\right\Vert _{L^{2}(\Omega_T,\mathbb{R}^N)}+||h_{m,i,j}||_{L^2(0,T,H^1_0(\Omega))}+\mu_{m,i,j,s}(\Omega_T)
         		\leq 2\tilde{\mu}_{i,j}(\Omega_T).\end{equation}
         	Here $\{\varphi_m\}$ is a sequence of mollifiers in $\mathbb{R}^{N+1}$.
         \end{description}    For any $n,k,m\in\mathbb{N}$, let $u_{n,k,m},u_{1,n,k,m}, u_{2,n,k,m}\in W $ be  solutions of problems
\begin{equation}\label{5hh080120144}
\left\{ \begin{array}{l}
  (u_{n,k,m})_t- \operatorname{div} (A(x,t,\nabla u_{n,k,m})) + T_k(|u_{n,k,m}|^{q - 1}u_{n,k,m})= \sum_{j=1}^{n}(\tilde{\mu}_{m,1,j}-\tilde{\mu}_{m,2,j}) ~\text{ in }~\Omega_T,  \\ 
  u_{n,k,m}=0~~~~~~~~~~~~~~~~\text{ on }\partial\Omega\times (0,T),\\
 u_{n,k,m}(0) = T_n(\sigma^+)-T_n(\sigma^-)~~~\text{on }~\Omega,  
 \end{array} \right.
\end{equation}
\begin{equation}\label{5hh080120145}
\left\{ \begin{array}{l}
  (u_{1,n,k,m})_t- \operatorname{div} (A(x,t,\nabla u_{1,n,k,m})) + T_k(u_{1,n,k,m}^q)= \sum_{j=1}^{n}\tilde{\mu}_{m,1,j} ~\text{ in }~\Omega_T,  \\ 
    u_{1,n,k,m}=0~~~~~~~~~~~~~~~~\text{ on }\partial\Omega\times (0,T),\\
   u_{1,n,k,m}(0) = T_n(\sigma^+)~~~\text{in}~\Omega,  \\  
 \end{array} \right.
\end{equation}
\begin{equation}\label{5hh080120146}
\left\{ \begin{array}{l}
  (u_{2,n,k,m})_t- \operatorname{div} (\tilde{A}(x,t,\nabla u_{2,n,k,m})) + T_k(u_{2,n,k,m}^q)= \sum_{j=1}^{n}\tilde{\mu}_{m,2,j}  ~\text{ in }~\Omega_T,  \\ 
 u_{2,n,k,m} = 0~~~~~~~~~~~~~~~~\text{ on }\partial\Omega\times (0,T),\\
    u_{2,n,k,m}(0) = T_n(\sigma^-)~~~\text{in }~\Omega,  \\ 
 \end{array} \right.
\end{equation}
where $\tilde{A}(x,t,\xi)=-A(x,t,-\xi)$ and 
$$W=\left\{z:z\in L^2(0,T,H^1_0(\Omega)),z_t\in L^2(0,T,H^{-1}(\Omega))\right\}.$$
Thanks to Comparison Principle Theorem and Theorem \ref{5hh141013112}, we have that  for any $m,k$ the sequences $\{u_{1,n,k,m}\}_n$ and $\{u_{2,n,k,m}\}_n$ are  increasing and 
\begin{align*}
-K \mathbb{I}_2[T_n(\sigma^-)\otimes\delta_{\{t=0\}}]-K \mathbb{I}_2[\mu_{2,n}*\varphi_m]&\leq -u_{2,n,k,m}\leq u_{n,k,m}\leq u_{1,,n,k,m}\\&\leq K \mathbb{I}_2[\mu_{1,n}*\varphi_m]+K \mathbb{I}_2[T_n(\sigma^+)\otimes\delta_{\{t=0\}}],
\end{align*} 
where a constant $K$ is in Theorem \ref{5hh141013112}.
Thus, 
\begin{align*}
-K \mathbb{I}_2[T_n(\sigma^-)\otimes\delta_{\{t=0\}}]-K \mathbb{I}_2[\mu_{2,n}]*\varphi_m&\leq -u_{2,n,k,m}\leq u_{n,k,m}\leq u_{1,,n,k,m}\\&\leq K \mathbb{I}_2[\mu_{1,n}]*\varphi_m+K \mathbb{I}_2[T_n(\sigma^+)\otimes\delta_{\{t=0\}}].
\end{align*}
Moreover, 
\begin{align*}
\int_{\Omega_T}T_k(u_{i,n,k,m}^q)dxdt&\leq \int_{\Omega_T}\varphi_m*\mu_{i,n}dxdt+|\sigma|(\Omega)
\leq |\mu|(\Omega_T)+|\sigma|(\Omega).
\end{align*}
As in \cite[Proof of Lemma 5.3]{55VHb},  thanks to Proposition \ref{5hhmun} and  Theorem \ref{5hhsta},  there exist subsequences of $\{u_{n,k,m}\}_m$ $\{u_{1,n,k,m}\}_m$, $\{u_{2,n,k,m}\}_m$, still denoted them, converging to renormalized solutions $u_{n,k}$ $u_{1,n,k}$, $u_{2,n,k}$ of equations
\begin{description}
	\item  \eqref{5hh080120144} with data $\mu_{1,n}-\mu_{2,n}$, $u_{n,k}(0)=T_n(\sigma^+)-T_n(\sigma^-)$ and the decomposition $(\sum_{j=1}^{n}f_{1,j}-\sum_{j=1}^{n}f_{2,j},\sum_{j=1}^{n}g_{1,j}-\sum_{j=1}^{n}g_{2,j},\sum_{j=1}^{n}h_{1,j}-\sum_{j=1}^{n}h_{2,j})$ of $\mu_{1,n,0}-\mu_{2,n,0}$,
	\item \eqref{5hh080120145} with data $\mu_{1,n}$, $u_{1,n,k}(0)=T_n(\sigma^+)$ and the decomposition $(\sum_{j=1}^{n}f_{1,j},$ $\sum_{j=1}^{n}g_{1,j},$ $\sum_{j=1}^{n}h_{1,j})$ of $\mu_{1,n,0}$,
	\item \eqref{5hh080120146} with data $\mu_{2,n}$, $u_{2,n,k}(0)=T_n(\sigma^-)$ and the decomposition $(\sum_{j=1}^{n}f_{2,j},$ $\sum_{j=1}^{n}g_{2,j},\sum_{j=1}^{n}h_{2,j})$ of $\mu_{2,n,0}$ respectively,
\end{description}   which satisfy \begin{align*}
-K \mathbb{I}_2[T_n(\sigma^-)\otimes\delta_{\{t=0\}}]-K \mathbb{I}_2[\mu_{2,n}]&\leq -u_{2,n,k}\leq u_{n,k}\leq u_{1,n,k}\\&\leq K \mathbb{I}_2[\mu_{1,n}]+K \mathbb{I}_2[T_n(\sigma^+)\otimes\delta_{\{t=0\}}].
\end{align*}
Next, as in \cite[Proof of Lemma 5.4]{55VHb} since $I_2[\mu_{i,n}]\in L^{q}(\Omega_T)$ for any $n$, 
thanks to Proposition \ref{5hhmun} and  Theorem \ref{5hhsta},  there exist subsequences of $\{u_{n,k}\}_k$ $\{u_{1,n,k}\}_k$, $\{u_{2,n,k}\}_k$, still denoted them, converging to renormalized solutions $u_{n}$ $u_{1,n}$, $u_{2,n}$ of equations 
\begin{equation}\label{5hh080120147}
\left\{ \begin{array}{l}
  (u_{n})_t- \operatorname{div} (A(x,t,\nabla u_{n})) + |u_{n}|^{q - 1}u_{n}= \mu_{1,n}-\mu_{2,n}~\text{ in }~\Omega_T,  \\ 
   u_{n}=0~~~~~~~~~~~~~~~~\text{ on }\partial\Omega\times (0,T),\\
  u_{n}(0) = T_n(\sigma^+)-T_n(\sigma^-)~~~\text{in }~\Omega, 
 \end{array} \right.
\end{equation}
\begin{equation}\label{5hh080120148}
\left\{ \begin{array}{l}
  (u_{1,n})_t- \operatorname{div} (A(x,t,\nabla u_{1,n})) + u_{1,n}^q= \mu_{1,n} ~\text{ in }~\Omega_T,  \\ 
  u_{1,n}=0~~~~~~~~~~~~~\text{ on }\partial\Omega\times (0,T),\\
    u_{1,n}(0) = T_n(\sigma^+)~~~\text{in }~\Omega,  \\ 
 \end{array} \right.
\end{equation}
\begin{equation}\label{5hh080120149}
\left\{ \begin{array}{l}
  (u_{2,n})_t- \operatorname{div} (\tilde{A}(x,t,\nabla u_{2,n})) + u_{2,n}^q= \mu_{2,n} ~\text{ in }~\Omega_T,  \\ 
  u_{2,n}=0~~~~~~~~~~~\text{ on }\partial\Omega\times (0,T),\\
    u_{2,n}(0) = T_n(\sigma^-)~~~\text{in }~\Omega,  \\ 
 \end{array} \right.
\end{equation} relative to the decomposition $(\sum_{j=1}^{n}f_{1,j}-\sum_{j=1}^{n}f_{2,j},\sum_{j=1}^{n}g_{1,j}-\sum_{j=1}^{n}g_{2,j},\sum_{j=1}^{n}h_{1,j}-\sum_{j=1}^{n}h_{2,j})$ of $\mu_{1,n,0}-\mu_{2,n,0}$, $(\sum_{j=1}^{n}f_{1,j},$ $\sum_{j=1}^{n}g_{1,j},$ $\sum_{j=1}^{n}h_{1,j})$ of $\mu_{1,n,0}$ and $(\sum_{j=1}^{n}f_{2,j},$ $\sum_{j=1}^{n}g_{2,j},\sum_{j=1}^{n}h_{2,j})$ of $\mu_{2,n,0}$ respectively,  which satisfy \begin{align*}
-K \mathbb{I}_2[T_n(u_0^-)\otimes\delta_{\{t=0\}}]-K \mathbb{I}_2[\mu_{2,n}]&\leq -u_{2,n}\leq u_{n}\leq u_{1,n}\\&\leq K \mathbb{I}_2[\mu_{1,n}]+K \mathbb{I}_2[T_n(u_0^+)\otimes\delta_{\{t=0\}}].
\end{align*}
Moreover, the sequences $\{u_{1,n}\}_n$ and $\{u_{2,n}\}_n$ are  increasing and $$
\int_{\Omega_T}u_{i,n}^qdxdt\leq |\mu|(\Omega_T)+|\sigma|(\Omega).$$
Note that from \eqref{5hh130320141} we have
$$
||f_{i,j}||_{L^1(\Omega_T)}+\left\Vert g_{i,j}\right\Vert _{L^{2}(\Omega_T,\mathbb{R}^N)}+||h_{i,j}||_{L^2(0,T,H^1_0(\Omega))}
         \leq 2\tilde{\mu}_{i,j}(\Omega_T)$$
which implies
$$
\sum_{j=1}^{n}\left(||f_{i,j}||_{L^1(\Omega_T)}+||g_{i,j}|| _{L^{2}(\Omega_T,\mathbb{R}^N)}+||h_{i,j}||_{L^2(0,T,H^1_0(\Omega))}\right)
         \leq 2\mu_{i,n}(\Omega_T)\leq 2|\mu|(\Omega_T).$$
Finally, as in \cite[Proof of Theorem 5.2]{55VHb} thanks to Proposition \ref{5hhmun}, Theorem \ref{5hhsta} and Monotone Convergence Theorem  there exist subsequences of $\{u_{n}\}_n$, $\{u_{1,n}\}_n$, $\{u_{2,n}\}_n$, still denoted them, converging to renormalized solutions $u$, $u_{1}$, $u_{2}$ of equations
\begin{description}
	\item \eqref{5hh080120147} with data $\mu$, $u(0)=\sigma$ and the decomposition $(\sum_{j=1}^{\infty}f_{1,j}-\sum_{j=1}^{\infty}f_{2,j},\sum_{j=1}^{\infty}g_{1,j}-\sum_{j=1}^{\infty}g_{2,j},\sum_{j=1}^{\infty}h_{1,j}-\sum_{j=1}^{\infty}h_{2,j})$ of $\mu_{0}$, 
	\item  \eqref{5hh080120148} with data $\mu^+$, $u_1(0)=\sigma^+$ and the decomposition $(\sum_{j=1}^{\infty}f_{1,j},\sum_{j=1}^{\infty}g_{1,j},\sum_{j=1}^{\infty}h_{1,j})$ of $\mu_0^+$, 
	\item \eqref{5hh080120149} with data $\mu^-$, $u_2(0)=\sigma^-$ and the decomposition $(\sum_{j=1}^{\infty}f_{2,j},\sum_{j=1}^{\infty}g_{2,j},\sum_{j=1}^{\infty}h_{2,j})$ of $\mu_{0}^-$, respectively 
\end{description} satisfying
$$
-K \mathbb{I}_2[\sigma^-\otimes\delta_{\{t=0\}}]-K \mathbb{I}_2[\mu^-]\leq -u_{2}\leq u\leq u_{1}\leq K \mathbb{I}_2[\mu^+]+K \mathbb{I}_2[\sigma^+\otimes\delta_{\{t=0\}}].$$
We now have remark: if $\sigma\equiv0$ and $\text{supp}(\mu)\subset\overline{\Omega}\times [a,T]$, $a>0$, then $u=u_1=u_2=0$ in $\Omega\times(0,a)$ since $u_{n,k}=u_{1,n,k}=u_{2,n,k}=0$ in $\Omega\times(0,a)$.
\\Next, we will consider $\sigma\in \mathfrak{M}_b(\Omega)$ such that $\sigma$ is absolutely continuous with respect to the capacity $\text{Cap}_{\mathbf{G}_{\frac{2}{q},q'}}$ in $\Omega$. So, $\chi_{\Omega_T}\mu+\sigma\otimes\delta_{\{t=0\}}$ is absolutely continuous with respect to the capacity $\text{Cap}_{2,1,q'}$ in $\Omega\times (-T,T)$. As above, we verify that there exists a renormalized solution $u$ of
\begin{equation}
\left\{ \begin{array}{l}
  u_t- \operatorname{div} (A(x,t,\nabla u)) + |u|^{q - 1}u= \chi_{\Omega_T}\mu+\sigma\otimes\delta_{\{t=0\}} ~\text{ in }~\Omega\times (-T,T),  \\ 
   u=0~~~~~~~~~~\text{ on }\partial\Omega\times (-T,T),\\
  u(-T) = 0~~~\text{on }~\Omega, 
 \end{array} \right.
\end{equation}
satisfying $u=0$ in $\Omega\times (-T,0)$ and 
\begin{equation*}
-K \mathbb{I}_2[\sigma^-\otimes\delta_{\{t=0\}}]-K \mathbb{I}_2[\mu^-]\leq u\leq K \mathbb{I}_2[\mu^+]+K \mathbb{I}_2[\sigma^+\otimes\delta_{\{t=0\}}].
\end{equation*}
Finally, from remark \ref{5hh060420141} we get the result. This completes the proof.
\end{proof}\medskip\\
\begin{proof}[Proof of Theorem \ref{5hh070120147}]Let $\{\mu_{n,i}\}\subset C_c^\infty(\Omega_T), \sigma_{i,n}\in C_c^\infty(\Omega)$ for $i=1,2$ be as in the proof of Theorem  \ref{5hh141013112}. We have $0\leq\mu_{n,1}\leq \varphi_n*\mu^+,0\leq\mu_{n,2}\leq \varphi_n*\mu^-, 0\leq\sigma_{1,n}\leq \varphi_{1,n}*\sigma^+, 0\leq\sigma_{2,n}\leq \varphi_{1,n}*\sigma^-$ for any $n\in\mathbb{N}$ where $\{\varphi_n\}$ and $\{\varphi_{1,n}\}$ are  sequences of standard mollifiers in $\mathbb{R}^{N+1},\mathbb{R}^N$ respectively.\medskip\\  We prove that the problem \eqref{5hh070120149} has a solution with data $\mu=\mu_{n_0}=\mu_{n_0,1}-\mu_{n_0,2},\sigma=\sigma_{n_0}=\sigma_{1,n_0}-\sigma_{2,n_0}$ for $n_0\in\mathbb{N}$.
 Put
       \begin{align*}
       J&=\left\{u\in L^q(\Omega_T):u^+\leq\frac{q K}{q-1}\mathbb{I}_2^{2T_0,\delta}[\mu_{n_0,1}+\sigma_{1,n_0}\otimes\delta_{\{t=0\}}] \right.
       \\&~~~~~~~~~~~~~~\text{ and }~\left.u^-\leq\frac{q K}{q-1}\mathbb{I}_2^{2T_0,\delta}[\mu_{n_0,2}+\sigma_{2,n_0}\otimes\delta_{\{t=0\}}]~\right\},
       \end{align*} 
       where $\max\{-\frac{N+2}{q'}+2,0\}<\delta<2$.\\
Clearly, $J$ is closed under the strong topology of $L^q(\Omega_T)$ and convex. \\
       We consider a map $S:J\to J$ defined for each $v\in J$ by $S(v)=u$, where $u\in L^1(\Omega_T)$  is the unique renormalized solution of 
       \begin{equation}\label{5hh090420142}
                         \left\{
                         \begin{array}
                         [c]{l}%
                         {u_{t}}-\operatorname{div}(A(x,t,\nabla u))=|v|^{q-1}v+\mu_{n_0,1}-\mu_{n_0,2}~\text{in }\Omega_T,\\ 
                                             u=0~~~~~~~\text{on}~~
                                          \partial\Omega \times (0,T),
                                             \\
                                             u(0) = \sigma_{1,n_0}-\sigma_{2,n_0}~~~\text{in}~~ \Omega. \\ 
                         \end{array}
                         \right.  
                         \end{equation} 
 By Theorem \ref{5hh141013112}, we have 
      \begin{align*}
     & u^+\leq K\mathbb{I}_2^{2T_0}[(v^+)^q]+ K\mathbb{I}_2^{2T_0}[\mu_{n_0,1}+\sigma_{1,n_0}\otimes\delta_{\{t=0\}}],\\&
      u^-\leq K\mathbb{I}_2^{2T_0}[(v^-)^q]+ K\mathbb{I}_2^{2T_0}[\mu_{n_0,2}+\sigma_{2,n_0}\otimes\delta_{\{t=0\}}],
      \end{align*}  
      where $K$ is the constant in Theorem \ref{5hh141013112}. 
      Thus, 
      \begin{align*}
           & u^+\leq K\left(\frac{q K}{q-1}\right)^q\mathbb{I}_2^{2T_0,\delta}\left[\left(\mathbb{I}_2^{2T_0,\delta}[\mu_{n_0,1}+\sigma_{1,n_0}\otimes\delta_{\{t=0\}}]\right)^q\right]+ K\mathbb{I}_2^{2T_0,\delta}[\mu_{n_0,1}+\sigma_{1,n_0}\otimes\delta_{\{t=0\}}],\\&
            u^-\leq K\left(\frac{q K}{q-1}\right)^q\mathbb{I}_2^{2T_0,\delta}\left[\left(\mathbb{I}_2^{2T_0,\delta}[\mu_{n_0,2}+\sigma_{2,n_0}\otimes\delta_{\{t=0\}}]\right)^q\right]+ K\mathbb{I}_2^{2T_0,\delta}[\mu_{n_0,2}+\sigma_{2,n_0}\otimes\delta_{\{t=0\}}].
            \end{align*}
            Thus, thanks to Theorem \ref{5hh1410136} there exists $\varepsilon_0=\varepsilon_0(N,K,\delta,q)$ such that 
              if for  every compact sets $E\subset \mathbb{R}^{N+1}$,  \begin{equation}\label{5hh020420141}
                     |\mu_{n_0,i}|(E)+(|\sigma_{i,n_0}|\otimes\delta_{\{t=0\}})(E)\leq \varepsilon_0 \text{Cap}_{E_2^{2T_0,\delta},q'}(E).
                     \end{equation}   
                     then $\mathbb{I}_2^{2T_0,\delta}[\mu_{n_0,i}+\sigma_{i,n_0}\otimes\delta_{\{t=0\}}]\in L^q(\mathbb{R}^{N+1})$ and 
                                                \begin{align*}
                                                \mathbb{I}_2^{2T_0,\delta}\left[\left(\mathbb{I}_2^{2T_0,\delta}[\mu_{n_0,i}+\sigma_{i,n_0}\otimes\delta_{\{t=0\}}]\right)^q\right]\leq \frac{(q-1)^{q-1}}{(Kq)^q}\mathbb{I}_2^{2T_0,\delta}[\mu_{n_0,i}+\sigma_{i,n_0}\otimes\delta_{\{t=0\}}]~i=1,2.
                                                \end{align*}        which implies $u\in L^q(\Omega_T)$ and
                                                $$ u^+\leq  \frac{q K}{q-1}\mathbb{I}_2^{2T_0}[\mu_{n_0,1}+\sigma_{1,n_n}\otimes\delta_{\{t=0\}}],\quad
                                                      u^-\leq  \frac{q K}{q-1}\mathbb{I}_2^{2T_0}[\mu_{n_0,2}+\sigma_{2,n_0}\otimes\delta_{\{t=0\}}].$$                                                                        
Now we assume  that \eqref{5hh020420141} is satisfied, so   $S$ is well defined. 
         Therefore, if we can show that the map $S:J\to J$ is continuous and $S(J)$ is pre-compact  under the strong topology of $L^q(\Omega_T)$ then by the Schauder Fixed Point Theorem, $S$ has a fixed point on $J$. Hence the problem \eqref{5hh070120149} has a solution with data $\mu=\mu_{n_0},\sigma=\sigma_{n_0}$.\medskip\\
         Now we show that \textbf{S is continuous}.
         Let $\{v_n\}$ be a sequence in $J$ such that $v_n$ converges strongly in $L^q(\Omega_T)$ to a function $v\in J$. Set $u_n=S(v_n)$. We need to show that $u_n\to S(v)$ in $L^q(\Omega_T)$.

         By Proposition \ref{5hhmun},  there exists a subsequence of $\{u_n\} $, still denoted by it, converging to $u$ a.e in $\Omega_T$. Since $$|u_n|\leq \sum_{i=1,2}\frac{qK}{q-1}\mathbb{I}_2^{2T_0,\delta}[\mu_{n_0,i}+\sigma_{i,n_0}\otimes\delta_{\{t=0\}}]\in L^q(\Omega_T)~~\forall~n\in\mathbb{N}$$
       Thanks to the Dominated Convergence Theorem, we have $u_n\to u$ in $L^q(\Omega_T)$. Hence, thanks to Theorem \ref{5hhsta} we get $u=S(v)$.  \\       
Next we show that \textbf{S is pre-compact}. 
         Indeed if $\{u_n\}=\{S(v_n)\}$ is a sequence in $S(J)$. By Proposition \ref{5hhmun}, there exists a subsequence of $\{u_n\} $, still denoted by it, converging to $u$ a.e in $\Omega_T$. Again,  the Dominated Convergence Theorem we get $u_n\to u$ in $L^q(\Omega_T)$. So \textbf{S is pre-compact}. \medskip\\
         Next,  thanks to Corollary \ref{5hh250320146} and Remark \ref{5hh270320141} we have
         \begin{align*}
         [\mu_{n,i}+\sigma_{i,n}\otimes\delta_{\{t=0\}}]_{\mathfrak{M}^{\mathcal{G}_2,q'}}\lesssim [|\mu|+|\sigma|\otimes\delta_{\{t=0\}}]_{\mathfrak{M}^{\mathcal{G}_2,q'}}~~\forall~n\in\mathbb{N}, i=1,2.
         \end{align*} 
         In addition, by the proof of Corollary \ref{5hh250320146} we get
$$
         \text{Cap}_{E^{2T_0,\delta}_2,q'}(E)\lesssim_{T_0} \text{Cap}_{\mathcal{G}_2,q'}(E)$$  for every compact set $E$ with $\text{diam}(E)\leq 2T_0$.                       
    Thus, there is $\varepsilon_1=\varepsilon_1(N,K,\delta,q,T_0)$ such that if \begin{equation}\label{5hh010420142}
    [|\mu|+|\sigma|\otimes\delta_{\{t=0\}}]_{\mathfrak{M}^{\mathcal{G}_2,q'}}\leq \varepsilon_1,
    \end{equation} then \eqref{5hh020420141} holds for any $n_0\in \mathbb{N}$.\medskip\\
    Now we suppose that  \eqref{5hh010420142} holds, it is equivalent to  \eqref{5hh020420143}, by Remark \ref{5hh130320146}. Therefore, for any $n\in\mathbb{N}$ there exists a renormalized solution $u_n$ of
   \begin{equation}
                            \left\{
                            \begin{array}
                            [c]{l}%
                            {(u_n)_{t}}-\operatorname{div}(A(x,t,\nabla u_n))=|u_n|^{q-1}u_n+\mu_{n,1}-\mu_{n,2}~\text{in }\Omega_T,\\ 
                                                u_n=0~~~~~~~\text{on}~~
                                             \partial\Omega \times (0,T),
                                                \\
                                                u_n(0) = \sigma_{1,n}-\sigma_{2,n}~~~\text{in}~~ \Omega, \\ 
                            \end{array}
                            \right.  
                            \end{equation}
                            which satisfies
                          \begin{align*}
                        -\frac{q K}{q-1}\mathbb{I}_2^{2T_0,\delta}[\mu_{n,2}+\sigma_{2,n}\otimes\delta_{\{t=0\}}]\leq u_n\leq \frac{q K}{q-1}\mathbb{I}_2^{2T_0,\delta}[\mu_{n,1}+\sigma_{1,n}\otimes\delta_{\{t=0\}}].
                          \end{align*}  
 Thus, for every $(x,t)\in \Omega_T$,
 \begin{align*}
       &-\frac{q K}{q-1}\varphi_n*\mathbb{I}_2^{2T_0,\delta}[\mu^-](x,t)-\frac{q K}{q-1}\varphi_{1,n}*(\mathbb{I}_2^{2T_0,\delta}[\sigma^-\otimes\delta_{\{t=0\}}](.,t))(x)\leq u_n(x,t)\\&~~~~~~~~~~~~~~~~~\leq\frac{q K}{q-1}\varphi_n*(\mathbb{I}_2^{2T_0,\delta}[\mu^-])(x,t)+\frac{q K}{q-1}\varphi_{1,n}*(\mathbb{I}_2^{2T_0,\delta}[\sigma^-\otimes\delta_{\{t=0\}}](.,t))(x).
                           \end{align*}               Since    $\varphi_n*\mathbb{I}_2^{2T_0,\delta}[\mu^\pm](x,t), \varphi_{1,n}*(\mathbb{I}_2^{2T_0,\delta}[\sigma^\pm\otimes\delta_{\{t=0\}}](.,t))(x)$  converge to  $\mathbb{I}_2^{2T_0,\delta}[\mu^\pm](x,t), \mathbb{I}_2^{2T_0,\delta}[\sigma^\pm\otimes\delta_{\{t=0\}}](x,t)$ in $L^q(\mathbb{R}^{N+1})$ as $n\to\infty$, respectively, so $|u_n|^q$ is equi-integrable. \\            
                                    By Proposition \ref{5hhmun}, there exists a subsequence of $\{u_n\} $, still denoted by its, converging to $u$ a.e in $\Omega_T$. It follows $|u_n|^{q-1}u_n\to|u|^{q-1}u$ in $L^1(\Omega_T)$. \\
                                    Consequently, by Proposition \ref{5hhmun} and Theorem \ref{5hhsta}, we obtain that     $u$ is
a distributional solution (a renormalized solution if $\sigma\in L^1(\Omega)$) of \eqref{5hh070120149} with data $\mu$, $\sigma$, and satisfies \eqref{5hh020420142}. 
Furthermore, by Corollary \ref{5hh250320146} we have 
$$
  \left[\left(\mathbb{I}_{2 }^{2T_0,\delta}[|\mu|+|\sigma|\otimes\delta_{\{t=0\}}]\right)^{q}\right]_{\mathfrak{M}^{\mathcal{G}_2,q'}}\sim_{T_0} \left[|\mu|+|\sigma|\otimes\delta_{\{t=0\}}\right]^{q}_{\mathfrak{M}^{\mathcal{G}_2,q'}}$$
which implies
$
  \left[|u|^q\right]_{\mathfrak{M}^{\mathcal{G}_2,q'}}\lesssim_{T_0}1$
and we get \eqref{5hh270420141}. 
The proof  is complete.
\end{proof}
\begin{remark}\label{5hh020420144} In view of above proof, we can see that
\begin{description}
\item[i.] The Theorem \ref{5hh070120147} also holds when we replace assumption \eqref{5hh020420143} by  \begin{equation*}
       |\mu|(E)\leq \varepsilon_0 \text{Cap}_{\mathcal{H}_2,q'}(E)~~\text{and}~~|\sigma|(F)\leq \varepsilon_0\text{Cap}_{\mathbf{I}_{\frac{2}{q}},q'}(F).
       \end{equation*} 
    for every compact sets $E\subset \mathbb{R}^{N+1}, F\subset \mathbb{R}^{N}$ and  $\varepsilon_0>0$ small enough.
\item[ii.] If $\sigma\equiv0$ and $\text{supp}(\mu)\subset\overline{\Omega}\times [a,T]$, $a>0$, then we can show that the solution $u$ in Theorem \ref{5hh070120147} satisfies $u=0$ in $\Omega\times(0,a)$ since we can replace the set $E$ by $E'$:
\begin{align*}
       E'&=\left\{u\in L^q(\Omega_T):u=0~\text{in}~\Omega\times(0,a)~\text{and}~u^+\leq\frac{q K}{q-1}\mathbb{I}_2^{2T_0,\delta}[\mu_{n_0,1}+\sigma_{1,n_0}\otimes\delta_{\{t=0\}}], \right.
       \\&~~~~~~~~~~~~~~~~~~~\left.u^-\leq\frac{q K}{q-1}\mathbb{I}_2^{2T_0,\delta}[\mu_{n_0,2}+\sigma_{2,n_0}\otimes\delta_{\{t=0\}}]~\right\}.
       \end{align*} 
\end{description} 
\end{remark}
\subsection{Quasilinear Lane-Emden Parabolic Equations in $\mathbb{R}^{N}\times (0,\infty)$ and $\mathbb{R}^{N+1}$}
This section is devoted to prove  Theorem \ref{5hh1303201411} and Theorem \ref{5hh130320142}.\\
\begin{proof}[Proof of the Theorem \ref{5hh1303201411}] Since $\omega$ is absolutely continuous with respect to the capacity $\text{Cap}_{2,1,q'}$ in $\mathbb{R}^{N+1}$, so does $|\omega|$. Set $D_n=B_n(0)\times(-n^2,n^2)$.  From the proof of Theorem \ref{5hh070120146}, there exist  renormalized solutions $u_n,v_{n}$ of 
  \[
  \left\{
  \begin{array}
  [c]{l}%
  (u_{n})_{t}-\operatorname{div}(A(x,t,\nabla u_{n}))+|u_n|^{q-1}u_n=\chi_{D_n}\omega~~~\text{in
  }~D_n,\\
  u_{n}=0\qquad\text{on }\partial B_{n}(0)\times (-n^2,n^2),\\
  u_{n}(-n^2)=0 ~~~~\text{ in } ~B_n(0),
  \end{array}
  \right.
  \] 
  and
  \[
    \left\{
    \begin{array}
    [c]{l}%
    (v_{n})_{t}-\operatorname{div}(A(x,t,\nabla v_{n}))+v_n^q=\chi_{D_n}|\omega|~~~\text{in
    }~D_n,\\
    v_{n}=0\qquad\text{on }\partial B_{n}(0)\times (-n^2,n^2),\\
    v_{n}(-n^2)=0 ~~~~\text{ in } ~B_n(0),
    \end{array}
    \right.
    \] 
  relative to  decompositions $(f_n,g_n,h_n)$ of $\chi_{D_n}\omega_0$ and  $(\overline{f}_n,\overline{g}_n,\overline{h}_n)$ of $\chi_{B_{n}(0)\times (0,n^2)}|\omega_0|$, satisfied
  \eqref{5hh110320141}, \eqref{5hh110320147} in Proposition \ref{5hh1203201411} with $1<q_0<q$, $L(u_n)=|u_n|^{q-1}u_n$, $L(v_n)=v_n^q$ and $\mu$ is replaced by $\chi_{D_n}\omega$ and $\chi_{D_n}|\omega|$ respectively. Moreover, there hold
  \begin{align}
  -KI_2[\omega^-]\leq u_n\leq KI_2[\omega^+],~0\leq v_n\leq KI_2[|\omega|] ~\text{ in }~ D_n,
  \end{align}
  and $v_{n+1}\geq v_{n}$, $|u_n|\leq v_n$ in $D_n$.  \\   
 By Remark \ref{5hh1203201410}, we can assume that 
 \begin{align*}
 &||f_n||_{L^1(D_i)}+||g_n||_{L^2(D_i,\mathbb{R}^N)}+|||h_n|+|\nabla h_n|||_{L^2(D_i)}\leq 2|\omega|(D_{i+1}),\\&||\overline{f}_n||_{L^1(D_i)}+||\overline{g}_n||_{L^2(D_i,\mathbb{R}^N)}+|||\overline{h}_n|+|\nabla \overline{h}_n|||_{L^2(D_i)}\leq 2|\omega|(D_{i+1}),
 \end{align*}
  for any $i=1,...,n-1$ and $h_n,\overline{h}_n$ are convergent in $L^1_{\text{loc}}(\mathbb{R}^{N+1})$. 
  On the other hand, since $u_n,v_n$ satisfy \eqref{5hh110320141} in Proposition \ref{5hh1203201411} with $1<q_0<q$, $L(u_n)=|u_n|^{q-1}u_n$, $L(v_n)=v_n^q$ and thanks to H\"older inequality: for any $\varepsilon\in (0,1)$
$$
 \left(|u_n|+1\right)^{q_0}\leq \varepsilon |u_n|^{q}+c(\varepsilon), \left(|v_n|+1\right)^{q_0}\leq \varepsilon |v_n|^{q}+c(\varepsilon),
$$
we get
  \begin{align}\int_{D_i}|u_n|^qdxdt+ \int_{D_i}|u_n|^{q_0}dxdt+ \int_{D_i}v_n^qdxdt+ \int_{D_i}v_n^{q_0}dxdt \leq C(i)+c|\omega|(D_{i+1}).\label{5hh1303201414}
  \end{align}
  for $i=1,...,n-1$, where the constant $C(i)$ depends on $N,\Lambda_1,\Lambda_2,q_0,q$ and $i$.\\ 
 Consequently,  we can apply Proposition \ref{5hh1203201412} with $\mu_n=-|u_n|^{q-1}u_n+\chi_{D_n}\omega, -v_n^q+\chi_{D_n}|\omega|$ and obtain that there are subsequences of $u_n,v_n$, still denoted by them, converging to some $u,v$ in $L^1_{loc}(\mathbb{R};W^{1,1}_{loc}(\mathbb{R}^N))$ resp. So, $\frac{|\nabla u|^2}{(|u|+1)^{\alpha+1}}\in L^1_{loc}(\mathbb{R}^{N+1})$ for all $\alpha>0$ and $u\in L^q_{loc}(\mathbb{R}^{N+1})$ satisfies \eqref{5hh020420145}. In addition, using H\"older inequality we get $u\in L^\gamma_{\text{loc}
               }(\mathbb{R};W^{1,\gamma}_{\text{loc}}(\mathbb{R}^N))$ for any $1\leq \gamma<\frac{2q}{q+1}$. \medskip\\ 
Thanks to \eqref{5hh1303201414} and the Monotone Convergence Theorem we get $v_n\to v$ in $L^q_{loc}(\mathbb{R}^{N+1})$. After, we also have $u_n\to u$ in $L^q_{loc}(\mathbb{R}^{N+1})$ by  $|u_n|\leq v_n$ and the Dominated Convergence Theorem. \\
Consequently, $u$ is a distributional solution of problem \eqref{5hhparabolic3'} which satisfies \eqref{5hh020420145}. \medskip\\
 If $\omega=\mu+\sigma\otimes\delta_{\{t=0\}}$ with $\mu\in\mathfrak{M}(\mathbb{R}^N\times (0,\infty))$ and $\sigma\in \mathfrak{M}(\mathbb{R}^N)$, then by  the proof of Theorem \ref{5hh070120146} we obtain that $u_n=0$ in $B_n(0)\times (-n^2,0)$. So, $u=0$ in $\mathbb{R}^N\times(-\infty,0)$. Therefore, clearly ${\left. u \right|_{\mathbb{R}^N\times [0,\infty)}}$ is a distributional solution to \eqref{5hhparabolic3}.
              This completes the proof.
              \end{proof}\\\\            
\begin{proof}[Proof of the Theorem \ref{5hh130320142}]By the proof of Theorem  \ref{5hh070120147} and Remark \ref{5hh020420144}, \ref{5hh130320146}, there exists a constant $\varepsilon_0=\varepsilon_0(N,q,\Lambda_1,\Lambda_2)$ such that 
if $\omega$ satisfies 
 for every compact set $E\subset \mathbb{R}^{N+1}$, 
              \begin{equation}\label{5hh130320149}
              |\omega|(E)\leq \varepsilon_0 \text{Cap}_{\mathcal{H}_2,q'}(E),
              \end{equation}
  then there is a renormalized solution $u_n$ of 
  \[
  \left\{
  \begin{array}
  [c]{l}%
  (u_{n})_{t}-\operatorname{div}(A(x,t,\nabla u_{n}))=|u_n|^{q-1}u_n+\chi_{D_n}\omega~\text{in
 }~D_n\\
  u_{n}=0\qquad\text{on }\partial B_{n}(0)\times (-n^2,n^2),\\
  u_{n}(-n^2)=0~~\text{ in } ~B_n(0),
  \end{array}
  \right.
  \] 
  relative to  a decomposition $(f_n,g_n,h_n)$ of $\chi_{D_n}\omega_0$, satisfying
  \eqref{5hh110320141}, \eqref{5hh110320147} in Proposition \ref{5hh1203201411} with $q_0=q$, $L\equiv 0$ and $\mu$ is replaced by $|u_{n}|^{q-1}u_n+\chi_{D_n}\omega$ and 
  \begin{align}
  -\frac{q K}{q-1}\mathbb{I}_{2}[\omega^-](x,t)\leq u_n\leq \frac{q K}{q-1}\mathbb{I}_{2}[\omega^+](x,t)\label{5hh020420146}
  \end{align} 
   for a.e $(x,t)$ in  $D_n$  and 
 $I_2[\omega^\pm]\in L^q_{loc}(\mathbb{R}^{N+1})$.\\ 
 Besides, thanks to Remark \ref{5hh1203201410}, we can assume that $f_n,g_n,h_n$ satisfies \eqref{5hh130320148} in proof of Theorem \eqref{5hh1203201417} and $h_n$ is convergent in $L^1_{\text{loc}}(\mathbb{R}^{N+1})$.\medskip\\
 Consequently,  we can apply Proposition \ref{5hh1203201412} and obtain that there exist a subsequence of $u_n$, still denoted by it, converging to some $u$ a.e in $\mathbb{R}^{N+1}$ and in $L^1_{loc}(\mathbb{R};W^{1,1}_{loc}(\mathbb{R}^N))$. Also, $u_n\to u$ in  $L^q_{\text{loc}}(\mathbb{R}^{N+1})$ by Dominated Convergence Theorem,  $\frac{|\nabla u|^2}{(|u|+1)^{\alpha+1}}\in L^1_{loc}(\mathbb{R}^{N+1})$ for all $\alpha>0$. Using H\"older inequality we get $u\in L^\gamma_{\text{loc}
               }(\mathbb{R};W^{1,\gamma}_{\text{loc}}(\mathbb{R}^N))$ for any $1\leq \gamma<\frac{2q}{q+1}$.\medskip\\
 Thus we obtain that $u$ is a distributional solution of \eqref{5hhparabolic4'} which satisfies \eqref{5hh130320144}. 
Since \eqref{5hh130320149} holds, thus by Theorem \ref{5hh1410136} we get
$$\left[\left(\mathbb{I}_{2 }[|\omega|]\right)^{q}\right]_{\mathfrak{M}^{\mathcal{H}_2,q'}}\sim  \left[|\omega|\right]^{q}_{\mathfrak{M}^{\mathcal{H}_2,q'}},$$
so we have 
$
  \left[|u|^{q}\right]_{\mathfrak{M}^{\mathcal{H}_2,q'}}\lesssim 1.$
              It follows \eqref{5hh280420141}. \medskip\\
               If $\omega=\mu+\sigma\otimes\delta_{\{t=0\}}$ with $\mu\in\mathfrak{M}(\mathbb{R}^N\times (0,\infty))$ and $\sigma\in \mathfrak{M}(\mathbb{R}^N)$, then by  Remark \ref{5hh020420144} we obtain that $u_n=0$ in $B_n(0)\times (-n^2,0)$. So, $u=0$ in $\mathbb{R}^N\times(-\infty,0)$. Therefore, clearly ${\left. u \right|_{\mathbb{R}^N\times [0,\infty)}}$ is a distributional solution to \eqref{5hhparabolic4}.
                            The proof is complete.                       
\end{proof}
    \section{Interior Estimates and Boundary Estimates for Parabolic Equations}

   In this section we always assume that  $u\in C(-T,T,L^2(\Omega))\cap L^2(-T,T,H^1_0(\Omega))$ is a solution to equation \eqref{5hhparabolic1} in $\Omega\times (-T,T)$ with $\mu\in L^2(\Omega\times (-T,T))$ and $u(-T)=0$. We extend $u$ by zero to $\Omega\times(-\infty,-T)$, clearly $u$ is a solution to equation
   \begin{equation}\label{5hheq6}
         \left\{ \begin{array}{l}
           {u_t} - \operatorname{div}\left( {A(x,t,\nabla u)} \right) = \chi_{(-T,T)}(t)\mu ~~\text{ in }~\Omega\times(-\infty,T), \\ 
           u = 0~~~\text{ on }~~ \partial\Omega\times (-\infty,T).\\ 
           \end{array} \right.
         \end{equation} 
      \subsection{Interior Estimates}
       For each ball $B_{2R}=B_{2R}(x_0)\subset\subset\Omega$ and $t_0\in (-T,T)$, one considers the unique solution 
      \begin{equation}
      w\in C(t_0-4R^2,t_0;L^2(B_{2R}))\cap L^2(t_0-4R^2,t_0;H^1(B_{2R}))
      \end{equation}
      to the following equation 
      \begin{equation}
       \label{5hheq3}\left\{ \begin{array}{l}
         {w_t} - \operatorname{div}\left( {A(x,t,\nabla w)} \right) = 0 \;in\;Q_{2R}, \\ 
         w = u\quad \quad on~~\partial_{p}Q_{2R}, \\ 
         \end{array} \right.
       \end{equation}
       where $Q_{2R}=B_{2R}  \times (t_0-4R^2,t_0)$ and $\partial_{p}Q_{2R}= \left( {\partial B_{2R}  \times (t_0-4R^2,t_0)} \right) \cup \left( {B_{2R}  \times \left\{ {t = t_0-4R^2} \right\}} \right) $.
       \begin{theorem} \label{5hh1510135}
      There exist  constants $\theta_1>2$,  $\beta_1\in (0,\frac{1}{2}]$ such that the following estimates are true 
       \begin{equation}\label{5hhineq1}
       \fint_{Q_{2R}}|\nabla u-\nabla w|dxdt \lesssim \frac{|\mu|(Q_{2R})}{R^{N+1}},
       \end{equation}
       \begin{equation}\label{5hhineq2}
       \left(\fint_{Q_{\rho/2}(y,s)}|\nabla w|^{\theta_1} dxdt\right)^{\frac{1}{\theta_1}}\lesssim\fint_{Q_{\rho}(y,s)}|\nabla w| dxdt,
       \end{equation}
   
       \begin{equation}\label{5hhineq3}
       \left(\fint_{Q_{\rho_1}(y,s)}|w-\overline{w}_{Q_{\rho_1}(y,s)}|^2dxdt\right)^{1/2}\lesssim \left(\frac{\rho_1}{\rho_2}\right)^{\beta_1}\left(\fint_{Q_{\rho_2}(y,s)}|w-\overline{w}_{Q_{\rho_2}(y,s)}|^2dxdt\right)^{1/2},
       \end{equation} 
       \begin{equation}\label{5hhineq3'}
              \left(\fint_{Q_{\rho_1}(y,s)}|\nabla w|^2dxdt\right)^{1/2}\lesssim \left(\frac{\rho_1}{\rho_2}\right)^{\beta_1-1}\left(\fint_{Q_{\rho_2}(y,s)}|\nabla w|^2dxdt\right)^{1/2}
              \end{equation}
               for any  $Q_{\rho}(y,s)\subset Q_{2R}$, and $Q_{\rho_1}(y,s)\subset Q_{\rho_2}(y,s)\subset Q_{2R}$. 
       \end{theorem}
       \begin{proof} Inequalities \eqref{5hhineq1}, \eqref{5hhineq2} and \eqref{5hhineq3} were proved by Duzaar and  Mingione in 
      \cite{55DuzaMing}. So, it remains to prove \eqref{5hhineq3'} in case $\rho_1\leq \rho_2/2$. 
      By the interior Caccioppoli inequality we have 
$$
       \left(\fint_{Q_{\rho_1}(y,s)}|\nabla w|^2dxdt\right)^{1/2}\lesssim \frac{1}{\rho_1} \left(\fint_{Q_{2\rho_1}(y,s)}|w-\overline{w}_{Q_{2\rho_1}(y,s)}|^2dxdt\right)^{1/2}.
$$
      On the other hand, by a Sobolev inequality there holds 
      \begin{equation*}
      \left(\fint_{Q_{\rho_2}(y,s)}|w-\overline{w}_{Q_{\rho_2}(y,s)}|^2dxdt\right)^{1/2}\lesssim \rho_2 \left(\fint_{Q_{\rho_2}(y,s)}|\nabla w|^2dxdt\right)^{1/2}.
      \end{equation*}
      Therefore, \eqref{5hhineq3'} follows from \eqref{5hhineq3}. 
       \end{proof}
       \begin{corollary}\label{5hh060120148} Let $\beta_1$ be  the constant in Theorem \ref{5hh1510135} and  $2-\beta_1<\theta<N+2$. There holds for any $B_{\rho}(y)\subset B_{\rho_0}(y)\subset\subset \Omega$,  $s\in (-T,T)$
       \begin{equation}\label{5hh060120141}
       \int_{Q_\rho(y,s)}|\nabla u|dxdt \lesssim_{\theta} \rho^{N+3-\theta}\left(\left(\frac{T_0}{\rho_0}\right)^{N+3-\theta}+1\right)||\mathbb{M}_{\theta}[\mu]||_{L^\infty(\Omega\times (-T,T))}.
       \end{equation}
       \end{corollary}
       \begin{proof} Take $B_{\rho_2}(y)\subset\subset \Omega$ and $s\in (-T,T)$. Set $Q_{\rho}:=Q_{\rho_1}(y,s)$.
       For any $Q_{\rho_1}\subset Q_{\rho_2}$ with $\rho_1\leq\rho_2/2$, we take $w$ as in Theorem \ref{5hh1510135} with $Q_{2R}=Q_{\rho_2}(y,s)$. Thus,
      $$\int_{Q_{\rho_1}}|\nabla w|dxdt\lesssim \left(\frac{\rho_1}{\rho_2}\right)^{N+\beta_1+1}\int_{Q_{\rho_2}}|\nabla w|dxdt,\quad \int_{Q_{\rho_2}}|\nabla u-\nabla w|dxdt\lesssim\rho_2 |\mu|(Q_{\rho_2}).$$
         It follows that
         \begin{align*}
         \int_{Q_{\rho_1}}|\nabla u|dxdt&\leq \int_{Q_{\rho_1}}|\nabla w |dxdt + \int_{Q_{\rho_1}}|\nabla u-\nabla w |dxdt\\&\lesssim \left(\frac{\rho_1}{\rho_2}\right)^{N+\beta_1+1}\int_{Q_{\rho_2}}|\nabla w|dxdt+ \int_{Q_{\rho_2}}|\nabla u-\nabla w |dxdt    
                \\&\lesssim \left(\frac{\rho_1}{\rho_2}\right)^{N+\beta_1+1}\int_{Q_{\rho_2}}|\nabla u|dxdt+ \rho_2 |\mu|(Q_{\rho_2}).
         \end{align*}     
       This implies 
$$
       \int_{Q_{\rho_1}}|\nabla u|dxdt\lesssim\left(\frac{\rho_1}{\rho_2}\right)^{N+\beta_1+1}\int_{Q_{\rho_2}}|\nabla u|dxdt+ \rho_2^{N+3-\theta}||\mathbb{M}_{\theta}[\mu]||_{L^\infty(\Omega\times(-T,T))}.$$
       Since $N+3-\beta<N+\beta_1+1$, applying  \cite[Lemma 4.6, page 54]{55Li3} we obtain
       $$
              \int_{Q_{\rho}}|\nabla u|dxdt\lesssim\left(\frac{\rho}{\rho_0}\right)^{N+3-\theta}||\nabla u||_{L^1(\Omega_\times(-T,T))}+ \rho^{N+3-\theta}||\mathbb{M}_{\theta}[\mu]||_{L^\infty(\Omega\times(-T,T))},$$
              for any $B_{\rho}(y)\subset B_{\rho_0}(y)\subset\subset \Omega$,  $s\in (-T,T)$.
       On the other hand, by Remark \ref{5hh070420143}
$$
       ||\nabla u||_{L^1(\Omega\times (-T,T))}\lesssim T_0|\mu|(\Omega\times (-T,T))\lesssim T_0^{N+3-\theta}||\mathbb{M}_{\theta}[\mu]||_{L^\infty(\Omega\times (-T,T))}.$$
        Hence, we get the desired result. 
       \end{proof}\\

      To continue, we consider the unique solution
       \begin{equation}
        v\in C(t_0-R^2,t_0;L^2(B_{R}))\cap L^2(t_0-R^2,t_0;H^1(B_{R}))
        \end{equation}
        to the following equation 
         \begin{equation}\label{5hheq4}
         \left\{ \begin{array}{l}
              {v_t} - \operatorname{div}\left( {\overline{A}_{B_R(x_0)}(t,\nabla v)} \right) = 0 \;in\;Q_{R}, \\ 
              v = w\quad \quad on~~\partial_{p}Q_{R}, \\ 
              \end{array} \right.
         \end{equation}
          where $Q_{R}=B_{R}(x_0)  \times (t_0-R^2,t_0)$ and $\partial_{p}Q_{R}= \left( {\partial B_{R}  \times (t_0-R^2,t_0)} \right) \cup \left( {B_{R}  \times \left\{ {t = t_0-R^2} \right\}} \right) $.
       \begin{lemma}\label{5hh21101319} Let $\theta_1$ be  the constant in Theorem \ref{5hh1510135}. There holds  \begin{eqnarray}
               \left(\fint_{Q_R}|\nabla w-\nabla v|^2dxdt\right)^{1/2}\lesssim [A]_{s_1}^{R} \fint_{Q_{2R}}|\nabla w|dxdt, \label{5hh18094}
               \end{eqnarray}
       with $s_1=\frac{2\theta_1}{\theta_1-2}$ and \begin{equation}\label{5hh18091}
             \int_{Q_R}|\nabla w|^2dxdt\sim\int_{Q_R}|\nabla v|^2dxdt.
                \end{equation}
       \end{lemma}
       \begin{proof} We can choose $\varphi=w-v$ as a test function for equations \eqref{5hheq3}, \eqref{5hheq4} and since \begin{eqnarray*}
             \int_{Q_R}w_t(w-v)dxdt- \int_{Q_R}v_t(w-v)dxdt =\frac{1}{2}\int_{B_R}(w-v)^2(t_0)dx\geq  0,
             \end{eqnarray*}
              we find       
      \begin{equation*}
      -\int_{Q_R}\overline{A}_{B_R(x_0)}(t,\nabla v)\nabla (w-v)dxdt\leq -\int_{Q_R}A(x,t,\nabla w)\nabla (w-v)dxdt.
      \end{equation*}
      By using inequalities \eqref{5hhconda} and \eqref{5hhcondb} together with H\"older's inequality we get 
$$
      \int_{Q_R}|\nabla w|^2dxdt\sim\int_{Q_R}|\nabla v|^2dxdt,$$
       and  we also have
       \begin{align*}
       \Lambda_2\int_{Q_R}|\nabla w-\nabla v|^2dxdt&\leq \int_{Q_R}\left(\overline{A}_{B_R(x_0)}(t,\nabla w)-\overline{A}_{B_R(x_0)}(t,\nabla v)\right)\left(\nabla w-\nabla v\right)dxdt\\&\leq \int_{Q_R}\left(\overline{A}_{B_R(x_0)}(t,\nabla w)-A(x,t,\nabla w)\right)\left(\nabla w-\nabla v\right)dxdt\\&\leq \int_{Q_R}\Theta(A,B_R(x_0))(x,t)|\nabla w||\nabla w-\nabla v|dxdt. 
       \end{align*}    
      Here we used the definition of $\Theta(A,B_R(x_0))$ in the last inequality. Using H\"older's inequality with exponents $s_1=\frac{2\theta_1}{\theta_1-2}, \theta_1$ and $2$ one gets
      \begin{align*}
       \Lambda_2\fint_{Q_R}|\nabla w-\nabla v|^2&\leq \left(\fint_{Q_R}\Theta(A,B_R(x_0))(x,t)^{s_1}dxdt\right)^{1/s_1} \left(\fint_{Q_R}|\nabla w|^{\theta_1} dxdt\right)^{1/\theta_1} \\&~~~~~~\times\left(\fint_{Q_R}|\nabla w-\nabla v|^2 dxdt\right)^{1/2}.
      \end{align*}    
       In other words, 
$$
         \left(\fint_{Q_R}|\nabla w-\nabla v|^2dxdt\right)^{1/2}\lesssim [A]_{s_1}^{R}\left(\fint_{Q_R}|\nabla w|^{\theta_1} dxdt\right)^{1/\theta_1}.  
$$
         After using the inequality \eqref{5hhineq2} in Theorem \ref{5hh1510135} we get \eqref{5hh18094}.
       \end{proof}
       \begin{lemma}\label{5hh24092}Let $\theta_1$ be  the constant in Theorem \ref{5hh1510135}. 
       There exists a function $v\in C(t_0-R^2,t_0;L^2(B_{R}))\cap L^2(t_0-R^2,t_0;H^1(B_{R}))\cap L^\infty(t_0-\frac{1}{4}R^2,t_0;W^{1,\infty}(B_{R/2}))$ such that 
       \begin{equation}\label{5hh18092}
       ||\nabla v||_{L^\infty(Q_{R/2})}\lesssim \fint_{Q_{2R}}|\nabla u| dxdt +\frac{|\mu|(Q_{2R})}{R^{N+1}},
       \end{equation}
         \begin{equation}
               \fint_{Q_R}|\nabla u-\nabla v|dxdt
               \lesssim\frac{|\mu|(Q_{2R})}{R^{N+1}}+ [A]_{s_1}^{R}\left(\fint_{Q_{2R}}|\nabla u|dxdt+ \frac{|\mu|(Q_{2R})}{R^{N+1}}\right),\label{5hh18093}
               \end{equation}
               where $s_1=\frac{2\theta_1}{\theta_1-2}$.
       \end{lemma}
       \begin{proof}
       Let $w$ and $v$ be in equations \eqref{5hheq3} and \eqref{5hheq4}. By standard interior regularity and inequality \eqref{5hhineq2} in Theorem \ref{5hh1510135} and \eqref{5hh18091} in Lemma \ref{5hh21101319} we have 
$$
      ||\nabla v||_{L^\infty(Q_{R/2})}\lesssim \left(\fint_{Q_R}|\nabla v|^2dxdt\right)^{1/2}\lesssim \left(\fint_{Q_R}|\nabla w|^2dxdt\right)^{1/2}\lesssim \fint_{Q_{2R}}|\nabla w|dxdt. $$
      Combining this with \eqref{5hhineq1}, we get \eqref{5hh18092}.
        On the other hand, \eqref{5hh18094} in Lemma \ref{5hh21101319} and H\"older's inequality yield  
$$
                  \fint_{Q_R}|\nabla w-\nabla v|dxdt\lesssim [A]_{s_1}^{R} \fint_{Q_{2R}}|\nabla w| dxdt. 
$$
     It leads
$$
                        \fint_{Q_R}|\nabla u-\nabla v|dxdt\lesssim \fint_{Q_R}|\nabla u-\nabla w|dxdt+  [A]_{s_1}^{R}\fint_{Q_{2R}}|\nabla w| dxdt. $$
                Consequently, we get \eqref{5hh18093}
                            from this and  \eqref{5hhineq1}  in Theorem \ref{5hh1510135}. The proof is complete. 
       \end{proof}  
       \subsection{Boundary Estimates}
       In this subsection, we focus on the corresponding estimates near the boundary. \\
    Let $x_0\in \partial\Omega$ be a boundary point and for $R>0$ and $t_0\in (-T,T)$. We set $\tilde{\Omega}_{6R}=\tilde{\Omega}_{6R}(x_0,t_0)=\left(\Omega\cap B_{6R}(x_0)\right)\times (t_0-(6R)^2,t_0)$ and $Q_{6R}=Q_{6R}(x_0,t_0)$.\\ We consider the unique solution $w$ to the equation
       \begin{equation}
         \label{5hheq5}\left\{ \begin{array}{l}
           {w_t} - \operatorname{div}\left( {A(x,t,\nabla w)} \right) = 0 \;in\;\tilde{\Omega}_{6R}, \\ 
           w = u\quad \quad on~~\partial_{p}\tilde{\Omega}_{6R}. \\ 
           \end{array} \right.
         \end{equation}
         In what follows we extend $\mu$ and $u$ by zero to $\left(\Omega\times (-\infty,T)\right)^c$ and then extend $w$ by $u$ to $\mathbb{R}^{N+1}\backslash \tilde{\Omega}_{6R}$. \\\\
         In order to obtain  estimates for $w$ as in Theorem \ref{5hh1510135} we require the domain $\Omega$ to be satisfied $2-$Capacity uniform thickness condition. 
        \subsubsection{2-Capacity uniform thickness domain}
            
   It is well known that if $\mathbb{R}^N\backslash\Omega$ satisfies  a uniformly $2-$thick condition with constants $c_0,r_0>0$, there exist $p_0\in (\frac{2N}{N+2},2)$ and $C=C(N,c_0)>0$ such that 
   \begin{equation}\label{5hh090120141}
   \text{Cap}_{p_0}(\overline{B_r(x)}\cap(\mathbb{R}^N\backslash \Omega), B_{2r}(x))\geq Cr^{N-p_0},
   \end{equation} 
   for all $0<r\leq r_0$ and all $x\in \mathbb{R}^N\backslash\Omega$, see \cite{55Le,55Mik}.
  \begin{theorem}\label{5hh24093}
  Suppose that $\mathbb{R}^N\backslash\Omega$ is uniformly $2-$thick with constants $c_0,r_0$. Let $w$ be in \eqref{5hheq5} with $0<6R\leq r_0$. There exist  constants $\theta_2>2$,  $\beta_2\in (0,\frac{1}{2}]$  such that
  \begin{align}
  	\label{5hhineq9}
  	&\fint_{Q_{6R}}|\nabla u-\nabla w|dxdt \lesssim \frac{|\mu|(\tilde{\Omega}_{6R})}{R^{N+1}},\\&
  	\label{5hhineq8}
  	\left(\fint_{Q_{\rho/2}(z,s)}|\nabla w|^{\theta_2} dxdt\right)^{\frac{1}{\theta_2}}\lesssim\fint_{Q_{3\rho}(z,s)}|\nabla w| dxdt,\\&
  	\label{5hhineq10}
  	\left(\fint_{Q_{\rho_1}(y,s)}|w|^2dxdt\right)^{1/2}\lesssim \left(\frac{\rho_1}{\rho_2}\right)^{\beta_2}\left(\fint_{Q_{\rho_2}(y,s)}|w|^2dxdt\right)^{1/2},\\&
  	\label{5hhineq10'}
  	\left(\fint_{Q_{\rho_1}(z,s)}|\nabla w|^2dxdt\right)^{1/2}\lesssim \left(\frac{\rho_1}{\rho_2}\right)^{\beta_2-1}\left(\fint_{Q_{\rho_2}(z,s)}|\nabla w|^2dxdt\right)^{1/2},
  \end{align}
   for any $Q_{3\rho}(z,s)\subset Q_{6R}$, $y\in \partial \Omega$,  $Q_{\rho_1}(y,s)\subset Q_{\rho_2}(y,s)\subset Q_{6R}$ and $Q_{\rho_1}(z,s)\subset Q_{\rho_2}(z,s)\subset Q_{6R}$.
  \end{theorem}
  \begin{proof}1. For $\eta\in C^\infty_c ([t_0-(6R)^2,t_0))$ ,  $0\leq \eta\leq 1$, $\eta_t\leq 0$ and $\eta(t_0-(6R)^2)=1$. Using $\varphi=T_k(u-w)\eta$, for any $k>0$, as a test function for \eqref{5hheq6} and \eqref{5hheq5}, we get
 \begin{align*}
 & \int_{\tilde{\Omega}_{6R}}(u-w)_tT_k(u-w)\eta dxdt\\&~~~~~~~~~~~~+\int_{\tilde{\Omega}_{6R}}\left(A(x,t,\nabla u)-A(x,t,\nabla w)\right)\nabla T_k(u-w)\eta dxdt=\int_{\tilde{\Omega}_{6R}}T_k(u-w)\eta d\mu.
 \end{align*}                             Thanks to  \eqref{5hhcondb}, we obtain
            \begin{equation*}
                       -\int_{\tilde{\Omega}_{6R}}\overline{T}_k(u-w)\eta_t dxdt+\Lambda_2 \int_{\tilde{\Omega}_{6R}}|\nabla T_k(u-w)|^2\eta dxdt\leq k|\mu|(\tilde{\Omega}_{6R}),
                       \end{equation*}
                       where $\overline{T}_k(s)=\int_{0}^{s}T_k(\tau)d\tau$. 
 As in  \cite[Proposition 2.8]{55VH}, we also verify that
$$
 |||\nabla (u-w)|||_{L^{\frac{N+2}{N+1},\infty}(\tilde{\Omega}_{6R})}\lesssim|\mu|(\tilde{\Omega}_{6R}).$$
  Hence we get \eqref{5hhineq9}.\\ 
 2. We need to prove that 
 \begin{eqnarray}
                  \fint_{Q_{r/4}(z,s)}|\nabla w|^2 dxdt\leq \frac{1}{2}\fint_{Q_{\frac{26}{10}r}(z,s)}|\nabla w|^2dxdt+ c\left(\fint_{Q_{\frac{26}{10}r}(z,s)}|\nabla w|^{p_0}dxdt\right)^{\frac{2}{p_0}},\label{5hh19096}
                  \end{eqnarray}
   for all $Q_{\frac{26}{10}r}(z,s)\subset Q_{6R}=Q_{6R}(x_0,t_0)$. Here the constant $p_0$ is in inequality \eqref{5hh090120141}.    \\                                                                
 Suppose that $B_{r}(z)\subset \Omega$. Take $\rho\in (0,r]$.  Let $\varphi\in C_c^\infty(B_\rho(z))$, $\eta\in C^\infty_c((s-\rho^2,s])$ be such that $0\leq \varphi,\eta\leq 1 $, $\varphi=1$ in $B_{\rho/2}(z)$, $\eta=1$ in $[s-\rho^2/4,s]$ and $|\nabla \varphi|\leq c_1/\rho$, $|\eta_t|\leq c_1/\rho^2$. We denote 
 \begin{equation*}
 \tilde{w}_{B_\rho(z)}(t)=\left(\int_{B_{\rho}(z)}\varphi(x)^2dx\right)^{-1}\int_{B_\rho(z)}w(x,t)\varphi(x)^2dx.
 \end{equation*}
  Using  $\varphi=(w-\tilde{w}_{B_\rho(z)})\varphi^2\eta^2$ 
  as a test function for the equation \eqref{5hheq5} we have for all $s'\in [s-\rho^2/4,s]$
  \begin{align*}
 &\int_{B_\rho(z)\times (s-\rho^2,s')}(w-\tilde{w}_{B_\rho(z)})_t(w-\tilde{w}_{B_\rho(z)})\varphi^2\eta^2dxdt\\&~~~~~~~~~~~~~~~~~ +\int_{B_\rho(z)\times (s-\rho^2,s')}A(x,t,\nabla w)\nabla \left((w-\tilde{w}_{B_\rho(z)})\varphi^2\eta^2\right)dxdt=0.
  \end{align*}
 Here we used the equality 
 $
   \int_{B_\rho(z)\times (s-\rho^2,s')}\left(\tilde{w}_{B_\rho(z)}\right)_t(w-\tilde{w}_{B_\rho(z)})\varphi^2\eta^2dxdt=0.
$\\
     Thus, we can  write 
     \begin{align*}
   & \frac{1}{2} \int_{B_\rho(z)}(w(s')-\tilde{w}_{B_\rho(z)}(s'))^2\varphi^2dx +\int_{B_\rho(z)\times (s-\rho^2,s')}A(x,t,\nabla w)\nabla w\varphi^2\eta^2 dxdt\\&~~~~~~~~~~~~~= -2\int_{B_\rho(z)\times (s-\rho^2,s')}A(x,t,\nabla w)\nabla \varphi \varphi\eta^2(w-\tilde{w}_{B_\rho(z)})dxdt \\&~~~~~~~~~~~~~~~~~+ \int_{B_\rho(z)\times (s-\rho^2,s')}(w-\tilde{w}_{B_\rho(z)})^2\varphi^2\eta \eta_t dxdt.
     \end{align*}
          From conditions \eqref{5hhconda} and \eqref{5hhcondb},  we get 
      \begin{align*}
      & \int_{B_\rho(z)}(w(s')-\tilde{w}_{B_\rho(z)}(s'))^2\varphi^2dx +\int_{B_{\rho}(z)\times (s-\rho^2,s')}|\nabla w|^2\varphi^2\eta^2 dxdt\\&\quad\quad\quad\lesssim\int_{B_\rho(z)\times (s-\rho^2,s')}|\nabla w||\nabla \varphi| \varphi\eta^2|w-\tilde{w}_{B_\rho(z)}|dxdt + \frac{1}{\rho^2}\int_{Q_{\rho}(z,s)}(w-\tilde{w}_{B_\rho(z)})^2 dxdt.
      \end{align*}
 Using H\"older's inequality we can verify that
          \begin{align}
          \nonumber
          &\mathop {\sup }\limits_{s'\in[s-\rho^2/4,s]}\int_{B_\rho(z)}(w(s')-\tilde{w}_{B_\rho(z)}(s'))^2\varphi^2dx \\&~~~~~~~~~~~~~~~~~~~~~~~~+\int_{Q_{\rho/2}(z,s)}|\nabla w|^2 dxdt \lesssim \frac{1}{\rho^2}\int_{Q_{\rho}(z,s)} |w-\tilde{w}_{B_\rho(z)}|^2dxdt.\label{5hh19092} 
          \end{align}         On the other hand, for any $s'\in[s-\rho^2/4,s]$ 
                        \begin{equation}\label{5hh230320141}
                        \int_{B_{\rho/2}(z)}(w(s')-\tilde{w}_{B_{\rho/2}(z)}(s'))^2dx\lesssim \int_{B_\rho(z)}(w(s')-\tilde{w}_{B_\rho(z)}(s'))^2\varphi^2dx,
                        \end{equation}
                        where $\varphi_1(x)=\varphi(z+2(x-z))$ for all $x\in B_{\rho/2}(z)$ and 
                        $$\tilde{w}_{B_{\rho/2}(z)}=\left(\int_{B_{\rho/2}(z)}\varphi_1(x)^2dx\right)^{-1}\int_{B_{\rho/2}(z)}w(x,t)\varphi_1(x)^2dx.$$      
                        In fact, since $0\leq \varphi\leq 1$ and $\varphi=1$ in $B_{\rho/2}(z)$, we have  
                        \begin{align*}
                        	&\int_{B_{\rho/2}(z)}(w(s')-\tilde{w}_{B_{\rho/2}(z)}(s'))^2dx\\&~~~\lesssim\int_{B_{\rho/2}(z)}(w(s')-\tilde{w}_{B_{\rho}(z)}(s'))^2dx+(\tilde{w}_{B_{\rho/2}(z)}(s')-\tilde{w}_{B_{\rho}}(z)(s'))^2|B_{\rho/4}(z)|                        
                        	\\&~~~\lesssim \int_{B_\rho(z)}(w(s')-\tilde{w}_{B_\rho(z)}(s'))^2\varphi^2dx +\int_{B_{\rho/2}(z)}(w(s')-\tilde{w}_{B_{\rho/2}(z)}(s'))^2\varphi^2_1dx\\&~~~~~+\int_{B_{\rho/2}(z)}(w(s')-\tilde{w}_{B_{\rho}(z)}(s'))^2\varphi^2_1dx. 
                        \end{align*}                       
                      
                        which yields \eqref{5hh230320141} due to  the following inequality
$$
                                                 \int_{B_{\rho/2}(z)}(w(s')-\tilde{w}_{B_{\rho/2}(z)}(s'))^2\varphi_1^2dx\leq \int_{B_{\rho/2}(z)}(w(s')-l)^2\varphi_1^2dx ~~\forall l\in\mathbb{R}.$$
                        Therefore, 
                        \begin{align}
                        \nonumber&\mathop {\sup }\limits_{s'\in[s-\rho^2/4,s]}\int_{B_{\rho/2}(z)}(w(s')-\tilde{w}_{B_{\rho/2}(z)}(s'))^2dx \\&~~~~~~~~~~~~~~~~~~~~~+\int_{Q_{\rho/2}(z,s)}|\nabla w|^2 dxdt \lesssim \frac{1}{\rho^2}\int_{Q_{\rho}(z,s)} |w-\tilde{w}_{B_\rho(z)}|^2dxdt. \label{5hh090120142}
                        \end{align}                                               
                     Now we use estimate  \eqref{5hh090120142} for $\rho=r/2$, we have
                     \begin{align*}
                     \int_{Q_{r/4}(z,s)}|\nabla w|^2 dxdt&\lesssim \frac{}{r^2}\int_{Q_{r/2}(z,s)}(w-\tilde{w}_{B_{r/2}(z)})^2dxdt\\&\lesssim \frac{1}{r^2}\left(\mathop {\sup }\limits_{s'\in[s-r^2/4,s]}\int_{B_{r/2}(z)}(w(s')-\tilde{w}_{B_{r/2}(z)}(s'))^2dx\right)^{\frac{2}{N+2}}\\&~~~~~~\times\int_{s-r^2/4}^{s}\left(\int_{B_{r/2}(z)}(w-\tilde{w}_{B_{r/2}(z)})^2dx\right)^{\frac{N}{N+2}}dt.
                     \end{align*}
                    After we  use estimate  \eqref{5hh090120142} for $\rho=r$ we get
                     \begin{align*}
       \int_{Q_{r/4}(z,s)}|\nabla w|^2 dxdt&\lesssim \frac{1}{r^2}\left(\mathop { }\frac{1}{r^2}\int_{Q_{r}(z,s)} |w-\tilde{w}_{B_r(z)}|^2dxdt\right)^{\frac{2}{N+2}}\\&~~~~~\times\int_{s-r^2/4}^{s}\left(\int_{B_{\rho/2}(z)}(w-\tilde{w}_{B_{r/2}(z)})^2dx\right)^{\frac{N}{N+2}}dt .
                                         \end{align*}                                    
                 Thanks to a Sobolev-Poincare inequality, we obtain 
$$\int_{Q_{r/4}(z,s)}|\nabla w|^2 dxdt\lesssim \frac{1}{r^2}\left(\int_{Q_r(z,s)}|\nabla w|^2dxdt\right)^{\frac{2}{N+2}}\int_{Q_{r/2}(z,s)}|\nabla w|^{\frac{2N}{N+2}}dxdt.$$
    Since $p_0\in (\frac{2N}{N+2},2)$, thanks to H\"older's inequality we get \eqref{5hh19096}. \\
          Finally, we consider the case $B_{r}(z)\cap \Omega\not= \emptyset$. In this  case we choose $z_0\in \partial \Omega$ such that $|z-z_0|=\text{dist}(z,\partial\Omega)$. Then $|z_0-z|<r$ and thus $\frac{1}{4}r\leq\rho_1\leq \frac{1}{2}r$,
          \begin{align}\label{5hh240320141}
          B_{\frac{1}{4}r}(z)\subset B_{\frac{5}{4}r}(z_0)\subset B_{\rho_1+r}(z_0)\subset B_{\rho_1
                    +\frac{11}{10}r}(z_0)\subset B_{\frac{16}{10}r}(z_0)\subset B_{\frac{26}{10}r}(z)\subset B_{6R}(x_0).
          \end{align}
          Let $\varphi\in C_c^\infty(B_{\rho_1
                    +\frac{11}{10}r}(z_0))$ be such that  $0\leq \varphi\leq1$, $\varphi=1$ in $ B_{\rho_1+r}(z_0)$ and $|\nabla \varphi|\leq C/r$. For $\frac{1}{2}r\leq \rho_2\leq r$, let $\eta\in C^\infty_c((s-\rho_2^2,s])$ be such that $0\leq \eta\leq 1$, $\eta=1$ in $[s-\rho_2^2/4,s]$ and $|\eta_t|\leq c/r^2$.  Using $\phi=w\varphi^2\eta^2$ as a test function for \eqref{5hheq5} we have for any $s'\in (s-\rho_2^2,s)$
\begin{align*}
&\int_{(B_{\rho_1
       +\frac{11}{10}r}(z_0)\cap \Omega)\times (s-\rho_2^2,s')}w_t w \varphi^2\eta^2dxdt\\&~~~~~~~~~~~~~~~~~~~~+\int_{(B_{\rho_1
              +\frac{11}{10}r}(z_0)\cap \Omega)\times (s-\rho_2^2,s')}A(x,t,\nabla w)\nabla\left(w\varphi^2\eta^2\right)dxdt=0.
\end{align*}                         
        As above we also get 
        \begin{align*}
        &\mathop {\sup }\limits_{s'\in[s-\rho_2^2/4,s]} \int_{B_{\rho_1+r}(z_0)}w^2(s')  dx\\&~~~~~~~~~~~~~~~+\int_{B_{\rho_1+r}(z_0)\times (s-\rho_2^2/4,s)}|\nabla w|^2dxdt\lesssim \frac{1}{r^2}\int_{B_{\rho_1
                          +\frac{11}{10}r}(z_0)\times (s-\rho_2^2,s)}w^2dxdt.
        \end{align*}           
           In particular, for $\rho_1=\frac{1}{4}r$, $\rho_2=\frac{1}{2}r$ and using \eqref{5hh240320141} yields
                        \begin{equation}\label{5hh240320143}
                        \int_{Q_{\frac{1}{4}r}(z,s)}|\nabla w|^2dxdt \lesssim \frac{1}{r^2}\int_{B_{\frac{29}{20}r}(z_0)\times (s-r^2/4,s)}w^2dxdt,
                        \end{equation}   
                        and $\rho_1=(\frac{1}{4}+\frac{1}{10})r,\rho_2=r$,
                        \begin{equation*}\label{5hh240320142}
     \mathop {\sup }\limits_{s'\in[s-r^2/4,s]} \int_{B_{\frac{1}{4}r+\frac{11}{10}r}(z_0)}w^2(s')dx\lesssim \frac{1}{r^2}\int_{B_{\frac{29}{20}r}(z_0)\times (s-r^2,s)}w^2dxdt.
                                            \end{equation*} 
                                            Set $ K_1=\{w=0\}\cap \overline{B}_{\frac{29}{20}r}(z_0)$ and $K_2=\{w=0\}\cap \overline{B}_{\frac{1}{4}r+\frac{11}{10}r}(z_0)$, Since $\mathbb{R}^N\backslash\Omega$ satisfies an uniformly $2-$thick, we have the following estimates 
                \begin{equation*}
                \text{Cap}_2(K_1,B_{\frac{29}{10}r}(z_0))\gtrsim  r^{N-2} ~\text{ and }~ \text{Cap}_{p_0}(K_2,B_{\frac{1}{2}r+\frac{11}{5}r}(z_0))\gtrsim  r^{N-p_0}.
                \end{equation*}               
                   So, by  Sobolev-Poincare's inequality we get 
                    \begin{equation}     
                                   \fint_{B_{\frac{29}{20}r}(z_0)}w^2dx\lesssim r^{2}\fint_{B_{\frac{5}{2}r}(z)}|\nabla w|^2dx,\label{5hh240320144}
                    \end{equation}               
                 and
$$
                 \fint_{B_{\frac{1}{4}r+\frac{11}{10}r}(z_0)}w^2dxdt\lesssim  r^2\left(\fint_{B_{\frac{1}{4}r+\frac{11}{10}r}(z_0)}|\nabla w|^{p_0}dx\right)^{\frac{2}{p_0}}
                             \lesssim  r^2\left(\fint_{B_{\frac{5}{2}r}(z_0)}|\nabla w|^{p_0}dx\right)^{\frac{2}{p_0}}.$$
Hence,
                               \begin{equation}\label{5hh19094}
                                             \mathop {\sup }\limits_{s'\in[s-r^2/4,s]} \int_{B_{\frac{1}{4}r+\frac{11}{10}r}(z_0)}w^2(s')dx\lesssim\int_{Q_{\frac{5}{2}r}(z,s)}|\nabla w|^2dxdt,
                                               \end{equation}   and\begin{equation}\label{5hh19095}
                                                               \int_{B_{\frac{1}{4}r+\frac{11}{10}r}(z_0)}w^2(t)dx\lesssim r^{N+2}\left(\fint_{B_{\frac{5}{2}r}(z_0)}|\nabla w|^{p_0}(t)dx\right)^{\frac{2}{p_0}}.
                                                                \end{equation}         
   From \eqref{5hh240320143}, we have
   \begin{align*}
    &\fint_{Q_{\frac{1}{4}r}(z,s)}|\nabla w|^2dxdt \lesssim \frac{1}{r^{N+4}}\int_{B_{\frac{1}{4}r
                         +\frac{11}{10}r}(z_0)\times (s-r^2/4,s)}w^2dxdt
                          \\&~\lesssim \frac{1}{r^{N+4}}  \left(\mathop {\sup }\limits_{s'\in[s-r^2/4,s]} \int_{B_{\frac{1}{4}r+\frac{11}{10}r}(z_0)}w^2(s')dx\right)^{1-\frac{p_0}{2}}\int_{s-r^2/4}^{s} \left(\int_{B_{\frac{1}{4}r+\frac{11}{10}r}(z_0)}w^2(t)dx\right)^{\frac{p_0}{2}}dt.                         
   \end{align*}
  Using \eqref{5hh19095}, \eqref{5hh19094} and H\"older's inequality we get
  \begin{align*}
  \fint_{Q_{\frac{1}{4}r}(z,s)}|\nabla w|^2dxdt&\lesssim \frac{1}{r^{N+4}}  \left(\int_{Q_{\frac{5}{2}r}(z,s)}|\nabla w|^2dxdt\right)^{1-\frac{p_0}{2}}r^{\frac{N+2}{2}p_0-N}\int_{Q_{\frac{5}{2}r}(z,s)}|\nabla w|^{p_0}dxdt 
         \\&\lesssim \left(\fint_{Q_{\frac{5}{2}r}(z,s)}|\nabla w|^2dxdt\right)^{1-\frac{p_0}{2}}\fint_{Q_{\frac{5}{2}r}(z,s)}|\nabla w|^{p_0}dxdt  
                  \\&\leq \frac{1}{2} \fint_{Q_{\frac{26}{10}r}(z,s)}|\nabla w|^2dxdt+c\left(\fint_{Q_{\frac{26}{10}r}(z,s)}|\nabla w|^{p_0}dxdt \right)^{\frac{2}{p_0}}. 
  \end{align*}                 
     So we proved \eqref{5hh19096}.  \\
   Therefore, By Gehring's Lemma (see \cite{55Nau}) we get     \eqref{5hhineq8}. \\
   3. Now we prove \eqref{5hhineq10}. Let $y\in \partial \Omega$, $Q_{\rho_1}(y,s)\subset Q_{\rho_2}(y,s)\subset Q_{6R}$ with  $\rho_1\leq\rho_2/4$. First, we will show that there exists a constant $\beta_2=\beta_2(N,\Lambda_1,\Lambda_2,c_0)\in (0,1/2]$ such that 
                      \begin{equation}\label{5hh060120145}
                      \text{osc}(w,Q_{\rho_1}(y,s))\lesssim\left(\frac{\rho_1}{\rho_2}\right)^{\beta_2} \text{osc}(w,Q_{\rho_2/2}(y,s)),
                      \end{equation}
                      where $\text{osc}(w,O)=\sup_{O}w-\inf_{O}w$.\\
                      Indeed, since 
   \begin{equation*}
   \int_{0}^1\frac{\text{Cap}_{1,2}(\Omega^c\cap B_r(z),B_{2r}(z)
    )}{r^{N-2}}\frac{dr}{r}=+\infty ~~\forall z\in\partial\Omega.
   \end{equation*} thus by the  Wiener criterion (see \cite{55Zi1}), we have $w$ is continuous up to $\partial_p \tilde{\Omega}_{6R}$.                     So, we can  choose  $\varphi=\left(V-M_{4\rho_1}\right)\eta^2\in L^2(-\infty,T;H^1_0(\Omega\cap B_{6R}(x_0)))$ as test function in \eqref{5hheq5},
    where 
    \begin{description}
    	\item[a)]  $\eta\in C^\infty (Q_{4\rho_1}(y,s))$, $0\leq \eta\leq 1$ such that $\eta=1$ in $Q_{\rho_1/2}(y,s-\frac{17}{4}\rho_1^2)$, $\text{supp}(\eta)\subset\subset Q_{\rho_1}(y,s-4\rho_1^2)$ and $|\nabla \eta|\leq c_{27}/\rho_1$, $|\eta_t|\leq c_{28}/\rho_1^2$,
    	\item[b)]  $M_{4\rho_1}= \sup_{Q_{4\rho_1}(y,s)}w$ and $V=\inf\{M_{4\rho_1}-w,M_{4\rho_1}\}$ in $\tilde{\Omega}_{6R}$, $V=M_{4\rho_1}$ outside $\tilde{\Omega}_{6R}$.
    \end{description}
   We have 
   \begin{align*}
   & \int_{\tilde{\Omega}_{6R}}w_t\left(V-M_{4\rho_1}\right)\eta^2dxdt\\&~~~~~~~+\int_{\tilde{\Omega}_{6R}}2\eta A(x,t,\nabla w)\nabla \eta\left(V-M_{4\rho_1}\right)dxdt+\int_{\tilde{\Omega}_{6R}}\eta^2A(x,t,
      \nabla w)\nabla Vdxdt=0,
   \end{align*}
   which implies 
   \begin{align*}
   &\int_{\tilde{\Omega}_{6R}}\eta^2A(x,t,
         -\nabla V)(-\nabla V)dxdt=\int_{\tilde{\Omega}_{6R}}2\eta A(x,t,-\nabla V)\nabla \eta\left(V-M_{4\rho_1}\right)dxdt\\&~~~~~~~~~~~~-\int_{\tilde{\Omega}_{6R}}\left(V-M_{4\rho_1}\right)_t\left(V-M_{4\rho_1}\right)\eta^2dxdt.
   \end{align*}
   Using \eqref{5hhconda} and \eqref{5hhcondb} we get 
   \begin{align*}
    &\Lambda_2\int_{\tilde{\Omega}_{6R}}\eta^2|\nabla V|^2dxdt\\&~~~~\leq 2\Lambda_1\int_{\tilde{\Omega}_{6R}}\eta |\nabla V||\nabla \eta||V-M_{4\rho_1}|dxdt-1/2\int_{\tilde{\Omega}_{6R}}\left(\left(V-M_{4\rho_1}\right)^2-M_{4\rho_1}^2\right)(\eta^2)_tdxdt
         \\&~~~~\leq 2\Lambda_1M_{4\rho_1}\int_{\tilde{\Omega}_{6R}}\eta |\nabla V||\nabla \eta|dxdt+2M_{4\rho_1}\int_{\tilde{\Omega}_{6R}}\eta V |\eta_t|dxdt.
   \end{align*}   
      Since $\text{supp}(|\nabla V|)\cap \text{supp}(\eta)\subset \tilde{\Omega}_{6R}$, thus 
      \begin{align}
      \nonumber \int_{\mathbb{R}^{N+1}}|\nabla (\eta V)|^2dxdt
                  &\lesssim M_{4\rho_1}\left(\int_{\mathbb{R}^{N+1}}\eta |\nabla V||\nabla \eta|dxdt+\int_{\mathbb{R}^{N+1}}V \left(\eta|\eta_t|+|\nabla\eta|^2\right)dxdt\right)\\& \lesssim M_{4\rho_1}\left(\int_{\mathbb{R}^{N+1}}\eta |\nabla V||\nabla \eta|dxdt+\frac{1}{\rho_1^2}\int_{Q_{\rho_1}(y,s-4\rho_1^2)}V dxdt\right).\label{5hh060120143}
      \end{align}
            By \cite[Theorem 6.31, p. 132]{55Li3}, for any $\sigma\in (0,1+2/N)$ there holds 
            
            \begin{equation}\label{5hh060120142}
            \left(\fint_{Q_{\rho_1}(y,s-4\rho_1^2)}V^\sigma dxdt\right)^{1/\sigma}\lesssim \inf_{Q_{\rho_1}(y,s)}V=M_{4\rho_1}-\sup_{Q_{\rho_1}(y,s)}w=M_{4\rho_1}-M_{\rho_1}.
            \end{equation}
            In particular, 
            \begin{equation}\label{5hh060120144}
            \frac{1}{\rho_1^2}\int_{Q_{\rho_1}(y,s-4\rho_1^2)}V dxdt\lesssim\rho_1^N (M_{4\rho_1}-M_{\rho_1}).
            \end{equation}
            We need to estimate $\int_{\tilde{\Omega}_{6R}}\eta |\nabla V||\nabla \eta|dxdt$.
            Using H\"older inequality and \eqref{5hh060120142}, for $\varepsilon\in (0,\min\{2/N,1\})$ we have
            \begin{align*}
             \int_{\tilde{\Omega}_{6R}}\eta |\nabla V||\nabla \eta|dxdt&\leq\left(\int_{\tilde{\Omega}_{6R}}\eta^2 V^{-(1+\varepsilon)}|\nabla V|^2 dxdt\right)^{1/2}\left(\int_{\tilde{\Omega}_{6R}}V^{1+\varepsilon}|\nabla \eta|^2dxdt\right)^{1/2}
             \\&  \lesssim\left(\int_{\tilde{\Omega}_{6R}}\eta^2 V^{-(1+\varepsilon)}|\nabla V|^2 dxdt\right)^{1/2}\left(\int_{Q_{\rho_1}(y,s-4\rho_1^2)}V^{1+\varepsilon}dxdt\right)^{1/2}
                        \\&\lesssim \left(\int_{\tilde{\Omega}_{6R}}\eta^2 V^{-(1+\varepsilon)}|\nabla V|^2 dxdt\right)^{1/2}\rho_1^{\frac{N+2}{2}}(M_{4\rho_1}-M_{\rho_1})^{(1+\varepsilon)/2}.
            \end{align*}          
            To estimate  $\left(\int_{\tilde{\Omega}_{6R}}\eta^2 V^{-(1+\varepsilon)}|\nabla V|^2 dxdt\right)^{1/2}$, we can choose $\varphi=((V+\delta)^{-\varepsilon}-(M_{4\rho_1}+\delta)^{-\varepsilon})\eta^2$, for $\delta>0$, as test function in \eqref{5hheq5}, we will get 
            \begin{align*}
             &\int_{\tilde{\Omega}_{6R}}\eta^2 (V+\delta)^{-(1+\varepsilon)}|\nabla V|^2 dxdt\\&~~~~~~~~~\lesssim \int_{\tilde{\Omega}_{6R}}\eta (V+\delta)^{-\varepsilon}|\nabla V||\nabla \eta| dxdt+\int_{\tilde{\Omega}_{6R}}\eta (V+\delta)^{1-\varepsilon} |\eta_t| dxdt.
            \end{align*} 
            Thanks to H\"older's inequality, we obtain
            \begin{align*}
             \int_{\tilde{\Omega}_{6R}}\eta^2 (V+\delta)^{-(1+\varepsilon)}|\nabla V|^2 dxdt&\lesssim  \int_{\tilde{\Omega}_{6R}} (V+\delta)^{1-\varepsilon} \left(\eta|\eta_t|+|\nabla\eta|^2\right) dxdt \\& \lesssim  \rho_1^2\int_{Q_{\rho_1}(y,s-4\rho_1^2)} (V+\delta)^{1-\varepsilon} dxdt. 
            \end{align*}       
           Letting $\delta\to 0$ and using \eqref{5hh060120142}, we get
$$\int_{\tilde{\Omega}_{6R}}\eta^2 V^{-(1+\varepsilon)}|\nabla V|^2 dxdt\lesssim\rho_1^{-2}\int_{Q_{\rho_1}(y,s-4\rho_1^2)} V^{1-\varepsilon} dxdt
               \lesssim\rho_1^{N}\left(M_{4\rho_1}-M_{\rho_1}\right)^{1-\varepsilon}.$$
        Thus, $$
                    \int_{\tilde{\Omega}_{6R}}\eta |\nabla V||\nabla \eta|dxdt\lesssim \rho_1^{N}(M_{4\rho_1}-M_{\rho_1}).
                   $$    Combining this with \eqref{5hh060120143} and \eqref{5hh060120144}, 
$$
\int_{\mathbb{R}^{N+1}}|\nabla (\eta V)|^2dxdt\lesssim\rho_1^NM_{4\rho_1}\left(M_{4\rho_1}-M_{\rho_1}\right).$$
Note that $\eta V=M_{4\rho_1}$ in $\left(\Omega^c\cap B_{\rho_1/2}(y)\right)\times (s-\frac{9}{2}\rho_1^2,s-\frac{17}{4}\rho_1^2)$ thus 
\begin{align*}
\int_{\mathbb{R}^{N+1}}|\nabla (\eta V)|^2dxdt&\geq \int_{s-\frac{9}{2}\rho_1^2}^{s-\frac{17}{4}\rho_1^2}\int_{\mathbb{R^N}}|\nabla (\eta V)|^2dxdt\\&\geq \int_{s-\frac{9}{2}\rho_1^2}^{s-\frac{17}{4}\rho_1^2}M_{4\rho_1}^2\text{Cap}_{1,2}(\Omega^c\cap B_{\rho_1/2}(y),B_{\rho_1}(y))dt
\\&\gtrsim M_{4\rho_1}^2\rho_1^N.
\end{align*}
Here we have used ${Cap}_{1,2}(\Omega^c\cap B_{\rho_1/2}(y),B_{\rho_1}(y))\gtrsim \rho_1^{N-2} $ in the last inequality.
It follows 
\begin{equation*}
M_{4\rho_1}\leq c(M_{4\rho_1}-M_{\rho_1}).
\end{equation*}
So 
\begin{equation*}
\sup_{Q_{\rho_1}(y,s)}w\leq \gamma\sup_{Q_{4\rho_1}(y,s)}w~~\text{ where }~\gamma=\frac{c}{c+1}<1.
\end{equation*}
Clearly, above estimate is also true when we replace $w$ by $-w$. These give, 
\begin{equation*}
\text{osc}(w,Q_{\rho_1}(y,s))\leq \gamma \text{osc}(w,Q_{4\rho_1}(y,s)).
\end{equation*}
It follows \eqref{5hh060120145}.\\
We come back the proof of \eqref{5hhineq10}.\\
Since $w=0$ outside $\Omega_T$ this leads to 
$$
              \left(\fint_{Q_{\rho_1}(y,s)}|w|^2dxdt\right)^{1/2}\lesssim \left(\frac{\rho_2}{\rho_1}\right)^{\beta_2}\text{osc}(w,Q_{\rho_2/2}(y,s)).$$
 On the other hand, By \cite[Theorem 6.30, p. 132]{55Li3} we have 
 \begin{align*}
 &\sup_{Q_{\rho_2/2}(y,s)}w\lesssim \left(\fint_{Q_{\rho_2}(y,s)}(w^+)^2\right)^{1/2},
  \sup_{Q_{\rho_2/2}(y,s)}(-w)\lesssim \left(\fint_{Q_{\rho_2}(y,s)}(w^-)^2\right)^{1/2}.
 \end{align*}                  
Thus, we get \eqref{5hhineq10}. \\
Next, we have \eqref{5hhineq10'} for case $z=y\in\partial\Omega$ since from Caccippoli's inequality, 
$$
                    \int_{Q_{\rho_1}(z,s)}|\nabla w|^2dxdt\lesssim \frac{1}{\rho_1^2}  \int_{Q_{2\rho_1}(z,s)}| w|^2dxdt,$$
and using Sobolev-Poincare's  inequality as in \eqref{5hh240320144}, 
$$\int_{Q_{\rho_2}(z,s)}| w|^2dxdt\lesssim \rho_2^2\int_{Q_{\rho_2}(z,s)}| \nabla w|^2dxdt.          
         $$              We now prove \eqref{5hhineq10'}. Take $Q_{\rho_1}(z,s)\subset Q_{\rho_2}(z,s)\subset Q_{6R}$, it is enough to consider the case $\rho_1\leq\rho_2/20$.  Clearly, if $B_{\rho_2/4}(z)\subset \Omega$ then \eqref{5hhineq10'} follows from \eqref{5hhineq3'} in Theorem \ref{5hh1510135}. We consider $B_{\rho_2/4}(z)\cap\partial\Omega\not= \emptyset$, let $z_0\in B_{\rho_2/4}(z)\cap\partial\Omega$ such that $|z-z_0|=\text{dist}(z,\partial\Omega)\leq \rho_2/4$. 
          Obviously,  if  $\rho_1<|z-z_0|/4$ and $z\notin\Omega$, then \eqref{5hhineq10'} is trivial.  If  $\rho_1<|z-z_0|/4$ and $z\in\Omega$, then \eqref{5hhineq10'} follows from \eqref{5hhineq3'} in Theorem \ref{5hh1510135}. \\
          Now assume  $\rho_1\geq |z-z_0|/4$ then since $Q_{\rho_1}(z,s)\subset Q_{5\rho_1}(z_0,s)$
          \begin{align*}
           \left(\fint_{Q_{\rho_1}(z,s)}|\nabla w|^2dxdt\right)^{1/2}&\lesssim \left(\fint_{Q_{5\rho_1}(z_0,s)}|\nabla w|^2dxdt\right)^{1/2}
                    \\&\lesssim \left(\frac{\rho_1}{\rho_2}\right)^{\beta_2-1}\left(\fint_{Q_{\rho_2/4}(z_0,s)}|\nabla w|^2dxdt\right)^{1/2}
                    \\&\lesssim \left(\frac{\rho_1}{\rho_2}\right)^{\beta_2-1}\left(\fint_{Q_{\rho_2/2}(z,s)}|\nabla w|^2dxdt\right)^{1/2},
          \end{align*}    
         which implies \eqref{5hhineq10'}.       
  \end{proof}
  \begin{corollary}\label{5hh060120149}
    Suppose that $\mathbb{R}^N\backslash\Omega$ satisfies uniformly $2-$thick with constants $c_0,r_0$. Let $\beta_2$ be the constant in Theorem \ref{5hh24093}. For $2-\beta_2<\theta<N+2$, there holds for any $B_{\rho}(y)\cap\partial\Omega\not=\emptyset$,  $s\in (-T,T)$, $0<\rho\leq r_0$
         \begin{equation}\label{5hh060120147}
         \int_{Q_\rho(y,s)}|\nabla u|dxdt \lesssim \rho^{N+3-\theta}\left(\left(\frac{T_0}{r_0}\right)^{N+3-\theta}+1\right)||\mathbb{M}_{\theta}[\mu]||_{L^\infty(\Omega\times(-T,T))},
         \end{equation}
         where $T_0=diam(\Omega)+T^{1/2}$.
         \end{corollary}
         \begin{proof} Take $B_{\rho_2/4}(y)\cap \partial \Omega\not=\emptyset$ and $s\in (-T,T)$, $\rho_2\leq 2r_0$. Let $y_0\in B_{\rho_2/4}(y)\cap \partial \Omega$ be such that $|y-y_0|=dist(y,\partial\Omega)\leq \rho_2/4$. Thus $ Q_{\rho_2/4}(y,s)\subset Q_{\rho_2/2}(y_0,s)$
         For any $Q_{\rho_1}(y,s)\subset Q_{\rho_2}(y,s)$ with $\rho_1\leq\rho_2/4$, we take $w$ as in Theorem \ref{5hh24093} with $Q_{6R}=Q_{\rho_2/2}(y_0,s)$. Thus, 
         \begin{align*}
        & \int_{Q_{\rho_1}(y,s)}|\nabla w|dxdt\lesssim \left(\frac{\rho_1}{\rho_2}\right)^{N+\beta_1+1}\int_{Q_{\rho_2/4}(y,s)}|\nabla w|dxdt,\\&
         \int_{Q_{\rho_2/2}(y_0,s)}|\nabla u-\nabla w|dxdt\lesssim\rho_2 |\mu|(Q_{\rho_2/2}(y_0,s)).
         \end{align*}     
           As in the proof of Corollary \ref{5hh060120148}, we get the result. 
         \end{proof}\\
          \subsubsection{Reifenberg flat domain}
  In this subsection, we always assume that $A$ satisfies \eqref{5hhcondc}. Also, 
  we  assume that $\Omega$ is a $(\delta,R_0)$- Reifenberg flat domain with $0<\delta<1/2$ . Fix $x_0\in\partial \Omega$ and $0<R<R_0/6$. We have a density estimate 
  \begin{equation}
  |B_t(x)\cap (\mathbb{R}^N\backslash\Omega)|\geq c|B_t(x)|~ \forall x\in\partial\Omega, 0<t<R_0,
  \end{equation}
  with $c=((1-\delta)/2)^N\geq 4^{-N}$.\\
  In particular, $\mathbb{R}^N\backslash\Omega$ satisfies a uniformly $2-$thick condition  with constants $c,r_0=R_0$. \\
  Next we set $\rho=R(1-\delta)$ so that $0<\rho/(1-\delta)<R_0/6$. By the definition of Reifenberg flat domains,  there exists a coordinate system $\{y_1,y_2,...,y_N\}$ with the
  origin $0\in\Omega$ such that in this coordinate system $x_0=(0,...,0,-\rho\delta/(1-\delta))$ and 
  \begin{equation*}
  B^+_\rho(0)\subset \Omega\cap B_\rho(0)\subset B_\rho(0)\cap \{y=(y_1,y_2,....,y_N):y_N>-2\rho\delta/(1-\delta)\}.
  \end{equation*}
  Since $\delta<1/2$ we have 
   \begin{equation*}
    B^+_\rho(0)\subset \Omega\cap B_\rho(0)\subset B_\rho(0)\cap \{y=(y_1,y_2,....,y_N):y_N>-4\rho\delta\},
    \end{equation*}
    where $B^+_\rho(0):=B_\rho(0)\cap\{y=(y_1,y_2,...,y_N):y_N>0\}$.\\
    Furthermore we consider the unique solution
           \begin{equation}
            v\in C(t_0-\rho^2,t_0;L^2(\Omega\cap B_\rho(0)))\cap L^2(t_0-\rho^2,t_0;H^1(\Omega\cap B_\rho(0)))
            \end{equation}
            to the following equation
             \begin{equation}\label{5hh1610131}
             \left\{ \begin{array}{l}
                  {v_t} - \operatorname{div}\left( {\overline{A}_{B_{\rho}(0)}(t,\nabla v)} \right) = 0 \;in\;\tilde{\Omega}_\rho(0),\\ 
                  v = w\quad \quad on~~\partial_{p}\tilde{\Omega}_\rho(0), \\ 
                  \end{array} \right.
             \end{equation}
              where $\tilde{\Omega}_\rho(0)=\left(\Omega\cap B_{\rho}(0)\right)\times (t_0-\rho^2,t_0)$ ($-T<t_0<T$). \\
              We put $v=w$ outside $\tilde{\Omega}_\rho(0)$. As Lemma \ref{5hh21101319},  we have the following Lemma. 
              \begin{lemma}\label{5hh1610139} Let $\theta_2$ be the constant in Theorem \ref{5hh24093}. There holds \begin{eqnarray}\label{5hh1610132}
                             \left(\fint_{Q_{\rho}(0,t_0)}|\nabla w-\nabla v|^2\right)^{1/2}\lesssim [A]_{s_2}^{R} \fint_{Q_{\rho}(0,t_0)}|\nabla w| dxdt, 
                             \end{eqnarray}
                     with $s_2=\frac{2\theta_2}{\theta_2-2}$ and \begin{equation}\label{5hh1610133}
                              \int_{Q_{\rho}(0,t_0)}|\nabla w|^2dxdt\sim  \int_{Q_{\rho}(0,t_0)}|\nabla v|^2dxdt.
                              \end{equation}
                     \end{lemma}
                    We can see that if the boundary of $\Omega$ is bad enough, then the $L^\infty$-norm of $\nabla v$ up to $\partial\Omega\cap B_\rho(0)\times (t_0-\rho^2,t_0)$ could be unbounded. For our purpose, we will consider another  equation: 
            \begin{equation}\label{5hh1610134}
               \left\{ \begin{array}{l}
       {V_t} - \operatorname{div}\left( {\overline{A}_{B_{\rho}(0)}(t,\nabla V)} \right) = 0 ~~\text{in}~~Q_\rho^+(0,t_0),\\ 
                V = 0\quad \quad \text{on}~~ T_\rho(0,t_0), \\ 
                            \end{array} \right.
                                             \end{equation}
                                             where $Q_\rho^+(0,t_0)=B_\rho^+(0)\times (t_0-\rho^2,t_0)$ and  $T_\rho(0,t_0)=Q_\rho(0,t_0)\cap\{x_N=0\}$.\\
A weak solution $V$ of above problem is understood in the following sense: the zero extension of $V$  to $Q_\rho(0,t_0)$ is in  $ V\in C(t_0-\rho^2,t_0;L^2( B_\rho(0)))\cap L^2_{\text{loc}}(t_0-\rho^2,t_0;H^1( B_\rho(0)))$  and for every $\varphi\in C_c^1(Q_\rho^+(0,t_0))$ there holds
\begin{align*}
-\int_{Q_\rho^+(0,t_0)}V\varphi_tdxdt+\int_{Q_\rho^+(0,t_0)}\overline{A}_{B_{\rho}(0)}(t,\nabla V)\nabla \varphi dxdt=0. 
\end{align*}                                      
                                         We have the following gradient $L^\infty$ estimate up to the boundary for $V$.
                                         \begin{lemma}[see \cite{55Li1,55Li2}]\label{5hh2110139}
                                         For any weak solution $V\in C(t_0-\rho^2,t_0;L^2( B_\rho^+(0)))\cap L^2_{\text{loc}}(t_0-\rho^2,t_0;H^1( B_\rho^+(0)))$ of \eqref{5hh1610134}, we have 
$$
                                         ||\nabla V||_{L^\infty (Q_{\rho'/2}^+(0,t_0))}\lesssim \fint_{Q_{\rho'}^+(0,t_0)}|\nabla V|^2dxdt~~\forall~ 0<\rho'\leq \rho.
$$
                                        Moreover, $\nabla V$ is continuous up to $T_\rho(0,t_0)$.
                                         \end{lemma}
                                         \begin{lemma}\label{5hh21101312}
                                         If  $V\in C(t_0-\rho^2,t_0;L^2( B_\rho^+(0)))\cap L^2(t_0-\rho^2,t_0;H^1( B_\rho^+(0)))$ is a weak solution of  \eqref{5hh1610134}, then its zero extension from $Q_{\rho}^+(0,t_0)$ to $Q_{\rho}(0,t_0)$ solves
                                         \begin{equation}
                                         {V_t} - \operatorname{div}\left( {\overline{A}_{B_{\rho}(0)}(t,\nabla V)} \right) = \frac{\partial F}{\partial x_N},
                                         \end{equation}
                                         weakly in $Q_{\rho}(0,t_0)$. Here,
                                         $\overline{A}_{B_{\rho}(0)}=(\overline{A}_{B_{\rho}(0)}^1,...,\overline{A}_{B_{\rho}(0)}^N)$ and $F(x,t)=\chi_{x_N<0}\overline{A}_{B_{\rho}(0)}^N(t, \nabla V(x',0,t))$  for $(x,t)=(x',x_N,t)\in Q_{\rho}(0,t_0)$.
                                         \end{lemma}
                                         \begin{proof}
Let $g\in C^\infty(\mathbb{R})$ with $g=0$ on $(-\infty,1/2)$ and $g=1$ on $(1,\infty)$. Then, for any $\varphi\in C^\infty_c(Q_\rho(0,t_0))$ and $n\in\mathbb{N}$, we have $\varphi_n(x,t)=\varphi_n(x',x_N,t)=g(nx_N)\varphi(x,t)\in C^\infty_c(Q^+_\rho(0,t_0)$. Thus, we get 
\begin{equation*}
\int_{Q^+_\rho(0,t_0)}V_t\varphi_n dxdt +\int_{Q^+_\rho(0,t_0)}\overline{A}_{B_\rho(0)}(t,\nabla V)\nabla \left(g(nx_N)\varphi(x,t)\right) dxdt=0,
\end{equation*}      
which implies  
$$
\int_{Q^+_\rho(0,t_0)}V_t\varphi_n dxdt +\int_{Q^+_\rho(0,t_0)}\overline{A}_{B_\rho(0)}(t,\nabla V)\nabla\varphi(x,t) g(nx_N)dxdt
=- \int_{0}^{\rho}G(z)g'(nz)ndz.$$
Here $$G(z)=\int_{t_0-\rho^2}^{t_0}\int_{|x'|<\sqrt{\rho^2-z^2}}\overline{A}^N_{B_\rho(0)}(t,\nabla V)\varphi(x',z,t)dx'dt\in C([0,\infty)). $$       
Letting $n\to \infty$ we get  
$$
\int_{Q^+_\rho(0,t_0)}V_t\varphi dxdt +\int_{Q^+_\rho(0,t_0)}\overline{A}_{B_\rho(0)}(t,\nabla V)\nabla\varphi(x,t) dxdt=-\int_{Q_\rho(0,t_0)}F\frac{\partial \varphi}{\partial x_N}dxdt.$$
Since $\nabla V=0,V=0$ outside $Q^+_\rho$, therefore we get  the result.     
                                         \end{proof}\\ \\
                                         We now consider a scaled version of equation \eqref{5hh1610131}
                               \begin{equation}\label{5hh1610135}
                              \left\{ \begin{array}{l}
                       {v_t} - \operatorname{div}\left( {\overline{A}_{B_{1}(0)}(t,\nabla v)} \right) = 0 ~~\text{ in }~\tilde{\Omega}_1(0),\\ 
                      v = 0~~\text{ on }~\partial_{p}\tilde{\Omega}_1(0)\backslash \left(\Omega\times (-T,T)\right), \\ 
                        \end{array} \right.
                                   \end{equation}
                                                      under assumption 
                                                      \begin{equation}\label{5hh1610136}
                                                          B^+_1\subset \Omega\cap B_1\subset B_1\cap \{x_N>-4\delta\}
                                                          \end{equation}
                                                      with $B_\rho=B_\rho(0).$
    \begin{lemma}\label{5hh2110138}
    For any $\varepsilon>0$ there exists a small $\delta=\delta(N,\Lambda_1,\Lambda_2,\varepsilon)>0$ such that if $v\in C(t_0-1,t_0;L^2(\Omega\cap B_1))\cap L^2(t_0-1,t_0;H^1(\Omega\cap B_1))$ is a solution of \eqref{5hh1610135} and \eqref{5hh1610136} is satisfied and 
    \begin{equation}
    \fint_{Q_1(0,t_0)}|\nabla v|^2dxdt\leq 1,
    \end{equation}
    then there exists a weak solution $V\in C(t_0-1,t_0;L^2( B_1^+))\cap L^2(t_0-1,t_0;H^1( B_1^+))$ of  \eqref{5hh1610134} with $\rho=1$, whose zero extension to $Q_1(0,t_0)$ satisfies 
    \begin{equation}
       \fint_{Q_1(0,t_0)}|v-V|^2dxdt\leq \varepsilon^2,
        \end{equation}
    \end{lemma}
    \begin{proof}
    We argue by contradiction. Suppose that the conclusion is false. Then, there exist a constant $\varepsilon_0>0$, $t_0\in \mathbb{R}$ and a sequence of nonlinearities $\{A_k\}$ satisfying  \eqref{5hhconda} and \eqref{5hhcondc}, a sequence of domains $\{\Omega^k\}$, and a sequence of functions $\{v_k\}\subset C(t_0-1,t_0;L^2(\Omega^k\cap B_1))\cap L^2(t_0-1,t_0;H^1(\Omega^k\cap B_1))$ such that 
    \begin{equation}\label{5hh2110131}
              B^+_1\subset \Omega^k\cap B_1\subset B_1\cap \{x_N>-1/2 k\},
                                                              \end{equation}
\begin{equation}\label{5hh2110132}
                       \left\{ \begin{array}{l}
                   {(v_k)_t} - \operatorname{div}\left( {\overline{A}_{k,B_{1}}(t,\nabla v_k)} \right) = 0 ~~\text{ in }~\tilde{\Omega}^k_1(0),\\ 
              v_k = 0~~\text{ on }~~(\partial_{p}\tilde{\Omega}^k_1(0))\backslash (\Omega^k\times(-T,T)), \\ 
                      \end{array} \right.
                                                      \end{equation}
and the zero extension of each $v_k$ to $Q_1(0,t_0)$ satisfies   
\begin{align}\label{5hh2110134}
         &\fint_{Q_1(0,t_0)}|\nabla v_k|^2dxdt\leq 1~\text{ but }\\&
         \label{5hh2110135}
                    \fint_{Q_1(0,t_0)}|v_k-V_k|^2dxdt\geq \varepsilon_0^2,\end{align}                             
for any weak solution $V_k$ of 
 \begin{equation}\label{5hh2110133}
                                             \left\{ \begin{array}{l}
                                                  {(V_k)_t} - \operatorname{div}\left( {\overline{A}_{k,B_{1}}(t,\nabla V_k)} \right) = 0, ~\text{in }~Q_1^+(0,t_0),\\ 
                                                  V_k = 0~~\text{on}~~T_1(0,t_0). \\ 
                                                  \end{array} \right.
                                             \end{equation} 
By \eqref{5hh2110131} and \eqref{5hh2110134} and Poincare's inequality  it follows that 
\begin{align*}
	&||v_k||_{L^2(t_0-1,t_0;H^1(B_1))}\lesssim||\nabla v_k||_{L^2(Q_1(0,t_0)}\lesssim 1,\\&
	||(v_k)_t ||_{L^2(t_0-1,t_0;H^{-1}(B_1))}\lesssim  \int_{Q_{1}(0,t_0)}|\nabla v_k|^2dxdt\lesssim 1.
\end{align*}
Therefore, using Aubin$-$Lions Lemma, one can find $v_0$ and a subsequence of $\{v_k\}$, still denoted by $\{v_k\}$ such that 
\begin{equation*}
v_k\to v_0 \text{ weakly in }  L^2(t_0-1,t_0,H^1(B_1)) ~\text{ and strongly in }  L^2(t_0-1,t_0,L^2(B_1)),
\end{equation*}    
and
\begin{equation*}
(v_k)_t\to (v_0)_t~~\text{ weakly in }~ L^2(t_0-1,t_0,H^{-1}(B_1)). 
\end{equation*}
Moreover, $v_0=0$ in $Q^-_1(0,t_0):=(B_1\cap\{x_N<0\})\times (1-t_0,1)$ since $v_k=0$ on outside $\Omega^k\cap Q_1(0,t_0)$ for all $k$. \\
 To get a contradiction we take $V_k$ to be the unique solution of $ {(V_k)_t} - \operatorname{div}\left( {\overline{A}_{k,B_{1}}(t,\nabla V_k)} \right) = 0$ in $Q_1^+(0,t_0)$ and $V_k-v_0\in L^2(t_0-1,t_0,H^1_0(B^+_1))$ and $V_k(t_0-1)=v_0(t_0-1)$. As above, one can find $V_0$ and a subsequence of $\{V_k\}$, still denoted by $\{V_k\}$ such that 
    \begin{equation*}
    V_k\to V_0 \text{ weakly in }  L^2(t_0-1,t_0,H^1(B_1))~\text{ and strongly in }  L^2(t_0-1,t_0,L^2(B_1)),
    \end{equation*}    
    and
    \begin{equation*}
    (V_k)_t\to (V_0)_t~~\text{ weakly in }~ L^2(t_0-1,t_0,H^{-1}(B_1)), 
    \end{equation*}
    for some $V_0\in v_0 +L^2(t_0-1,t_0,H^1_0(B^+_1)$ and $V_0(t_0-1)=v_0(t_0-1)$.\\
   Thanks to \eqref{5hh2110135}, the proof would be complete if we could show that  $v_0=V_0$. In fact,\\
    Let $\mathcal{J}_k: X\to L^2(Q_1^+(0,t_0),\mathbb{R}^N)$ determine
    by $$
    \mathcal{J}_k(\phi(x,t))={\overline{A}_{k,B_{1}}(t,\nabla\phi(x,t))} ~\text{ for any }~\phi\in X,$$
    where $X\subset L^2(t_0-1,t_0,H^1(B_1) )$ is closures (in the strong topology of $L^2(t_0-1,t_0,H^1(B_1))$) of convex combinations of $\{v_k\}_{k\geq 1}\cup\{V_k\}_{k\geq 1}\cup\{0\}$.\\ Since $v_k,V_k$ converge weakly to $v_0,V_0$  in $ L^2(t_0-1,t_0,H^1(B_1)) $ resp.,  thus by  Mazur Theorem,  $X$ is compact subset of  $L^2(t_0-1,t_0,H^1(B_1))$ and $v_0,V_0\in X$. \\Thanks to \eqref{5hhconda} and \eqref{5hhcondc}, we get $\mathcal{J}_k(0)=0$ and
    \begin{align*}
    ||\mathcal{J}_k(\phi_1)-\mathcal{J}_k(\phi_2)||_{L^2(Q_1^+(0,t_0),\mathbb{R}^N)}\leq \Lambda_1 ||\phi_1-\phi_2||_{L^2(t_0-1,t_0,H^1(B_1))},
    \end{align*}
    for every  $\phi_1,\phi_2\in X$ and $k\in\mathbb{N}$.
    Thus, by Ascoli Theorem, there exist $\mathcal{J}\in C(X,L^2(Q_1^+(0,t_0),\mathbb{R}^N) )$ and a subsequence of $\{\mathcal{J}_k\}$, still denoted by it, such that 
    \begin{align}\label{5hh310320141}
    \sup_{\phi\in X}||\mathcal{J}_k(\phi)-\mathcal{J}(\phi)||_{L^2(Q_1^+(0,t_0),\mathbb{R}^N)}\to 0~~\text{ as }~ k\to\infty,
    \end{align}
    and also for any $\phi_1,\phi_2\in X$, 
    \begin{align}\label{5hh310320144}
    \int_{Q_1^+(0,t_0)}\left(\mathcal{J}(\phi_1)-\mathcal{J}(\phi_2)\right).\left(\nabla \phi_1-\nabla\phi_2\right)dxdt\geq \Lambda_2 |||\nabla \phi_1-\nabla\phi_2|||_{L^2(Q_1^+(0,t_0))}.
    \end{align}  
    From  \eqref{5hh2110132}, we deduce
    \begin{align*}
    &\int_{Q^+_1(0,t_0)}(v_k-V_k)_t (v_0-V_0) dxdt\\&~~~~~+\int_{Q^+_1(0,t_0)}\left(\overline{A}_{k,B_1}(t,\nabla v_k)-\overline{A}_{k,B_1}(t,\nabla V_k)\right).\nabla (v_0-V_0) dxdt =0.
    \end{align*}
    We have 
    \begin{align*}
    &\int_{Q^+_1(0,t_0)}|\overline{A}_{k,B_1}(\nabla v_k)|^2dxdt\lesssim \int_{Q^+_1(0,t_0)}|\nabla v_k|^2dxdt\lesssim 1,\\&
    \int_{Q^+_1(0,t_0)}|\overline{A}_{k,B_1}(\nabla V_k)|^2dxdt\lesssim \int_{Q^+_1(0,t_0)}|\nabla V_k|^2dxdt\lesssim 1.
    \end{align*}
    for every $k$.    
    \\
 Thus there exist  a subsequence, still denoted by    $\{\overline{A}_{k,B_1}(t,\nabla v_k),\overline{A}_{k,B_1}(t,\nabla V_k)\}$ and a vector field $A_1,A_2$ belonging to $L^2(Q^+_1(0,t_0),\mathbb{R}^N)$ such that 
    \begin{equation*}
    \overline{A}_{k,B_1}(t,\nabla v_k)\to A_1 ~\text{ and }~  \overline{A}_{k,B_1}(t,\nabla V_k)\to A_2,
    \end{equation*}
    weakly in $L^2(Q^+_1(0,t_0),\mathbb{R}^N)$.
    It follows
     \begin{align*}
        \int_{Q^+_1(0,t_0)}(v_0-V_0)_t (v_0-V_0) dxdt+\int_{Q^+_1(0,t_0)}(A_1-A_2).\nabla (v_0-V_0) dxdt =0.
        \end{align*}
        Since
        \begin{align*}
        \int_{Q^+_1(0,t_0)}(v_0-V_0)_t (v_0-V_0) dxdt= \int_{B_1^+(0)}(v_0-V_0)^2(t_0)dx\geq 0,
        \end{align*}
        we get 
        \begin{align}\label{5hh310320145}
        \int_{Q^+_1(0,t_0)}(A_1-A_2).\nabla (v_0-V_0) dxdt\leq 0.
        \end{align}
        For our purpose, we need to show that 
        \begin{align}
        \label{5hh310320142}
        &\int_{Q^+_1(0,t_0)}(A_1-\mathcal{J}(v_0)).\nabla (v_0-V_0) dxdt\geq 0,\\&
        \label{5hh310320143}
                        \int_{Q^+_1(0,t_0)}(A_2-\mathcal{J}(V_0)).\nabla (V_0-v_0) dxdt\geq 0. 
        \end{align}        
 To do this, we fix a function $g\in X$ and any $\varphi\in C^1_c(Q^+_1(0,t_0))$ such that $\varphi\geq 0$. We have
 \begin{align*}
 0&\leq \int_{Q^+_1(0,t_0)}\varphi \left(\overline{A}_{k,B_1}(t,\nabla v_k)-\overline{A}_{k,B_1}(t,\nabla g)\right)\left(\nabla  v_k-\nabla g\right)dxdt\\&= \int_{Q^+_1(0,t_0)}\varphi \overline{A}_{k,B_1}(t,\nabla v_k)\nabla  v_k dxdt -\int_{Q^+_1(0,t_0)}\varphi \overline{A}_{k,B_1}(t,\nabla v_k)\nabla  g dxdt\\&- \int_{Q^+_1(0,t_0)}\varphi \overline{A}_{k,B_1}(t,\nabla g)\left(\nabla  v_k-\nabla g\right) dxdt
  \\&:=L_1+L_2+L_3.
 \end{align*}   
 It is easy to see that 
 \begin{equation*}
 \lim_{k\to\infty} L_2=  -\int_{Q^+_1(0,t_0)}\varphi A_1\nabla  g dxdt~~\text{ and }~~\lim_{k\to\infty} L_3=- \int_{Q^+_1(0,t_0)}\varphi \mathcal{J}( g)\left(\nabla  v_0-\nabla g\right) dxdt .
 \end{equation*}  
 Moreover, we have 
$$
 L_1=\frac{1}{2}\int_{Q^+_1(0,t_0)}v_k^2\varphi_t dxdt -\int_{Q^+_1(0,t_0)} \overline{A}_{k,B_1}(\nabla v_k)\nabla \varphi  v_kdxdt .$$
 Thus, 
 \begin{align*}
  \lim_{k\to\infty} L_1&=\frac{1}{2}\int_{Q^+_1(0,t_0)}v_0^2\varphi_tdxdt -\int_{Q^+_1(0,t_0)} A_1\nabla \varphi  v_0dxdt\\&=
  -\int_{Q^+_1(0,t_0)}(v_0)_t\varphi v_0dxdt -\int_{Q^+_1(0,t_0)} A_1\nabla (\varphi  v_0)dxdt +\int_{Q^+_1(0,t_0)}\varphi A_1\nabla  v_0dxdt\\&=\int_{Q^+_1(0,t_0)}\varphi A_1\nabla  v_0dxdt.
 \end{align*}
 Hence, 
$$
 0\leq \int_{Q^+_1(0,t_0)}\varphi \left(A_1-\mathcal{J}(g)\right)\left(\nabla  v_0-\nabla g\right)dxdt$$
 holds for all $\varphi \in  C^1_c(Q^+_1(0,t_0))$, $\varphi\geq 0$ and $g\in X$.   
 Now we choose $g=v_0-\xi (v_0-V_0)=(1-\xi)v_0+\xi V_0\in X$ for $\xi\in (0,1)$, 
 so $$
   0\leq \int_{Q^+_1(0,t_0)}\varphi \left(A-\mathcal{J}(v_0-\xi (v_0-V_0))\right)\left(\nabla  v_0-\nabla V_0\right)dxdt.$$
   Letting $\xi\to 0^+$ and $\varphi\to \chi_{Q_1^+(0,t_0)}$, we get \eqref{5hh310320142}. Similarly, we also obtain \eqref{5hh310320143}.
   \\
   Thus,
$$
           \int_{Q^+_1(0,t_0)}(A_1-A_2)\nabla (v_0-V_0) dxdt \geq \int_{Q^+_1(0,t_0)}(\mathcal{J}(v_0)-\mathcal{J}(V_0))\nabla (v_0-V_0) dxdt.        $$
 Combining this with  \eqref{5hh310320144}, \eqref{5hh310320145} and $v_0-V_0\in L^2(t_0-1,t_0,H^1_0(B^+_1))$ yield $v_0=V_0$.
 This completes the proof.
    \end{proof}
    \begin{lemma}\label{5hh21101313} For any $\varepsilon>0$ there exists a small $\delta=\delta(N,\Lambda_1,\Lambda_2,\varepsilon)>0$ such that if $v\in C(t_0-1,t_0;L^2(\Omega\cap B_1))\cap L^2(t_0-1,t_0;H^1(\Omega\cap B_1))$ is a solution of \eqref{5hh1610135} and \eqref{5hh1610136} is satisfied and 
                    \begin{equation}\label{5hh240320145}
                    \fint_{Q_1(0,t_0)}|\nabla v|^2dxdt\leq 1,
                    \end{equation}
                    then there exists a weak solution $V\in C(t_0-1,t_0;L^2( B_1^+))\cap L^2(t_0-1,t_0;H^1( B_1^+))$ of  \eqref{5hh1610134} with $\rho=1$, whose zero extension to $Q_1(0,t_0)$ satisfies 
                    \begin{align}
                    \label{5hh21101310}
                & ||\nabla V||_{L^\infty(Q_{1/4}(0,t_0))}\lesssim 1,\\&\label{5hh21101311}
         \fint_{Q_{1/8}(0,t_0)}|\nabla v-\nabla V|^2dxdt\leq \varepsilon^2. \end{align}                   
                        
                    \end{lemma}
                    \begin{proof}
Given $\varepsilon_1\in(0,1)$ by applying Lemma   \ref{5hh2110138},  one finds a small $\delta=\delta(N,\Lambda_1,\Lambda_2,\varepsilon_1)>0$ and a weak solution $V\in C(t_0-1,t_0;L^2( B_1^+(0)))\cap L^2(t_0-1,t_0;H^1( B_1^+(0)))$ of  \eqref{5hh1610134} with $\rho=1$ such that 
\begin{equation}\label{5hh240320146}
       \fint_{Q_1(0,t_0)}|v-V|^2dxdt\leq \varepsilon_1^2,
        \end{equation} 
Using $\phi^2V$ with $\phi\in C^\infty_c(B_1\times (t_0-1,t_0])$, $0\leq \phi \leq 1$  and $\phi=1$ in $Q_{1/2}(0,t_0)$ as test function in \eqref{5hh1610134}, we obtain 
\begin{equation*}
\int_{Q_{1/2}(0,t_0)}|\nabla V|^2dxdt\lesssim \int_{Q_{1}(0,t_0)}|V|^2dxdt.
\end{equation*}  
This implies
\begin{align*}
\int_{Q_{1/2}(0,t_0)}|\nabla V|^2dxdt\lesssim \int_{Q_{1}(0,t_0)}\left(|v-V|^2+|\nabla v|^2\right)dxdt\lesssim 1,
\end{align*}
since \eqref{5hh240320145}, \eqref{5hh240320146} and Poincar\'e's inequality.
Thus, using Lemma \ref{5hh2110139} we get \eqref{5hh21101310}. \\
 Next, we will prove \eqref{5hh21101311}. By Lemma \ref{5hh21101312}, the zero extension of $V$ to $Q_{1}(0,t_0)$ satisfies
    \begin{equation*}
           {V_t} - \operatorname{div}\left( {\overline{A}_{B_{1}}(t,\nabla V)} \right) = \frac{\partial F}{\partial x_N}~~\text{ in weakly }~Q_1(0,t_0).
                                           \end{equation*} 
where $F(x,t)=\chi_{x_N<0}\overline{A}_{B_{\rho}}^N(t, \nabla V(x',0,t)).$ Thus, we can write 
\begin{align*}
&\int_{\tilde{\Omega}_1(0,t_0)}(V-v)_t\varphi dxdt\\&~~~~+\int_{\tilde{\Omega}_1(0,t_0)} \left({\overline{A}_{B_{1}}(t,\nabla V)}-{\overline{A}_{B_{1}}(t,\nabla v)}\right)\nabla \varphi dxdt=-\int_{\tilde{\Omega}_1(0,t_0)}F\frac{\partial \varphi}{\partial x_N}dxdt,
\end{align*}                                                         for any $\varphi\in L^2(t_0-1,t_0,H^1_0(\Omega\cap B_1)) $.\\
We take $\varphi=\phi^2 (V-v)$ where $\phi\in C^\infty_c(B_{1/4}\times(t_0-(1/4)^2,t_0])$ , $0\leq \phi \leq 1$ and $\phi=1$ on $\overline{Q}_{1/8}(0,t_0)$, so 
\begin{align*}
&\int_{\tilde{\Omega}_1(0,t_0)}\phi^2 \left({\overline{A}_{B_{1}}(t,\nabla V)}-{\overline{A}_{B_{1}}(t,\nabla v)}\right)\left(\nabla V-\nabla v\right)dxdt\\&~~~~~~=-2\int_{\tilde{\Omega}_1(0,t_0)}\phi(V-v) \left({\overline{A}_{B_{1}}(t,\nabla V)}-{\overline{A}_{B_{1}}(t,\nabla v)}\right)\nabla \phi dxdt
\\&~~~~~~~ -\int_{\tilde{\Omega}_1(0,t_0)}\phi^2(V-v)_t(V-v) dxdt 
\\&~~~~~~~-\int_{\tilde{\Omega}_1(0,t_0)}\left(\phi^2F \frac{\partial (V-v)}{\partial x_N}+2\phi F(V-v)\frac{\partial \phi}{\partial x_N}\right)dxdt.
\end{align*}
We can rewrite $I_1=I_2+I_3+I_4$.\\
One has
\begin{equation*}
I_1\gtrsim\int_{\tilde{\Omega}_1(0,t_0)}\phi^2|\nabla V-\nabla v|^2 dxdt
\end{equation*}
and using H\"older's inequality
\begin{align*}
|I_2|&\lesssim \int_{\tilde{\Omega}_1(0,t_0)}\phi|V-v|(|\nabla V|+|\nabla v|)|\nabla \phi| dxdt
\\&\leq  \varepsilon_2 \int_{\tilde{\Omega}_1(0,t_0)}\phi^2(|\nabla V|^2+|\nabla v|^2)dxdt+c(\varepsilon_2)\int_{\tilde{\Omega}_1(0,t_0)}|V-v|^2|\nabla \phi|^2dxdt.
\end{align*}
Similarly, we also have 
\begin{align*}
|I_4|&\leq \varepsilon_2\int_{\tilde{\Omega}_1(0,t_0)}\phi^2(|\nabla V|^2+|\nabla v|^2)dxdt+c(\varepsilon_2)\int_{\tilde{\Omega}_1(0,t_0)}|V-v|^2|\nabla \phi|^2dxdt\\&~~~~~~~~~~+c(\varepsilon_2)\int_{\tilde{\Omega}_1(0,t_0)}|F|^2\phi^2 dxdt,
\end{align*}
and 
$$
I_3\leq  \int_{\tilde{\Omega}_1(0,t_0)} \phi_t\phi(V-v)^2dxdt
\lesssim \int_{\tilde{\Omega}_{1/4}(0,t_0)}|V-v|^2dxdt.$$
Hence,
\begin{align*}
\int_{\tilde{\Omega}_{1/8}(0,t_0)}|\nabla V-\nabla v|^2& \lesssim \varepsilon_2\int_{\tilde{\Omega}_{1/4}(0,t_0)}(|\nabla V|^2+|\nabla v|^2)+c(\varepsilon_2)\int_{\tilde{\Omega}_{1/4}(0,t_0)}(|V-v|^2+|F|^2) \\&~\lesssim\varepsilon_2 +c(\varepsilon_2)\left(\varepsilon_1^2+\int_{\tilde{\Omega}_{1/4}(0,t_0)\cap\{-4\delta<x_N<0\}}|\nabla V(x',0,t)|^2dxdt\right)
\\&\lesssim \varepsilon_2 +c(\varepsilon_2)\left(\varepsilon_1^2+\delta\right).
\end{align*}
Finally, for any $\varepsilon>0$ by choosing $\varepsilon_2,\varepsilon_1$ and $\delta$ appropriately  we get \eqref{5hh21101311}.  The proof  is complete.
                    \end{proof}
                    
    \begin{lemma}\label{5hh21101314}
        For any $\varepsilon>0$ there exists a small $\delta=\delta(N,\Lambda_1,\Lambda_2,\varepsilon)>0$ such that if $v\in C(t_0-\rho^2,t_0;L^2(\Omega\cap B_\rho(0)))\cap L^2(t_0-\rho^2,t_0;H^1(\Omega\cap B_\rho(0)))$ is a solution of 
\begin{equation}\label{5hh1610137}
                                  \left\{ \begin{array}{l}
                                           {v_t} - \operatorname{div}\left( {\overline{A}_{B_{\rho}(0)}(t,\nabla v)} \right) = 0 \;in\;\tilde{\Omega}_\rho(0),\\ 
                                                           v = 0\quad \quad on~~\partial_{p}\tilde{\Omega}_\rho(0)\backslash (\Omega\times (-T,T)), \\ 
                                                           \end{array} \right.
                                                      \end{equation}
                                                      and  
                                                      \begin{equation}\label{5hh1610138}
                                                          B^+_\rho(0)\subset \Omega\cap B_\rho(0)\subset B_\rho(0)\cap \{x_N>-4\rho\delta\}.
                                                          \end{equation}
 then there exists a weak solution $V\in C(t_0-\rho^2,t_0;L^2( B_\rho^+(0)))\cap L^2(t_0-\rho^2,t_0;H^1( B_\rho^+(0)))$ of  \eqref{5hh1610134}, whose zero extension to $Q_1(0,t_0)$ satisfies 
 \begin{align}
 &||\nabla V||^2_{L^\infty(Q_{\rho/4}(0,t_0))}\lesssim \fint_{Q_{\rho}(0,t_0)} |\nabla v|^2dxdt~\text{ and }\\&
  \fint_{Q_{\rho/8}(0,t_0)}|\nabla v-\nabla V|^2dxdt\leq \varepsilon^2\fint_{Q_{\rho}(0,t_0)} |\nabla v|^2dxdt.
 \end{align}
        \end{lemma}
        \begin{proof}
  We set
  \begin{align*}
   \mathcal{A}(x,t,\xi)=A(\rho x,t_0+\rho^2(t-t_0),\kappa  \xi)/\kappa~\text{ and }
   \tilde{v}(x,t)=v(\rho x,t_0+\rho^2(t-t_0))/(\rho \kappa)
  \end{align*} 
  where $\kappa=\left(\frac{1}{|Q_{\rho}(0,t_0)|}\int_{Q_{\rho}(0,t_0)} |\nabla v|^2dxdt\right)^{1/2}$. 
  Then $\mathcal{A}$ satisfies conditions \eqref{5hhconda} and \eqref{5hhcondc}  with the same constants $\Lambda_1$ and $\Lambda_2$. Moreover, $\tilde{v}$ is a solution of   
  \begin{equation}
                   \left\{ \begin{array}{l}
                                   {\tilde{v}_t} - \operatorname{div}\left( {\overline{\mathcal{A}}_{B_{1}(0)}(t,\nabla \tilde{v})} \right) = 0 ~~\text{ in }~\tilde{\Omega}_1^\rho(0)\\ 
                                                             \tilde{v} = 0~~~~~\text{on }~~\left((\partial\Omega^\rho\cap B_1(0))\times(t_0-1,t_0)\right) \cup \left((\Omega^\rho\cap B_1(0))\times\{t=t_0-1\}\right) \\ 
                                                             \end{array} \right.
                                                        \end{equation}
                                                        where $\Omega^\rho=\{z=x/\rho:x\in\Omega\}$ 
and satisfies
       $
       \fint_{Q_{1}(0,t_0)} |\nabla \tilde{v}|^2dxdt=1.
      $
 We also have
$$ B^+_1(0)\subset \Omega^\rho\cap B_1(0)\subset B_1(0)\cap \{x_N>-4\delta\}.$$          Therefore, applying Lemma \ref{5hh21101313} for any $\varepsilon>0$, there exist a constant $\delta=\delta(N,\Lambda_1,\Lambda_2,\varepsilon)>0$ and   $\tilde{V}$ satisfying
$$
||\nabla \tilde{V}||_{L^\infty(Q_{1/4}(0,t_0))}\lesssim 1~~\text{and}\quad\fint_{Q_{1/8}(0,t_0)}|\nabla \tilde{v}-\nabla \tilde{V}|^2dxdt\leq \varepsilon^2.
$$  We complete the proof by choosing $V(x,t)=k\rho \tilde{V}(x/\rho,t_0+(t-t_0)/\rho^2)$.                      
        \end{proof}
       
\begin{lemma}\label{5hh16101310}
 Let $s_2$ be as in Lemma \ref{5hh1610139}. For any $\varepsilon\in (0,1)$ there exists a small $\delta=\delta(N,\Lambda_1,\Lambda_2,\varepsilon)>0$ such that the following holds. If $\Omega$ is a $(\delta,R_0)$-Reifenberg flat domain and $u\in C(0,T;L^2(\Omega))\cap L^2(0,T;H^1(\Omega))$ is a solution to equation \eqref{5hhparabolic1} with $\mu\in L^2(\Omega\times (-T,T))$ and $u(-T)=0$, for  $x_0\in\partial\Omega$, $-T<t_0<T$ and $0<R<R_0/6$ then there is a function $V\in L^\infty(t_0-(R/9)^2,t_0;W^{1,\infty}( B_{R/9}(x_0)))$ such that 
 \begin{equation}\label{5hh21101317}
 ||\nabla V||_{L^\infty(Q_{R/9}(x_0,t_0))}\lesssim \fint_{Q_{6R}(x_0,t_0)}|\nabla u|dxdt+\frac{|\mu|(Q_{6R}(x_0,t_0))}{R^{N+1}},
 \end{equation}
 and 
 \begin{align}
 \nonumber&\fint_{Q_{R/9}(x_0,t_0)}|\nabla u-\nabla V|dxdt\\&~~~~\lesssim (\varepsilon+[A]_{s_2}^{R_0})\fint_{Q_{6R}(x_0,t_0)}|\nabla u|dxdt+ (1+[A]_{s_2}^{R_0})\frac{|\mu|(Q_{6R}(x_0,t_0))}{R^{N+1}}.\label{5hh21101318}
 \end{align}

 \end{lemma}
    \begin{proof}
 Let $x_0\in \partial \Omega$, $-T<t_0<T$ and $\rho=R(1-\delta)$, we may assume that $0\in \Omega$, $x_0=(0,...,-\delta\rho/(1-\delta))$
 and 
    \begin{equation}
                                     B^+_\rho(0)\subset \Omega\cap B_\rho(0)\subset B_\rho(0)\cap \{x_N>-4\rho\delta\}.
                                                              \end{equation}
We also have 
\begin{equation}\label{5hh090520141}
Q_{R/9}(x_0,t_0)\subset Q_{\rho/8}(0,t_0)\subset Q_{\rho/4}(0,t_0)\subset Q_{\rho}(0,t_0)\subset Q_{6\rho}(0,t_0)\subset Q_{6R}(x_0,t_0),
\end{equation}                                                         provided that $0<\delta<1/625$.\\
Let $w$ and $v$ be in Theorem \ref{5hh24093} and Lemma  \ref{5hh1610139}. 
    By Lemma  \ref{5hh21101314} for any $\varepsilon>0$                                                          we can find a small positive $\delta=\delta(N,\alpha,\beta,\varepsilon)<1/625$ such that there is a function $V\in L^\infty(t_0-\rho^2,t_0;W^{1,\infty}( B_{\rho}(0)))$ satisfying 
$$ ||\nabla V||^2_{L^\infty(Q_{\rho/4}(0,t_0))}\lesssim \fint_{Q_{\rho}(0,t_0)} |\nabla v|^2,\quad 
   \fint_{Q_{\rho/8}(0,t_0)}|\nabla v-\nabla V|^2\leq \varepsilon^2\fint_{Q_{\rho}(0,t_0)} |\nabla v|^2.$$
        Then, by \eqref{5hh1610133} in Lemma  \ref{5hh1610139} and \eqref{5hhineq8} in Theorem \ref{5hh24093}  and \eqref{5hh090520141} we get 
       \begin{align}
       &	||\nabla V||_{L^\infty(Q_{R/9}(x_0,t_0))} \lesssim \left(\fint_{Q_{\rho}(0,t_0)} |\nabla w|^2\right)^{1/2}\lesssim \fint_{Q_{6R}(x_0,t_0)} |\nabla w|\label{5hh21101316},\\& \fint_{Q_{\rho/8}(0,t_0)}|\nabla v-\nabla V|\lesssim\varepsilon \left(\fint_{Q_{\rho}(0,t_0)} |\nabla w|^2\right)^{1/2}                
       \lesssim\varepsilon \fint_{Q_{6R}(x_0,t_0)} |\nabla w|.\label{5hh21101320}
       \end{align}
        Therefore, from \eqref{5hhineq9} in Theorem  \ref{5hh24093} and \eqref{5hh21101316} we get \eqref{5hh21101317}.\\
         Now we prove \eqref{5hh21101318}. One has
         \begin{align*}
         &\fint_{Q_{R/9}(x_0,t_0)}|\nabla u-\nabla V|\lesssim\fint_{Q_{\rho/8}(0,t_0)}|\nabla u-\nabla V|
                    \\&\lesssim \fint_{Q_{\rho/8}(0,t_0)}|\nabla u-\nabla w|+\fint_{Q_{\rho/8}(0,t_0)}|\nabla w-\nabla v|+\fint_{Q_{\rho/8}(0,t_0)}|\nabla v-\nabla V|.
         \end{align*}
        
           From Lemma \ref{5hh1610139} and Theorem \ref{5hh24093} and \eqref{5hh21101320} it follows that \begin{align*}
           \fint_{Q_{\rho/8}(0,t_0)}|\nabla u-\nabla w|&\lesssim \frac{|\mu|(Q_{6R}(x_0,t_0))}{R^{N+1}},
           \\\fint_{Q_{\rho/8}(0,t_0)}|\nabla v-\nabla w| &\lesssim[A]_{s_2}^{R_0} \fint_{Q_{6\rho}(0,t_0)}|\nabla w|\lesssim[A]_{s_2}^{R_0} \fint_{Q_{6 R}(x_0,t_0)}|\nabla w|\\&\lesssim[A]_{s_2}^{R_0}\left(\fint_{Q_{6 R}(x_0,t_0)}|\nabla u|+\frac{|\mu|(Q_{6R}(x_0,t_0))}{R^{N+1}}\right),
           \end{align*}
$$
          \fint_{Q_{\rho/8}(0,t_0)}|\nabla v-\nabla V|  \lesssim\varepsilon \fint_{Q_{6R}(x_0,t_0)} |\nabla w| \lesssim\varepsilon \left(\fint_{Q_{6 R}(x_0,t_0)}|\nabla u|+\frac{|\mu|(Q_{6R}(x_0,t_0))}{R^{N+1}}\right).$$
           Hence we get \eqref{5hh21101318}. The proof is complete.

                         \end{proof}  
    \section{Global Integral Gradient Bounds for Parabolic equations }
    \subsection{Global estimates on 2-Capacity uniform thickness domains }
   We use the Theorem \ref{5hh1510135} and  \ref{5hh24093}  to prove the following theorem.
  \begin{theorem}\label{5hh1510139} Suppose that $\mathbb{R}^N\backslash\Omega$ satisfies a uniformly $2-$thick condition with constants $c_0,r_0$. Let $\theta_1,\theta_2$ be in Theorem \ref{5hh1510135} and \ref{5hh24093}. Set  $\theta=\min\{\theta_1,\theta_2\}>2$ and $T_0=diam(\Omega)+T^{1/2}$,  $Q=B_{\text{diam}(\Omega)}(x_0)\times(0,T)$. Let $B_1=\tilde{Q}_{R_1}(y_0,s_0)$, $B_2=4B_1:=\tilde{Q}_{4R_1}(y_0,s_0)$  for $R_1>0$. For $\mu\in\mathfrak{M}_b(\Omega_T)$, $\sigma\in\mathfrak{M}_b(\Omega)$, set $\omega=|\mu|+|\sigma|\otimes\delta_{\{t=0\}}$,  there exist a distributional solution $u$ of  equation \eqref{5hhparabolic1} with data $\mu$, $u_0=\sigma$ and constants $\varepsilon_1=\varepsilon_1(N,\Lambda_1,\Lambda_2,c_0,T_0/r_0), \varepsilon_2=\varepsilon_2(N,\Lambda_1,\Lambda_2,c_0)>0$  such that  
  \begin{equation}\label{5hh15101310}
  |\{\mathbb{M}(|\nabla u|)>\varepsilon^{-1/\theta}\lambda, \mathbb{M}_{1}[\omega]\leq \varepsilon^{1-\frac{1}{\theta}}\lambda\}\cap Q|\lesssim_{c_0,T_0/r_0} \varepsilon |\{\mathbb{M}(|\nabla u|)>\lambda \}\cap Q|,
  \end{equation}
  for all $\lambda>0,\varepsilon\in(0,\varepsilon_1)$
  and
  \begin{equation}\label{5hh070120141}
    |\{\mathbb{M}(\chi_{B_2}|\nabla u|)>\varepsilon^{-1/\theta}\lambda, \mathbb{M}_{1}[\chi_{B_2}\omega]\leq \varepsilon^{1-\frac{1}{\theta}}\lambda\}\cap B_1|\lesssim_{c_0,T_0/r_0} \varepsilon |\{\mathbb{M}(\chi_{B_2}|\nabla u|)>\lambda \}\cap B_1|,~~
    \end{equation}
    for all $\lambda> \varepsilon^{-1+\frac{1}{\theta}}||\nabla u||_{L^1(\Omega_T\cap B_2)}R_2^{-N-2}$, $\varepsilon\in (0,\varepsilon_2)$ with $R_2=\inf\{r_0,R_1\}/16$.\\  
    Moreover, if $\sigma\in L^1(\Omega)$ then $u$ is a renormalized solution. 
  \end{theorem} 
  \begin{proof}[Proof of Theorem \ref{5hh1510139}] Let $\{\mu_n\}\subset C_c^\infty(\Omega_T), \{\sigma_n\}\subset C_c^\infty(\Omega)$ be as in the proof of Theorem  \ref{5hh141013112}. We have $|\mu_n|\leq \varphi_n*|\mu|$ and $|\sigma_n|\leq \varphi_{1,n}*|\sigma|$ for any $n\in\mathbb{N}$, where $\{\varphi_n\}, \{\varphi_{1,n}\}$ are sequences of standard mollifiers in $\mathbb{R}^{N+1},\mathbb{R}^N$, respectively.\\ 
                                   Let $u_n$ be solution of equation
                                   \begin{equation}
                                            \left\{
                                            \begin{array}
                                            [c]{l}%
                                            {(u_{n})_t}-\operatorname{div}(A(x,t,\nabla u_n))=\mu_{n}~~\text{in }\Omega_T,\\
                                            {u}_{n}=0\qquad\text{on }\partial\Omega\times (0,T),\\
                                            u_n(0)=\sigma_n~~~\text{in}~~\Omega.
                                            \end{array}
                                            \right.  
                                            \end{equation}                                              
  By Proposition \ref{5hhmun} and Theorem \ref{5hhsta}, there exists a subsequence of $\{u_n\}$, still denoted by $\{u_n\}$, converging to a distributional solution $u$ of  \eqref{5hhparabolic1} with data $\mu\in\mathfrak{M}_b(\Omega_T)$ and $u_0=\sigma$ such that $ u_n\to  u$ in  $L^s(0,T,W^{1,s}_0(\Omega))$ for any $s\in\left[1,\frac{N+2}{N+1}\right)$ and if $\sigma\in L^1(\Omega)$ then $u$ is a renormalized solution. \\
  By Remark \ref{5hh240220142} and Theorem \ref{5hhsta}, a sequence $\{u_{n,m}\}_m$ of solutions to equations
                                               \begin{equation*}
                                                     \left\{
                                                  \begin{array}
                                                                           [c]{l}%
                                                                                  {(u_{n,m})_t}-\operatorname{div}(A(x,t,\nabla u_{n,m}))=\mu_{n,m}~~\text{in }\Omega\times (-T,T),\\
                                                                                  u_{n,m}=0~~~\text{on}~\partial \Omega\times (-T,T),\\
                                                                                  {u}_{n,m}(-T)=0~~\text{on }\Omega,
                                                                                  \end{array}
                                                                                  \right.  
                             \end{equation*}   
                              converges to $\chi_{\Omega_T}u_{n}$ in $L^s(-T,T,W^{1,s}_0(\Omega))$ for any $s\in\left[1,\frac{N+2}{N+1}\right)$,                                          where $\mu_{n,m}=\left(g_{n,m}\right)_t+\chi_{\Omega_T}\mu_{n}$,  $g_{n,m}(x,t)=\sigma_n(x)\int_{-T}^{t}\varphi_{2,m}(s)ds$ and $\{\varphi_{2,m}\}$ is a sequence of mollifiers in $\mathbb{R}$.\medskip\\     
  Set 
  \begin{align*}
  & E^1_{\lambda,\varepsilon}=\{\mathbb{M}(|\nabla u|)>\varepsilon^{-1/\theta}\lambda, \mathbb{M}_{1}[\omega]\leq \varepsilon^{1-\frac{1}{\theta}}\lambda\}\cap Q,~~F^1_\lambda=\{\mathbb{M}(|\nabla u|)>\lambda \}\cap Q,\\&
  E^2_{\lambda,\varepsilon}=\{\mathbb{M}(\chi_{B_2}|\nabla u|)>\varepsilon^{-1/\theta}\lambda, \mathbb{M}_{1}[\chi_{B_2}\omega]\leq \varepsilon^{1-\frac{1}{\theta}}\lambda\}\cap B_1, ~~F^2_\lambda=\{\mathbb{M}(\chi_{B_2}|\nabla u|)>\lambda \}\cap B_1,
  \end{align*}
  for $\varepsilon\in (0,1)$ and $\lambda>0$.\\
  We verify that
  \begin{equation}\label{5hh1510133}
  |E^1_{\lambda,\varepsilon}|\lesssim_{T_0/r_0}\varepsilon |\tilde{Q}_{R_3}| ~~\forall ~\lambda>0, \varepsilon\in (0,1),
  \end{equation}
 \begin{equation}\label{5hh1510133'}
     |E^2_{\lambda,\varepsilon}|\lesssim\varepsilon |\tilde{Q}_{R_2}| ~~\forall ~~\lambda> \varepsilon^{-1+\frac{1}{\theta}}||\nabla u||_{L^1(\Omega_T\cap A)}R_2^{-N-2}, \varepsilon\in (0,1)
     \end{equation}
  with  $R_3=\inf\{r_0,T_0\}/16$.\\
  In fact, we can assume that $E^1_{\lambda,\varepsilon}\not=\emptyset$. So, $|\mu|(\Omega_T)+|\sigma|(\Omega)\leq T_0^{N+1}\varepsilon^{1-\frac{1}{\theta}}\lambda$. We have
  $$
|E^1_{\lambda,\varepsilon}|\lesssim \frac{1}{\varepsilon^{-1/\theta}\lambda}\int_{\Omega_{T}}|\nabla u|dxdt.$$
By Remark \ref{5hh070420143},
  $
   \int_{\Omega_T}|\nabla u_n|dxdt\lesssim T_0 \left(|\mu_n|(\Omega_T)+|\sigma_n|(\Omega)\right)
$ for all $n$.
   Letting $n\to\infty$ we get 
  $
     \int_{\Omega_T}|\nabla u|dxdt\lesssim T_0 \left(|\mu|(\Omega_T)+|\sigma|(\Omega)\right)
    $. 
 Thus, $$
  |E^1_{\lambda,\varepsilon}|\lesssim \frac{1}{\varepsilon^{-1/\theta}\lambda} T_0 \left(|\mu|(\Omega_T)+|\sigma|(\Omega)\right) \lesssim \frac{1}{\varepsilon^{-1/\theta}\lambda} T_0^{N+2} \varepsilon^{1-\frac{1}{\theta}}\lambda = c\varepsilon |\tilde{Q}_{R_3}|.$$
 Hence, \eqref{5hh1510133} holds.\\
 For any $\lambda> \varepsilon^{-1+\frac{1}{\theta}}||\nabla u||_{L^1(\Omega_T\cap B_2)}R_2^{-N-2}$ we have 
$$
 |E^2_{\lambda,\varepsilon}|\lesssim \frac{1}{\varepsilon^{-1/\theta}\lambda}\int_{\Omega_{T}}\chi_{B_2}|\nabla u|dxdt \leq   c \varepsilon |\tilde{Q}_{R_2}|.$$
 Hence, \eqref{5hh1510133'} holds.\medskip\\
 Next we verify that for all $(x,t)\in Q$, $r\in (0,R_3]$ and $\lambda>0,
 \varepsilon\in (0,1)$, we have
 $
    \tilde{Q}_r(x,t)\cap Q\subset F_\lambda^1
  $
    if $
       |E^1_{\lambda,\varepsilon}\cap \tilde{Q}_r(x,t)|\geq c\varepsilon |\tilde{Q}_r(x,t)|
      $
       where the constant $c$ does not depend on $\lambda$ and $\varepsilon$.
       Indeed,
 take $(x,t)\in Q$ and $0<r\leq R_3$.
            Now assume that $\tilde{Q}_r(x,t)\cap Q\cap (\mathbb{R}^{N+1}\backslash F^1_\lambda)\not= \emptyset$ and $E^1_{\lambda,\varepsilon}\cap \tilde{Q}_r(x,t)\not = \emptyset$ i.e, there exist $(x_1,t_1),(x_2,t_2)\in \tilde{Q}_r(x,t)\cap Q$ such that $\mathbb{M}(|\nabla u|)(x_1,t_1)\leq \lambda$ and $\mathbb{M}_1[\omega](x_2,t_2)\le \varepsilon^{1-\frac{1}{\theta}} \lambda$.  
             We need to prove that
             \begin{equation}\label{5hh1510134}
                    |E^1_{\lambda,\varepsilon}\cap \tilde{Q}_r(x,t))|<c \varepsilon |\tilde{Q}_r(x,t)| 
                                     \end{equation}
                                   Obviously, we have for all $(y,s)\in \tilde{Q}_r(x,t)$ there holds 
$$
                                     \mathbb{M}(|\nabla u|)(y,s)\leq \max\left\{\mathbb{M}\left(\chi_{\tilde{Q}_{2r}(x,t)}|\nabla u|\right)(y,s),3^{N+2}\lambda\right\}.
$$
 So, for all $\lambda>0$ and $\varepsilon\in(0,\varepsilon_0)$ with $\varepsilon_0\leq 3^{-(N+2)\theta}$,
          \begin{align}
         E^1_{\lambda,\varepsilon}\cap \tilde{Q}_r(x,t)=\left\{\mathbb{M}\left(\chi_{\tilde{Q}_{2r}(x,t)}|\nabla u|\right)>\varepsilon^{-1/\theta}\lambda, \mathbb{M}_{1}[\omega]\leq \varepsilon^{1-\frac{1}{\theta}}\lambda\right\}\cap Q \cap \tilde{Q}_r(x,t).\label{5hh1510138}
          \end{align}       
          In particular, $E^1_{\lambda,\varepsilon}\cap \tilde{Q}_r(x,t)=\emptyset$ if $\overline{B}_{4r}(x)\subset\subset \mathbb{R}^{N}\backslash \Omega$. 
          Thus, it is enough to consider the case $B_{4r}(x)\subset\subset\Omega$ and $B_{4r}(x)\cap\Omega\not=\emptyset$.\\          
          We consider the case  $B_{4r}(x)\subset\subset\Omega$. Let $w_{n,m}$ be as in Theorem \ref{5hh1510135} with $Q_{2R}=Q_{4r}(x,t_0)$ and $u=u_{n,m}$ where $t_0=\min\{t+2r^2,T\}$. We have
          \begin{align}
          \label{5hh1510136}
                           &\fint_{Q_{4r}(x,t_0)}|\nabla u_{n,m}-\nabla w_{n,m}|\lesssim \frac{|\mu_{n,m}|(Q_{4r}(x,t_0))}{r^{N+1}},\\&
                          \fint_{Q_{2r}(x,t_0)}|\nabla w_{n,m}|^\theta \lesssim\left(\fint_{Q_{4r}(x,t_0)}|\nabla w_{n,m}|\right)^\theta. \label{5hh1510137} 
          \end{align}                                         
         From  \eqref{5hh1510138}, we have
         \begin{align*}
        & |E^1_{\lambda,\varepsilon}\cap \tilde{Q}_r(x,t)|\leq  |\{\{\mathbb{M}\left(\chi_{\tilde{Q}_{2r}(x,t)}|\nabla w_{n,m}|\right)>\varepsilon^{-1/\theta}\lambda/4\}\cap \tilde{Q}_r(x,t)\}| \\&~~~~~~~~~+|\{\mathbb{M}\left(\chi_{\tilde{Q}_{2r}(x,t)}|\nabla u_{n,m}-\nabla w_{n,m}|\right)>\varepsilon^{-1/\theta}\lambda/4\} \cap \tilde{Q}_r(x,t)|
                   \\&~~~~~~~~~+|\{\mathbb{M}\left(\chi_{\tilde{Q}_{2r}(x,t)}|\nabla u_{n,m}-\nabla u_{n}|\right)>\varepsilon^{-1/\theta}\lambda/4\} \cap \tilde{Q}_r(x,t)|
                   \\&~~~~~~~~~+|\{\mathbb{M}\left(\chi_{\tilde{Q}_{2r}(x,t)}|\nabla u_{n}-\nabla u|\right)>\varepsilon^{-1/\theta}\lambda/4\} \cap \tilde{Q}_r(x,t)|\\&~~~~~~~\lesssim \varepsilon \lambda^{-\theta}\int_{\tilde{Q}_{2r}(x,t)}|\nabla w_{n,m}|^\theta + \varepsilon^{1/\theta}\lambda^{-1}\int_{\tilde{Q}_{2r}(x,t)}|\nabla u_{n,m}-\nabla w_{n,m}|\\&~~~~~~~~~+ \varepsilon^{1/\theta}\lambda^{-1}\int_{\tilde{Q}_{2r}(x,t)}|\nabla u_{n,m}-\nabla u_n|+ \varepsilon^{1/\theta}\lambda^{-1}\int_{\tilde{Q}_{2r}(x,t)}|\nabla u_n-\nabla u|. 
         \end{align*}      
          Thanks to \eqref{5hh1510136} and \eqref{5hh1510137} we can continue 
     \begin{align*}
      &|E^1_{\lambda,\varepsilon}\cap \tilde{Q}_r(x,t)| \lesssim\varepsilon \lambda^{-\theta} |\tilde{Q}_{r}(x,t)|\left(\fint_{Q_{4r}(x,t_0)}|\nabla u_{n,m}| dxdt\right)^\theta \\&~+ \varepsilon \lambda^{-\theta} |\tilde{Q}_{r}(x,t)|\left(\frac{|\mu_{n,m}|(Q_{4r}(x,t_0))}{r^{N+1}}\right)^\theta 
                           + \varepsilon^{1/\theta}\lambda^{-1}|\tilde{Q}_{r}(x,t)|\frac{|\mu_{n,m}|(Q_{4r}(x,t_0))}{r^{N+1}}
                           \\&~+\varepsilon^{1/\theta}\lambda^{-1}\int_{Q_{2r}(x,t_0)}|\nabla u_{n,m}-\nabla u_n|+\varepsilon^{1/\theta}\lambda^{-1}\int_{Q_{2r}(x,t_0)}|\nabla u_n-\nabla u|.
     \end{align*} 
    Letting $m\to\infty$ and $n\to\infty$, we get
 \begin{align*}
 &|E_{\lambda,\varepsilon}\cap \tilde{Q}_r(x,t)|
 \lesssim\varepsilon \lambda^{-\theta} |\tilde{Q}_{r}(x,t)| \left(\fint_{Q_{4r}(x,t_0)}|\nabla u|\right)^\theta  \\&~~~~+ \varepsilon \lambda^{-\theta} |\tilde{Q}_{r}(x,t)|\left(\frac{\omega(\overline{Q_{4r}(x,t_0)})}{r^{N+1}}\right)^\theta  
 + \varepsilon^{1/\theta}\lambda^{-1}|\tilde{Q}_{r}(x,t)|\frac{\omega(\overline{Q_{4r}(x,t_0)})}{r^{N+1}}.                                                  
 \end{align*}           
           Since, $\mathbb{M}(|\nabla u|)(x_1,t_1)\leq \lambda$ and $\mathbb{M}_1[\omega](x_2,t_2)\le \varepsilon^{1-\frac{1}{\theta}} \lambda$ we have 
           \begin{align*}
           &\int_{Q_{4r}(x,t_0)}|\nabla u| \leq \int_{\tilde{Q}_{9r}(x_1,t_1)}|\nabla u|\leq |\tilde{Q}_{9r}(x_1,t_1)|\lambda,\\&
            \omega(\overline{Q_{4r}(x,t_0)})\leq \omega(\tilde{Q}_{8r}(x,t))\leq \omega(\tilde{Q}_{9r}(x_2,t_2)) \leq  \varepsilon^{1-\frac{1}{\theta}} \lambda (9r)^{N+1}.
           \end{align*}         
           Thus 
         $$
            |E_{\lambda,\varepsilon}\cap \tilde{Q}_r(x,t)|                                  \leq c \varepsilon |\tilde{Q}_{r}(x,t)|.
          $$       
            Next,  we consider the case  $B_{4r}(x)\cap\Omega\not=\emptyset$. Let $x_3\in \partial\Omega$ such that $|x_3-x|=\text{dist}(x,\partial\Omega)$. Let $w_n$ be as in Theorem \ref{5hh24093}  with $\tilde{\Omega}_{6R}=\tilde{\Omega}_{16r}(x_3,t_0)$ and $u=u_{n,m}$ where $t_0=\min\{t+2r^2,T\}$. We have $Q_{12r}(x,t_0)\subset Q_{16r}(x_3,t_0)$,
            \begin{align*}
           & \fint_{Q_{12r}(x,t_0)}|\nabla u_{n,m}-\nabla w_{n,m}|  \lesssim \frac{|\mu_{n,m}|(\tilde{\Omega}_{16r}(x_3,t_0))}{r^{N+1}},
           \\&\left(\fint_{Q_{2r}(x,t_0)}|\nabla w_{n,m}|^{\theta} \right)^{\frac{1}{\theta}}\lesssim\fint_{Q_{12r}(x,t_0)}|\nabla w_{n,m}|.
            \end{align*} As above we also obtain
                         \begin{align*}
                         &|E^1_{\lambda,\varepsilon}\cap \tilde{Q}_r(x,t)|\lesssim\varepsilon \lambda^{-\theta} |\tilde{Q}_{r}(x,t)|\left(\fint_{Q_{12r}(x,t_0)}|\nabla u| dxdt\right)^\theta \\&~~~+ \varepsilon \lambda^{-\theta} |\tilde{Q}_{r}(x,t)|\left(\frac{\omega(\overline{Q_{16r}(x_3,t_0)})}{r^{N+1}}\right)^\theta              + \varepsilon^{1/\theta}\lambda^{-1}|\tilde{Q}_{r}(x,t)|\frac{\omega(\overline{Q_{16r}(x_3,t_0)})}{r^{N+1}}.
                         \end{align*}                                 
Since, $ \mathbb{M}(|\nabla u|)(x_1,t_1)\leq \lambda$ and $\mathbb{M}_1[\omega](x_2,t_2)\le \varepsilon^{1-\frac{1}{\theta}} \lambda$ we have 
\begin{align*}
	&\int_{Q_{12r}(x,t_0)}|\nabla u| dxdt\leq \int_{\tilde{Q}_{24r}(x,t)}|\nabla u| dxdt\leq \int_{\tilde{Q}_{25r}(x_1,t_1)}|\nabla u| dxdt\leq |\tilde{Q}_{25r}(x_1,t_1)|\lambda,\\&
	\omega(\overline{Q_{16r}(x_3,t_0)})\leq \omega(\tilde{Q}_{32r}(x_3,t))\leq \omega(\tilde{Q}_{36r}(x,t))\leq \omega(\tilde{Q}_{37r}(x_2,t_2)) \leq  \varepsilon^{1-\frac{1}{\theta}} \lambda (37r)^{N+1}.
\end{align*}
           Thus 
          $$
                    |E^1_{\lambda,\varepsilon}\cap \tilde{Q}_r(x,t)|\leq c \varepsilon |\tilde{Q}_{r}(x,t)|.
                   $$                         
                    Hence, \eqref{5hh1510134} holds.\medskip\\
                    Similarly, we also prove that 
                    for all $(x,t)\in B_1$ and $r\in (0,R_2]$ and $\lambda>0,
                     \varepsilon\in (0,1)$ we have
                        $
                        \tilde{Q}_r(x,t)\cap B_1\subset F_\lambda^2
                       $
                        if $
                           |E^2_{\lambda,\varepsilon}\cap \tilde{Q}_r(x,t)|\geq c\varepsilon |\tilde{Q}_r(x,t)|
                          $
                           where the constant $c$ does not depend on $\lambda$ and $\varepsilon$. 
                     We apply Lemma \ref{5hhvitali3} for $E=E^1_{\lambda,\varepsilon},F=F^1_\lambda$ to  get  \eqref{5hh15101310} and  for $E=E^2_{\lambda,\varepsilon},F=F^2_\lambda$  get  \eqref{5hh070120141}.
                             The proof is complete.                 
  \end{proof}\medskip\\
   
    \begin{proof}[Proof of Theorem \ref{5hh0701201411}]
    By theorem \ref{5hh1510139}, there exist constants $c>0$, $0<\varepsilon_0<1$ and  a renormalized solution $u$ of equation \eqref{5hhparabolic1} with data $\mu$,  $u_0=\sigma$ such that for any $\varepsilon\in (0,1)$, $\lambda>0$
$$|\{\mathbb{M}(|\nabla u|)>\varepsilon^{-1/\theta}\lambda, \mathbb{M}_{1}[\omega]\leq \varepsilon^{1-\frac{1}{\theta}}\lambda\}\cap Q|\leq c\varepsilon |\{\mathbb{M}(|\nabla u|)>\lambda \}\cap Q|.      $$                                           
                                             Therefore, 
 if $0<s<\infty$ 
 \begin{align*}
  ||\mathbb{M}(|\nabla u|)||_{L^{p,s}(Q)}^s&= \varepsilon^{-s/\theta}p\int_{0}^{\infty}\lambda^s|\{(x,t)\in Q:\mathbb{M}(|\nabla u|)>\varepsilon^{-1/\theta}\lambda \}|^{\frac{s}{p}}\frac{d\lambda}{\lambda}\\&\leq 
              c \varepsilon^{\frac{s(\theta-p)}{\theta p}}p\int_{0}^{\infty}\lambda^s|\{(x,t)\in Q:\mathbb{M}(|\nabla u|)>\lambda \}|^{\frac{s}{p}}\frac{d\lambda}{\lambda}\\&+ \varepsilon^{-s/\theta}p\int_{0}^{\infty}\lambda^s|\{(x,t)\in Q:\mathbb{M}_{1}[\omega]>\varepsilon^{1-\frac{1}{\theta}}\lambda \}|^{\frac{s}{p}}\frac{d\lambda}{\lambda}\\&= c \varepsilon^{\frac{s(\theta-p)}{\theta p}}||\mathbb{M}(|\nabla u|)||_{L^{p,s}(Q)}^s+ \varepsilon^{-s}||\mathbb{M}_{1}[\omega]||_{L^{p,s}(Q)}^s.
 \end{align*}         
            Since $p<\theta $, we can choose $0<\varepsilon<\varepsilon_0$ such that $ c \varepsilon^{\frac{s(\theta-p)}{\theta p}}\le 1/2$. So,  we get the result for case $0<s<\infty$. Similarly, we also get the result for case $s=\infty$. \medskip\\
        Also, we get \eqref{5hh200320142} by using \eqref{5hh230120141'} in Proposition \ref{5hh230120143}, \eqref{5hh240120145'} in Proposition \ref{5hh240120146}. The proof is complete.
    \end{proof} 
    \begin{remark}
    Thanks to Proposition \ref{5hh23101315} we have that for any $s\in\left(\frac{N+2}{N+1},\frac{N+2+\theta}{N+2}\right)$  if $\mu\in L^{\frac{(s-1)(N+2)}{s},\infty}(\Omega_T)$ and $\sigma\equiv 0$ then
$$
                                            |||\nabla u|^s||_{L^{\frac{(s-1)(N+2)}{s},\infty}(\Omega_T)}\lesssim_{s,c_0,T_0/r_0}||\mu||^s_{L^{\frac{(s-1)(N+2)}{s},\infty}(\Omega_T)}.
$$
                                            
    \end{remark}     
        As the  proof of Theorem \ref{5hh0701201411}, we also get
        \begin{theorem}\label{5hh070120143} Suppose that $\mathbb{R}^N\backslash\Omega$ is uniformly $2-$thick with constants $c_0,r_0$. Let $\theta$ be as in Theorem \ref{5hh1510139}. Let $1 \leq p< \theta $, $0<s \leq \infty$ and  $\mu\in \mathfrak{M}_b(\Omega_T)$, $\sigma\in\mathfrak{M}_b(\Omega)$, set $\omega=|\mu|+|\sigma|\otimes\delta_{\{t=0\}}$.  There exists  a distributional solution $u$ of equation \eqref{5hhparabolic1} with data $\mu$ and $u_0=\sigma$ such that 
                \begin{align}
                \nonumber||\mathbb{M}(\chi_{\tilde{Q}_{4R}(y_0,s_0)}|\nabla u|)||_{L^{p,s}(\tilde{Q}_{R}(y_0,s_0))}&\lesssim_{p,s,c_0}R^{\frac{N+2}{p}}\inf\{r_0,R\}^{-N-2}||\nabla u||_{L^1(\tilde{Q}_{4R}(y_0,s_0))}\\&~~+  ||\mathbb{M}_1[\chi_{\tilde{Q}_{4R}(y_0,s_0)}\omega]||_{L^{p,s}(\tilde{Q}_{R}(y_0,s_0))},\label{5hh070120142}
                \end{align}               
                    for any $\tilde{Q}_{R}(y_0,s_0)\subset\mathbb{R}^{N+1}$. Moreover, if $\sigma\in L^1(\Omega)$ then $u$ is a renormalized solution.                
                \end{theorem}              
                  \begin{proof}[Proof of Theorem \ref{5hh0701201412}] Let $\{u_{n,m}\}$ and $\mu_{n,m}$ be in the proof of Theorem \ref{5hh1510139}.  From Corollary \ref{5hh060120148} and \ref{5hh060120149} we assert:
                                    for $2-\inf\{\beta_1,\beta_2\}<\gamma<N+2$, $0<\rho\leq T_0$ we have 
$$
                                             \int_{Q_\rho(y,s)}|\nabla u_{n,m}| \lesssim_{c_0,\gamma,T_0/r_0} \rho^{N+3-\gamma}||\mathbb{M}_{\gamma}[|\mu_{n,m}|]||_{L^\infty(\Omega\times (-T,T))},$$
                                             where $\beta_1,\beta_2$ are constants in Theorem \ref{5hh1510135} and Theorem \ref{5hh24093}. 
It is easy to see that   $$||\mathbb{M}_{\gamma}[|\mu_{n,m}|]||_{L^\infty(\Omega\times (-T,T))}\leq ||\mathbb{M}_{\gamma}[\omega]||_{L^\infty(\Omega\times (-T,T))} = ||\mathbb{M}_{\gamma}[\omega]||_{L^\infty(\Omega_T)},$$ for any $n,m$ large enough. \\
Letting $m\to\infty,n\to\infty$, yield 
  $$
         \int_{Q_\rho(y,s)}|\nabla u|dxdt \lesssim_{c_0,\gamma,T_0/r_0} \rho^{N+3-\gamma}||\mathbb{M}_{\gamma}[\omega]||_{L^\infty(\Omega_T)}$$                 By Theorem \ref{5hh070120143}  we get 
\begin{align*}
 |||\nabla u|||_{L^{p,s}(\tilde{Q}_{R}(y_0,s_0)\cap\Omega_T)}&\lesssim_{p,s,c_0} c(T_0/r_0)R^{\frac{N+2}{p}+1-\gamma}||\mathbb{M}_{\gamma}[\omega]||_{L^\infty(\Omega_T)}\\&~~+  ||\mathbb{M}_1[\chi_{\tilde{Q}_{R}(y_0,s_0)}\omega]||_{L^{p,s}(\tilde{Q}_{R}(y_0,s_0))}
\end{align*}                                              
 for any $\tilde{Q}_{R}(y_0,s_0)\subset\mathbb{R}^{N+1}$ and $0<R\leq T_0$. It follows \eqref{5hh070520141}.

               Finally, if  $\mu\in L_{*}^{\frac{(\gamma-1)p}{\gamma},\frac{(\gamma-1)s}{\gamma};(\gamma-1)p}(\Omega_T)  $ and $\sigma\equiv 0$, then clearly $u$ is a unique renormalized solution. It suffices to show that 
               \begin{align}&\label{5hh070520143}
             ~~~~~~||\mathbb{M}_{\gamma}[|\mu|]||_{L^\infty(\Omega_T)}\lesssim c_1 ||\mu||_{L_{*}^{\frac{(\gamma-1)p}{\gamma},\frac{(\gamma-1)s}{\gamma};(\gamma-1)p}(\Omega_T)},\\&\label{5hh070520144}
             R^{\frac{p(\gamma-1)-N-2}{p}}||\mathbb{M}_1[\chi_{\tilde{Q}_{R}(y,s_0)}|\mu|]||_{L^{p,s}(\tilde{Q}_{R}(y_0,s_0))}\lesssim c_1 ||\mu||_{L_{*}^{\frac{(\gamma-1)p}{\gamma},\frac{(\gamma-1)s}{\gamma};(\gamma-1)p}(\Omega_T)}
               \end{align}  
                  for any $\tilde{Q}_{R}(y_0,s_0)\subset\mathbb{R}^{N+1}$ and $0<R\leq T_0$,
               where $c_1=c_1(p,s,\gamma,c_0,T_0/r_0)$.
               \\
               In fact, for $0<\rho<T_0$ and $(x,t)\in\Omega_T$ we have 
               \begin{align*}
               ||\mu||_{L_{*}^{\frac{(\gamma-1)p}{\gamma},\frac{(\gamma-1)s}{\gamma};(\gamma-1)p}(\Omega_T)}&\geq ||\mu||_{L_{*}^{\frac{(\gamma-1)p}{\gamma},\infty;(\gamma-1)p}(\Omega_T)}
               \\&\geq \rho^{\frac{(\gamma-1)p-N-2}{\frac{(\gamma-1)p}{\gamma}}}||\mu||_{L^{\frac{(\gamma-1)p}{\gamma},\infty}(\tilde{Q}_\rho(x,t)\cap\Omega_T)}
               \\&\gtrsim c_1 \rho^{\frac{(\gamma-1)p-N-2}{\frac{(\gamma-1)p}{\gamma}}}|\tilde{Q}_\rho(x,t)|^{-1+\frac{\gamma}{(\gamma-1)p}}|\mu|(\tilde{Q}_\rho(x,t)\cap\Omega_T)
               \\&\gtrsim c_1 \frac{|\mu|(\tilde{Q}_\rho(x,t)\cap\Omega_T)}{\rho^{N+2-\gamma}},
               \end{align*}
               which obviously implies \eqref{5hh070520143}.\\
               Next,  note that 
               \begin{align*}
               \mathbb{M}_1[\chi_{\tilde{Q}_{R}(y_0,s_0)}|\mu|](x,t)\lesssim\left(\mathbb{M}\left(\chi_{\tilde{Q}_{R}(y_0,s_0)}|\mu|\right)(x,t)\right)^{1-\frac{1}{\gamma}}||\mu||_{L_{*}^{\frac{(\gamma-1)p}{\gamma},\frac{(\gamma-1)s}{\gamma};(\gamma-1)p}(\Omega_T)}^{\frac{1}{\gamma}}.
               \end{align*}
   We derive 
   \begin{align*}
    & R^{\frac{p(\gamma-1)-N-2}{p}}||\mathbb{M}_1[\chi_{\tilde{Q}_{R}(y,s_0)}|\mu|]||_{L^{p,s}(\tilde{Q}_{R}(y_0,s_0))}\\&~~~\lesssim R^{\frac{p(\gamma-1)-N-2}{p}}||\mathbb{M}\left(\chi_{\tilde{Q}_{R}(y_0,s_0)}|\mu|\right)||_{L^{\frac{(\gamma-1)p}{\gamma},\frac{(\gamma-1)s}{\gamma}}(\tilde{Q}_{R}(y_0,s_0))}^{1-\frac{1}{\gamma}}||\mu||_{L_{*}^{\frac{(\gamma-1)p}{\gamma},\frac{(\gamma-1)s}{\gamma};(\gamma-1)p}(\Omega_T)}^{\frac{1}{\gamma}}
    \\&~~~\lesssim R^{\frac{p(\gamma-1)-N-2}{p}}|||\mu|||_{L^{\frac{(\gamma-1)p}{\gamma},\frac{(\gamma-1)s}{\gamma}}(\tilde{Q}_{R}(y_0,s_0))}^{1-\frac{1}{\gamma}}||\mu||_{L_{*}^{\frac{(\gamma-1)p}{\gamma},\frac{(\gamma-1)s}{\gamma};(\gamma-1)p}(\Omega_T)}^{\frac{1}{\gamma}}.
   \end{align*}
  Here we have used the boundedness property of $\mathbb{M}$ in  $L^{\frac{(\gamma-1)p}{\gamma},\frac{(\gamma-1)s}{\gamma}}(\mathbb{R}^{N+1})$ for $\frac{(\gamma-1)p}{\gamma}>1$. Therefore, immediately we get  \eqref{5hh070520144}.   This completes the proof.          
                
                  \end{proof}
   \subsection{Global estimates on Reifenberg flat domains }
                  Now we prove results for Reifenberg flat domain. First,
           we will use Lemma \ref{5hh24092}, \ref{5hh16101310} and Lemma \ref{5hhvitali2} to get the following result.
  \begin{theorem}\label{5hh23101312}  Suppose that $A$ satisfies \eqref{5hhcondc}. Let $s_1,s_2$ be in Lemma \ref{5hh21101319} and \ref{5hh1610139}, set $s_0=\max\{s_1,s_2\}$. Let $w\in A_\infty$, $\mu\in\mathfrak{M}_b(\Omega_T)$, $\sigma\in\mathfrak{M}_b(\Omega)$, set $\omega=|\mu|+|\sigma|\otimes\delta_{\{t=0\}}$. There exists a distributional solution of \eqref{5hhparabolic1} with data $\mu$ and $u_0=\sigma$ such that the following holds. For any $\varepsilon>0,R_0>0$ one finds  $\delta_1=\delta_1(N,\Lambda_1,\Lambda_2,\varepsilon,[w]_{A_\infty})\in (0,1)$ and $\delta_2=\delta_2(N,\Lambda_1,\Lambda_2,\varepsilon,[w]_{A_\infty},T_0/R_0)\in (0,1)$ and $\Lambda=\Lambda(N,\Lambda_1,\Lambda_2)>0$ such that if $\Omega$ is  $(\delta_1,R_0)$- Reifenberg flat domain and $[\mathcal{A}]_{s_0}^{R_0}\le \delta_1$ then 
   \begin{equation}\label{5hh16101311}
   w(\{\mathbb{M}(|\nabla u|)>\Lambda\lambda, \mathbb{M}_1[\omega]\le \delta_2\lambda \}\cap \Omega_T)\le C\varepsilon w(\{ \mathbb{M}(|\nabla u|)> \lambda\}\cap \Omega_T)
   \end{equation}
   for all $\lambda>0$, 
      where the constant $C$  depends only on $N,\Lambda_1,\Lambda_2, T_0/R_0, [w]_{A_\infty}$.\\
      Furthermore, if $\sigma\in L^1(\Omega)$ then  $u$ is a renormalized solution.  
  \end{theorem}
  \begin{proof}Let $\{\mu_{n}\},\{\sigma_n\},\{\mu_{n,m}\},\{u_n\},\{u_{n,m}\},u$ be as in the proof of Theorem \ref{5hh1510139}.                              
  Let $\varepsilon$ be in $(0,1)$. Set $E_{\lambda,\delta_2}=\{\mathbb{M}(|\nabla u|)>\Lambda\lambda, \mathbb{M}_{1}[\omega]\leq \delta_2\lambda\}\cap \Omega_T $ and $F_\lambda=\{\mathbb{M}(|\nabla u|)>\lambda \}\cap \Omega_T$ for $\varepsilon\in (0,1)$ and $\lambda>0$.
    Let $\{y_i\}_{i=1}^L\subset \Omega$ and a ball $B_0$ with radius $2T_0$ such that 
   $$
    \Omega\subset \bigcup\limits_{i = 1}^L {{B_{r_0}}({y_i})}  \subset {B_0}$$
    where $r_0=\min\{R_0/1080,T_0\}$. Let 
    $s_j=T-jr_0^2/2$ for all $j=0,1,...,[\frac{2T}{r_0^2}]$ and $Q_{2T_0}=B_0\times (T-4T_0^2,T)$. So,
$$
        \Omega_T\subset \bigcup\limits_{i,j} {{Q_{r_0}}({y_i,s_j})}  \subset {Q_{2T_0}}.$$
   We verify that
   \begin{equation}\label{5hh2310131}
   w(E_{\lambda,\delta_2})\leq \varepsilon w({\tilde{Q}_{r_0}}({y_i,s_j})) ~~\forall ~\lambda>0
   \end{equation}
   for some $\delta_2$ small enough, depending on $N,p,\alpha,\beta,\epsilon,[w]_{A_\infty},T_0/R_0$.\\
   In fact, we can assume that $E_{\lambda,\delta_2}\not=\emptyset$,  so $|\mu| (\Omega_T)+|\sigma|(\Omega)\leq T_0^{N+1}\delta_2\lambda$. We have
 $$
 |E_{\lambda,\delta_2}|\lesssim \frac{1}{\Lambda\lambda}\int_{\Omega_T}|\nabla u|.$$
  We also have
$$\int_{\Omega_T}|\nabla u|\lesssim T_0 (|\mu| (\Omega_T)+|\sigma|(\Omega)).
$$
  Thus, $$
   |E_{\lambda,\varepsilon}|\lesssim \frac{1}{\Lambda\lambda} T_0( |\mu| (\Omega_T)+|\sigma|(\Omega))\lesssim \frac{1}{\Lambda\lambda} T_0^{N+2} \delta_2\lambda =  \delta_2 |Q_{2T_0}| .$$
  which implies
$$
    w(E_{\lambda,\delta_2})\lesssim_C \left(\frac{|E_{\lambda,\delta_2}|}{|Q_{2T_0}|}\right)^\nu w(Q_{2T_0})\leq  c\delta_2^\nu w(Q_{2T_0})
$$
   where $(C,\nu)$ is a pair of $A_\infty$ constants of $w$. It is known that (see, e.g \cite{55Gra}) there exist $A_1=A_1(N,C,\nu)$ and $\nu_1=\nu_1(N,C,\nu)$ such that 
$$
   \frac{w(\tilde{Q}_{2T_0})}{w({\tilde{Q}_{r_0}}({y_i,s_j}))}\leq A_1\left(\frac{|\tilde{Q}_{2T_0}|}{|{\tilde{Q}_{r_0}}({y_i,s_j})|}\right)^{\nu_1}~~\forall i,j.
$$
  So, 
$$
      w(E_{\lambda,\delta_2})\leq C\left(c\delta_2\right)^\nu \left(\frac{|\tilde{Q}_{T_0}|}{|{\tilde{Q}_{r_0}}({y_i,s_j})|}\right)^{\nu_1} w({\tilde{Q}_{r_0}}({y_i,s_j}))
      < \varepsilon w({\tilde{Q}_{r_0}}({y_i,s_j}))~~\forall ~i,j$$
     where $\delta_2\leq c\left(\frac{\varepsilon}{(T_0r_0^{-1})^{(N+2)\nu_1}}\right)^{1/\nu}$. It follows \eqref{5hh2310131}.\\
  Next we verify that for all $(x,t)\in \Omega_T$ and $r\in (0,2r_0]$ and $\lambda>0$ we have
    $
     \tilde{Q}_r(x,t)\cap \Omega_T\subset F_\lambda
     $
     if $
        w(E_{\lambda,\delta_2}\cap \tilde{Q}_r(x,t))\geq \varepsilon w(Q_r(x,t))
      $
     for some $\delta_2\leq c\left(\frac{\varepsilon}{(T_0r_0^{-1})^{(N+2)\nu_1}}\right)^{1/\nu}$. \\
        Indeed,
  take $(x,t)\in \Omega_T$ and $0<r\leq 2r_0$.
             Now assume that $\tilde{Q}_r(x,t)\cap \Omega_T\cap F^c_\lambda\not= \emptyset$ and $E_{\lambda,\delta_2}\cap \tilde{Q}_r(x,t)\not = \emptyset$ i.e, there exist $(x_1,t_1),(x_2,t_2)\in \tilde{Q}_r(x,t)\cap \Omega_T$ such that $\mathbb{M}(|\nabla u|)(x_1,t_1)\leq \lambda$ and $\mathbb{M}_1[\omega](x_2,t_2)\le \delta_2 \lambda$.
              We need to prove that
              \begin{equation}\label{5hh2310133}
                     w(E_{\lambda,\delta_2}\cap \tilde{Q}_r(x,t)))< \varepsilon w(\tilde{Q}_r(x,t)). 
                                      \end{equation}
                                   Clearly,
                                      \begin{equation*}
                                      \mathbb{M}(|\nabla u|)(y,s)\leq \max\left\{\mathbb{M}\left(\chi_{\tilde{Q}_{2r}(x,t)}|\nabla u|\right)(y,s),3^{N+2}\lambda\right\}~~\forall (y,s)\in \tilde{Q}_r(x,t).
                                      \end{equation*}
           Therefore, for all $\lambda>0$ and $\Lambda\geq 3^{N+2}$,
           \begin{eqnarray}\label{5hh2310134}E_{\lambda,\delta_2}\cap \tilde{Q}_r(x,t)=\left\{\mathbb{M}\left(\chi_{\tilde{Q}_{2r}(x,t)}|\nabla u|\right)>\Lambda\lambda, \mathbb{M}_{1}[\omega]\leq \delta_2\lambda\right\}\cap \Omega_T \cap \tilde{Q}_r(x,t).
           \end{eqnarray}
           In particular, $E_{\lambda,\delta_2}\cap \tilde{Q}_r(x,t)=\emptyset$ if $\overline{B}_{8r}(x)\subset\subset \mathbb{R}^{N}\backslash \Omega$.
           Thus, it is enough to consider the case $B_{8r}(x)\subset\subset\Omega$ and $B_{8r}(x)\cap\Omega\not=\emptyset$.\\   
           We consider the case $B_{8r}(x)\subset\subset\Omega$. Let $v_{n,m}$ be as in Lemma \ref{5hh24092} with $Q_{2R}=Q_{8r}(x,t_0)$ and $u=u_{n,m}$ where $t_0=\min\{t+2r^2,T\}$. We have  
           \begin{equation}
           \label{5hh2310135}
         ||\nabla v_{n,m}||_{L^\infty(Q_{2r}(x,t_0))}\lesssim \fint_{Q_{8r}(x,t_0)}|\nabla u_{n,m}|  +\frac{|\mu_{n,m}|(Q_{8r}(x,t_0))}{r^{N+1}},
           \end{equation}
                  \begin{align*}
                  \fint_{Q_{4r}(x,t_0)}|\nabla u_{n,m}-\nabla v_{n,m}|&\lesssim\frac{|\mu_{n,m}|(Q_{8r}(x,t_0))}{r^{N+1}}+ [A]_{s_0}^{R_0}\left(\fint_{Q_{8r}(x,t_0)}|\nabla u_{n,m}|\right.\\&~~~~~~~~\left.+ \frac{|\mu_{n,m}|(Q_{8r}(x,t_0))}{r^{N+1}}\right). 
                  \end{align*}
                          Thanks to $\mathbb{M}(|\nabla u|)(x_1,t_1)\leq \lambda$ and $\mathbb{M}_1[\omega](x_2,t_2)\le \delta_2 \lambda$ with $(x_1,t_1),(x_2,t_2)\in Q_r(x,t)$, we get 
                          \begin{align*}
                           \mathop {\limsup }\limits_{n \to \infty } \mathop {\limsup }\limits_{m \to \infty }||\nabla v_{n,m}||_{L^\infty(Q_{2r}(x,t))}&\lesssim  \fint_{\tilde{Q}_{17r}(x_1,t_1)}|\nabla u| dxdt + \frac{\omega(\overline{\tilde{Q}_{17r}(x_2,t_2)})}{r^{N+1}}                                   
                                                             \\&\lesssim\lambda+ \delta_2\lambda
                                  \lesssim\lambda,  
                          \end{align*}
                                     and 
\begin{align*}
&\mathop {\limsup }\limits_{n \to \infty }\mathop {\limsup }\limits_{m \to \infty }\fint_{Q_{4r}(x,t_0)}|\nabla u_n-\nabla v_n|  
                          \\&~~~~~~~~~~~~~\lesssim\frac{\omega(\overline{\tilde{Q}_{17r}(x_2,t_2)})}{r^{N+1}}+ [A]_{s_0}^{R_0}\left(\fint_{\tilde{Q}_{17r}(x_1,t_1)}|\nabla u|+ \frac{\omega(\overline{\tilde{Q}_{17r}(x_2,t_2)})}{r^{N+1}}\right)  
                          \\&~~~~~~~~~~~~~\lesssim\delta_2\lambda + [A]_{s_0}^{R_0}\left(\lambda+\delta_2\lambda\right)  
                          \lesssim\left(\delta_2+\delta_1(1+\delta_2)\right)\lambda.
\end{align*}                                                                Here we have used $[A]_{s_0}^{R_0}\leq \delta_1$ in the last inequality. \\
                          So, we can find $n_0$ large enough and a sequence $\{k_{n}\}$ such that 
                          \begin{equation}\label{5hh2310136}
                          ||\nabla v_{n,m}||_{L^\infty(\tilde{Q}_{2r}(x,t))}=||\nabla v_{n,m}||_{L^\infty(Q_{2r}(x,t_0))}\leq c\lambda,
                          \end{equation}                        
                          \begin{equation}\label{5hh2310137} 
                          \fint_{Q_{4r}(x,t_0)}|\nabla u_{n,m}-\nabla v_{n,m}|dxdt\leq c\left(\delta_2+\delta_1(1+\delta_2)\right)\lambda,
                          \end{equation}
                          for all $n\geq n_0$ and $m\geq k_n$.\\
                           In view of \eqref{5hh2310136} we see that for $\Lambda\geq \max\{3^{N+2},8c\}$ and $n\geq n_0$, $m\geq k_n$,
                           \begin{align*}
                           |\{\mathbb{M}\left(\chi_{\tilde{Q}_{2r}(x,t)}|\nabla v_{n,m}|\right)>\Lambda\lambda/4\}\cap \tilde{Q}_r(x,t)|=0.
                           \end{align*}
                           Leads to
\begin{align*}
|E_{\lambda,\delta_2}\cap \tilde{Q}_r(x,t)|&\leq   |\{\mathbb{M}\left(\chi_{\tilde{Q}_{2r}(x,t)}|\nabla u_{n,m}-\nabla v_{n,m}|\right)>\Lambda\lambda/4\}\cap \tilde{Q}_r(x,t)|
                           \\&+ |\{\mathbb{M}\left(\chi_{\tilde{Q}_{2r}(x,t)}|\nabla u_n-\nabla u_{n,m}|\right)>\Lambda\lambda/4\}\cap \tilde{Q}_r(x,t)|
                           \\&+ |\{\mathbb{M}\left(\chi_{\tilde{Q}_{2r}(x,t)}|\nabla u-\nabla u_{n}|\right)>\Lambda\lambda/4\}\cap \tilde{Q}_r(x,t)|.                          
\end{align*}                           
Therefore, by  \eqref{5hh2310137} and $\tilde{Q}_{2r}(x,t)\subset Q_{4r}(x,t_0)$ we obtain for any $n\geq n_0$ and $m\geq k_n$
\begin{align*}
|E_{\lambda,\delta_2}\cap \tilde{Q}_r(x,t)|&\leq  \frac{1}{\lambda}\int_{\tilde{Q}_{2r}(x,t)} |\nabla u_{n,m}-\nabla v_{n,m}| \\&~+\frac{1}{\lambda}\int_{\tilde{Q}_{2r}(x,t)} |\nabla u_n-\nabla u_{n,m}| +\frac{1}{\lambda}\int_{\tilde{Q}_{2r}(x,t)} |\nabla u-\nabla u_{n}|                          
                          \\&\lesssim \left(\delta_2+\delta_1(1+\delta_2)\right)|Q_r(x,t)|  \\&~+\frac{1}{\lambda}\int_{\tilde{Q}_{2r}(x,t)} |\nabla u_n-\nabla u_{n,m}| +\frac{1}{\lambda}\int_{\tilde{Q}_{2r}(x,t)} |\nabla u-\nabla u_{n}|.
\end{align*}
                          Letting $m\to\infty$ and $n\to\infty$ we get 
                          \begin{equation*}
                          |E_{\lambda,\delta_2}\cap \tilde{Q}_r(x,t)|\lesssim \left(\delta_2+\delta_1(1+\delta_2)\right)|\tilde{Q}_r(x,t)|.
                          \end{equation*}
Thus,  
\begin{align*}
w(E_{\lambda,\delta_2}\cap \tilde{Q}_r(x,t))&\leq C\left(\frac{|E_{\lambda,\delta_2}\cap \tilde{Q}_r(x,t) |}{|\tilde{Q}_r(x,t)|}\right)^\nu w(\tilde{Q}_r(x,t))
    \\&\lesssim \left(\delta_2+\delta_1(1+\delta_2)\right)^\nu w(\tilde{Q}_r(x,t))
    \\&< \varepsilon w(\tilde{Q}_r(x,t)).
\end{align*} 
    where $\delta_2,\delta_1$ are appropriately chosen,  $(C,\nu)$ is a pair of $A_\infty$ constants of $w$.\\
    Next we consider the case $B_{8r}(x)\cap\Omega\not=\emptyset$. Let $x_3\in\partial \Omega$ be such that $|x_3-x|=\text{dist}(x,\partial\Omega)$. Set $t_0=\min\{t+2r^2,T\}$. We have 
    \begin{equation}\label{5hh2310138}
    Q_{2r}(x,t_0)\subset Q_{10r}(x_3,t_0)\subset Q_{540r}(x_3,t_0)\subset \tilde{Q}_{1080r}(x_3,t)\subset \tilde{Q}_{1088r}(x,t)\subset \tilde{Q}_{1089r}(x_1,t_1),
    \end{equation}
    and 
    \begin{equation}\label{5hh2310139}
        Q_{540r}(x_3,t_0)\subset \tilde{Q}_{1080r}(x_3,t)\subset \tilde{Q}_{1088r}(x,t)\subset \tilde{Q}_{1089r}(x_2,t_2).
        \end{equation}
     Let $V_{n,m}$ be as in Lemma 
    \ref{5hh16101310} with $Q_{6R}=Q_{540r}(x_3,t_0)$, $u=u_{n,m}$ and $\varepsilon=\delta_3\in (0,1)$. We have 
     \begin{equation*}
     ||\nabla V_{n,m}||_{L^\infty(Q_{10r}(x_3,t_0))}\lesssim \fint_{Q_{540r}(x_3,t_0)}|\nabla u_{n,m}|+\frac{|\mu_{n,m}|(Q_{540r}(x_3,t_0))}{R^{N+1}}
     \end{equation*}
     and 
     \begin{align*}
      &\fint_{Q_{10r}(x_3,t_0)}|\nabla u_{n,m}-\nabla V_{n,m}|\\&\quad\quad\lesssim  (\delta_3+[A]_{s_0}^{R_0})\fint_{Q_{540r}(x_3,t_0)}|\nabla u_{n,m}|+ (1+[A]_{s_0}^{R_0})\frac{|\mu_{n,m}|(Q_{540r}(x_3,t_0))}{R^{N+1}}.
     \end{align*}       
Since $\mathbb{M}(|\nabla u|)(x_1,t_1)\leq \lambda$,  $\mathbb{M}_1[\omega](x_2,t_2)\le \delta_2 \lambda$ and \eqref{5hh2310138}, \eqref{5hh2310139} we get 
\begin{align*}
\mathop {\limsup }\limits_{n \to \infty }\mathop {\limsup }\limits_{m \to \infty }||\nabla V_{n,m}||_{L^\infty(Q_{2r}(x,t_0))}&\lesssim \fint_{Q_{540r}(x_3,t_0)}|\nabla u|+\frac{\omega(\overline{Q_{540r}(x_3,t_0)})}{R^{N+1}}
\\&\lesssim \fint_{\tilde{Q}_{1089r}(x_1,t_1)}|\nabla u|+\frac{\omega(\tilde{Q}_{1089r}(x_2,t_2))}{R^{N+1}}
\\&\lesssim \lambda+\delta_2 \lambda\lesssim\lambda 
\end{align*}
and 
\begin{align*}
 &\mathop {\limsup }\limits_{n \to \infty }\mathop {\limsup }\limits_{m\to \infty }\fint_{Q_{2r}(x,t_0)}|\nabla u_{n,m}-\nabla V_{n,m}|dxdt
 \\&~~~\lesssim (\delta_3+[A]_{s_0}^{R_0})\fint_{Q_{540r}(x_3,t_0)}|\nabla u|dxdt+ (1+[A]_{s_0}^{R_0})\frac{\omega(\overline{Q_{540 r}(x_3,t_0)})}{r^{N+1}}
 \\&~~~\lesssim (\delta_3+[A]_{s_0}^{R_0})\fint_{\tilde{Q}_{1089r}(x_1,t_1)}|\nabla u|dxdt+ (1+[A]_{s_0}^{R_0})\frac{\omega(\tilde{Q}_{1089}(x_2,t_2))}{r^{N+1}}
 \\&~~~\lesssim (\delta_3+[A]_{s_0}^{R_0})\lambda+ (1+[A]_{s_0}^{R_0}) \delta_2 \lambda
 \\&~~~\leq \left( (\delta_3+\delta_1)+ (1+\delta_1) \delta_2 \right)\lambda.
\end{align*}
   Here we used $[A]_{s}^{R_0}\leq \delta_1$ in the last inequality.\\
   So, we can find $n_0$ large enough and a sequence $\{k_n\}$ such that 
                             \begin{equation}\label{5hh23101310}
                            ||\nabla V_{n,m}||_{L^\infty(\tilde{Q}_{2r}(x,t))}= ||\nabla V_{n,m}||_{L^\infty(Q_{2r}(x,t_0))}\leq c\lambda,
                             \end{equation}                           
                             \begin{equation}\label{5hh23101311} 
                             \fint_{Q_{2r}(x,t_0)}|\nabla u_{n,m}-\nabla V_{n,m}|\leq c\left( (\delta_3+\delta_1)+ (1+\delta_1) \delta_2 \right)\lambda,
                             \end{equation}
                             for all $n\geq n_0$ and $m\geq k_n$.\\
  Now set $\Lambda= \max\{3^{N+2},8c\}$. As above we also have  for $n\geq n_0$, $m\geq k_n$
 \begin{align*}
 |E_{\lambda,\delta_2}\cap \tilde{Q}_r(x,t)|&\leq   |\{\mathbb{M}\left(\chi_{\tilde{Q}_{2r}(x,t)}|\nabla u_{n,m}-\nabla V_{n,m}|\right)>\Lambda\lambda/4\}\cap \tilde{Q}_r(x,t)|
                            \\&+ |\{\mathbb{M}\left(\chi_{\tilde{Q}_{2r}(x,t)}|\nabla u_n-\nabla u_{n,m}|\right)>\Lambda\lambda/4\}\cap \tilde{Q}_r(x,t)|
                            \\&+ |\{\mathbb{M}\left(\chi_{\tilde{Q}_{2r}(x,t)}|\nabla u-\nabla u_{n}|\right)>\Lambda\lambda/4\}\cap \tilde{Q}_r(x,t)|.                          
 \end{align*}                         
 Therefore from  \eqref{5hh23101311} we obtain 
 \begin{align*}
 |E_{\lambda,\delta_2}\cap \tilde{Q}_r(x,t)|&\lesssim  \frac{1}{\lambda}\int_{\tilde{Q}_{2r}(x,t)} |\nabla u_{n,m}-\nabla V_{n,m}| \\&~+\frac{1}{\lambda}\int_{\tilde{Q}_{2r}(x,t)} |\nabla u_{n}-\nabla u_{n,m}| +\frac{1}{\lambda}\int_{\tilde{Q}_{2r}(x,t)} |\nabla u-\nabla u_{n}|                             
                            \\&\lesssim   \left( (\delta_3+\delta_1)+ (1+\delta_1) \delta_2 \right)|\tilde{Q}_r(x,t)|  \\&~+\frac{1}{\lambda}\int_{\tilde{Q}_{2r}(x,t)} |\nabla u_{n}-\nabla u_{n,m}| +\frac{1}{\lambda}\int_{\tilde{Q}_{2r}(x,t)} |\nabla u-\nabla u_{n}|. 
 \end{align*}
                           Letting $m\to\infty$ and $n\to\infty$ we get
$$
                           |E_{\lambda,\delta_2}\cap \tilde{Q}_r(x,t)|\lesssim \left( (\delta_3+\delta_1)+ (1+\delta_1) \delta_2 \right)|\tilde{Q}_r(x,t)| .$$
 Thus
 \begin{align*}
  w(E_{\lambda,\delta_2}\cap \tilde{Q}_r(x,t))&\leq C\left(\frac{|E_{\lambda,\delta_2}\cap \tilde{Q}_r(x,t)|}{|\tilde{Q}_r(x,t)|}\right)^\nu w(\tilde{Q}_r(x,t))
      \\&\leq  c \left( (\delta_3+\delta_1)+ (1+\delta_1) \delta_2 \right)^\nu w(\tilde{Q}_r(x,t))
      \\&< \varepsilon w(\tilde{Q}_r(x,t)),
 \end{align*}   
     where $\delta_3,\delta_1,\delta_2$ are appropriately chosen,  $(C,\nu)$ is a pair of $A_\infty$ constants of $w$.\medskip\\  
    Therefore, for all $(x,t)\in \Omega_T$ and $r\in (0,2r_0]$ and $\lambda>0$ if  
                $w(E_{\lambda,\delta_2}\cap \tilde{Q}_r(x,t))\geq \varepsilon w(\tilde{Q}_r(x,t))$                
then            
        $ \tilde{Q}_r(x,t)\cap \Omega_T\subset F_\lambda$    
         where $\delta_1=\delta_1(\varepsilon,[w]_{A_\infty})\in (0,1)$ and $\delta_2=\delta_2(\varepsilon,[w]_{A_\infty},T_0/R_0)\in (0,1)$. Thanks to Lemma \ref{5hhvitali2} we get the result.
              \end{proof}\\\\
      \begin{proof}[Proof of Theorem \ref{5hh2410131}] As in the proof of 
Theorem \ref{5hh0701201411}, we can prove \eqref{5hh16101312} by using  estimate \eqref{5hh16101311} in Theorem \ref{5hh23101312}.
              In particular,  thanks to Proposition \ref{5hh23101315} for $q>\frac{N+2}{N+1}$, $\mu\in L^{\frac{(N+2)(q-1)}{q},\infty}(\Omega_T)$  and $\sigma\equiv0$,    one has                               
                                          \begin{equation}\label{5hh23101314}
                                             |||\nabla u|^q||_{L^{\frac{(N+2)(q-1)}{ q },\infty}(\Omega_T)}\lesssim_{q,T_0/R_0} ||\mu||^{q}_{L^{\frac{(N+2)(q-1)}{q},\infty}(\Omega_T)}.
                                             \end{equation}                 The proof is complete.                                             
            \end{proof}\medskip\\
         The following corollary is a consequence of Theorem \ref{5hh2410131}.

        \begin{corollary}
        	\label{5hh2410131'} Suppose that $A$ satisfies \eqref{5hhcondc}. Let $\mu\in\mathfrak{M}_b(\Omega_T)$, $\sigma\in\mathfrak{M}_b(\Omega)$, set $\omega=|\mu|+|\sigma|\otimes\delta_{\{t=0\}}$. Let $s_0$ be in  Theorem \ref{5hh2410131}.  There exists a distributional solution of \eqref{5hhparabolic1} with data $\mu, \sigma$ such that the following holds. For any $1<q<\infty, 0<\beta<N+2$, we find $\delta=\delta(N,\Lambda_1,\Lambda_2,q)\in (0,1)$ such that if $\Omega$ is a $(\delta,R_0)$- Reifenberg flat domain and $[A]_{s_0}^{R_0}\le \delta$ for some $R_0$ then
        	\begin{equation}\label{keypoint}
        		\mathbb{I}_\beta[|\nabla u|^q\chi_{\Omega_{T}}]\lesssim C_1	\mathbb{I}_\beta[\mathbb{M}_1[\omega]^q\chi_{\Omega_{T}}],~~
        	\end{equation}
        	\begin{align}\label{5hh070520147}
        		\mathop {\sup }\limits_{\scriptstyle \text{compact } K\subset\mathbb{R}^{N+1} \hfill \atop 
        			\scriptstyle \text{Cap}_{\mathcal{G}_1,q'}(K)>0 \hfill}\left(\frac{\int_{K\cap\Omega_T}|\nabla u|^qdxdt}{\text{Cap}_{\mathcal{G}_1,q'}(K)}\right) \lesssim C_1 \mathop {\sup }\limits_{\scriptstyle \text{compact } K\subset\mathbb{R}^{N+1} \hfill \atop 
        			\scriptstyle \text{Cap}_{\mathcal{G}_1,q'}(K)>0 \hfill}\left(\frac{\omega(K)}{\text{Cap}_{\mathcal{G}_1,q'}(K)}\right)^q,                              
        	\end{align}
        	and if $q>\frac{N+2}{N+1}$, 
        	\begin{align}\label{5hh070520148}                                    
        		\mathop {\sup }\limits_{\scriptstyle \text{compact } K\subset\mathbb{R}^{N+1} \hfill \atop 
        			\scriptstyle \text{Cap}_{\mathcal{H}_1,q'}(K)>0 \hfill}\left(\frac{\int_{K\cap\Omega_T}|\nabla u|^qdxdt}{\text{Cap}_{\mathcal{H}_1,q'}(K)}\right) \lesssim C_2 \mathop {\sup }\limits_{\scriptstyle \text{compact } K\subset\mathbb{R}^{N+1} \hfill \atop 
        			\scriptstyle \text{Cap}_{\mathcal{H}_1,q'}(K)>0 \hfill}\left(\frac{\omega(K)}{\text{Cap}_{\mathcal{H}_1,q'}(K)}\right)^q,                             
        	\end{align}
        	
        	where $C_1=C_1(q,R_0,T_0)$ and $C_2=C_2(q,T_0/R_0)$.      
        \end{corollary} 
        \begin{remark}If $1<q<2$, we have
$$
        		 \mathop {\sup }\limits_{\scriptstyle \text{compact } O\subset\mathbb{R}^{N} \hfill \atop 
        			\scriptstyle \text{Cap}_{\mathbf{G}_{\frac{2}{q}-1},q'}(O)>0 \hfill}\frac{|\sigma|(O)}{\text{Cap}_{\mathbf{G}_{\frac{2}{q}-1},q'}(O)}\sim_q
        		\mathop {\sup }\limits_{\scriptstyle \text{compact } K\subset\mathbb{R}^{N+1} \hfill \atop 
        			\scriptstyle \text{Cap}_{\mathcal{G}_1,q'}(K)>0 \hfill}\frac{(|\sigma|\otimes\delta_{\{t=0\}})(K)}{\text{Cap}_{\mathcal{G}_1,q'}(K)}.                       
$$
        	If $\frac{N+2}{N+1}<q<2$, then this estimate is true when two capacities $\text{Cap}_{\mathcal{G}_1,q'}$, ,$\text{Cap}_{\mathbf{G}_{\frac{2}{q}-1},q'}$ are replaced by $\text{Cap}_{\mathcal{H}_1,q'}$ ,$\text{Cap}_{\mathbf{I}_{\frac{2}{q}-1},q'}$ respectively,                          
        	see Remark \ref{5hh130320146}. 
        \end{remark}        
             \begin{proof}[Proof of Corollary \ref{5hh2410131'}] By Theorem \ref{5hh2410131}, there exists a renormalized solution of \eqref{5hhparabolic1} with data $\mu$, $u(0)=\sigma$ satisfying 
             \begin{equation}\label{5hh2503201410}
             \int_{\Omega_T}|\nabla u|^qdw\lesssim c_1 \int_{\Omega_T}\left(\mathbb{M}_1[\omega]\right)^qdw
             \end{equation}
             for any $w\in A_\infty$, where $c_1=c_1(q,T_0/R_0,[w]_{A_\infty})$. Thanks to Corollary \ref{keylam}, we obtain \eqref{keypoint}.\medskip
             \\ We have $\mathbb{M}_1[\omega]\lesssim\mathbb{I}_1^{2T_0,1}[\omega]$ in $\Omega_T$. Thus, \eqref{5hh2503201410} can be rewritten
$$
              \int_{\Omega_T}|\nabla u|^qdw\lesssim c_1  \int_{\Omega_T}\left(\mathbb{I}_1^{2T_0,1}[\omega]\right)^qdw.
$$
              Thanks to Proposition \ref{5hh140420143} and  Corollary \ref{5hh250320146} and \ref{5hh250320145} we obtain \eqref{5hh070520147} and \eqref{5hh070520148}.
             \end{proof}\medskip\\
          As in the proof of 
         Theorem \ref{5hh0701201411}, we obtain the following corollary by using  estimate \eqref{5hh16101311} in Theorem \ref{5hh23101312}.
         \begin{corollary} \label{keyre}  Suppose that $A$ satisfies \eqref{5hhcondc}. Let $\mu\in\mathfrak{M}_b(\Omega_T)$, $\sigma\in\mathfrak{M}_b(\Omega)$, set $\omega=|\mu|+|\sigma|\otimes\delta_{\{t=0\}}$. There exists a distributional solution of \eqref{5hhparabolic1} with data $\mu,\sigma$ such that the following holds. For any  $1\leq q<\infty$, we find  $\delta=\delta(N,\Lambda_1,\Lambda_2,q)\in (0,1)$ and $s_0=s_0(N,\Lambda_1,\Lambda_2)>0$ such that if $\Omega$ is  $(\delta,R_0)$-Reifenberg flat domain $\Omega$ and $[A]_{s_0}^{R_0}\le \delta$ for some $R_0$ then                                       
$$
       	\int_{\lambda_0}^{\infty}\lambda^{q-1}\left|\{\mathbb{M}(|\nabla u|)>\Lambda\lambda \}\cap \Omega_T\right|d\lambda\leq C	\int_{\lambda_0/C}^{\infty}\lambda^{q-1}\left|\{\mathbb{M}_1[\omega]>\lambda \}\cap \Omega_T\right| d\lambda
$$
         for any $\lambda_0\geq 0$.	Here $C$ depends  on $N,\Lambda_1,\Lambda_2,q$ and $T_0/R_0$.                       
         \end{corollary} 
                      
  In the following, we shall use the classical  Minkowski inequality, for  convenience we recall it.  For any $0<q_1\leq q_2<\infty$ there holds
  \begin{align*}
  \left(\int_{X}\left(\int_{Y}|f(x,y)|^{q_1}d\mu_2(y)\right)^{\frac{q_2}{q_1}}d\mu_1(x)\right)^{\frac{1}{q_2}}\leq \left(\int_{Y}\left(\int_{X}|f(x,y)|^{q_2}d\mu_1(x)\right)^{\frac{q_1}{q_2}}d\mu_2(y)\right)^{\frac{1}{q_1}}
  \end{align*}
  for any measure function $f$ in $X\times Y$, where $\mu_1,\mu_2$ are nonnegative measures in $X$ and $Y$ respectively.\\ 
          \begin{proof}[Proof of Theorem \ref{5hh190320143}]We will consider only the case $s\not=\infty$ and leave the case $s=\infty$ to the readers. 
                              Take $\kappa_1\in (0,\kappa)$. It is easy to see that for $(x_0,t_0)\in\Omega_T$ and $0<\rho<\text{diam}(\Omega)+T^{1/2}$ \begin{align*}
                              w(x,t)=\min\{\rho^{-N-2+\kappa-\kappa_1},\max\{|x-x_0|,\sqrt{2|t-t_0|}\}^{-N-2+\kappa-\kappa_1}\}\in A_{\infty}
                              \end{align*}
                              where $[w]_{A_\infty}$ is independent of $ (x_0,t_0)$ and $\rho$. Thus, from \eqref{5hh16101312} in Theorem \ref{5hh2410131} we have 
\begin{align}\nonumber
||\mathbb{M}(|\nabla u|)||^s_{L^{q,s}(\tilde{Q}_\rho(x_0,t_0)\cap\Omega_T)}&= \rho^{\frac{(N+2-\kappa+\kappa_1)s}{q}}||\mathbb{M}(|\nabla u|)||^s_{L^{q,s}(\tilde{Q}_\rho(x_0,t_0)\cap\Omega_T,dw)}
\\&\lesssim\nonumber \rho^{\frac{(N+2-\kappa+\kappa_1)s}{q}}||\mathbb{M}_1[\omega]||^s_{L^{q,s}(\Omega_T,dw)}
\\&=\nonumber q\rho^{\frac{(N+2-\kappa+\kappa_1)s}{q}}\int_{0}^{\infty}\left(\lambda^q\int_{0}^{\infty}|\{\mathbb{M}_1[\omega]>\lambda,w>\tau\}\cap\Omega_T|d\tau\right)^{\frac{s}{q}}\frac{d\lambda}{\lambda}
\\&=:\rho^{\frac{(N+2-\kappa+\kappa_1)s}{q}} A.\label{5hh190320145}
\end{align}   
Since $w\leq \rho^{-N-2+\kappa-\kappa_1}$ and    $\{\mathbb{M}_1[\omega]>\lambda,w>\tau\}\subset \{\mathbb{M}_1[\omega]>\lambda\}\cap \tilde{Q}_{\tau^{\frac{1}{-N-2+\kappa-\kappa_1}}}(x_0,t_0)$, 
\begin{align*}
A\leq q \int_{0}^{\infty}\left(\lambda^q\int_{0}^{\rho^{-N-2+\kappa-\kappa_1}}|\{\mathbb{M}_1[\omega]>\lambda\}\cap \tilde{Q}_{\tau^{\frac{1}{-N-2+\kappa-\kappa_1}}}(x_0,t_0)\cap\Omega_T|d\tau\right)^{\frac{s}{q}}\frac{d\lambda}{\lambda}.
\end{align*} 
We divide into two cases. \\
Case 1: $0<s\leq q$. We can verify that for any non-increasing function $F$ in $(0,\infty)$ and $0<a\leq 1$ we have 
$$
\left(\int_{0}^{\infty}F(\tau)d\tau\right)^a\leq 4 \int_{0}^{\infty}(\tau F(\tau))^a\frac{d\tau}{\tau}.$$
Hence, 
\begin{align*}
A&\leq 4 q \int_{0}^{\infty}\int_{0}^{\rho^{-N-2+\kappa-\kappa_1}}\left(\lambda^q\tau|\{\mathbb{M}_1[\omega]>\lambda\}\cap \tilde{Q}_{\tau^{\frac{1}{-N-2+\kappa-\kappa_1}}}(x_0,t_0)\cap\Omega_T|\right)^{\frac{s}{q}}\frac{d\tau}{\tau}\frac{d\lambda}{\lambda}
\\&= 4 \int_{0}^{\rho^{-N-2+\kappa-\kappa_1}} ||\mathbb{M}_1[\omega]||^s_{L^{q,s}(\tilde{Q}_{\tau^{\frac{1}{-N-2+\kappa-\kappa_1}}}(x_0,t_0)\cap\Omega_T)}\tau^{\frac{s}{q}}\frac{d\tau}{\tau}
\\&\leq 4 \int_{0}^{\rho^{-N-2+\kappa-\kappa_1}} ||\mathbb{M}_1[\omega]||^s_{L^{q,s;\kappa}(\Omega_T)}\tau^{\frac{(N+2-\kappa)s}{(-N-2+\kappa-\kappa_1)q}}\tau^{\frac{s}{q}}\frac{d\tau}{\tau}
\\&\lesssim ||\mathbb{M}_1[\omega]||^s_{L^{q,s;\kappa}(\Omega_T)}\rho^{-\frac{s \kappa_1}{q}}.
\end{align*} 
Case 2: $s>q$. Using the Minkowski inequality, yields 
\begin{align*}
A&\lesssim \left(\int_{0}^{\rho^{-N-2+\kappa-\kappa_1}} \left(\int_{0}^{\infty}\left(\lambda^q|\{\mathbb{M}_1[\omega]>\lambda\}\cap \tilde{Q}_{\tau^{\frac{1}{-N-2+\kappa-\kappa_1}}}(x_0,t_0)\cap\Omega_T|\right)^{\frac{s}{q}}\frac{d\lambda}{\lambda}\right)^{\frac{q}{s}}d\tau\right)^{\frac{s}{q}}
\\&\lesssim \left(\int_{0}^{\rho^{-N-2+\kappa-\kappa_1}} \left(||\mathbb{M}_1[\omega]||^s_{L^{q,s;\kappa}(\Omega_T)}\tau^{\frac{(N+2-\kappa)s}{(-N-2+\kappa-\kappa_1)q}}\right)^{\frac{q}{s}}d\tau\right)^{\frac{s}{q}}
\\&=  ||\mathbb{M}_1[\omega]||^s_{L^{q,s;\kappa}(\Omega_T)}\rho^{-\frac{s \kappa_1}{q}}.
\end{align*}                   
Therefore, we always have 
$$
A\lesssim ||\mathbb{M}_1[\omega]||^s_{L^{q,s;\kappa}(\Omega_T)}\rho^{-\frac{s \kappa_1}{q}}.$$
which implies  \eqref{5hh190320144} from \eqref{5hh190320145}. \\
Similarly, we obtain estimate \eqref{5hh050520142} by adapting
$$
                             w(x,t)=\min\{\rho^{-N+\vartheta-\vartheta_1},|x-x_0|^{-N+\vartheta-\vartheta_1}\}\in A_{\infty}$$in above argument,   
                  where $0<\vartheta_1<\vartheta$, $x_0\in\Omega$ and $0<\rho<\text{diam}(\Omega)$  and $[w]_{A_\infty}$ is independent of $ x_0$ and $\rho$. \\
                  Next, to obtain \eqref{5hh050520141} we need to show that for any ball $B_\rho\subset\mathbb{R}^N$                                                      
                 \begin{equation}\label{5hh050520144}                                
                                                                                                                                \left(\int_0^T|\text{osc}_{B_\rho\cap\overline{\Omega}}u(t)|^qdt\right)^{\frac{1}{q}}\lesssim \rho^{1-\frac{\vartheta}{q}}|||\nabla u|||_{L_{**}^{q;\vartheta}(\Omega_T)}.
                                                                                                                                    \end{equation} 
  Since the extension of $u$ over $(\Omega_T)^c$ is zero and $u\in L^1(0,T,W^{1,1}_0(\Omega))$ thus we have for a.e $t\in (0,T)$, $u(.,t)\in W^{1,1}(\mathbb{R}^N)$. Applying \cite[Lemma 7.16]{55GiTr} to a ball $B_\rho\subset\mathbb{R}^N$, we get for a.e $t\in (0,T)$ and $x\in B_\rho$
$$
 |u(x,t)-u_{B\rho}(t)|\lesssim\int_{B_\rho}\frac{|\nabla u(y,t)|}{|x-y|^{N-1}}dy \lesssim \int_{0}^{3\rho}\frac{\int_{B_r(x)}|\nabla u(y,t)|dy}{r^{N-1}}\frac{dr}{r}.$$
Here $u_{B\rho}(t)$ is the average of $u(.,t)$ over $B_\rho$, i.e $u_{B\rho}(t)=\frac{1}{|B_\rho|}\int_{B_\rho}u(x,t)dx$.  \\Using the Minkowski and the H\"older inequality, we observe that for a.e  $x\in B_\rho$ 
\begin{align*}
\left(\int_0^T|u(x,t)-u_{B\rho}(t)|^qdt\right)^{\frac{1}{q}}&\lesssim\left(\int_{0}^{T}\left(\int_{0}^{3\rho}\frac{\int_{B_r(x)}|\nabla u(y,t)|dy}{r^{N-1}}\frac{dr}{r} \right)^qdt\right)^{\frac{1}{q}}
\\&\lesssim \int_{0}^{3\rho}\int_{B_r(x)}\left(\int_{0}^{T}|\nabla u(y,t)|^qdt\right)^{\frac{1}{q}}dy\frac{dr}{r^{N}}
\\&\lesssim\int_{0}^{3\rho}\left(\int_{B_r(x)}\int_{0}^{T}|\nabla u(y,t)|^qdtdy\right)^{\frac{1}{q}}|B_r(x)|^{\frac{q-1}{q}}\frac{dr}{r^{N}}
\\&\lesssim\int_{0}^{3\rho}r^{\frac{N-\vartheta}{q}}r^{\frac{N(q-1)}{q}}\frac{dr}{r^{N}}|||\nabla u|||_{L_{**}^{q;\vartheta}(\Omega_T)}
\\&\lesssim\rho^{1-\frac{\vartheta}{q}}|||\nabla u|||_{L_{**}^{q;\vartheta}(\Omega_T)}.
\end{align*}                                                       
 Therefore, we get \eqref{5hh050520144}. The proof is complete.                               \end{proof} \vspace{0.4cm}\\
In the following, we give some estimates for the norm of $\mathbb{M}_1[\omega]$ in $L_{*}^{q;\kappa}(\mathbb{R}^{N+1})$ and in $L_{**}^{q;\vartheta}(\mathbb{R}^{N+1})$ 
\begin{proposition}\label{5hh050520145} Let $1<\kappa\leq N+2$, $0<\vartheta\leq N$ and $q,q_1>1$. Suppose that  $\mu\in \mathfrak{M}^+(\mathbb{R}^{N+1})$. Then $\mathbb{M}_1[\mu]\leq 2^{N+2}\mathbb{I}_1[\mu]$ and 
	\begin{description}
		\item[a.]  If $q>\frac{\kappa}{\kappa-1}$ then  
		\begin{equation}\label{5hh050520142}                                 
			||\mathbb{I}_1[\mu]||_{L_{*}^{q;\kappa}(\mathbb{R}^{N+1})}\lesssim_{q,\kappa} ||\mu||_{L_{*}^{\frac{q\kappa}{q+\kappa};\kappa}(\mathbb{R}^{N+1})}.                                                                                                                 \end{equation}  

		\item[b.]If $1<q<2$ then  \begin{align}\label{5hh050520146}
			||\mathbb{I}_1[\mu](x,.)||_{L^q(\mathbb{R})}\leq \mathbf{I}_{\frac{2}{q}-1}[\mu_1](x),
		\end{align}
		where $\mu_1$ is a nonnegative radon measure in $\mathbb{R}^N$ defined by $\mu_1(A)=\mu(A\times\mathbb{R})$ for every  Borel set $A\subset\mathbb{R}^N$. In particular, 
		\begin{align}
			\label{5hh050520148}||\mathbb{I}_1[\mu]||_{L_{**}^{q;\vartheta}(\mathbb{R}^{N+1})}\leq ||\mathbf{I}_{\frac{2}{q}-1}[\mu_1]||_{L^{q;\vartheta}(\mathbb{R}^{N})},
		\end{align} and if $\vartheta>\frac{2-q}{q-1}$ there holds
		\begin{align}
			\label{5hh0505201410}||\mathbb{I}_1[\mu]||_{L_{**}^{q;\vartheta}(\mathbb{R}^{N+1})}\lesssim_{q,\vartheta}||\mu_1||_{L^{\frac{\vartheta q}{\vartheta+2-q};\vartheta}(\mathbb{R}^{N})}.
		\end{align}
		\item[c.]If $\frac{2q}{q+2}<q_1\leq q$ then   \begin{align}\label{5hh050520147}
			||\mathbb{I}_1[\mu](x,.)||_{L^q(\mathbb{R})}\leq\mathbf{I}_{\frac{2}{q}+1-\frac{2}{q_1}}[\mu_2](x),
		\end{align}
		where $d\mu_2(x)=||\mu(x,.)||_{L^{q_1}(\mathbb{R})}dx$. In particular, 
		\begin{align}
			\label{5hh050520149}||\mathbb{I}_1[\mu]||_{L_{**}^{q;\vartheta}(\mathbb{R}^{N+1})}\leq ||\mathbf{I}_{\frac{2}{q}+1-\frac{2}{q_1}}[\mu_2]||_{L^{q;\vartheta}(\mathbb{R}^{N})},
		\end{align} and if $\vartheta>\frac{1}{q-1}\left(2+q-\frac{2q}{q_1}\right)$ there holds
		\begin{align}
			\label{5hh0505201411}||\mathbb{I}_1[\mu]||_{L_{**}^{q;\vartheta}(\mathbb{R}^{N+1})}\lesssim_{q,\vartheta} ||\mu_2||_{L^{\frac{\vartheta qq_1}{(\vartheta+2+q)q_1-2q};\vartheta}(\mathbb{R}^{N})}=||\mu||_{L^{\frac{\vartheta qq_1}{(\vartheta+2+q)q_1-2q};\vartheta}(\mathbb{R}^{N}, L^{q_1}(\mathbb{R}))}.
		\end{align}
	\end{description}
	
\end{proposition}  
\begin{remark} Let $1<q<2$, $0<\vartheta\leq N$ and  $\sigma\in\mathfrak{M}(\mathbb{R}^N)$. From  \eqref{5hh050520148} and \eqref{5hh0505201410} in Proposition \ref{5hh050520145} we assert that 
	\begin{align*}
		||\mathbb{I}_1[|\sigma|\otimes\delta_{\{t=0\}}]||_{L_{**}^{q;\vartheta}(\mathbb{R}^{N+1})}\leq ||\mathbf{I}_{\frac{2}{q}-1}[|\sigma|]||_{L^{q;\vartheta}(\mathbb{R}^{N})},
	\end{align*}
	and 
	\begin{align*}
		||\mathbb{I}_1[|\sigma|\otimes\delta_{\{t=0\}}]||_{L_{**}^{q;\vartheta}(\mathbb{R}^{N+1})}\lesssim_{q,\vartheta}||\sigma||_{L^{\frac{\vartheta q}{\vartheta+2-q};\vartheta}(\mathbb{R}^{N})}~\text{ if } \vartheta>\frac{2-q}{q-1}.
	\end{align*}
	Furthermore, from preceding inequality and \eqref{5hh0505201411} in Proposition \ref{5hh050520145}  we can state that
	\begin{align*}
		||\mathbb{I}_1[|\sigma|\otimes\delta_{\{t=0\}}+|\mu|]||_{L_{**}^{q;\vartheta}(\mathbb{R}^{N+1})}\lesssim_{q,\vartheta}||\sigma||_{L^{\frac{\vartheta q}{\vartheta+2-q};\vartheta}(\mathbb{R}^{N})}+||\mu||_{L^{\frac{\vartheta qq_1}{(\vartheta+2+q)q_1-2q};\vartheta}(\mathbb{R}^{N}, L^{q_1}(\mathbb{R}))},
	\end{align*}
	provided \begin{align*}
		& 1<q_1\leq q<2,\\& \max\left\{\frac{2-q}{q-1},\frac{1}{q-1}\left(2+q-\frac{2q}{q_1}\right)\right\}<\vartheta\leq N.
	\end{align*}
 Here 
	\begin{align*}
		||\mu||_{L^{q_2;\vartheta}(\mathbb{R}^{N}, L^{q_1}(\mathbb{R}))}=\sup_{\rho>0, x\in \mathbb{R}^N}\rho^{\frac{\vartheta-N}{q_2}}\left(\int_{B_\rho(x)}\left(\int_{\mathbb{R}}|\mu(y,t)|^{q_1}dt\right)^{\frac{q_2}{q_1}}dy\right)^{\frac{1}{q_2}},
	\end{align*}
	with $q_2=\frac{\vartheta qq_1}{(\vartheta+2+q)q_1-2q}$.
\end{remark}     
 \begin{proof}[Proof of Proposition \ref{5hh050520145}] Clearly, estimate 
 \eqref{5hh050520142} is followed by \eqref{5hh190320142} in Proposition \ref{5hh190320148}. We want to emphasize that  
 almost every estimates in this proof will be used the Minkowski inequality.
For a ball $B_\rho\subset\mathbb{R}^N$, we have for a.e $x\in\mathbb{R}^N$
\begin{align}\nonumber
||\mathbb{I}_1[\mu](x,.)||_{L^q(\mathbb{R})}&=\left(\int_{-\infty}^{+\infty}\left(\int_{0}^{\infty}\frac{\mu(\tilde{Q}_r(x,t))}{r^{N+1}}\frac{dr}{r}\right)^qdt\right)^{\frac{1}{q}}
\\&\leq \int_{0}^{\infty}\left(\int_{-\infty}^{+\infty}(\mu(\tilde{Q}_r(x,t)))^qdt\right)^{\frac{1}{q}}\frac{dr}{r^{N+2}}.\label{5hh0505201412}
\end{align}
Now, we need to estimate $\left(\int_{-\infty}^{+\infty}(\mu(\tilde{Q}_r(x,t)))^qdt\right)^{\frac{1}{q}}$. \\
\textbf{b.} We have 
\begin{align*}
\left(\int_{-\infty}^{+\infty}(\mu(\tilde{Q}_r(x,t)))^qdt\right)^{\frac{1}{q}}&=\left(\int_{-\infty}^{+\infty}\left(\int_{\mathbb{R}^{N+1}}\chi_{\tilde{Q}_r(x,t)}(x_1,t_1)d\mu(x_1,t_1)\right)^qdt\right)^{\frac{1}{q}}\\&\leq \int_{\mathbb{R}^{N+1}}\left(\int_{-\infty}^{+\infty}\chi_{\tilde{Q}_r(x,t)}(x_1,t_1)dt\right)^{\frac{1}{q}}d\mu(x_1,t_1)
\\&=r^{\frac{2}{q}}\mu_1(B_r(x)).
\end{align*}
Combining this with \eqref{5hh0505201412} we obtain \eqref{5hh050520146} and \eqref{5hh0505201410}.\\ Thus,   we also assert \eqref{5hh0505201410} from \cite[Theorem 3.1 ]{55Ad}.\\
\textbf{c.} Set $d\mu_{2}(x)=||\mu(x,.)||_{L^{q_1}(\mathbb{R})}dx$. Using H\"older's inequality, 
\begin{align*}
\mu(\tilde{Q}_r(x,t))\leq r^{\frac{2(q_1-1)}{q_1}}\int_{B_r(x)}\left(\int_{t-\frac{\rho^2}{2}}^{t+\frac{\rho^2}{2}}(\mu(x_1,t_1))^{q_1}dt_1\right)^{\frac{1}{q_1}}dx_1.
\end{align*}
Leads to 
\begin{align*}
\left(\int_{-\infty}^{+\infty}(\mu(\tilde{Q}_r(x,t)))^qdt\right)^{\frac{1}{q}}&\leq r^{\frac{2(q_1-1)}{q_1}}\int_{B_r(x)}\left(\int_{-\infty}^{+\infty}\left(\int_{t-\frac{\rho^2}{2}}^{t+\frac{\rho^2}{2}}(\mu(x_1,t_1))^{q_1}dt_1\right)^{\frac{q}{q_1}}dt\right)^{\frac{1}{q}}dx_1.
\end{align*}
Note that 
$$\left(\int_{-\infty}^{+\infty}\left(\int_{t-\frac{\rho^2}{2}}^{t+\frac{\rho^2}{2}}(\mu(x_1,t_1))^{q_1}dt_1\right)^{\frac{q}{q_1}}dt\right)^{\frac{q_1}{q}}\leq \rho^{\frac{2q_1}{q}}\int_{-\infty}^{+\infty}(w(x_1,t_1))^{q_1}dt_1.$$
Hence 
$$
\left(\int_{-\infty}^{+\infty}(\mu(\tilde{Q}_r(x,t)))^qdt\right)^{\frac{1}{q}}\leq  r^{\frac{2(q_1-1)}{q_1}+\frac{2}{q}}\int_{B_r(x)}||\mu(x_1,.)||_{L^{q_1}(\mathbb{R})}dx_1=r^{\frac{2(q_1-1)}{q_1}+\frac{2}{q}}\mu_{2}(B_r(x)).$$
Consequently, since \eqref{5hh0505201412} we derive  \eqref{5hh050520147} and \eqref{5hh050520149}.\\ We also obtain  \eqref{5hh0505201411} from \cite[Theorem 3.1 ]{55Ad}.
 \end{proof}

                                   \subsection{Global estimates in $\mathbb{R}^N\times (0,\infty)$ and $\mathbb{R}^{N+1}$}
                                 Now, we prove Theorem  \ref{5hh040420149} and   \ref{5hh040420145}.\\                   \begin{proof}[Proof of Theorem \ref{5hh040420149} and Theorem \ref{5hh040420145}] 
                                   For any $n\geq 1$, it is easy to see that 
                                   \begin{description}
                                   \item[i.]$\mathbb{R}^N\backslash B_n(0)$ satisfies a uniformly $2-$thick condition with constants $c_0=\frac{\text{Cap}_p(B_{1/4}(z_0),B_2(0))}{\text{Cap}_p(B_{1}(0),B_2(0))}$, $z_0=(1/2,0,...,0)\in\mathbb{R}^N$ and $r_0=n$.
                                   \item[ii.] for any $\delta\in (0,1)$, $B_n(0)$ is a $(\delta, 2n\delta)-$ Reifenberg flat domain. 
                                   \item[iii.] $[\mathcal{A}]_{s_0}^n\leq [\mathcal{A}]_{s_0}^\infty$. 
                                   \end{description}
                                Assume that $||\mathbb{M}_1[|\omega|]||_{L^{p,s}(\mathbb{R}^{N+1}                                                                )}<\infty$. Thus by Remark \ref{5hh0404201410} we have  
                                   \begin{align}\label{5hh0404201411}
                                   \mathbb{I}_2[|\omega|](x,t)<\infty~\text{ for a.e }(x,t)\in \mathbb{R}^{N+1}.
                                   \end{align}                                   
                                   In view of the proof of the Theorem \ref{5hh1203201417} and applying Theorem \ref{5hh0701201411} to $B_n(0)\times (-n^2,n^2)$ and with data $\chi_{B_{n-1}(0)\times (-(n-1)^2,(n-1)^2)}\omega$ for any $n\geq 2$, there exists a sequence of renormalized solutions $\{u_n\}$ of 
                                    \begin{equation*}                  \left\{
                                                     \begin{array}
                                                     [c]{l}%
                                                     {(u_{n})_t}-\operatorname{div}(A(x,t,\nabla u_n))=\chi_{B_{n-1}(0)\times (-(n-1)^2,(n-1)^2)}\omega~\text{in }B_n(0)\times (-n^2,n^2),\\
                                                     u_n=0~~~\text{on}~\partial B_n(0)\times (-n^2,n^2),\\
                                                     {u}_{n}(-n^2)=0~~\text{in }B_n(0),
                                                     \end{array}
                                                     \right.  
                                                     \end{equation*} 
                                   converging to a distributional solution $u$ in $L^1_{\text{loc}}(\mathbb{R}; W^{1,1}_{\text{loc}}(\mathbb{R}^N))$ of \ref{5hhparabolic2'} with data $\mu=\omega$ such that     
$$|||\nabla u_n|||_{L^{p,s}(B_n(0)\times (-n^2,n^2))}\lesssim ||\mathbb{M}_1[|\omega|]||_{L^{p,s}(\mathbb{R}^{N+1})}.$$
                                             Here the constant $c$ does not depend on $n$ since $\frac{T_0}{r_0}=\frac{2n+(1+n^2)^{1/2}}{n}\approx 1$. \\
                                   Using Fatou's Lemma, we get estimate \eqref{5hh040420141}.  \\                                
                                   As above, we also obtain \eqref{5hh040420143}.
                                   And similarly, we can prove Theorem  \ref{5hh040420145} by this way. The proof is complete.                              
                                   \end{proof}        
 \begin{remark}[sharpness]The inequality \eqref{5hh040420146} is optimal in the following  sense : 
 \begin{align}\label{5hh070420141}
|||\nabla \mathcal{H}_2|*|\omega|||_{L^q(\mathbb{R}^N\times (0,\infty))}\sim ||\mathbb{M}_1[|\omega|]||_{L^q(\mathbb{R}^{N+1})}.
 \end{align}
 Indeed, we have 
 \begin{align*}
 \nabla \mathcal{H}_2(x,t)=-\frac{C}{2}\frac{\chi_{(0,\infty)}(t)}{t^{(N+1)/2}}\exp(-\frac{|x|^2}{4t})\frac{x}{\sqrt{t}},
 \end{align*}
which leads to 
$$
      \frac{1 }{t^{\frac{N+1}{2}}}\chi_{t>0}\chi_{\frac{1}{2}\sqrt{t}\lesssim|x|\leq2 \sqrt{t}}\lesssim |\nabla\mathcal{H}_\alpha(x,t)|\lesssim \frac{1}{\max\{|x|,\sqrt{2|t|}\}^{N+1}}.$$
Immediately, we get 
$$
  \int_{0}^{\infty}\frac{\omega\left((B_r(x)\backslash B_{r/2}(x))\times (t-r^2,t-r^2/4)\right)}{r^{N+1}}\frac{dr}{r}
  \lesssim |\nabla \mathcal{H}_2|*|\omega| (x,t)\lesssim \mathbb{I}_1[\omega](x,t).$$  
Then,  Theorem \ref{5hh051120131} gives  the right-hand side inequality of \eqref{5hh070420141}. So, it is enough to show that 
\begin{equation}\label{5hh070420142}
A:=\int_{\mathbb{R}^{N+1}}\left(\int_{0}^{\infty}\frac{\omega\left((B_r(x)\backslash B_{r/2}(x))\times (t-r^2,t-r^2/4)\right)}{r^{N+1}}\frac{dr}{r}\right)^qdxdt
\gtrsim||\mathbb{M}_1[\omega]||_{L^q(\mathbb{R}^{N+1})}^q.
\end{equation}
To do this, we take $r_k=(3/2)^k$ for $k\in\mathbb{Z}$, 
\begin{align*}
&\left(\int_{0}^{\infty}\frac{\omega\left((B_r(x)\backslash B_{r/2}(x))\times (t-r^2,t-r^2/4)\right)}{r^{N+1}}\frac{dr}{r}\right)^q\\& ~~~~~~~\gtrsim\sum_{k=-\infty}^{\infty}\left(\frac{\omega\left((B_{r_k}(x)\backslash B_{3r_k/4}(x))\times(t-r_k^2,t-9 r_k^2/16)\right)}{r_k^{N+1}}\right)^q.
\end{align*}
We deduce that 
\begin{align*}
A\gtrsim\sum_{k=-\infty}^{\infty}\int_{\mathbb{R}^{N+1}}\left(\frac{\omega\left((B_{r_k}(x)\backslash B_{3r_k/4}(x))\times(t-r_k^2,t-9 r_k^2/16)\right)}{r_k^{N+1}}\right)^qdxdt.
\end{align*}
For any $k$, put $y=x+\frac{7}{8}r_k$ and $s=t-\frac{25}{32}r_k^2$, so $B_{r_k}(x)\backslash B_{3r_k/4}(x)\supset B_{r_k/8}(y)$ and 
\begin{align*}
&\int_{\mathbb{R}^{N+1}}\left(\frac{\omega\left((B_{r_k}(x)\backslash B_{3r_k/4}(x))\times(t-r_k^2,t-9 r_k^2/16)\right)}{r_k^{N+1}}\right)^qdxdt\\&~~~~~\geq \int_{\mathbb{R}^{N+1}}\left(\frac{\omega\left(B_{r_k/8}(y)\times(s-7r_k^2/32,t+7r_k^2/32)\right)}{r_k^{N+1}}\right)^qdyds.
\end{align*}
Consequently,
$$
A\gtrsim\int_{\mathbb{R}^{N+1}}\sum_{k=-\infty}^{\infty}\left(\frac{\omega\left(B_{r_k/8}(y)\times(s-7r_k^2/32,t+7r_k^2/32)\right)}{r_k^{N+1}}\right)^qdyds.$$
It follows \eqref{5hh070420142}. 
   
 \end{remark}
             \section{Quasilinear Riccati Type Parabolic Equations}
             \subsection{Quasilinear Riccati Type Parabolic Equation in $\Omega_T$}
             We provide below only the proof of Theorem \ref{5hh260320144},  \ref{5hh2410136}. The proof of Theorem \ref{5hh0701201413} can be proceeded by a similar argument.\\\\  
              \begin{proof}[Proof of Theorem \ref{5hh260320144}] Let $\{\mu_n\}\subset C_c^\infty(\Omega_T)$ be as in the proof of Theorem  \ref{5hh141013112}. We have $|\mu_n|(\Omega_T)\leq |\mu|(\Omega_T)$ for any $n\in\mathbb{N}$. Let $\sigma_{n}\in C_c^\infty(\Omega)$ converge to $\sigma$ in the narrow topology of measures and in $L^1(\Omega)$ if $\sigma\in L^1(\Omega)$ such that $||\sigma_{n}||_{L^1(\Omega)}\leq |\sigma|(\Omega)$.             
                                 For $n_0\in \mathbb{N}$, we prove that the problem \eqref{5hh0701201410} has a solution with data $\mu=\mu_{n_0}$ and $\sigma=\sigma_{n_0}$.
                                   Now we put 
                                   $$\mathbf{E}_\Lambda=\left\{u\in L^1(0,T,W^{1,1}_0(\Omega)):|||\nabla u|||_{L^{\frac{N+2}{N+1},\infty}(\Omega_T)}\leq \Lambda\right\},$$ 
                                          where $\Lambda>0$,  $L^{\frac{N+2}{N+1},\infty}(\Omega_T)$ is the Lorentz space with norm 
                                          \begin{equation*}
                                          ||f||_{L^{\frac{N+2}{N+1},\infty}(\Omega_T)}:=\mathop {\sup }\limits_{0<|D|<\infty }\left(|D|^{-\frac{1}{N+2}} \int_{D\cap\Omega_T} |f|\right).
                                          \end{equation*}     
                                   By Fatou's lemma, $\mathbf{E}_\Lambda$ is closed under the strong topology of $L^1(0,T,W^{1,1}_0(\Omega))$ and convex. \\
                                   We consider a map $S:\mathbf{E}_\Lambda\to \mathbf{E}_\Lambda$ defined for each $v\in \mathbf{E}_\Lambda$ by $S(v)=u$, where $u\in L^1(0,T,W^{1,1}_0(\Omega))$  is the unique solution of 
                  \begin{equation}\label{5hh270320142}\left\{ \begin{array}{l}
            {u_t} - \operatorname{div}\left( {A(x,t,\nabla u)} \right) = |\nabla v|^q+\mu_{n_0}~\text{ in }\Omega_T, \\ 
                 u = 0~~~\text{on }~~\partial\Omega\times (0,T) \\
              u(0)=\sigma_{n_0}. 
         \end{array} \right.\end{equation}
             By Remark \ref{5hh070420143}, we have              
             \begin{align*}
             |||\nabla u|||_{L^{\frac{N+2}{N+1},\infty}(\Omega_T)}\lesssim|||\nabla v|^q||_{L^1(\Omega_T)}+|\mu_{n_{0}}|(\Omega_T)+||\sigma_{n_0}||_{L^1(\Omega)}.
             \end{align*} It leads to 
             \begin{align*}
             |||\nabla u|||_{L^{\frac{N+2}{N+1},\infty}(\Omega_T)}&\lesssim |\Omega_T|^{1-\frac{q(N+1)}{N+2}}|||\nabla v|||_{L^{\frac{N+2}{N+1},\infty}(\Omega_T)}^q+|\mu|(\Omega_T)+|\sigma|(\Omega)
             \\&\lesssim |\Omega_T|^{1-\frac{q(N+1)}{N+2}}\Lambda^q+|\mu|(\Omega_T)+|\sigma|(\Omega).
             \end{align*} 
              Thus, we now suppose that $$|\Omega_T|^{-1+\frac{q'}{N+2}}\left(|\mu|(\Omega_T)+|\sigma|(\Omega)\right)\leq\varepsilon_0,$$ then $$|||\nabla u|||_{L^{\frac{N+2}{N+1},\infty}(\Omega_T)}\leq \Lambda:=c(|\mu|(\Omega)+|\sigma|(\Omega)),$$ 
                 for some $\varepsilon_0>0,c>0$.             So, $S$ is well defined.\\
                                     Now we show that $S$\textbf{ is continuous}.
                                              Let $\{v_n\}$ be a sequence in $\mathbf{E}_\Lambda$ such that $v_n$ converges strongly in $L^1(0,T,W^{1,1}_0(\Omega))$ to a function $v\in \mathbf{E}_\Lambda$. Set $u_n=S(v_n)$. We need to show that $u_n\to S(v)$ in $L^1(0,T,W^{1,1}_0(\Omega))$. We have 
                                               \begin{equation}\label{5hh270320144}\left\{ \begin{array}{l}
                                                                        {(u_n)_t} - \operatorname{div}\left( {A(x,t,\nabla u_n)} \right) = |\nabla v_n|^q+\mu_{n_0}~\text{ in }~\Omega_T, \\ 
                                                                        u_n = 0~~~\text{ on }~~\partial\Omega\times (0,T),\\
                                                                        u_n(0)=\sigma_{n_0} ~~\text{ in }~\Omega,
                                                                        \end{array} \right.\end{equation}
                                                                        satisfying
                                                  \begin{equation*}
                        |||\nabla u_n|||_{L^{\frac{N+2}{N+1},\infty}(\Omega_T)}\leq \Lambda,~~|||\nabla v_n|||_{L^{\frac{N+2}{N+1},\infty}(\Omega_T)}\leq \Lambda.
                                                                        \end{equation*}                                                           
                                                                        Thus, $|\nabla v_n|^q\to |\nabla v|^q$ in $L^1(\Omega_T)$. 
                                                                        Therefore, it is easy to see that we get $u_n\to S(v)$ in $L^1(0,T,W^{1,1}_0(\Omega))$  by Theorem \ref{5hhsta}.\\
                                     Next we show that $S$ \textbf{ is pre-compact}. 
                                              Indeed if $\{u_n\}=\{S(v_n)\}$ is a sequence in $S(\mathbf{E}_\Lambda)$. By Proposition \ref{5hhmun}, there exists a subsequence of $\{u_n\} $ converging to some $u$ in $L^1(0,T,W^{1,1}_0(\Omega))$. Consequently, 
                                      by the Schauder Fixed Point Theorem, $S$ has a fixed point on $\mathbf{E}_\Lambda$. So,  the problem \eqref{5hh0701201410} has a solution with data $\mu_{n_0},\sigma_{n_0}$.\medskip\\                        
                             Therefore, for any $n\in\mathbb{N}$, 
                                                   there exists a renormalized solution $u_n$ of 
                      \begin{equation}\label{5hh270320143}\left\{ \begin{array}{l}
                                                               {(u_n)_t} - \operatorname{div}\left( {A(x,t,\nabla u_n)} \right) = |\nabla u_n|^q+\mu_{n}~~\text{ in }~\Omega_T, \\ 
                                                               u = 0~~~\text{on }~~\partial\Omega\times (0,T), \\
                                                               u_n(0)=\sigma_{n},
                                                               \end{array} \right.\end{equation}
which satisfies                                                          
                                                               \begin{equation*}
                                                                |||\nabla u_n|||_{L^{\frac{N+2}{N+1},\infty}(\Omega_T)}\leq \Lambda.
                                                               \end{equation*}                     Thanks to Proposition \ref{5hhmun}, there exists a subsequence of $\{u_n\} $ converging to $u$ in $L^1(0,T,W^{1,1}_0(\Omega))$. So, $ |||\nabla u|||_{L^{\frac{N+2}{N+1},\infty}(\Omega_T)}\leq \Lambda$ and $|\nabla u_n|^q\to |\nabla u|^q$ in $L^1(\Omega)$ since $\{|\nabla u_n|^q\}$ is equi-integrable. It follows the result by Proposition \ref{5hhmun} and  Theorem \ref{5hhsta}. 
                                     \end{proof}\\\\                  
            \begin{proof}[Proof of Theorem \ref{5hh2410136}]Let $\{\mu_n\}\subset C_c^\infty(\Omega_T), \sigma_{n}\in C_c^\infty(\Omega)$ be as in the proof of Theorem  \ref{5hh141013112}. We have $|\mu_n|\leq \varphi_n*|\mu|, |\sigma_{n}|\leq \varphi_{1,n}*|\sigma|$ for any $n\in\mathbb{N}$, $\{\varphi_n\}$, $\{\varphi_{1,n}\}$ are  sequences of standard mollifiers in $\mathbb{R}^{N+1},\mathbb{R}^N$ respectively. Set $\omega_n=|\mu_n|+|\sigma_n|\otimes\delta_{\{t=0\}}$ and $\omega=|\mu|+|\sigma|\otimes\delta_{\{t=0\}}$.\medskip\\ 
         For $n_0\in\mathbb{N}$,      $\varepsilon>0$, we prove that the problem \eqref{5hh0701201410} has a solution with data $\mu=\mu_{n_0},\sigma=\sigma_{n_0}$.
       Now we put 
        $$\mathbf{E}=\left\{u\in L^q(0,T,W^{1,q}_0(\Omega)): \mathbb{I}_1[|\nabla u|^q\chi_{\Omega_{T}}]\leq \mathbb{I}_1[\omega]  \right\}.$$ 
       By Fatou's lemma, $\mathbf{E}_\Lambda$ is closed under the strong topology of $L^1(0,T,W^{1,1}_0(\Omega))$ and convex. \\
       We consider a map $S:\mathbf{E}\to \mathbf{E}$ defined for each $v\in \mathbf{E}$ by $S(v)=u$, where $u\in L^q(0,T,W^{1,q}_0(\Omega))$  is the unique solution of problem \eqref{5hh270320142}.        
        By Corollary  \ref{5hh2410131'} and \ref{keyre}, there exist   $\delta=\delta(N,\Lambda_1,\Lambda_2,q)\in (0,1)$ and  $s_0=s_0(N,\Lambda_1,\Lambda_2)>0$ such that $\Omega$ is $(\delta,R_0)$- Reifenberg flat domain  and $[\mathcal{A}]_{s_0}^{R_0}\le \delta$ for some $R_0$ we have
       \begin{equation}\label{Z2}
       	\mathbb{I}_1[|\nabla u|^q\chi_{\Omega_{T}}]\lesssim	\mathbb{I}_1[\mathbb{M}_1[|\nabla v|^q+\omega_{n_0}]^q\chi_{\Omega_{T}}],~~
       \end{equation}
   and 
   \begin{equation}\label{Z5}
   	\int_{\lambda_0}^{\infty}\lambda^{q-1}\left|\{\mathbb{M}(|\nabla u|)>\Lambda\lambda \}\cap \Omega_T\right|d\lambda\leq c	\int_{\lambda_0/c}^{\infty}\lambda^{q-1}\left|\{\mathbb{M}_1[|\nabla v|^q+\omega_{n_0}]>\lambda \}\cap \Omega_T\right| d\lambda
   \end{equation}
   for any $\lambda_0\geq 0$
                where $c=c(N,\Lambda_1,\Lambda_2,q,T_0/R_0,T_0)$.\\
                We will prove that if 
                \begin{equation}\label{Z3}
                	[\omega]_{\mathfrak{M}^{\mathcal{G}_1,q'}}\leq \varepsilon_0
                \end{equation}
                for some $\varepsilon_0>0$ small enough, then  $S$ is well defined.\\
           By Corollary \ref{5hh250320146},     
            \begin{equation}\label{Z4}
            	[\omega_{n_0}]_{\mathfrak{M}^{\mathcal{G}_1,q'}}\lesssim	[\omega]_{\mathfrak{M}^{\mathcal{G}_1,q'}}\lesssim \varepsilon_0.
            \end{equation}
             Thus, thanks to  Theorem \ref{5hh1410136}
             \begin{align*}
&\mathbb{I}_1^{2T_0,1}[\mathbb{I}_1^{2T_0,1}[\omega]^q]\lesssim\varepsilon_0^{q-1}\mathbb{I}_1^{2T_0,1}[\omega],\quad\quad\mathbb{I}_1^{2T_0,1}[\mathbb{I}_1^{2T_0,1}[\omega_{n_0}]^q]\lesssim\varepsilon_0^{q-1}\mathbb{I}_1^{2T_0,1}[\omega_{n_0}].
             \end{align*}  
Note that   $$\mathbb{I}_1^{2T_0,1}[\omega_{n_0}]\leq \varphi_{n_0}\star \mathbb{I}_1^{2T_0,1}[\omega]\overset{\eqref{5hh09262}}\lesssim \mathbb{I}_1^{2T_0,1}[\omega].$$        
         Thus, \eqref{Z2} and $\mathbb{I}_1[|\nabla v|^q\chi_{\Omega_{T}}]\leq \mathbb{I}_1[\omega]$ imply
         \begin{align*}
	\mathbb{I}_1[|\nabla u|^q\chi_{\Omega_{T}}]&\lesssim \mathbb{I}_1[\mathbb{I}_1[\omega]^q\chi_{\Omega_{T}}]+	\mathbb{I}_1[\mathbb{M}_1[\omega_{n_0}]^q\chi_{\Omega_{T}}]\\&\lesssim \mathbb{I}_1^{2T_0,1}[\mathbb{I}_1^{2T_0,1}[\omega]^q]+	\mathbb{I}_1^{2T_0,1}[\mathbb{I}_1^{2T_0,1}[\omega_{n_0}]^q]\\&\lesssim\varepsilon_0^{q-1} \mathbb{I}_1^{2T_0,1}[\omega]\\&\leq  \mathbb{I}_1[\omega]
         \end{align*}      
         for  $\varepsilon_0>0$ small enough. Hence, $S$ is well defined.\\Moreover, it follows from \eqref{Z5} that
          \begin{equation}\label{Z6}
         	\int_{\lambda_0}^{\infty}\lambda^{q-1}\left|\{\mathbb{M}(|\nabla S(v)|)>\lambda \}\cap \Omega_T\right|d\lambda\leq c	\int_{\lambda_0/c}^{\infty}\lambda^{q-1}\left|\{\mathbb{I}_1[\omega]>\lambda \}\cap \Omega_T\right| d\lambda
         \end{equation}
         for any $\lambda_0\geq 0$ and $v\in \mathbf{E}.$\medskip\\
         Now we assume \eqref{Z3}. We show that $S$\textbf{ is continuous}.
                  Let $\{v_n\}$ be a sequence in $\mathbf{E}$ such that $v_n$ converges strongly in $L^q(0,T,W^{1,q}_0(\Omega))$ to a function $v\in \mathbf{E}$. Set $u_n=S(v_n)$. We need to show that $u_n\to S(v)$ in $L^q(0,T,W^{1,q}_0(\Omega))$. Thanks to Proposition \ref{5hhmun}, $u_n\to S(v)$ a.e. By \eqref{Z6}, $\{|\nabla u_n|^q\}$ is equi-integrable.
                                            Thus,  $u_n\to S(v)$ in $L^q(0,T,W^{1,q}_0(\Omega))$.
Similarly, we also obtain that $S$\textbf{ is  pre-compact}.
\medskip\\                                             
         Thus, by the Schauder Fixed Point Theorem, $S$ has a fixed point on $\mathbf{E}$. Hence the problem \eqref{5hh0701201410} has a solution with data $\mu=\mu_{n_0},\sigma=\sigma_{n_0}$.                              This means, for any $n\in\mathbb{N}$,                                                          there exists a solution $u_n\in \mathbf{E}$ of problem \eqref{5hh270320143}. Since $u_n\in \mathbf{E}$ satisfies \eqref{Z6} with $S(v)=u_n$, so, $\{|\nabla u_n|^q\}$ is equi-integrable.
                By Proposition \ref{5hhmun}, there exists a subsequence of $\{u_n\} $ converging to some function $u$ in $L^1(0,T,W^{1,1}_0(\Omega))$. Thus, 
                             $|\nabla u_n|^q\to |\nabla u|^q$ in $L^1(\Omega)$. The results follow by Proposition \ref{5hhmun} and  Theorem \ref{5hhsta}.  The proof is complete.                     \end{proof}                                                         
\subsection{Quasilinear Riccati Type Parabolic Equation in $\mathbb{R}^N\times (0,\infty)$ and $\mathbb{R}^{N+1}$} 
In this subsection,  we provide the proofs of  Theorem   \ref{5hh0404201417}. We shall follow the same strategy as the proof of  Theorem \ref{5hh0404201412}. \\               
\begin{proof}[Proof of Theorem \ref{5hh0404201417}]Let $D_n=B_n(0)\times (-n^2,n^2)$,  $\mu_n=\varphi_n*(\chi_{D_{n-1}}\mu)$ for any $n\geq 2$. Here $\{\varphi_n\}$ is a  sequence of standard mollifiers in $\mathbb{R}^{N+1}$.  We have $\mu_n\in C_c^\infty (\mathbb{R}^{N+1})$ with $\text{supp}(\mu_n)\subset D_n$ and $\mu_n\to\mu$ weakly in $\mathfrak{M}(\mathbb{R}^{N+1})$.\\ 
	Assume that 
	 \begin{equation}\label{Z8}
		[\mu]_{\mathfrak{M}^{\mathcal{H}_1,q'}}\leq \varepsilon_0.
	\end{equation}
 By Corollary \ref{5hh250320146},     
$$
	[\mu_{n}]_{\mathfrak{M}^{\mathcal{H}_1,q'}}\lesssim	[\mu]_{\mathfrak{M}^{\mathcal{H}_1,q'}}\lesssim \varepsilon_0.$$
Therefore, thanks to Theorem \ref{5hh2410136}, there exist   $\delta=\delta(N,\Lambda_1,\Lambda_2,q)\in (0,1)$ and  $s_0=s_0(N,\Lambda_1,\Lambda_2)>0$ such that if  $[\mathcal{A}]_{s_0}^{\infty}\le \delta$, then    for any $n$, the problem 
\begin{equation*}\left\{ \begin{array}{l}
		{u_t} - \operatorname{div}\left( {A(t,x, \nabla u)} \right) = |\nabla u|^q+\mu_{n}~\text{in }~D_{n}, \\
		u=0~~~~\text{ on }~~\partial B_{n}(0)\times (-n^2,n^2), \\                                                   u(-n^2)=0 \quad \text{in }~~B_{n}(0),
	\end{array} \right.\end{equation*}
 has a solution $u_n\in \mathbf{E}_n$ with data $\mu=\mu_{n}$ provided $\varepsilon_0>0$ small enough. Here
	$$\mathbf{E}_n=\left\{v\in L^q(-n^2,n^2,W^{1,q}_0(B_{n}(0))): \mathbb{I}_1[|\nabla v|^q\chi_{D_n}]\leq \mathbb{I}_1[\mu]  \right\}.$$ 
Moreover, $u_n$ satisfies 
 $$\int_{K\cap D_n} |\nabla u_n|^{q}dxdt\lesssim \text{Cap}_{\mathcal{H}_1,q'}(K)~~\text{ for every compact subset } K\subset \mathbb{R}^{N+1}.
$$
By \eqref{Z6}, one has 
	 \begin{equation}\label{Z10}
		\int_{\lambda_0}^{\infty}\lambda^{q-1}\left|\{\mathbb{M}(|\nabla u_n|)>\lambda \}\cap D_n\right|d\lambda\leq c	\int_{\lambda_0/c}^{\infty}\lambda^{q-1}\left|\{\mathbb{I}_1[\mu]>\lambda \}\right| d\lambda,
	\end{equation}
for some $c>0$. 
So, $\{|\nabla u_n|^{q}\}_{n\geq n_0}$ is equi-integrable in $D_{n_0}$ for any $n_0>0$. \\
Hence, by Corollary \ref{5hh090420144}, there exists a subsequence of $\{u_n\}$ converging to a distributional solution $u$ of \eqref{5hh270420142} satisfying 
$
	\mathbb{I}_1[|\nabla u|^q]\leq \mathbb{I}_1[\mu]
$ and 
 $$\int_{K} |\nabla u_n|^{q}dxdt\lesssim \text{Cap}_{\mathcal{H}_1,q'}(K)~~\text{ for every compact subset } K\subset \mathbb{R}^{N+1}.
$$
 Furthermore, if $\text{supp}(\mu)\subset \mathbb{R}^N\times [0,\infty]$ and $\sigma\in \mathfrak{M}(\mathbb{R}^N)$, then $u_{n}=0$ in $B_{n}(0)\times (-n^2,-\frac{2}{n})$ where $\text{supp}(\omega_n)\subset \mathbb{R}^N\times (-\frac{2}{n},\infty)$. So, $u=0$ in $\mathbb{R}^N\times(-\infty,0)$.
 The proof is complete.
\end{proof}                      
            \section{Appendix}
            \begin{proof}[Proof of the Remark \ref{5hh020520141}]            
            For $\omega\in\mathfrak{M}^+(\mathbb{R}^{N+1})$, $0<\alpha<N+2$ if $\mathbb{I}_\alpha[\omega](x_0,t_0)<\infty$ for some $(x_0,t_0)\in\mathbb{R}^{N+1}$ then for any $0<\beta\leq \alpha$,  $\mathbb{I}_\beta[\omega]\in L^{s}_{\text{loc}}(\mathbb{R}^{N+1})$ for any $0<s<\frac{N+2}{N+2-\beta}$. Indeed, by Remark \ref{5hh120320141} we have     $\mathbb{I}_\alpha[\omega]\in L^{s}_{\text{loc}}(\mathbb{R}^{N+1})$ for any $0<s<\frac{N+2}{N+2-\beta}$.\\ Take $0<\beta\leq \alpha$ and $0<s<\frac{N+2}{N+2-\beta}$. For $R>0$, by Proposition \ref{5hh23101315} we have
            $\mathbb{I}_\beta[\chi_{\tilde{Q}_{2R}(0,0)}\omega]\in L^s_{\text{loc}}(\mathbb{R}^{N+1})$.
            Thus,
            \begin{align*}
            &\int_{\tilde{Q}_R(0,0)}\left(\mathbb{I}_\beta[\omega](x,t)\right)^sdxdt\\&~~~\lesssim \int_{\tilde{Q}_R(0,0)}\left(\mathbb{I}_\beta[\chi_{\tilde{Q}_{2R}(0,0)}\omega](x,t)\right)^sdxdt+  \int_{\tilde{Q}_R(0,0)}\left(\mathbb{I}_\beta[\chi_{\tilde{Q}_{2R}(0,0)^c}\omega](x,t)\right)^sdxdt
            \\&~~~\lesssim \int_{\tilde{Q}_R(0,0)}\left(\mathbb{I}_\beta[\chi_{\tilde{Q}_{2R}(0,0)}\omega](x,t)\right)^sdxdt+ R^{-s(\alpha-\beta)} \int_{\tilde{Q}_R(0,0)}\left(\mathbb{I}_\alpha[\omega](x,t)\right)^sdxdt\\&~~~<\infty.
            \end{align*}
            For $0<\beta<\alpha<N+2$, we consider 
            \begin{align*}
            \omega (x,t)=\sum_{k=4}^{\infty}\frac{a_k}{|\tilde{Q}_{k+1}(0,0)\backslash\tilde{Q}_{k}(0,0)|}\chi_{\tilde{Q}_{k+1}(0,0)\backslash\tilde{Q}_{k}(0,0)}(x,t),
            \end{align*}
            where $a_k=2^{n(N+2-\theta)}$ if $k=2^n$ and $a_k=0$ otherwise with $\theta\in (\beta,\alpha]$.\\
             It is easy to see that  $\mathbb{I}_\alpha[\omega]\equiv \infty$ and $\mathbb{I}_\beta[\omega]<\infty$ in $\mathbb{R}^{N+1}$.
            \end{proof}\medskip\\
            \begin{proof}[Proof of the Remark \ref{5hh0404201410}] For $\omega\in\mathfrak{M}^+(\mathbb{R}^{N+1})$, 
                           since $\mathbb{I}_2[\omega]\lesssim I_1[I_1[\omega]]$ thus:\\
                           If $\mathbb{I}_1[\omega]\in L^{s,\infty}(\mathbb{R}^{N+1})$ with $1<s<N+2$, then by Proposition \ref{5hh23101315} 
$$
                           ||\mathbb{I}_2[\omega]||_{L^{\frac{s(N+1)}{N+2-s},\infty}(\mathbb{R}^{N+1})}\lesssim ||\mathbb{I}_1[\omega]||_{L^{s,\infty}(\mathbb{R}^{N+1})}<\infty.$$
                           If $\mathbb{I}_1[\omega]\in L^{N+2,\infty}(\mathbb{R}^{N+1})$, then by Theorem \ref{5hh170220145}, one has 
$\mathbb{I}_2[\omega]\in L^{s_0}_{\text{loc}}(\mathbb{R}^{N+1})~~\forall ~s_0>1.$ So, $ \mathbb{I}_2[\omega]<\infty$ a.e in $\mathbb{R}^{N+1}$ if $\mathbb{I}_1[\omega]\in L^{s,\infty}(\mathbb{R}^{N+1})$ with $1<s\leq N+2$.\\
                           For $s>N+2$, there exists $\omega\in\mathfrak{M}^+(\mathbb{R}^{N+1})$ such that  $ \mathbb{I}_2[\omega]\equiv\infty$ in $\mathbb{R}^{N+1}$ and $\mathbb{I}_1[\omega]\in L^{s}(\mathbb{R}^{N+1})$. Indeed, consider
$$
                           \omega (x,t)=\sum_{k=1}^{\infty}\frac{k^{N-1}}{|\tilde{Q}_{k+1}(0,0)\backslash\tilde{Q}_{k}(0,0)|}\chi_{\tilde{Q}_{k+1}(0,0)\backslash\tilde{Q}_{k}(0,0)}(x,t).$$
                           We have for $(x,t)\in \mathbb{R}^{N+1}$ and $n_0\in\mathbb{N}$ with $n_0>\log_2(\max\{|x|,\sqrt{2|t|}\}) $
                           \begin{align*}
                           \mathbb{I}_2[\omega](x,t)&\gtrsim\sum_{n_0}^{\infty}\frac{\omega(\tilde{Q}_{2^{n}}(x,t))}{2^{nN}}
                           \gtrsim\sum_{n_0}^{\infty}\frac{\omega(\tilde{Q}_{2^{n-1}}(0,0))}{2^{nN}}
                           \\&\gtrsim\sum_{n_0}^{\infty}\frac{\sum_{k=1}^{2^{n-1}-1}k^{N-1}}{2^{nN}}
                           \gtrsim\sum_{k=1}^{\infty}\left(\sum_{n_0}^{\infty}\chi_{k\leq 2^{n-1}-1}\frac{1}{2^{nN}}\right)k^{N-1}
                           \\&\gtrsim\sum_{k=n_0}^{\infty}k^{-1}=\infty.
                           \end{align*}
                           On the other hand, for $s_1>\frac{N+2}{2}$
$$
                          \int_{\mathbb{R}^{N+1}}\omega^{s_1}dxdt=c\sum_{k=1}^{\infty}\frac{k^{s(N-1)}}{((k+1)^{N+2}-k^{N+2})^{s_1-1}}
                         \lesssim \sum_{k=1}^{\infty}\frac{k^{s_1(N-1)}}{k^{(s_1-1)(N+1)}}<\infty, $$
                           since $(s_1-1)(N+1)-s_1(N-1)>1$.
                           Thus, 
$$
                           ||\mathbb{I}_1[\omega]||_{L^s(\mathbb{R}^{N+1})}\lesssim ||\omega||_{L^{\frac{s(N+2)}{N+2+s}}(\mathbb{R}^{N+1})}<\infty.$$                                     
            \end{proof}\\
            \begin{proof}[Proof of the Proposition \ref{5hh1203201411}] We will use an idea in \cite{55Bi1,55Bi2} to  prove \ref{5hh110320141}. For $S^{\prime}\in W^{1,\infty}(\mathbb{R})$ with $S(0)=0$, $S^{\prime\prime}\geq 0$, $S'(\tau)\tau\geq 0$ for all $\tau\in\mathbb{R}$ and  $||S^{\prime}||_{L^\infty(\mathbb{R})}\leq 1$  we have           
            \begin{align*}
            &-\int_{D}\eta_t S(u)dxdt+\int_{D}{S^{\prime}(u)A(x,t,\nabla u)\nabla\eta}dxdt \\&~~~~~~~~~~~~~~~+\int_{D}{S^{\prime\prime
            }(u)\eta A(x,t,\nabla u)\nabla u}dxdt
           +\int_{D}S^{\prime}(u)\eta L(u)dxdt=\int_{D}S^{\prime}(u)\eta d\mu.
            \end{align*}
            Thus, 
            \begin{align*}
            &\Lambda_2\int_{D}S^{\prime\prime}(u)\eta |\nabla u|^2dxdt
                     \\&~~~~~~+\int_{D}S^{\prime}(u)\eta L(u)dxdt\leq \Lambda_1\int_{D}|\nabla u||\nabla \eta| dxdt+\int_{D}\eta d|\mu|+\int_{D}|\eta_t||u|dxdt.
            \end{align*}            
            \textbf{a.} We choose $S^{\prime} \equiv \varepsilon^{-1}T_\varepsilon $ for $\varepsilon>0$ and  let $\varepsilon\to 0$
             we will obtain            
             \begin{align}
             \int_{D}\eta|L(u)|dxdt\leq \Lambda_1\int_{D}|\nabla u||\nabla\eta|dxdt+\int_{D}\eta d|\mu|+\int_{D}|\eta_t||u|dxdt.\label{5hh110320142}
              \end{align}             
            \textbf{b.} for $S^{\prime}(u)=(1-(|u|+1)^{-\alpha})\text{sign}(u)$ for $\alpha>0$ then 
            \begin{align*}
            \int_{D} \frac{|\nabla u|^2}{\left(|u|+1\right)^{\alpha+1}}\eta dxdt\lesssim\int_{D}|\nabla u||\nabla\eta|dxdt+\int_{D}\eta d|\mu|+\int_{D}|\eta_t||u|dxdt.
            \end{align*}          
            Using H\"older's inequality, we have
             \begin{equation}
             \int_{D}|\nabla u||\nabla\eta|dxdt+\int_{D} \frac{|\nabla u|^2}{\left(|u|+1\right)^{\alpha+1}}\eta dxdt\lesssim B.\label{5hh110320143}
             \end{equation}
             \textbf{c.} for $S^{\prime}(u)=\frac{-k+\delta+|u|}{2\delta}\text{sign}(u)\chi_{k-\delta<|u|<k+\delta}+\text{sign}(u)\chi_{|u|\geq k+\delta}$, $0<\delta\leq k$ then
             \begin{equation}\label{5hh110320144}
            \frac{1}{2\delta}\int_{k-\delta<|u|<k+\delta} |\nabla u|^2
            \eta dxdt\lesssim\int_{D}|\nabla u||\nabla\eta|dxdt+\int_{D}\eta d|\mu|+\int_{D}|\eta_t||u|dxdt. \end{equation} 
            In particular, 
            \begin{equation}\label{5hh110320145}
            \frac{1}{k}\int_{D} |\nabla T_k(u)|^2
            \eta dxdt\lesssim\int_{D}|\nabla u||\nabla\eta|dxdt+\int_{D}\eta d|\mu|+\int_{D}|\eta_t||u|dxdt~~\forall k>0.
            \end{equation}
            Consequently,  we deduce \eqref{5hh110320141}  from \eqref{5hh110320142}-\eqref{5hh110320145}.\\
            Next, take  $\varphi \in  C^\infty_c(D)$ and $S^{\prime}(u)=\chi_{|u|\leq k-\delta}+\frac{k+\delta-|u|}{2\delta}\chi_{k-\delta<|u|<k+\delta}$, $S(0)=0$
             we have
             \begin{align*}
             &-\int_{D}\varphi_t\eta S(u)dxdt+\int_{D}S^{\prime}(u)\eta A(x,t,\nabla u)\nabla \varphi dxdt+\int_{D}S^{\prime}(u)\varphi A(x,t,\nabla u)\nabla \eta dxdt\\&~~~~-\frac{1}{2\delta}\int_{k-\delta<|u|<k+\delta}\text{sign}(u)\varphi\eta  A(x,t,\nabla u)\nabla u dxdt+\int_{D}S'(u)\varphi\eta L(u)dxdt\\&\quad\quad=\int_{D}S'(u)\varphi\eta d\mu +\int_{D}\varphi\eta_t S(u)dxdt.
             \end{align*} 
             Combining with \eqref{5hh110320142}, \eqref{5hh110320143} and \eqref{5hh110320144}, we get
$$
             -\int_{D}\varphi_t\eta S(u)dxdt+\int_{D}S^{\prime}(u)\eta A(x,t,\nabla u)\nabla \varphi dxdt\lesssim||\varphi||_{L^\infty(D)}B.
$$
            Letting $\delta\to0$, we get
$$
             -\int_{D}\varphi_t\eta T_k(u)dxdt+\int_{D}\eta A(x,t,\nabla T_k(u))\nabla \varphi dxdt\lesssim||\varphi||_{L^\infty(D)}B.
$$
           We take $\varphi=T_{\varepsilon}(T_k(u)-\langle T_{k}(w)\rangle_\nu)$,
            \begin{align*}
             &-\int_{D}\frac{\partial}{\partial t}\left(T_{\varepsilon}(T_k(u)-\langle T_{k}(w)\rangle_\nu)\right)\eta T_k(u)dxdt\\&~~~~~~~~~~~~+\int_{D}\eta A(x,t,\nabla T_k(u))\nabla T_{\varepsilon}(T_k(u)-\langle T_{k}(w)\rangle_\nu) dxdt\lesssim\varepsilon B.
             \end{align*}
             Using integration by part, we have
             \begin{align*}
             &-\int_{D}\frac{\partial}{\partial t}\left(T_{\varepsilon}(T_k(u)-\langle T_{k}(w)\rangle_\nu)\right)\eta T_k(u)dxdt
             =\frac{1}{2}\int_D(T_\varepsilon(T_k(u)-\langle T_{k}(w)\rangle_\nu))^2\eta_tdxdt
              \\&\quad\quad\quad\quad+\int_DT_{\varepsilon}(T_k(u)-\langle T_{k}(w)\rangle_\nu)\langle T_{k}(w)\rangle_\nu\eta_tdxdt
             \\&\quad\quad\quad\quad\quad+\nu\int_D\eta (T_k(w)-\langle T_{k}(w)\rangle_\nu) T_{\varepsilon}(T_k(u)-\langle T_{k}(w)\rangle_\nu)dxdt.
             \end{align*}
            Thus,   
             \begin{align*}
             &-\int_{D}\frac{\partial}{\partial t}\left(T_{\varepsilon}(T_k(u)-\langle T_{k}(w)\rangle_\nu)\right)\eta T_k(u)dxdt\\&~~~~~~~
             \geq -\varepsilon(1+k)||\eta_t||_{L^1(D)}+ \nu\int_D \eta\left( T_{k}(w)-\langle T_{k}(w)\rangle_\nu\right)T_\varepsilon(T_k(u)-\langle T_{k}(w)\rangle_\nu) dxdt,
             \end{align*}
            which follows \eqref{5hh110320147}.
            \end{proof}\\
            \begin{proof}[Proof of the proposition \ref{5hh1203201412}]  Let $S_k\in W^{2,\infty}(\mathbb{R})$ such that $S_k(z)=z$ if $|z|\le k$ and $S_k(z)=sign(z)2k$ if $|z|>2k$. For $m\in \mathbb{N}$, let $\eta_{m}$ be  the cut off function on $D_m$ with respect to $D_{m+1}$. It is easy to see that from the assumption  and Remark \ref{5hh120320145}, Proposition \ref{5hh120320146} we get  $U_{m,n}=\eta_{m}S_k(v_n)$, $v_n=u_n-h_n$
            \begin{align*}
            &\mathop {\textrm{sup} }\limits_{n\geq m+1} \left( ||\left(U_{m,n}\right)_t||_{L^{2}(-m^2,m^2,H^{-1}(B_m(0)))+L^1(D_m)}+||U_{m,n}||_{L^2(-m^2,m^2,H_0^1(B_m(0)))}\right.\\&~~~~~~~~~~~~~\left.+||u_n||_{L^1(D_m)}+||v_n||_{L^1(D_m)}\right) \leq M_m< \infty. 
            \end{align*} 
            Thus,  $\{U_{m,n}\}_{n\geq m+1}$ is relatively compact in $L^1(D_m)$. 
            On the other hand, for any $n_1,n_2\geq m+1$
            \begin{align*}
            &|\{|v_{n_1}-v_{n_2}|>\lambda\}\cap D_m|= |\{|\eta_{m}v_{n_1}-\eta_{m}v_{n_2}|>\lambda\}\cap D_m|
            \\&~~~\leq \frac{1}{k}\left(||v_{n_1}||_{L^1(D_m)}+||v_{n_2}||_{L^1(D_m)}\right)+\frac{1}{\lambda}||\eta_{m}S_k(v_{n_1})-\eta_{m}S_k(v_{n_2})||_{L^1(D_m)}
            \\&~~~\leq \frac{2M_m}{k}+\frac{1}{\lambda}||U_{m,n_1}-U_{m,n_2}||_{L^1(D_m)},
            \end{align*}
            and $h_n$ is convergent in $L^1_{\text{loc}}(\mathbb{R}^{N+1})$. 
            So, for any $m\in \mathbb{N}$ there is a subsequence of $\{ u_n\}$, still denoted by  $\{u_n\}$ such that $\{u_n\}$ is a Cauchy sequence (in measure) in $D_m$.
            Therefore, there is a subsequence of $\{u_n\}$ converging to some function $u$ a.e in $\mathbb{R}^{N+1}$. Clearly, $u\in L^1_{\text{loc}}(\mathbb{R}; W^{1,1}_{loc}(\mathbb{R}^N))$.
            Now, we prove that $\nabla {u_n} \to \nabla u$ a.e in $\mathbb{R}^{N+1}$.\\
            From 
            \eqref{5hh110320147} with $D=D_{m+2}$, $\eta=\eta_{m}$ and $T_k(w)=T_k(\eta_{m+1}u)$ we have 
            \begin{align}\nonumber
             &\nu\int_{D_{m+2}} \eta_{m}\left( T_{k}(\eta_{m+1}u)-\langle T_{k}(\eta_{m+1}u)\rangle_\nu\right)T_\varepsilon(T_k(u_n)-\langle T_{k}(\eta_{m+1}u)\rangle_\nu) dxdt\\&\nonumber~~~~~~+\int_{D_{m+2}}\eta_{m} A(x,t,\nabla T_k(u_n))\nabla T_{\varepsilon}(T_k(u_n)-\langle T_{k}(\eta_{m+1}u)\rangle_\nu) dxdt \\&~~~~~~~\lesssim\varepsilon(1+k)C(n,m)~~\forall ~n\geq m+2,\label{5hh120320147}
             \end{align}
             where 
            \begin{align*}
            C(n,m)&=||(\eta_m)_t (|u_n|+1)||_{L^1(D_{m+2})}\\&~~~+\int_{D_{m+2}}(|u_n|+1)^{q_0}\eta dxdt+\int_{D_{m+2}}|\nabla \eta_m^{1/q_1}|^{q_1}dxdt+\int_{D_{m+2}}\eta_m d|\mu_n|,
            \end{align*}            
             with $q_1<\frac{q_0-1}{2q_0}$. By the assumption, we verify  that the right hand side of \eqref{5hh120320147} is bounded by  $c\varepsilon$, where $c$ does not depend on $n$.   \\  
            Since $\{\eta_{m}T_k(u_n)\}_{n\geq m+2}$ is bounded in $L^2(-(m+2)^2,(m+2)^2;H^1_0(B_{m+2}(0)))$, thus  there is a subsequence of $\{u_n\}$, still denoted by  $\{u_n\}$ such that
$$
            \lim\limits_{n\to\infty}\int\limits_{ \left| {{T_k}({u_n}) - \langle T_{k}(\eta_{m+1}u)\rangle_\nu} \right| \le \varepsilon } \eta_{m} A(x,t,\nabla T_k(u))\nabla \left(T_k(u_n)-T_k(u)\right)dxdt=0.
$$
            Therefore, thanks to $u_n\to u$ a.e in $D_{m+2}$ and $\langle T_{k}(\eta_{m+1}u)\rangle_\nu\to T_k(\eta_{m+1}u)$ in $L^2(-(m+2)^2,(m+2)^2;H^1_0(B_{m+2}(0)))$, we get 
$$
            \limsup\limits_{\nu\to\infty}\limsup\limits_{n\to\infty}\int\limits_{ \left| {{T_k}({u_n}) - \langle T_{k}(\eta_{m+1}u)\rangle_\nu} \right| \le \varepsilon } \eta_{1,m} \Phi_{n,k} dxdt\lesssim\varepsilon ~\forall~ \varepsilon\in (0,1),
$$
            where
            $
            \Phi_{n,k}=\left(A(x,t,T_k(u_n))-A(x,t,T_k(u))\right)\nabla\left(T_k(u_n)-T_k(u)\right).$ \medskip\\           
            Using H\"older's inequality, one has
            \begin{align*}
            &\int_{D_{m+2}}\eta_{m}\Phi_{k,n}^{1/2}dxdt=\int_{D_{m+2}}\eta_{m}\Phi_{k,n}^{1/2}\chi_{ \left| {{T_k}({u_n}) - \langle T_{k}(\eta_{m+1}u)\rangle_\nu} \right| \le \varepsilon}dxdt\\&~~~+\int_{D_{m+2}}\eta_{m}\Phi_{k,n}^{1/2}\chi_{ \left| {{T_k}({u_n}) - \langle T_{k}(\eta_{m+1}u)\rangle_\nu} \right| > \varepsilon}dxdt
            \\&~~~\leq ||\eta_{1,m}||_{L^1(D_{m+2})}^{1/2}\left(\int\limits_{ \left| {{T_k}({u_n}) - \langle T_{k}(\eta_{m+1}u)\rangle_\nu} \right| \le \varepsilon } \eta_{m} \Phi_{n,k} dxdt\right)^{1/2}\\&~~~+|\{\left| {{T_k}({u_n}) - \langle T_{k}(\eta_{m+1}u)\rangle_\nu} \right| > \varepsilon\}\cap D_{m+1}|^{1/2}\left(\int_{D_{m+2}}\eta_{m}^2\Phi_{k,n}dxdt\right)^{1/2}\\&~~~=A_{n,\nu,\varepsilon}.
            \end{align*}
            Clearly, $\limsup\limits_{\varepsilon\to 0}\limsup\limits_{\nu\to\infty}\limsup\limits_{n\to\infty}A_{n,\nu,\varepsilon}=0$. It follows 
$$
            \limsup\limits_{n\to\infty}\int_{D_{m+2}}\eta_{m}\Phi_{k,n}^{1/2}dxdt=0.
$$
            Since $\Phi_{n,k}\geq\Lambda_2 |\nabla T_k(u_n)-\nabla T_k(u)|^2$, thus  $ \nabla T_k(u_n)\to \nabla T_k(u)$ in $L^1(D_m)$.\\
            Note that 
            \begin{align*}
            |\{|\nabla u_{n_1}-\nabla u_{n_2}|>\lambda\}\cap D_m|&\leq  \frac{1}{k}\left(|| u_{n_1}||_{L^1(D_m)}+|| u_{n_2}||_{L^1(D_m)}\right)\\&~+\frac{1}{\lambda}||\nabla T_k(u_{n_1})-\nabla T_k(u_{n_2})||_{L^1(D_m)}
            \\&\leq \frac{2M_m}{k}+\frac{1}{\lambda}|||\nabla T_k(u_{n_1})-\nabla T_k(u_{n_2})|||_{L^1(D_m)}.
            \end{align*}
            Thus, we can show that there is a subsequence of $\{\nabla u_n\}$ converging $\nabla u$ a.e in $\mathbb{R}^{N+1}$. The proof is complete. 
            \end{proof}

\end{document}